\documentclass[12pt,reqno]{amsart}
\expandafter\ifx\csname ProvidesPackage\endcsname\relax\toks0=\expandafter{%
\fi\ProvidesPackage{diagrams}[2024/11/20 v3.97 Paul Taylor's commutative
diagrams]%%
\toks0=\bgroup}%%
\ifx\diagram\isundefined\else\expandafter\endinput
\fi

\edef\cdrestoreat{%%
\noexpand\catcode`\noexpand\@=\the\catcode`\@%%
\noexpand\catcode`\noexpand\#=\the\catcode`\#%%
\noexpand\catcode`\noexpand\$=\the\catcode`\$%%
\noexpand\catcode`\noexpand\<=\the\catcode`\<%%
\noexpand\catcode`\noexpand\>=\the\catcode`\>%%
\noexpand\catcode`\noexpand\:=\the\catcode`\:%% Johannes L. Braams's
\noexpand\catcode`\noexpand\;=\the\catcode`\;%% Babel languages package
\noexpand\catcode`\noexpand\!=\the\catcode`\!%% makes these \active.
\noexpand\catcode`\noexpand\?=\the\catcode`\?%%
\noexpand\catcode`\noexpand\+=\the\catcode'53%% texinfo @+ is @outer@active
}\catcode`\@=11 \catcode`\#=6 \catcode`\<=12 \catcode`\>=12 \catcode'53=12
\catcode`\:=12 \catcode`\;=12 \catcode`\!=12 \catcode`\?=12

\ifx\diagram@help@messages\CD@qK\let\diagram@help@messages y\fi

\def\cdps@Rokicki#1{\special{ps:#1}}\let\cdps@dvips\cdps@Rokicki\let
\cdps@RadicalEye\cdps@Rokicki\let\CD@HB\cdps@Rokicki\let\CD@IK\cdps@Rokicki
\let\CD@HB\cdps@Rokicki%%
\def\cdps@Bechtolsheim#1{\special{dvitps: Literal "#1"}}%
\let\cdps@dvitps\cdps@Bechtolsheim\let\cdps@IntegratedComputerSystems
\cdps@Bechtolsheim%%
\def\cdps@Clark#1{\special{dvitops: inline #1}}%%
\let\cdps@dvitops\cdps@Clark%%
\let\cdps@OzTeX\empty\let\cdps@oztex\empty\let\cdps@Trevorrow\empty%%
\def\cdps@Coombes#1{\special{ps-string #1}}%%
\count@=\year\multiply\count@12 \advance\count@\month%%
\ifnum\count@>24382 %% (December 2029)
\ifnum\count@>24385 %% (March 2030)
\errhelp{You may press RETURN and carry on for the time being.}\errmessage{! remain
compatible with your files. Please get it NOW.}\fi\fi

\def\CD@DE{\global\let}\def\CD@RH{\outer\def}

{\escapechar\m@ne\xdef\CD@o{\string\{}\xdef\CD@yC{\string\}}%%
\catcode`\&=4 \CD@DE\CD@Q=&\xdef\CD@S{\string\&}%%ascii three ands
\catcode`\$=3 \CD@DE\CD@k=$\CD@DE\CD@ND=$%%ascii three dollars
\xdef\CD@nC{\string\$}\gdef\CD@LG{$$}%%ascii three dollars
\catcode`\_=8 \CD@DE\CD@lJ=_%%ascii two underlines
\obeylines\catcode`\^=7 \CD@DE\@super=^%%ascii two carets
\ifnum\newlinechar=10 \gdef\CD@uG{^^J}%%ascii two carets
\else\ifnum\newlinechar=13 \gdef\CD@uG{^^M}%%ascii two carets
\else\ifnum\newlinechar=-1 \gdef\CD@uG{^^J}%%ascii two carets
\else\CD@DE\CD@uG\space\expandafter\message{! input error: \noexpand
\newlinechar\space is ASCII \the\newlinechar, not LF=10 or CR=13.}%%
\fi\fi\fi}%%

\mathchardef\lessthan='30474 \mathchardef\greaterthan='30476

\ifx\tenln\CD@qK%%
\font\tenln=line10\relax%% Hit return - who needs diagonals?
\fi\ifx\tenlnw\CD@qK\ifx\tenln\nullfont\let\tenlnw\nullfont\else%%
\font\tenlnw=linew10\relax%% Hit return - who needs diagonals?
\fi\fi%%

\ifx\inputlineno\CD@qK\csname newcount\endcsname\inputlineno\inputlineno\m@ne
\fi

\def\cd@shouldnt#1{\CD@KB{* THIS (#1) SHOULD NEVER HAPPEN! *}}

\def\get@round@pair#1(#2,#3){#1{#2}{#3}}%%ascii round brackets ()
\def\get@square@arg#1[#2]{#1{#2}}%%ascii square brackets []
\def\CD@AE#1{\CD@PK\let\CD@DH\CD@@E\CD@@E#1,],}%%ascii sq brackets
\def\CD@m{[}\def\CD@RD{]}\def\commdiag#1{{\let\enddiagram\relax\diagram[]#1%
\enddiagram}}

\def\CD@BF{{\ifx\CD@EH[\aftergroup\get@square@arg\aftergroup\CD@YH\else
\aftergroup\CD@JH\fi}}%%
\def\CD@CF#1#2{\def\CD@YH{#1}\def\CD@JH{#2}\futurelet\CD@EH\CD@BF}

\def\CD@KK{|}

\def\CD@PB{%% arguments to maps inside diagrams
\tokcase\CD@DD:\CD@y\break@args;\catcase\@super:\upper@label;\catcase\CD@lJ:%
\lower@label;\tokcase{~}:\middle@label;%%ascii tilde
\tokcase<:\CD@iF;%%ascii less-than
\tokcase>:\CD@iI;%%ascii greater-than
\tokcase(:\CD@BC;%%)%ascii open round bracket
\tokcase[:\optional@;%%]%ascii open square bracket
\tokcase.:\CD@JJ;%%ascii dot 12.7.94
\catcase\space:\eat@space;\catcase\bgroup:\positional@;\default:\CD@@A
\break@args;\endswitch}

\def\switch@arg{%% arguments to horizontal maps outside diagrams
\catcase\@super:\upper@label;\catcase\CD@lJ:\lower@label;\tokcase[:\optional@
;%%]%ascii open square bracket
\tokcase.:\CD@JJ;%%ascii dot 12.7.94 % ; was : before 15.6.97
\catcase\space:\eat@space;\catcase\bgroup:\positional@;\tokcase{~}:%
\middle@label;%%ascii tilde (questionable!)
\default:\CD@y\break@args;\endswitch}

\let\CD@tJ\relax\ifx\protect\CD@qK\let\protect\relax\fi\ifx\AtEndDocument
\CD@qK\def\CD@PG{\CD@gB}\def\CD@GF#1#2{}\else\def\CD@PG#1{\edef\CD@CH{#1}%
\expandafter\CD@oC\CD@CH\CD@OD}\def\CD@oC#1\CD@OD{\AtEndDocument{\typeout{%
\CD@tA: #1}}}\def\CD@GF#1#2{\gdef#1{#2}\AtEndDocument{#1}}\fi\def\CD@ZA#1#2{%
\def#1{\CD@PG{#2\CD@mD\CD@W}\CD@DE#1\relax}}\def\CD@uF#1\repeat{\def\CD@p{#1}%
\CD@OF}\def\CD@OF{\CD@p\relax\expandafter\CD@OF\fi}\def\CD@sF#1\repeat{\def
\CD@q{#1}\CD@PF}\def\CD@PF{\CD@q\relax\expandafter\CD@PF\fi}\def\CD@tF#1%
\repeat{\def\CD@r{#1}\CD@QF}\def\CD@QF{\CD@r\relax\expandafter\CD@QF\fi}\def
\CD@tG#1#2#3{\def#2{\let#1\iftrue}\def#3{\let#1\iffalse}#3}\if y%
\diagram@help@messages\def\CD@rG#1#2{\csname newtoks\endcsname#1#1=%
\expandafter{\csname#2\endcsname}}\else\csname newtoks\endcsname\no@cd@help
\no@cd@help{See the manual}\def\CD@rG#1#2{\let#1\no@cd@help}\fi\chardef\CD@lF
=1 \chardef\CD@lI=2 \chardef\CD@MH=5 \chardef\CD@tH=6 \chardef\CD@sH=7
\chardef\CD@PC=9 \dimendef\CD@hI=2 \dimendef\CD@hF=3 \dimendef\CD@mF=4
\dimendef\CD@mI=5 \dimendef\CD@wJ=6 \dimendef\CD@tI=8 \dimendef\CD@sI=9
\skipdef\CD@uB=1 \skipdef\CD@NF=2 \skipdef\CD@tB=3 \skipdef\CD@ZE=4 \skipdef
\countdef\CD@JC=9 \countdef\CD@gD=8 \countdef\CD@A=7 \def\sdef#1#2{\def#1{#2}%
}\def\CD@L#1{\expandafter\aftergroup\csname#1\endcsname}\def\CD@RC#1{%
\expandafter\def\csname#1\endcsname}\def\CD@sD#1{\expandafter\gdef\csname#1%
\endcsname}\def\CD@vC#1{\expandafter\edef\csname#1\endcsname}\def\CD@nF#1#2{%
\expandafter\let\csname#1\expandafter\endcsname\csname#2\endcsname}\def\CD@EE
#1#2{\expandafter\CD@DE\csname#1\expandafter\endcsname\csname#2\endcsname}%
\def\CD@AK#1{\csname#1\endcsname}\def\CD@XJ#1{\expandafter\show\csname#1%
\endcsname}\def\CD@ZJ#1{\expandafter\showthe\csname#1\endcsname}\def\CD@WJ#1{%
\expandafter\showbox\csname#1\endcsname}\def\CD@tA{Commutative Diagram}\edef
\CD@kH{\string\par}\edef\CD@dC{\string\diagram}\edef\CD@HD{\string\enddiagram
}\edef\CD@EC{\string\\}\def\CD@eF{LaTeX}\ifx\@ignoretrue\CD@qK\expandafter
\CD@tG\csname if@ignore\endcsname\ignore@true\ignore@false\def\@ignoretrue{%
\global\ignore@true}\def\@ignorefalse{\global\ignore@false}\fi

\def\CD@g{{\ifnum0=`}\fi}\def\CD@wC{\ifnum0=`{\fi}}\def\catcase#1:{\ifcat
\noexpand\CD@EH#1\CD@tJ\expandafter\CD@kC\else\expandafter\CD@dJ\fi}\def
\tokcase#1:{\ifx\CD@EH#1\CD@tJ\expandafter\CD@kC\else\expandafter\CD@dJ\fi}%
\def\CD@kC#1;#2\endswitch{#1}\def\CD@dJ#1;{}\let\endswitch\relax\def\default:%
#1;#2\endswitch{#1}\ifx\at@\CD@qK\def\at@{@}\fi\edef\CD@P{\CD@o pt\CD@yC}%
\CD@RC{\CD@P>}#1>#2>{\CD@z\rTo\sp{#1}\sb{#2}\CD@z}\CD@RC{\CD@P<}#1<#2<{\CD@z
\lTo\sp{#1}\sb{#2}\CD@z}\CD@RC{\CD@P)}#1)#2){\CD@z\rTo\sp{#1}\sb{#2}\CD@z}%
\CD@RC{\CD@P(}#1(#2({\CD@z\lTo\sp{#1}\sb{#2}\CD@z}%%ascii brack
\def\CD@O{\def\endCD{\enddiagram}\CD@RC{\CD@P A}##1A##2A{\uTo<{##1}>{##2}%
\CD@z\CD@z}\CD@RC{\CD@P V}##1V##2V{\dTo<{##1}>{##2}\CD@z\CD@z}\CD@RC{\CD@P=}{%
\CD@z\hEq\CD@z}\CD@RC{\CD@P\CD@KK}{\vEq\CD@z\CD@z}\CD@RC{\CD@P\string\vert}{%
\vEq\CD@z\CD@z}\CD@RC{\CD@P.}{\CD@z\CD@z}\let\CD@z\CD@Q}\def\CD@IE{\let\tmp
\CD@JE\ifcat A\noexpand\CD@CH\else\ifcat=\noexpand\CD@CH\else\ifcat\relax
\noexpand\CD@CH\else\let\tmp\at@\fi\fi\fi\tmp}\def\CD@JE#1{\CD@nF{tmp}{\CD@P
\string#1}\ifx\tmp\relax\def\tmp{\at@#1}\fi\tmp}\def\CD@z{}\begingroup
\aftergroup\def\aftergroup\CD@T\aftergroup{\aftergroup\def\catcode`\@\active
\aftergroup @\endgroup{\futurelet\CD@CH\CD@IE}}\def\CD@uK#1{\CD@nF{x}{tex_#1:%
D}\ifx\x\relax\else\CD@nF{#1}{x}\fi} \def\CD@vK{\CD@uK{par}\CD@uK{everypar}%
\CD@uK{noindent}\CD@uK{parskip}\CD@uK{unskip}\CD@uK{hskip}\CD@uK{indent}}

\newcount\CD@uA\newcount\CD@vA\newcount\CD@wA\newcount\CD@xA\newdimen\CD@OA
\newdimen\CD@PA\CD@tG\CD@gE\CD@@A\CD@y\CD@tG\CD@hE\CD@EA\CD@BA\newdimen\CD@RA
\newdimen\CD@SA\newcount\CD@yA\newcount\CD@zA\newdimen\CD@QA\newbox\CD@DA
\CD@tG\CD@lE\CD@dA\CD@bA\newcount\CD@LH\newcount\CD@TC\def\CD@V#1#2{\ifdim#1<%
#2\relax#1=#2\relax\fi}\def\CD@X#1#2{\ifdim#1>#2\relax#1=#2\relax\fi}%
\newdimen\CD@XH\CD@XH=1sp \newdimen\CD@zC\CD@zC\z@\def\CD@cJ{\ifdim\CD@zC=1em%
\else\CD@nJ\fi}\def\CD@nJ{\CD@zC1em\def\CD@NC{\fontdimen8\textfont3 }\CD@@J
\CD@NJ\setbox0=\vbox{\CD@t\noindent\CD@k\null\penalty-9993\null\CD@ND\null
\endgraf\setbox0=\lastbox\unskip\unpenalty\setbox1=\lastbox\global\setbox
\CD@IG=\hbox{\unhbox0\unskip\unskip\unpenalty\setbox0=\lastbox}\global\setbox
\CD@KG=\hbox{\unhbox1\unskip\unpenalty\setbox1=\lastbox}}}\newdimen\CD@@I
\CD@@I=1true in \divide\CD@@I300 \def\CD@zH#1{\multiply#1\tw@\advance#1\ifnum
#1<\z@-\else+\fi\CD@@I\divide#1\tw@\divide#1\CD@@I\multiply#1\CD@@I}\def
\MapBreadth{\afterassignment\CD@gI\CD@LF}\newdimen\CD@LF\newdimen\CD@oI\def
\CD@gI{\CD@oI\CD@LF\CD@V\CD@@I{4\CD@XH}\CD@X\CD@@I\p@\CD@zH\CD@oI\ifdim\CD@LF
>\z@\CD@V\CD@oI\CD@@I\fi\CD@cJ}\def\CD@RJ#1{\CD@zD\count@\CD@@I#1\ifnum
\count@>\z@\divide\CD@@I\count@\fi\CD@gI\CD@NJ}\def\CD@NJ{\dimen@\CD@QC
\count@\dimen@\divide\count@5\divide\count@\CD@@I\edef\CD@OC{\the\count@}}%
\def\CD@AJ{\CD@QJ\z@}\def\CD@QJ#1{\CD@tI\axisheight\advance\CD@tI#1\relax
\advance\CD@tI-.5\CD@oI\CD@zH\CD@tI\CD@sI-\CD@tI\advance\CD@tI\CD@LF}%
\newdimen\CD@DC\CD@DC\z@\newdimen\CD@eJ\CD@eJ\z@\def\CD@CJ#1{\CD@sI#1\relax
\CD@tI\CD@sI\advance\CD@tI\CD@LF\relax}\def\horizhtdp{height\CD@tI depth%
\CD@sI}\def\axisheight{\fontdimen22\the\textfont\tw@}\def\script@axisheight{%
\fontdimen22\the\scriptfont\tw@}\def\ss@axisheight{\fontdimen22\the
\scriptscriptfont\tw@}\def\CD@NC{0.4pt}\def\CD@VK{\fontdimen3\textfont\z@}%
\def\CD@UK{\fontdimen3\textfont\z@}\newdimen\PileSpacing\newdimen\CD@nA\CD@nA
\z@\def\CD@RG{\ifincommdiag1.3em\else2em\fi}\newdimen\CD@YB\def\CellSize{%
\afterassignment\CD@kB\DiagramCellHeight}\newdimen\DiagramCellHeight
\DiagramCellHeight-\maxdimen\newdimen\DiagramCellWidth\DiagramCellWidth-%
\maxdimen\def\CD@kB{\DiagramCellWidth\DiagramCellHeight}\def\CD@QC{3em}%
\newdimen\MapShortFall\def\MapsAbut{\MapShortFall\z@\objectheight\z@
\objectwidth\z@}\newdimen\CD@iA\CD@iA\z@\CD@tG\CD@vE\CD@aB\CD@ZB\expandafter
\ifx\expandafter\iftrue\csname ifUglyObsoleteDiagrams\endcsname\CD@ZB\else
\CD@aB\fi\CD@nF{ifUglyObsoleteDiagrams}{relax}\newif\ifUglyObsoleteDiagrams
\def\CD@nK{\CD@aB\UglyObsoleteDiagramsfalse}\def\CD@oK{\CD@ZB
\UglyObsoleteDiagramstrue}\CD@vE\CD@nK\else\CD@oK\fi\CD@tG\CD@hK\CD@dK\CD@cK
\CD@cK\def\CD@sK{\ifx\pdfoutput\CD@qK\else\ifx\pdfoutput\relax\else\ifnum
\pdfoutput>\z@\CD@pK\fi\fi\fi} \def\CD@pK{\global\CD@dK\global\CD@aB\global
\UglyObsoleteDiagramsfalse\global\let\CD@n\empty\global\let\CD@oK\relax
\global\let\CD@pK\relax\global\let\CD@sK\relax}\def\CD@tK#1{}\ifx\pdfliteral\CD@qK\else\ifx\pdfliteral\relax\else\let\CD@tK
\pdfliteral\fi\fi\ifx\XeTeXrevision\CD@qK\else\ifx\XeTeXrevision\relax\else
\ifdim\XeTeXrevision pt<.996pt \expandafter\message{! XeTeX version
\XeTeXrevision\space does not support PDF literals, so diagonals will not work%
!}\else\expandafter\message{RUNNING UNDER XETEX \XeTeXrevision}\CD@pK\fi\fi
\fi\CD@sK\def\newarrowhead{\CD@mG h\CD@BG\CD@GG>}\def\newarrowtail{\CD@mG t%
\CD@BG\CD@GG>}\def\newarrowmiddle{\CD@mG m\CD@BG\hbox@maths\empty}\def
\newarrowfiller{\CD@mG f\CD@bE\CD@MK-}\def\CD@mG#1#2#3#4#5#6#7#8#9{\CD@RC{r#1%
:#5}{#2{#6}}\CD@RC{l#1:#5}{#2{#7}}\CD@RC{d#1:#5}{#3{#8}}\CD@RC{u#1:#5}{#3{#9}%
}\CD@vC{-#1:#5}{\expandafter\noexpand\csname-#1:#4\endcsname\noexpand\CD@MC}%
\CD@vC{+#1:#5}{\expandafter\noexpand\csname+#1:#4\endcsname\noexpand\CD@MC}}%
\CD@ZA\CD@MC{\CD@eF\space diagonals are used unless PostScript is set}\def
\defaultarrowhead#1{\edef\CD@sJ{#1}\CD@@J}\def\CD@@J{\CD@IJ\CD@sJ<>ht\CD@IJ
\CD@sJ<>th}\def\CD@IJ#1#2#3#4#5{\CD@HJ{r#4}{#3}{l#5}{#2}{r#4:#1}\CD@HJ{r#5}{#%
2}{l#4}{#3}{l#4:#1}\CD@HJ{d#4}{#3}{u#5}{#2}{d#4:#1}\CD@HJ{d#5}{#2}{u#4}{#3}{u%
#4:#1}}\def\CD@HJ#1#2#3#4#5{\begingroup\aftergroup\CD@GJ\CD@L{#1+:#2}\CD@L{#1%
:#2}\CD@L{#3:#4}\CD@L{#5}\endgroup}\def\CD@GJ#1#2#3#4{\csname newbox%
\endcsname#1\def#2{\copy#1}\def#3{\copy#1}\setbox#1=\box\voidb@x}\def\CD@sJ{}%
\CD@@J\def\CD@GJ#1#2#3#4{\setbox#1=#4}\ifx\tenln\nullfont\def\CD@sJ{vee}\else
\let\CD@sJ\CD@eF\fi\def\CD@xF#1#2#3{\begingroup\aftergroup\CD@wF\CD@L{#1#2:#3%
#3}\CD@L{#1#2:#3}\aftergroup\CD@yF\CD@L{#1#2:#3-#3}\CD@L{#1#2:#3}\endgroup}%
\def\CD@wF#1#2{\def#1{\hbox{\rlap{#2}\kern.4\CD@zC#2}}}\def\CD@yF#1#2{\def#1{%
\hbox{\rlap{#2}\kern.4\CD@zC#2\kern-.4\CD@zC}}}\CD@xF lh>\CD@xF rt>\CD@xF rh<%
\CD@xF rt<\def\CD@yF#1#2{\def#1{\hbox{\kern-.4\CD@zC\rlap{#2}\kern.4\CD@zC#2}%
}}\CD@xF rh>\CD@xF lh<\CD@xF lt>\CD@xF lt<\def\CD@wF#1#2{\def#1{\vbox{\vbox to%
\z@{#2\vss}\nointerlineskip\kern.4\CD@zC#2}}}\def\CD@yF#1#2{\def#1{\vbox{%
\vbox to\z@{#2\vss}\nointerlineskip\kern.4\CD@zC#2\kern-.4\CD@zC}}}\CD@xF uh>%
\CD@xF dt>\CD@xF dh<\CD@xF dt<\def\CD@yF#1#2{\def#1{\vbox{\kern-.4\CD@zC\vbox
to\z@{#2\vss}\nointerlineskip\kern.4\CD@zC#2}}}\CD@xF dh>\CD@xF ut>\CD@xF uh<%
\CD@xF ut<\def\CD@BG#1{\hbox{\mathsurround\z@\offinterlineskip\CD@k\mkern-1.5%
mu{#1}\mkern-1.5mu\CD@ND}}\def\hbox@maths#1{\hbox{\CD@k#1\CD@ND}}\def\CD@GG#1%
{\hbox to\CD@LF{\setbox0=\hbox{\offinterlineskip\mathsurround\z@\CD@k{#1}%
\CD@ND}\dimen0.5\wd0\advance\dimen0-.5\CD@oI\CD@zH{\dimen0}\kern-\dimen0%
\unhbox0\hss}}\def\CD@sB#1{\hbox to2\CD@LF{\hss\offinterlineskip\mathsurround
\z@\CD@k{#1}\CD@ND\hss}}\def\CD@vF#1{\hbox{\mathsurround\z@\CD@k{#1}\CD@ND}}%
\def\CD@bE#1{\hbox{\kern-.15\CD@zC\CD@k{#1}\CD@ND\kern-.15\CD@zC}}\def\CD@MK#%
1{\vbox{\offinterlineskip\kern-.2ex\CD@GG{#1}\kern-.2ex}}\def\@fillh{%
\xleaders\vrule\horizhtdp}\def\@fillv{\xleaders\hrule width\CD@LF}\CD@nF{rf:-%
}{@fillh}\CD@nF{lf:-}{@fillh}\CD@nF{df:-}{@fillv}\CD@nF{uf:-}{@fillv}\CD@nF{%
rh:}{null}\CD@nF{rm:}{null}\CD@nF{rt:}{null}\CD@nF{lh:}{null}\CD@nF{lm:}{null%
}\CD@nF{lt:}{null}\CD@nF{dh:}{null}\CD@nF{dm:}{null}\CD@nF{dt:}{null}\CD@nF{%
uh:}{null}\CD@nF{um:}{null}\CD@nF{ut:}{null}\CD@nF{+h:}{null}\CD@nF{+m:}{null%
}\CD@nF{+t:}{null}\CD@nF{-h:}{null}\CD@nF{-m:}{null}\CD@nF{-t:}{null}\def
\CD@@D{\hbox{\vrule height 1pt depth-1pt width 1pt}}\CD@RC{rf:}{\CD@@D}\CD@nF
{lf:}{rf:}\CD@nF{+f:}{rf:}\CD@RC{df:}{\CD@@D}\CD@nF{uf:}{df:}\CD@nF{-f:}{df:}%
\def\CD@BD{\CD@U\null\CD@@D\null\CD@@D\null}\edef\CD@lG{\string\newarrow}\def
\newarrow#1#2#3#4#5#6{\begingroup\edef\@name{#1}\edef\CD@oJ{#2}\edef\CD@iD{#3%
}\edef\CD@QG{#4}\edef\CD@jD{#5}\edef\CD@LE{#6}\let\CD@HE\CD@sG\let\CD@FK
\CD@BH\let\@x\CD@AH\ifx\CD@oJ\CD@iD\let\CD@oJ\empty\fi\ifx\CD@LE\CD@jD\let
\CD@LE\empty\fi\def\CD@LI{r}\def\CD@SF{l}\def\CD@IC{d}\def\CD@yJ{u}\def\CD@gH
{+}\def\@m{-}\ifx\CD@iD\CD@jD\ifx\CD@QG\CD@iD\let\CD@QG\empty\fi\ifx\CD@LE
\empty\ifx\CD@iD\CD@aE\let\@x\CD@yG\else\let\@x\CD@zG\fi\fi\else\edef\CD@a{%
\CD@iD\CD@oJ}\ifx\CD@a\empty\ifx\CD@QG\CD@jD\let\CD@QG\empty\fi\fi\fi\ifmmode
\aftergroup\CD@kG\else\CD@@A\CD@oB rh{head\space\space}\CD@LE\CD@oB rf{filler%
}\CD@iD\CD@oB rm{middle}\CD@QG\ifx\CD@jD\CD@iD\else\CD@oB rf{filler}\CD@jD\fi
\CD@oB rt{tail\space\space}\CD@oJ\CD@gE\CD@HE\CD@FK\@x\CD@nG l-2+2{lu}{nw}%
\NorthWest\CD@nG r+2+2{ru}{ne}\NorthEast\CD@nG l-2-2{ld}{sw}\SouthWest\CD@nG r%
+2-2{rd}{se}\SouthEast\else\aftergroup\CD@b\CD@L{r\@name}\fi\fi\endgroup}\def
\CD@sG{\CD@vG\CD@LI\CD@SF rl\Horizontal@Map}\def\CD@BH{\CD@vG\CD@IC\CD@yJ du%
\Vertical@Map}\def\CD@AH{\CD@vG\CD@gH\@m+-\Vector@Map}\def\CD@yG{\CD@vG\CD@gH
\@m+-\Slant@Map}\def\CD@zG{\CD@vG\CD@gH\@m+-\Slope@Map}\catcode`\/=\active
\def\CD@vG#1#2#3#4#5{\CD@jG#1#3#5t:\CD@oJ/f:\CD@iD/m:\CD@QG/f:\CD@jD/h:\CD@LE
//\CD@jG#2#4#5h:\CD@LE/f:\CD@jD/m:\CD@QG/f:\CD@iD/t:\CD@oJ//}\def\CD@jG#1#2#3%
#4//{\edef\CD@fG{#2}\aftergroup\sdef\CD@L{#1\@name}\aftergroup{\aftergroup#3%
\CD@M#4//\aftergroup}}\def\CD@M#1/{\edef\CD@EH{#1}\ifx\CD@EH\empty\else\CD@L{%
\CD@fG#1}\expandafter\CD@M\fi}\catcode`\/=12 \def\CD@nG#1#2#3#4#5#6#7#8{%
\aftergroup\sdef\CD@L{#6\@name}\aftergroup{\CD@L{#2\@name}\if#2#4\aftergroup
\CD@CI\else\aftergroup\CD@BI\fi\CD@L{#1\@name}%
\aftergroup(\aftergroup#3\aftergroup,\aftergroup#5\aftergroup)\aftergroup}}%
\def\CD@oB#1#2#3#4{\expandafter\ifx\csname#1#2:#4\endcsname\relax\CD@y\CD@gB{%
arrow#3 "#4" undefined}\fi}\CD@rG\CD@VE{All five components must be defined
before an arrow.}\CD@rG\CD@SE{\CD@lG, unlike \string\HorizontalMap, is a
declaration.}\def\CD@b#1{\CD@YA{Arrows \string#1 etc could not be defined}%
\CD@VE}\def\CD@kG{\CD@YA{misplaced \CD@lG}\CD@SE}\def\newdiagramgrid#1#2#3{%
\CD@RC{cdgh@#1}{#2,],}%% ASCII close square bracket
\CD@RC{cdgv@#1}{#3,],}}%% ASCII close square bracket
\CD@tG\ifincommdiag\incommdiagtrue\incommdiagfalse\CD@tG\CD@@F\CD@IF\CD@HF
\newcount\CD@VA\CD@VA=0 \def\CD@yH{\CD@VA6 }\def\CD@OB{\CD@VA1 \global\CD@yA1
\CD@DE\CD@YF\empty}\def\CD@YF{}\def\CD@nB#1{\relax\CD@MD\edef\CD@vJ{#1}%
\begingroup\CD@rE\else\ifcase\CD@VA\ifmmode\else\CD@YG\CD@E0\fi\or\CD@cE5\or
\CD@YG\CD@F5\or\CD@YG\CD@B5\or\CD@YG\CD@B5\or\CD@YG\CD@C5\or\CD@cE7\or\CD@YG
\CD@D7\fi\fi\endgroup\xdef\CD@YF{#1}}\def\CD@pB#1#2#3#4#5{\relax\CD@MD\xdef
\CD@vJ{#4}\begingroup\ifnum\CD@VA<#1 \expandafter\CD@cE\ifcase\CD@VA0\or#2\or
#3\else#2\fi\else\ifnum\CD@VA<6 \CD@tJ\CD@YG\CD@B#2\else\CD@YG\CD@G#2\fi\fi
\endgroup\CD@DE\CD@YF\CD@vJ\ifincommdiag\let\CD@ZD#5\else\let\CD@ZD\CD@LK\fi}%
\def\CD@yI{\global\CD@yA=\ifnum\CD@VA<5 1\else2\fi\relax}\def\CD@OI{\CD@VA
\CD@yA}\def\CD@cE#1{\aftergroup\CD@VA\aftergroup#1\aftergroup\relax}\def
\CD@HH{\def\CD@nB##1{\relax}\let\CD@pB\CD@FH\let\CD@yH\relax\let\CD@OB\relax
\let\CD@yI\relax\let\CD@OI\relax}\def\CD@FH#1#2#3#4#5{\ifincommdiag\let\CD@ZD
#5\else\xdef\CD@vJ{#4}\let\CD@ZD\CD@LK\fi}\def\CD@YG#1{\aftergroup#1%
\aftergroup\relax\CD@cE}\def\CD@B{\CD@YE\CD@S\CD@ME\CD@Q}\def\CD@G{\CD@YE{%
\CD@yC\CD@S}\CD@XE\CD@QD\CD@Q}\def\CD@F{\CD@YE{*\CD@S}\CD@RE\clubsuit\CD@Q}%
\def\CD@C{\CD@YE{\CD@S*\CD@S}\CD@RE\CD@Q\clubsuit\CD@Q}\def\CD@D{\CD@YE\CD@EC
\CD@TE\\}\def\CD@E{\CD@YE\CD@nC\CD@QE\CD@k}\def\CD@LK{\CD@YA{\CD@vJ\space
ignored \CD@dH}\CD@WE}\def\CD@FE{}\def\CD@d{\CD@YA{maps must never be enclosed
in braces}\CD@OE}\def\CD@dH{outside diagram}\def\CD@FC{\string\HonV, \string
\VonH\space and \string\HmeetV}\CD@rG\CD@ME{The way that horizontal and
vertical arrows are terminated implicitly means\CD@uG that they cannot be
mixed with each other or with \CD@FC.}\CD@rG\CD@XE{\string\pile\space is for
parallel horizontal arrows; verticals can just be put together in\CD@uG a cell%
. \CD@FC\space are not meaningful in a \string\pile.}\CD@rG\CD@RE{The
horizontal maps must point to an object, not each other (I've put in\CD@uG one
which you're unlikely to want). Use \string\pile\space if you want them
parallel.}\CD@rG\CD@TE{Parallel horizontal arrows must be in separate layers
of a \string\pile.}\CD@rG\CD@QE{Horizontal arrows may be used \CD@dH s, but
must still be in maths.}\CD@rG\CD@WE{Vertical arrows, \CD@FC\space\CD@dH s don%
't know where\CD@uG where to terminate.}\CD@rG\CD@OE{This prevents them from
stretching correctly.}\def\CD@YE#1{\CD@YA{"#1" inserted \ifx\CD@YF\empty
before \CD@vJ\else between \CD@YF\ifx\CD@YF\CD@vJ s\else\space and \CD@vJ\fi
\fi}}\count@=\year\multiply\count@12 \advance\count@\month\ifnum\count@>24391
\expandafter\endinput\fi\def
\Horizontal@Map{\CD@nB{horizontal map}\CD@LC\CD@TJ\CD@qD}\def\CD@TJ{\CD@GB-%
9999 \let\CD@ZD\CD@XD\ifincommdiag\else\CD@cJ\ifinpile\else\skip2\z@ plus 1.5%
\CD@VK minus .5\CD@UK\skip4\skip2 \fi\fi\let\CD@kD\@fillh\CD@nF{fill@dot}{rf:%
.}}\def\Vector@Map{\CD@HK4}\def\Slant@Map{\CD@HK{\CD@EF255\else6\fi}}\def
\Slope@Map{\CD@HK\CD@OC}\def\CD@HK#1#2#3#4#5#6{\CD@LC\def\CD@WK{2}\def\CD@aK{%
2}\def\CD@ZK{1}\def\CD@bK{1}\let\Horizontal@Map\CD@nI\def\CD@OG{#1}\def\CD@NI
{\CD@U#2#3#4#5#6}}\def\CD@nI{\CD@TJ\CD@JB\let\CD@ZD\CD@TD\CD@qD}\CD@tG\CD@pE
\CD@rA\CD@qA\CD@rA\def\cds@missives{\CD@rA}\def\CD@TD{\CD@vE\let\CD@OG\CD@OC
\CD@x\CD@zE\CD@WF\fi\setbox0\hbox{\incommdiagfalse\CD@HI}\CD@pE\CD@aD\else
\global\CD@YC\CD@bD\fi\ifvoid6 \ifvoid7 \CD@eE\fi\fi\CD@zE\else\CD@BD\global
\CD@YC\let\CD@CG\CD@IH\CD@YD\fi\else\CD@NI\CD@MI\global\CD@YC\CD@YD\fi}\def
\CD@LC{\begingroup\dimen1=\MapShortFall\dimen2=\CD@RG\dimen5=\MapShortFall
\setbox3=\box\voidb@x\setbox6=\box\voidb@x\setbox7=\box\voidb@x\CD@pD
\mathsurround\z@\skip2\z@ plus1fill minus 1000pt\skip4\skip2 \CD@TB}\CD@tG
\CD@tE\CD@UB\CD@TB\def\CD@U#1#2#3#4#5{\let\CD@oJ#1\let\CD@iD#2\let\CD@QG#3%
\let\CD@jD#4\let\CD@LE#5\CD@TB\ifx\CD@iD\CD@jD\CD@UB\fi}\def\CD@qD#1#2#3#4#5{%
\CD@U#1#2#3#4#5\CD@tD}\def\Vertical@Map{\CD@pB433{vertical map}\CD@cD\CD@LC
\CD@GB-9995 \let\CD@kD\@fillv\CD@nF{fill@dot}{df:.}\CD@qD}\def\break@args{%
\def\CD@tD{\CD@ZD}\CD@ZD\endgroup\aftergroup\CD@FE}\def\CD@MJ{\setbox1=\CD@oJ
\setbox5=\CD@LE\ifvoid3 \ifx\CD@QG\null\else\setbox3=\CD@QG\fi\fi\CD@@G2%
\CD@iD\CD@@G4\CD@jD}\def\CD@pF#1{\ifvoid1\else\CD@oF1#1\fi\ifvoid2\else\CD@oF
2#1\fi\ifvoid3\else\CD@oF3#1\fi\ifvoid4\else\CD@oF4#1\fi\ifvoid5\else\CD@oF5#%
1\fi} \def\CD@oF#1#2{\setbox#1\vbox{\offinterlineskip\box#1\dimen@\prevdepth
\advance\dimen@-#2\relax\setbox0\null\dp0\dimen@\ht0-\dimen@\box0}}\def\CD@@G
#1#2{\ifx#2\CD@kD\setbox#1=\box\voidb@x\else\setbox#1=#2\def#2{\xleaders\box#%
1}\fi}\CD@ZA\CD@BK{\string\HorizontalMap, \string\VerticalMap\space and
\string\DiagonalMap\CD@uG are obsolete - use \CD@lG\space to pre-define maps}%
\def\HorizontalMap#1#2#3#4#5{\CD@BK\CD@nB{old horizontal map}\CD@LC\CD@TJ\def
\CD@oJ{\CD@UH{#1}}\CD@SH\CD@iD{#2}\def\CD@QG{\CD@UH{#3}}\CD@SH\CD@jD{#4}\def
\CD@LE{\CD@UH{#5}}\CD@tD}\def\VerticalMap#1#2#3#4#5{\CD@BK\CD@pB433{vertical
map}\CD@cD\CD@LC\CD@GB-9995 \let\CD@kD\@fillv\def\CD@oJ{\CD@GG{#1}}\CD@VH
\CD@iD{#2}\def\CD@QG{\CD@GG{#3}}\CD@VH\CD@jD{#4}\def\CD@LE{\CD@GG{#5}}\CD@tD}%
\def\DiagonalMap#1#2#3#4#5{\CD@BK\CD@LC\def\CD@OG{4}\let\CD@kD\CD@qK\let
\CD@ZD\CD@YD\def\CD@WK{2}\def\CD@aK{2}\def\CD@ZK{1}\def\CD@bK{1}\def\CD@QG{%
\CD@vF{#3}}\ifPositiveGradient\let\mv\raise\def\CD@oJ{\CD@vF{#5}}\def\CD@iD{%
\CD@vF{#4}}\def\CD@jD{\CD@vF{#2}}\def\CD@LE{\CD@vF{#1}}\else\let\mv\lower\def
\CD@oJ{\CD@vF{#1}}\def\CD@iD{\CD@vF{#2}}\def\CD@jD{\CD@vF{#4}}\def\CD@LE{%
\CD@vF{#5}}\fi\CD@tD}\def\CD@aE{-}\def\CD@AD{\empty}\def\CD@SH{\CD@EG\CD@bE
\CD@aE\@fillh}\def\CD@VH{\CD@EG\CD@MK\CD@KK\@fillv}\def\CD@EG#1#2#3#4#5{\def
\CD@CH{#5}\ifx\CD@CH#2\let#4#3\else\let#4\null\ifx\CD@CH\empty\else\ifx\CD@CH
\CD@AD\else\let#4\CD@CH\fi\fi\fi}\def\CD@UH#1{\hbox{\mathsurround\z@
\offinterlineskip\def\CD@CH{#1}\ifx\CD@CH\empty\else\ifx\CD@CH\CD@AD\else
\CD@k\mkern-1.5mu{\CD@CH}\mkern-1.5mu\CD@ND\fi\fi}}\def\CD@yD#1#2{\setbox#1=%
\hbox\bgroup\setbox0=\hbox{\CD@k\labelstyle()\CD@ND}%% ASCII round brackets
\setbox1=\null\ht1\ht0\dp1\dp0\box1 \kern.1\CD@zC\CD@k\bgroup\labelstyle
\aftergroup\CD@LD\CD@xD}\def\CD@LD{\CD@ND\kern.1\CD@zC\egroup\CD@tD}\def
\CD@xD{\futurelet\CD@EH\CD@mJ}\def\CD@mJ{%% qualifiers on label arguments
\catcase\bgroup:\CD@v;\catcase\egroup:\missing@label;\catcase\space:\CD@TF;%
\tokcase[:\CD@XF;%%]%ascii close square bracket 
\default:\CD@zJ;\endswitch}\def\CD@v{\let\CD@MD\CD@c\let\CD@CH}\def\CD@zJ#1{%
\let\CD@UF\egroup{\let\actually@braces@missing@around@macro@in@label\CD@ZH
\let\CD@MD\CD@xC\let\CD@UF\CD@VF#1%
\actually@braces@missing@around@macro@in@label}\CD@UF}\def
\actually@braces@missing@around@macro@in@label{\let\CD@CH=}\def\missing@label
{\egroup\CD@YA{missing label}\CD@PE}\def\CD@xC{\egroup\missing@label}\outer
\def\CD@ZH{}\def\CD@UF{}\def\CD@VF{\CD@wC\CD@UF}\def\CD@MD{}\def\CD@XF{\let
\CD@N\CD@xD\get@square@arg\CD@AE}\CD@rG\CD@PE{The text which has just been
read is not allowed within map labels.}\def\CD@c{\egroup\CD@YA{missing \CD@yC
\space inserted after label}\CD@PE}\def\upper@label{\CD@oD\CD@yD6}\def
\lower@label{\def\positional@{\CD@@A\break@args}\CD@yD7}\def\middle@label{%
\CD@yD3}\CD@tG\CD@yE\CD@pD\CD@oD\def\CD@iF{\ifPositiveGradient\CD@tJ
\expandafter\upper@label\else\expandafter\lower@label\fi}\def\CD@iI{%
\ifPositiveGradient\CD@tJ\expandafter\lower@label\else\expandafter
\upper@label\fi}\def\positional@{\CD@gB{labels as positional arguments are
obsolete}\CD@yE\CD@tJ\expandafter\upper@label\else\expandafter\lower@label\fi
-}\def\CD@tD{\futurelet\CD@EH\switch@arg}\def\eat@space{\afterassignment
\CD@tD\let\CD@EH= }\def\CD@TF{\afterassignment\CD@xD\let\CD@EH= }\def\CD@BC{%
\get@round@pair\CD@uD}\def\CD@uD#1#2{\def\CD@WK{#1}\def\CD@aK{#2}\CD@tD}\def
\optional@{\let\CD@N\CD@tD\get@square@arg\CD@AE}\def\CD@JJ.{\CD@sC\CD@tD}\def
\CD@sC{\let\CD@iD\fill@dot\let\CD@jD\fill@dot\def\CD@MI{\let\CD@iD\dfdot\let
\CD@jD\dfdot}}\def\CD@MI{}\def\CD@@E#1,{\CD@nH#1,\begingroup\ifx\@name\CD@RD
\CD@FF\aftergroup\CD@e\fi\aftergroup\CD@jC\else\expandafter\def\expandafter
\CD@RF\expandafter{\csname\@name\endcsname}\expandafter\CD@vD\CD@RF\CD@KD\ifx
\CD@RF\empty\aftergroup\CD@pC\expandafter\aftergroup\csname\CD@FB\@name
\endcsname\expandafter\aftergroup\csname\CD@FB @\@name\endcsname\else\gdef
\CD@GE{#1}\CD@gB{\string\relax\space inserted before `[\CD@GE'}\message{(I was
trying to read this as a \CD@tA\ option.)}\aftergroup\CD@H\fi\fi\endgroup}%
\def\CD@vD#1#2\CD@KD{\def\CD@RF{#2}}\def\CD@jC{\let\CD@CH\CD@N\let\CD@N\relax
\CD@CH}\def\CD@H#1],{%% ASCII close square bracket
\CD@jC\relax\def\CD@RF{#1}\ifx\CD@RF\empty\def\CD@RF{[\CD@GE]}%
\else\def\CD@RF{[\CD@GE,#1]}%% ASCII open and close square bracket
\fi\CD@RF}\def\CD@pC#1#2{\ifx#2\CD@qK\ifx#1\CD@qK\CD@gB{option `\@name'
undefined}\else#1\fi\else\CD@FF\expandafter#2\CD@GK\CD@PK\else\CD@QK\fi\fi
\CD@DH}\CD@tG\CD@FF\CD@QK\CD@PK\def\CD@nH#1,{\CD@FF\ifx\CD@GK\CD@qK\CD@e\else
\expandafter\CD@oH\CD@GK,#1,(,),(,)[]%
\fi\fi\CD@FF\else\CD@mH#1==,\fi}\def\CD@e{\CD@gB{option `\@name' needs (x,y)
value}\CD@PK\let\@name\empty}\def\CD@mH#1=#2=#3,{\def\@name{#1}\def\CD@GK{#2}%
\def\CD@RF{#3}\ifx\CD@RF\empty\let\CD@GK\CD@qK\fi}%
\def\CD@oH#1(#2,#3)#4,(#5,#6)#7[]{\def\CD@GK{{#2}{#3}}\def\CD@RF{#1#4#5#6}%
\ifx\CD@RF\empty\def\CD@RF{#7}\ifx\CD@RF\empty\CD@e\fi\else\CD@e\fi}\def
\CD@FB{cds@}\let\CD@N\relax\def\CD@zD#1{\ifx\CD@GK\CD@qK\CD@gB{option `\@name
' needs a value}\else#1\CD@GK\relax\fi}\def\CD@BE#1#2{\ifx\CD@GK\CD@qK#1#2%
\relax\else#1\CD@GK\relax\fi}\def\cds@@showpair#1#2{\message{x=#1,y=#2}}\def
\cds@@diagonalbase#1#2{\edef\CD@ZK{#1}\edef\CD@bK{#2}}\def\CD@DI#1{\def\CD@CH
{#1}\CD@nF{@x}{cdps@#1}\ifx\CD@CH\empty\CD@f\CD@CH{cannot be used}\else\ifx
\CD@CH\relax\CD@f\CD@CH{unknown}\else\let\CD@IK\@x\fi\fi}\def\CD@f#1#2{\CD@gB
{PostScript translator `#1' #2}}\def\CD@PH{}\def\CD@PJ{\CD@fA\edef\CD@PH{%
\noexpand\CD@KB{\@name\space ignored within maths}}}\def\diagramstyle{\CD@cJ
\let\CD@N\relax\CD@CF\CD@AE\CD@AE}\CD@tG\CD@sE
\CD@SB\CD@RB\CD@tG\CD@qE\CD@EB\CD@DB\CD@tG\CD@oE\CD@pA\CD@oA\CD@tG\CD@iE
\CD@HA\CD@GA\CD@HA\CD@tG\CD@jE\CD@JA\CD@IA\CD@tG\CD@kE\CD@LA\CD@KA\CD@tG
\CD@EF\CD@DK\CD@CK\CD@tG\CD@rE\CD@JB\CD@IB\CD@tG\CD@mE\CD@gA\CD@fA\CD@tG
\CD@nE\CD@kA\CD@jA\CD@tG\CD@AF\CD@iG\CD@hG\CD@RC{cds@ }{}\CD@RC{cds@}{}\CD@RC
{cds@1em}{\CellSize1\CD@zC}\CD@RC{cds@1.5em}{\CellSize1.5\CD@zC}\CD@RC{cds@2%
em}{\CellSize2\CD@zC}\CD@RC{cds@2.5em}{\CellSize2.5\CD@zC}\CD@RC{cds@3em}{%
\CellSize3\CD@zC}\CD@RC{cds@3.5em}{\CellSize3.5\CD@zC}\CD@RC{cds@4em}{%
\CellSize4\CD@zC}\CD@RC{cds@4.5em}{\CellSize4.5\CD@zC}\CD@RC{cds@5em}{%
\CellSize5\CD@zC}\CD@RC{cds@6em}{\CellSize6\CD@zC}\CD@RC{cds@7em}{\CellSize7%
\CD@zC}\CD@RC{cds@8em}{\CellSize8\CD@zC}\def\cds@abut{\MapsAbut\dimen1\z@
\dimen5\z@}\def\cds@alignlabels{\CD@IA\CD@KA}\def\cds@amstex{\ifincommdiag
\CD@O\else\def\CD{\diagram[amstex]}%%ascii square brackets []
\fi\CD@T\catcode`\@\active}\def\cds@b{\let\CD@dB\CD@bB}\def\cds@balance{\let
\CD@hA\CD@AA}\let\cds@bottom\cds@b\def\cds@center{\cds@vcentre\cds@nobalance}%
\let\cds@centre\cds@center\def\cds@centerdisplay{\CD@HA\CD@PJ\cds@balance}%
\let\cds@centredisplay\cds@centerdisplay\def\cds@crab{\CD@BE\CD@DC{.5%
\PileSpacing}}\CD@RC{cds@crab-}{\CD@DC-.5\PileSpacing}\CD@RC{cds@crab+}{%
\CD@DC.5\PileSpacing}\CD@RC{cds@crab++}{\CD@DC1.5\PileSpacing}\CD@RC{cds@crab%
--}{\CD@DC-1.5\PileSpacing}\def\cds@defaultsize{\CD@BE{\let\CD@QC}{3em}\CD@NJ
}\def\cds@displayoneliner{\CD@DB}\let\cds@dotted\CD@sC\def\cds@dpi{\CD@RJ{1%
truein}}\def\cds@dpm{\CD@RJ{100truecm}}\let\CD@XA\CD@qK\def\cds@eqno{\let
\CD@XA\CD@GK\let\CD@EJ\empty}\def\cds@fixed{\CD@qA}\CD@tG\CD@fE\CD@J\CD@I\def
\cds@flushleft{\CD@I\CD@GA\CD@PJ\cds@nobalance\CD@BE\CD@nA\CD@nA}\def\cds@gap
{\CD@AJ\setbox3=\null\ht3=\CD@tI\dp3=\CD@sI\CD@BE{\wd3=}\MapShortFall} \def
\cds@grid{\ifx\CD@GK\CD@qK\let\h@grid\relax\let\v@grid\relax\else\CD@nF{%
h@grid}{cdgh@\CD@GK}\CD@nF{v@grid}{cdgv@\CD@GK}\ifx\h@grid\relax\CD@gB{%
unknown grid `\CD@GK'}\else\CD@WB\fi\fi}\let\h@grid\relax\let\v@grid\relax
\def\cds@gridx{\ifx\CD@GK\CD@qK\else\cds@grid\fi\let\CD@CH\h@grid\let\h@grid
\v@grid\let\v@grid\CD@CH}\def\cds@h{\CD@zD\DiagramCellHeight}\def\cds@hcenter
{\let\CD@hA\CD@aA}\let\cds@hcentre\cds@hcenter\def\cds@heads{\CD@BE{\let
\CD@sJ}\CD@sJ\CD@@J\CD@vE\else\ifx\CD@sJ\CD@eF\else\CD@MC\fi\fi}\let
\cds@height\cds@h\let\cds@hmiddle\cds@balance\def\cds@htriangleheight{\CD@BE
\DiagramCellHeight\DiagramCellHeight\DiagramCellWidth1.73205%
\DiagramCellHeight}\def\cds@htrianglewidth{\CD@BE\DiagramCellWidth
\DiagramCellWidth\DiagramCellHeight.57735\DiagramCellWidth}\CD@tG\CD@zE\CD@eE
\CD@dE\CD@eE\def\cds@hug{\CD@eE} \def\cds@inline{\CD@gA\let\CD@PH\empty}\def
\cds@inlineoneliner{\CD@EB}\CD@RC{cds@l>}{\CD@zD{\let\CD@RG}\dimen2=\CD@RG}%
\def\cds@labelstyle{\CD@zD{\let\labelstyle}}\def\cds@landscape{\CD@kA}\def
\cds@large{\CellSize5\CD@zC}\let\CD@EJ\empty\def\CD@FJ{\refstepcounter{%
equation}\def\CD@XA{\hbox{\@eqnnum}}}\def\cds@LaTeXeqno{\let\CD@EJ\CD@FJ}\def
\cds@lefteqno{\CD@pA}\def\cds@leftflush{\cds@flushleft\CD@J}\def
\cds@leftshortfall{\CD@zD{\dimen1 }}\def\cds@lowershortfall{%
\ifPositiveGradient\cds@leftshortfall\else\cds@rightshortfall\fi}\def
\cds@loose{\CD@VB}\def\cds@midhshaft{\CD@JA}\def\cds@midshaft{\CD@JA}\def
\cds@midvshaft{\CD@LA}\def\cds@moreoptions{\CD@@A}\let\cds@nobalance
\cds@hcenter\def\cds@nohcheck{\CD@HH}\def\cds@nohug{\CD@dE} \def
\cds@nooptions{\def\CD@aC{\CD@WD}}\let\cds@noorigin\cds@nobalance\def
\cds@nopixel{\CD@@I4\CD@XH\CD@cJ}\def\cds@UO{\CD@oK\global\let\CD@n\empty}%
\def\cds@UglyObsolete{\cds@UO\let\cds@PS\empty}\def\CD@rK#1{\CD@gB{option `#1%
' renamed as `UglyObsolete'}}\def\cds@noPostScript{\CD@rK{noPostScript}}\def
\cds@noPS{\CD@rK{noPostScript}}\def\cds@notextflow{\CD@RB}\def\cds@noTPIC{%
\CD@CK}\def\cds@objectstyle{\CD@zD{\let\objectstyle}}\def\cds@origin{\let
\CD@hA\CD@iB}\def\cds@p{\CD@zD\PileSpacing}\let\cds@pilespacing\cds@p\def
\cds@pixelsize{\CD@zD\CD@@I\CD@gI}\def\cds@portrait{\CD@jA}\def
\cds@PostScript{\CD@nK\global\let\CD@n\empty\CD@BE\CD@DI\empty}\def\cds@PS{%
\CD@nK\global\let\CD@n\empty}\CD@GF\CD@n{\typeout{\CD@tA: try the PostScript
option for better results}}\def\cds@repositionpullbacks{\let\make@pbk\CD@fH
\let\CD@qH\CD@pH}\def\cds@righteqno{\CD@oA}\def\cds@rightshortfall{\CD@zD{%
\dimen5 }}\def\cds@ruleaxis{\CD@zD{\let\axisheight}}\def\cds@cmex{\let\CD@GG
\CD@sB\let\CD@QJ\CD@CJ}\def\cds@s{\cds@height\DiagramCellWidth
\DiagramCellHeight}\def\cds@scriptlabels{\let\labelstyle\scriptstyle}\def
\cds@shortfall{\CD@zD\MapShortFall\dimen1\MapShortFall\dimen5\MapShortFall}%
\def\cds@showfirstpass{\CD@BE{\let\CD@nD}\z@}\def\cds@silent{\def\CD@KB##1{}%
\def\CD@gB##1{}}\let\cds@size\cds@s\def\cds@small{\CellSize2\CD@zC}\def
\cds@snake{\CD@BE\CD@eJ\z@}\def\cds@t{\let\CD@dB\CD@fB}\def\cds@textflow{%
\CD@SB\CD@PJ}\def\cds@thick{\let\CD@rF\tenlnw\CD@LF\CD@NC\CD@BE\MapBreadth{2%
\CD@LF}\CD@@J}\def\cds@thin{\let\CD@rF\tenln\CD@BE\MapBreadth{\CD@NC}\CD@@J}%
\def\cds@tight{\CD@WB}\let\cds@top\cds@t\def\cds@TPIC{\CD@DK}\def
\cds@uppershortfall{\ifPositiveGradient\cds@rightshortfall\else
\cds@leftshortfall\fi}\def\cds@vcenter{\let\CD@dB\CD@cB}\let\cds@vcentre
\cds@vcenter\def\cds@vtriangleheight{\CD@BE\DiagramCellHeight
\DiagramCellHeight\DiagramCellWidth.577035\DiagramCellHeight}\def
\cds@vtrianglewidth{\CD@BE\DiagramCellWidth\DiagramCellWidth
\DiagramCellHeight1.73205\DiagramCellWidth}\def\cds@vmiddle{\let\CD@dB\CD@eB}%
\def\cds@w{\CD@zD\DiagramCellWidth}\let\cds@width\cds@w\def\diagram{\relax
\protect\CD@bC}\def\enddiagram{\protect\CD@SG}\def\CD@bC{\CD@g\CD@uI
\incommdiagtrue\edef\CD@wI{\the\CD@NB}\global\CD@NB\z@\boxmaxdepth\maxdimen
\everycr{}\CD@sK\everymath{}\everyhbox{}\ifx\pdfsyncstop\CD@qK\else
\pdfsyncstop\fi\CD@aC}\def\CD@aC{\CD@y\let\CD@N\CD@ZC\CD@CF\CD@AE\CD@WD}\def
\CD@ZC{\CD@gE\expandafter\CD@aC\else\expandafter\CD@WD\fi}\def\CD@WD{\let
\CD@EH\relax\CD@nE\CD@vE\else\CD@hK\else\CD@KB{landscape ignored without
PostScript}\CD@jA\fi\fi\fi\CD@EJ\setbox2=\vbox\bgroup\CD@JF\CD@VD}\def\CD@cH{%
\CD@nE\CD@fB\else\CD@dB\fi\CD@hA\nointerlineskip\setbox0=\null\ht0-\CD@pI\dp0%
\CD@pI\wd0\CD@kI\box0 \global\CD@QA\CD@kF\global\CD@yA\CD@XB\ifx\CD@NK\CD@qK
\global\CD@RA\CD@kF\else\global\CD@RA\CD@NK\fi\egroup\CD@zF\CD@nE\setbox2=%
\hbox to\dp2{\vrule height\wd2 depth\CD@QA width\z@\global\CD@QA\ht2\ht2\z@
\dp2\z@\wd2\z@\CD@hK\CD@tK{q 0 1 -1 0 0 0 cm}\else\global\CD@iG\CD@IK{0 1
bturn}\fi\box2\CD@gK\hss}\CD@DB\fi\ifnum\CD@yA=1 \else\CD@DB\fi\global
\@ignorefalse\CD@mE\leavevmode\fi\ifvmode\CD@TA\else\ifmmode\CD@PH\CD@GI\else
\CD@qE\CD@gA\fi\ifinner\CD@gA\fi\CD@mE\CD@GI\else\CD@sE\CD@QB\else\CD@TA\fi
\fi\fi\fi\CD@dD}\def\CD@dD{\global\CD@NB\CD@wI\relax\CD@xE\global\CD@ID\else
\aftergroup\CD@mC\fi\if@ignore\aftergroup\ignorespaces\fi\CD@wC\ignorespaces}%
\def\CD@fB{\advance\CD@pI\dimen1\relax}\def\CD@eB{\advance\CD@pI.5\dimen1%
\relax}\def\CD@bB{}\def\CD@cB{\CD@fB\advance\CD@pI\CD@YB\divide\CD@pI2
\advance\CD@pI-\axisheight\relax}\def\CD@aA{}\def\CD@iB{\CD@kF\z@}\def\CD@AA{%
\ifdim\dimen2>\CD@kF\CD@kF\dimen2 \else\dimen2\CD@kF\CD@kI\dimen0 \advance
\CD@kI\dimen2 \fi}\def\CD@QB{\skip0\z@\relax\loop\skip1\lastskip\ifdim\skip1>%
\z@\unskip\advance\skip0\skip1 \repeat\vadjust{\prevdepth\dp\strutbox\penalty
\predisplaypenalty\vskip\abovedisplayskip\CD@UA\penalty\postdisplaypenalty
\vskip\belowdisplayskip}\ifdim\skip0=\z@\else\hskip\skip0 \global\@ignoretrue
\fi}\def\CD@TA{\CD@LG\kern-\displayindent\CD@UA\CD@LG\global\@ignoretrue}\def
\CD@UA{\hbox to\hsize{\CD@fE\ifdim\CD@RA=\z@\else\advance\CD@QA-\CD@RA\setbox
2=\hbox{\kern\CD@RA\box2}\fi\fi\setbox1=\hbox{\ifx\CD@XA\CD@qK\else\CD@k
\CD@XA\CD@ND\fi}\CD@oE\CD@iE\else\advance\CD@QA\wd1 \fi\wd1\z@\box1 \fi\dimen
0\wd2 \advance\dimen0\wd1 \advance\dimen0-\hsize\ifdim\dimen0>-\CD@nA\CD@HA
\fi\advance\dimen0\CD@QA\ifdim\dimen0>\z@\CD@KB{wider than the page by \the
\dimen0 }\CD@HA\fi\CD@iE\hss\else\CD@V\CD@QA\CD@nA\fi\CD@GI\hss\kern-\wd1\box
1 }}\def\CD@GI{\CD@AF\CD@@F\else\CD@SC\global\CD@hG\fi\fi\kern\CD@QA\box2 }%
\CD@tG\CD@wE\CD@YC\CD@XC\def\CD@JF{\CD@cJ\ifdim\DiagramCellHeight=-\maxdimen
\DiagramCellHeight\CD@QC\fi\ifdim\DiagramCellWidth=-\maxdimen
\DiagramCellWidth\CD@QC\fi\global\CD@XC\CD@IF\let\CD@FE\empty\let\CD@z\CD@Q
\let\overprint\CD@eH\let\CD@s\CD@rJ\let\enddiagram\CD@ED\let\\\CD@cC\let\par
\CD@jH\let\CD@MD\empty\let\switch@arg\CD@PB\let\shift\CD@iA\baselineskip
\DiagramCellHeight\lineskip\z@\lineskiplimit\z@\mathsurround\z@\tabskip\z@
\CD@OB}\def\CD@VD{\penalty-123 \begingroup\CD@jA\aftergroup\CD@K\halign
\bgroup\global\advance\CD@NB1 \vadjust{\penalty1}\global\CD@FA\z@\CD@OB\CD@j#%
#\CD@DD\CD@Q\CD@Q\CD@OI\CD@j##\CD@DD\cr}\def\CD@ED{\CD@MD\CD@GD\crcr\egroup
\global\CD@JD\endgroup}\def\CD@j{\global\advance\CD@FA1 \futurelet\CD@EH\CD@i
}\def\CD@i{\ifx\CD@EH\CD@DD\CD@tJ\hskip1sp plus 1fil \relax\let\CD@DD\relax
\CD@vI\else\hfil\CD@k\objectstyle\let\CD@FE\CD@d\fi}\def\CD@DD{\CD@MD\relax
\CD@yI\CD@vI\global\CD@QA\CD@iA\penalty-9993 \CD@ND\hfil\null\kern-2\CD@QA
\null}\def\CD@cC{\cr}\def\across#1{\span\omit\mscount=#1 \global\advance
\CD@FA\mscount\global\advance\CD@FA\m@ne\CD@sF\ifnum\mscount>2 \CD@fJ\repeat
\ignorespaces}\def\CD@fJ{\relax\span\omit\advance\mscount\m@ne}\def\CD@qJ{%
\ifincommdiag\ifx\CD@iD\@fillh\ifx\CD@jD\@fillh\ifdim\dimen3>\z@\else\ifdim
\dimen2>93\CD@@I\ifdim\dimen2>18\p@\ifdim\CD@LF>\z@\count@\CD@bJ\advance
\count@\m@ne\ifnum\count@<\z@\count@20\let\CD@aJ\CD@uJ\fi\xdef\CD@bJ{\the
\count@}\fi\fi\fi\fi\fi\fi\fi}\def\CD@cG#1{\vrule\horizhtdp width#1\dimen@
\kern2\dimen@}\def\CD@uJ{\rlap{\dimen@\CD@@I\CD@V\dimen@{.182\p@}\CD@zH
\dimen@\advance\CD@tI\dimen@\CD@cG0\CD@cG0\CD@cG2\CD@cG6\CD@cG6\CD@cG2\CD@cG0%
\CD@cG0\CD@cG2\CD@cG6\CD@cG0\CD@cG0\CD@cG2\CD@cG2\CD@cG6\CD@cG0\CD@cG0\CD@cG2%
\CD@cG6\CD@cG2\CD@cG2\CD@cG0\CD@cG0}}\def\CD@bJ{10}\def\CD@aJ{}\def\CD@XD{%
\CD@gE\CD@TB\fi\CD@x\CD@WF\CD@HI}\def\CD@x{\CD@QJ\CD@DC\CD@MJ\ifdim\CD@DC=\z@
\else\CD@pF\CD@DC\fi\ifvoid3 \setbox3=\null\ht3\CD@tI\dp3\CD@sI\else\CD@V{\ht
3}\CD@tI\CD@V{\dp3}\CD@sI\fi\dimen3=.5\wd3 \ifdim\dimen3=\z@\CD@tE\else\dimen
3-\CD@XH\fi\else\CD@TB\fi\CD@V{\dimen2}{\wd7}\CD@V{\dimen2}{\wd6}\CD@qJ
\advance\dimen2-2\dimen3 \dimen4.5\dimen2 \dimen2\dimen4 \advance\dimen2%
\CD@eJ\advance\dimen4-\CD@eJ\advance\dimen2-\wd1 \advance\dimen4-\wd5 \ifvoid
2 \else\CD@V{\ht3}{\ht2}\CD@V{\dp3}{\dp2}\CD@V{\dimen2}{\wd2}\fi\ifvoid4 \else
\CD@V{\ht3}{\ht4}\CD@V{\dp3}{\dp4}\CD@V{\dimen4}{\wd4}\fi\advance\skip2\dimen
2 \advance\skip4\dimen4 \CD@tE\advance\skip2\skip4 \dimen0\dimen5 \advance
\dimen0\wd5 \skip3-\skip4 \advance\skip3-\dimen0 \let\CD@jD\empty\else\skip3%
\z@\relax\dimen0\z@\fi}\def\CD@WF{\offinterlineskip\lineskip.2\CD@zC\ifvoid6
\else\setbox3=\vbox{\hbox to2\dimen3{\hss\box6\hss}\box3}\fi\ifvoid7 \else
\setbox3=\vtop{\box3 \hbox to2\dimen3{\hss\box7\hss}}\fi}\def\CD@HI{\kern
\dimen1 \box1 \CD@aJ\CD@iD\hskip\skip2 \kern\dimen0 \ifincommdiag\CD@jE
\penalty1\fi\kern\dimen3 \penalty\CD@GB\hskip\skip3 \null\kern-\dimen3 \else
\hskip\skip3 \fi\box3 \CD@jD\hskip\skip4 \box5 \kern\dimen5}\def\CD@MF{\ifnum
\CD@LH>\CD@TC\CD@V{\dimen1}\objectheight\CD@V{\dimen5}\objectheight\else\CD@V
{\dimen1}\objectwidth\CD@V{\dimen5}\objectwidth\fi}\def\CD@Y{\begingroup
\ifdim\dimen7=\z@\kern\dimen8 \else\ifdim\dimen6=\z@\kern\dimen9 \else\dimen5%
\dimen6 \dimen6\dimen9 \CD@KJ\dimen4\dimen2 \CD@dG{\dimen4}\dimen6\dimen5
\dimen7\dimen8 \CD@KJ\CD@iC{\dimen2}\ifdim\dimen2<\dimen4 \kern\dimen2 \else
\kern\dimen4 \fi\fi\fi\endgroup}\def\CD@jJ{\CD@JI\setbox\z@\hbox{\lower
\axisheight\hbox to\dimen2{\CD@DF\ifPositiveGradient\dimen8\ht\CD@MH\dimen9%
\CD@mI\else\dimen8\dp3 \dimen9\dimen1 \fi\else\dimen8 \ifPositiveGradient
\objectheight\else\z@\fi\dimen9\objectwidth\fi\advance\dimen8
\ifPositiveGradient-\fi\axisheight\CD@Y\unhbox\z@\CD@DF\ifPositiveGradient
\dimen8\dp3 \dimen9\dimen0 \else\dimen8\ht\CD@MH\dimen9\CD@mF\fi\else\dimen8
\ifPositiveGradient\z@\else\objectheight\fi\dimen9\objectwidth\fi\advance
\dimen8 \ifPositiveGradient\else-\fi\axisheight\CD@Y}}}\def\CD@bD{\dimen6
\CD@aK\DiagramCellHeight\dimen7 \CD@WK\DiagramCellWidth\CD@jJ
\ifPositiveGradient\advance\dimen7-\CD@ZK\DiagramCellWidth\else\dimen7 \CD@ZK
\DiagramCellWidth\dimen6\z@\fi\advance\dimen6-\CD@bK\DiagramCellHeight\CD@mK
\setbox0=\rlap{\kern-\dimen7 \lower\dimen6\box\z@}\ht0\z@\dp0\z@\raise
\axisheight\box0 }\def\CD@mK{\setbox0\hbox{\ht\z@\z@\dp\z@\z@\wd\z@\z@\CD@hK
\expandafter\CD@tK{q \CD@eK\space\CD@lK\space\CD@kK\space\CD@eK\space0 0 cm}%
\else\global\CD@iG\CD@eD{\the\CD@TC\space\ifPositiveGradient\else-\fi\the
\CD@LH\space bturn}\fi\box\z@\CD@gK}}\def\CD@vB{\advance\CD@hF-\CD@mI\CD@wJ
\CD@hF\advance\CD@wJ\CD@hI\ifvoid\CD@sH\ifdim\CD@wJ<.1em\ifnum\CD@gD=\@m\else
\CD@aG h\CD@wJ<.1em:objects overprint:\CD@FA\CD@gD\fi\fi\else\ifhbox\CD@sH
\CD@SK\else\CD@TK\fi\advance\CD@wJ\CD@mI\CD@bH{-\CD@mI}{\box\CD@sH}{\CD@wJ}%
\z@\fi\CD@hF-\CD@mF\CD@gD\CD@FA\CD@hI\z@}\def\CD@SK{\setbox\CD@sH=\hbox{%
\unhbox\CD@sH\unskip\unpenalty}\setbox\CD@tH=\hbox{\unhbox\CD@tH\unskip
\unpenalty}\setbox\CD@sH=\hbox to\CD@wJ{\CD@OA\wd\CD@sH\unhbox\CD@sH\CD@PA
\lastkern\unkern\ifdim\CD@PA=\z@\CD@UB\advance\CD@OA-\wd\CD@tH\else\CD@TB\fi
\ifnum\lastpenalty=\z@\else\CD@JA\unpenalty\fi\kern\CD@PA\ifdim\CD@hF<\CD@OA
\CD@JA\fi\ifdim\CD@hI<\wd\CD@tH\CD@JA\fi\CD@jE\CD@hI\CD@wJ\advance\CD@hI-%
\CD@OA\advance\CD@hI\wd\CD@tH\ifdim\CD@hI<2\wd\CD@tH\CD@aG h\CD@hI<2\wd\CD@tH
:arrow too short:\CD@FA\CD@gD\fi\divide\CD@hI\tw@\CD@hF\CD@wJ\advance\CD@hF-%
\CD@hI\fi\CD@tE\kern-\CD@hI\fi\hbox to\CD@hI{\unhbox\CD@tH}\CD@HG}}\CD@tG
\ifinpile\inpiletrue\inpilefalse\inpilefalse\def\pile{\protect\CD@UJ\protect
\CD@uH}\def\CD@uH#1{\CD@l#1\CD@QD}\def\CD@UJ{\CD@nB{pile}\setbox0=\vtop
\bgroup\aftergroup\CD@lD\inpiletrue\let\CD@FE\empty\let\pile\CD@KF\let\CD@QD
\CD@PD\let\CD@GD\CD@FD\CD@yH\baselineskip.5\PileSpacing\lineskip.1\CD@zC
\relax\lineskiplimit\lineskip\mathsurround\z@\tabskip\z@\let\\\CD@wH}\def
\CD@l{\CD@DE\CD@YF\empty\halign\bgroup\hfil\CD@k\let\CD@FE\CD@d\let\\\CD@vH##%
\CD@MD\CD@ND\hfil\CD@Q\CD@R##\cr}\CD@rG\CD@NE{pile only allows one column.}%
\CD@rG\CD@UE{you left it out!}\def\CD@R{\CD@QD\CD@Q\relax\CD@YA{missing \CD@yC
\space inserted after \string\pile}\CD@NE}\def\CD@PD{\CD@MD\crcr\egroup
\egroup}\def\CD@GD{\CD@MD}\def\CD@FD{\CD@MD\relax\CD@QD\CD@YA{missing \CD@yC
\space inserted between \string\pile\space and \CD@HD}\CD@UE}\def\CD@QD{%
\CD@MD}\def\CD@lD{\vbox{\dimen1\dp0 \unvbox0 \setbox0=\lastbox\advance\dimen1%
\dp0 \nointerlineskip\box0 \nointerlineskip\setbox0=\null\dp0.5\dimen1\ht0-%
\dp0 \box0}\ifincommdiag\CD@tJ\penalty-9998 \fi\xdef\CD@YF{pile}}\def\CD@vH{%
\cr}\def\CD@wH{\noalign{\skip@\prevdepth\advance\skip@-\baselineskip
\prevdepth\skip@}}\def\CD@KF#1{#1}\def\CD@TK{\setbox\CD@sH=\vbox{\unvbox
\CD@sH\setbox1=\lastbox\setbox0=\box\voidb@x\CD@tF\setbox\CD@sH=\lastbox
\ifhbox\CD@sH\CD@rC\repeat\unvbox0 \global\CD@QA\CD@ZE}\CD@ZE\CD@QA}\def
\CD@rC{\CD@jE\setbox\CD@sH=\hbox{\unhbox\CD@sH\unskip\setbox\CD@sH=\lastbox
\unskip\unhbox\CD@sH}\ifdim\CD@wJ<\wd\CD@sH\CD@aG h\CD@wJ<\wd\CD@sH:arrow in
pile too short:\CD@FA\CD@gD\else\setbox\CD@sH=\hbox to\CD@wJ{\unhbox\CD@sH}%
\fi\else\CD@gJ\fi\setbox0=\vbox{\box\CD@sH\nointerlineskip\ifvoid0 \CD@tJ\box
1 \else\vskip\skip0 \unvbox0 \fi}\skip0=\lastskip\unskip}\def\CD@gJ{\penalty7
\noindent\unhbox\CD@sH\unskip\setbox\CD@sH=\lastbox\unskip\unhbox\CD@sH
\endgraf\setbox\CD@tH=\lastbox\unskip\setbox\CD@tH=\hbox{\CD@JG\unhbox\CD@tH
\unskip\unskip\unpenalty}\ifcase\prevgraf\cd@shouldnt P\or\ifdim\CD@wJ<\wd
\CD@tH\CD@aG h\CD@wJ<\wd\CD@sH:object in pile too wide:\CD@FA\CD@gD\setbox
\CD@sH=\hbox to\CD@wJ{\hss\unhbox\CD@tH\hss}\else\setbox\CD@sH=\hbox to\CD@wJ
{\hss\kern\CD@hF\unhbox\CD@tH\kern\CD@hI\hss}\fi\or\setbox\CD@sH=\lastbox
\unskip\CD@SK\else\cd@shouldnt Q\fi\unskip\unpenalty}\def\CD@cD{\CD@MJ\ifvoid
3 \setbox3=\null\ht3\axisheight\dp3-\ht3 \dimen3.5\CD@LF\else\dimen4\dp3
\dimen3.5\wd3 \setbox3=\CD@GG{\box3}\dp3\dimen4 \ifdim\ht3=-\dp3 \else\CD@TB
\fi\fi\dimen0\dimen3 \advance\dimen0-.5\CD@LF\setbox0\null\ht0\ht3\dp0\dp3\wd
0\wd3 \ifvoid6\else\setbox6\hbox{\unhbox6\kern\dimen0\kern2pt}\dimen0\wd6 \fi
\ifvoid7\else\setbox7\hbox{\kern2pt\kern\dimen3\unhbox7}\dimen3\wd7 \fi
\setbox3\hbox{\ifvoid6\else\kern-\dimen0\unhbox6\fi\unhbox3 \ifvoid7\else
\unhbox7\kern-\dimen3\fi}\ht3\ht0\dp3\dp0\wd3\wd0 \CD@tE\dimen4=\ht\CD@MH
\advance\dimen4\dp5 \advance\dimen4\dimen1 \let\CD@jD\empty\else\dimen4\ht3
\fi\setbox0\null\ht0\dimen4 \offinterlineskip\setbox8=\vbox spread2ex{\kern
\dimen5 \box1 \CD@iD\vfill\CD@tE\else\kern\CD@eJ\fi\box0}\ht8=\z@\setbox9=%
\vtop spread2ex{\kern-\ht3 \kern-\CD@eJ\box3 \CD@jD\vfill\box5 \kern\dimen1}%
\dp9=\z@\hskip\dimen0plus.0001fil \box9 \kern-\CD@LF\box8 \CD@kE\penalty2 \fi
\CD@tE\penalty1 \fi\kern\PileSpacing\kern-\PileSpacing\kern-.5\CD@LF\penalty
\CD@GB\null\kern\dimen3}\def\CD@cI{\ifhbox\CD@VA\CD@KB{clashing verticals}\ht
\CD@MH.5\dp\CD@VA\dp\CD@MH-\ht5 \CD@yB\ht\CD@MH\z@\dp\CD@MH\z@\fi\dimen1\dp
\CD@VA\CD@xA\prevgraf\unvbox\CD@VA\CD@wA\lastpenalty\unpenalty\setbox\CD@VA=%
\null\setbox\CD@lI=\hbox{\CD@JG\unhbox\CD@lI\unskip\unpenalty\dimen0\lastkern
\unkern\unkern\unkern\kern\dimen0 \CD@HG}\setbox\CD@lF=\hbox{\unhbox\CD@lF
\dimen0\lastkern\unkern\unkern\global\CD@QA\lastkern\unkern\kern\dimen0 }%
\CD@tF\ifnum\CD@xA>4 \CD@zI\repeat\unskip\unskip\advance\CD@mF.5\wd\CD@VA
\advance\CD@mF\wd\CD@lF\advance\CD@mI.5\wd\CD@VA\advance\CD@mI\wd\CD@lI\ifnum
\CD@FA=\CD@lA\CD@OA.5\wd\CD@VA\edef\CD@NK{\the\CD@OA}\fi\setbox\CD@VA=\hbox{%
\kern-\CD@mF\box\CD@lF\unhbox\CD@VA\box\CD@lI\kern-\CD@mI\penalty\CD@wA
\penalty\CD@NB}\ht\CD@VA\dimen1 \dp\CD@VA\z@\wd\CD@VA\CD@tB\CD@vB}\def\CD@zI{%
\ifdim\wd\CD@lF<\CD@QA\setbox\CD@lF=\hbox to\CD@QA{\CD@JG\unhbox\CD@lF}\fi
\advance\CD@xA\m@ne\setbox\CD@VA=\hbox{\box\CD@lF\unhbox\CD@VA}\unskip\setbox
\CD@lF=\lastbox\setbox\CD@lF=\hbox{\unhbox\CD@lF\unskip\unpenalty\dimen0%
\lastkern\unkern\unkern\global\CD@QA\lastkern\unkern\kern\dimen0 }}\def\CD@yB
{\dimen1\dp\CD@VA\ifhbox\CD@VA\CD@xB\else\CD@zB\fi\setbox\CD@VA=\vbox{%
\penalty\CD@NB}\dp\CD@VA-\dp\CD@MH\wd\CD@VA\CD@tB}\def\CD@zB{\unvbox\CD@VA
\CD@wA\lastpenalty\unpenalty\ifdim\dimen1<\ht\CD@MH\CD@aG v\dimen1<\ht\CD@MH:%
rows overprint:\CD@NB\CD@wA\fi}\def\CD@xB{\dimen0=\ht\CD@VA\setbox\CD@VA=%
\hbox\bgroup\advance\dimen1-\ht\CD@MH\unhbox\CD@VA\CD@xA\lastpenalty
\unpenalty\CD@wA\lastpenalty\unpenalty\global\CD@RA-\lastkern\unkern\setbox0=%
\lastbox\CD@tF\setbox\CD@VA=\hbox{\box0\unhbox\CD@VA}\setbox0=\lastbox\ifhbox
0 \CD@kJ\repeat\global\CD@SA-\lastkern\unkern\global\CD@QA\CD@JK\unhbox\CD@VA
\egroup\CD@JK\CD@QA\CD@bH{\CD@SA}{\box\CD@VA}{\CD@RA}{\dimen1}}\def\CD@kJ{%
\setbox0=\hbox to\wd0\bgroup\unhbox0 \unskip\unpenalty\dimen7\lastkern\unkern
\ifnum\lastpenalty=1 \unpenalty\CD@UB\else\CD@TB\fi\ifnum\lastpenalty=2
\unpenalty\dimen2.5\dimen0\advance\dimen2-.5\dimen1\advance\dimen2-%
\axisheight\else\dimen2\z@\fi\setbox0=\lastbox\dimen6\lastkern\unkern\setbox1%
=\lastbox\setbox0=\vbox{\unvbox0 \CD@tE\kern-\dimen1 \else\ifdim\dimen2=\z@
\else\kern\dimen2 \fi\fi}\ifdim\dimen0<\ht0 \CD@aG v\dimen0<\ht0:upper part of
vertical too short:{\CD@tE\CD@NB\else\CD@wA\fi}\CD@xA\else\setbox0=\vbox to%
\dimen0{\unvbox0}\fi\setbox1=\vtop{\unvbox1}\ifdim\dimen1<\dp1 \CD@aG v\dimen
1<\dp1:lower part of vertical too short:\CD@NB\CD@wA\else\setbox1=\vtop to%
\dimen1{\ifdim\dimen2=\z@\else\kern-\dimen2 \fi\unvbox1 }\fi\box1 \kern\dimen
6 \box0 \kern\dimen7 \CD@HG\global\CD@QA\CD@JK\egroup\CD@JK\CD@QA\relax}%
\countdef\CD@u=14 \newcount\CD@CA\newcount\CD@XB\newcount\CD@NB\let\CD@LB
\insc@unt\newcount\CD@FA\newcount\CD@lA\let\CD@mA\CD@XB\newcount\CD@MB\CD@tG
\CD@DF\CD@bI\CD@aI\CD@aI\def\CD@nD{-1}\def\CD@K{\CD@t\ifnum\CD@nD<\z@\else
\begingroup\scrollmode\showboxdepth\CD@nD\showboxbreadth\maxdimen\showlists
\endgroup\fi\CD@bI\CD@zF\CD@CA=\CD@u\advance\CD@CA1 \CD@XB=\CD@CA\ifnum\CD@NB
=1 \CD@JA\fi\advance\CD@XB\CD@NB\dimen1\z@\skip0\z@\count@=\insc@unt\advance
\count@\CD@u\divide\count@2 \ifnum\CD@XB>\count@\CD@KB{The diagram has too
many rows! It can't be reformatted.}\else\CD@NG\CD@WI\fi\CD@cH}\def\CD@NG{%
\CD@NB\CD@CA\CD@uF\ifnum\CD@NB<\CD@XB\setbox\CD@NB\box\voidb@x\advance\CD@NB1%
\relax\repeat\CD@NB\CD@CA\skip\z@\z@\CD@uF\CD@GB\lastpenalty\unpenalty\ifnum
\CD@GB>\z@\CD@KE\repeat\ifnum\CD@GB=-123 \CD@tJ\unpenalty\else\cd@shouldnt D%
\fi\ifx\v@grid\relax\else\CD@NB\CD@XB\advance\CD@NB\m@ne\expandafter\CD@VJ
\v@grid\fi\CD@MB\CD@mA\CD@tB\z@\CD@XG\ifx\h@grid\relax\else\expandafter\CD@LJ
\h@grid\fi\count@\CD@XB\advance\count@\m@ne\CD@YB\ht\count@}\def\CD@KE{%
\ifcase\CD@GB\or\CD@MG\else\CD@uA-\lastpenalty\unpenalty\CD@vA\lastpenalty
\unpenalty\setbox0=\lastbox\CD@WG\fi\CD@wD}\def\CD@wD{\skip1\lastskip\unskip
\advance\skip0\skip1 \ifdim\skip1=\z@\else\expandafter\CD@wD\fi}\def\CD@MG{%
\setbox0=\lastbox\CD@pI\dp0 \advance\CD@pI\skip\z@\skip\z@\z@\advance\CD@NF
\CD@pI\CD@uE\ifnum\CD@NB>\CD@CA\CD@NF\DiagramCellHeight\CD@pI\CD@NF\advance
\CD@pI-\CD@qI\fi\fi\CD@qI\ht0 \CD@NF\CD@qI\setbox\CD@NB\hbox{\unhbox\CD@NB
\unhbox0}\dp\CD@NB\CD@pI\ht\CD@NB\CD@qI\advance\CD@NB1 }\def\CD@WG{\ifnum
\CD@uA<\z@\advance\CD@uA\CD@XB\ifnum\CD@uA<\CD@CA\CD@UG\else\CD@OA\dp\CD@uA
\CD@PA\ht\CD@uA\setbox\CD@uA\hbox{\box\z@\penalty\CD@vA\penalty\CD@GB\unhbox
\CD@uA}\dp\CD@uA\CD@OA\ht\CD@uA\CD@PA\fi\else\CD@UG\fi}\def\CD@UG{\CD@KB{%
diagonal goes outside diagram (lost)}}\def\CD@fI{\advance\CD@uA\CD@XB\ifnum
\CD@uA<\CD@CA\CD@UG\else\ifnum\CD@uA=\CD@NB\CD@VG\else\ifnum\CD@uA>\CD@NB
\cd@shouldnt M\else\CD@OA\dp\CD@uA\CD@PA\ht\CD@uA\setbox\CD@uA\hbox{\box\z@
\penalty\CD@vA\penalty\CD@GB\unhbox\CD@uA}\dp\CD@uA\CD@OA\ht\CD@uA\CD@PA\fi
\fi\fi}\def\CD@WI{\CD@AJ\setbox\CD@PC=\hbox{\CD@k A\@super f\CD@lJ f\CD@ND}%
\CD@ZE\z@\CD@JK\z@\CD@kI\z@\CD@kF\z@\CD@NB=\CD@XB\CD@NF\z@\CD@uB\z@\CD@uF
\ifnum\CD@NB>\CD@CA\advance\CD@NB\m@ne\CD@qI\ht\CD@NB\CD@pI\dp\CD@NB\advance
\CD@NF\CD@qI\CD@rI\advance\CD@uB\CD@NF\CD@KC\CD@ZI\CD@w\ht\CD@NB\CD@qI\dp
\CD@NB\CD@pI\nointerlineskip\box\CD@NB\CD@NF\CD@pI\setbox\CD@NB\null\ht\CD@NB
\CD@uB\repeat\CD@wB\nointerlineskip\box\CD@NB\CD@gG\CD@ZE\DiagramCellWidth{%
width}\CD@gG\CD@JK\DiagramCellHeight{height}\CD@VA\CD@LB\advance\CD@VA-\CD@lA
\advance\CD@VA\m@ne\advance\CD@VA\CD@mA\dimen0\wd\CD@VA\CD@tI\axisheight
\dimen1\CD@uB\advance\dimen1-\CD@YB\dimen2\CD@kI\advance\dimen2-\dimen0
\advance\CD@XB-\CD@CA\advance\CD@LB-\CD@lA}\count@\year\multiply\count@12
\advance\count@\month\ifnum\count@>24398 \loop\iftrue\repeat\fi\def\CD@wB{\CD@qI-\CD@NF\CD@pI\CD@NF\setbox\CD@MH=\null\dp
\CD@MH\CD@NF\ht\CD@MH-\CD@NF\CD@mF\z@\CD@mI\z@\CD@lA\CD@LB\advance\CD@lA-%
\CD@MB\advance\CD@lA\CD@mA\CD@FA\CD@LB\CD@VA\CD@MB\CD@sF\ifnum\CD@FA>\CD@lA
\advance\CD@FA\m@ne\advance\CD@VA\m@ne\CD@tB\wd\CD@VA\setbox\CD@FA=\box
\voidb@x\CD@yB\repeat\CD@w\ht\CD@NB\CD@qI\dp\CD@NB\CD@pI}\def\CD@gG#1#2#3{%
\ifdim#1>.01\CD@zC\CD@PA#2\relax\advance\CD@PA#1\relax\advance\CD@PA.99\CD@zC
\count@\CD@PA\divide\count@\CD@zC\CD@KB{increase cell #3 to \the\count@ em}%
\fi}\def\CD@rI{\CD@FA=\CD@LB\penalty4 \noindent\unhbox\CD@NB\CD@sF\unskip
\setbox0=\lastbox\ifhbox0 \advance\CD@FA\m@ne\setbox\CD@FA\hbox to\wd0{\null
\penalty-9990\null\unhbox0}\repeat\CD@lA\CD@FA\advance\CD@FA\CD@MB\advance
\CD@FA-\CD@mA\ifnum\CD@FA<\CD@LB\count@\CD@FA\advance\count@\m@ne\dimen0=\wd
\count@\count@\CD@MB\advance\count@\m@ne\CD@tB\wd\count@\CD@sF\ifnum\CD@FA<%
\CD@LB\CD@DJ\CD@XG\dimen0\wd\CD@FA\advance\CD@FA1 \repeat\fi\CD@sF\CD@GB
\lastpenalty\unpenalty\ifnum\CD@GB>\z@\CD@vA\lastpenalty\unpenalty\CD@VG
\repeat\endgraf\unskip\ifnum\lastpenalty=4 \unpenalty\else\cd@shouldnt S\fi}%
\def\CD@VG{\advance\CD@vA\CD@lA\advance\CD@vA\m@ne\setbox0=\lastbox\ifnum
\CD@vA<\CD@LB\setbox\CD@vA\hbox{\box0\penalty\CD@GB\unhbox\CD@vA}\else\CD@UG
\fi}\def\CD@bG{}\CD@tG\CD@uE\CD@WB\CD@VB\def\CD@DJ{\advance\dimen0\wd\CD@FA
\divide\dimen0\tw@\CD@uE\dimen0\DiagramCellWidth\else\CD@V{\dimen0}%
\DiagramCellWidth\CD@pJ\fi\advance\CD@tB\dimen0 }\def\CD@XG{\setbox\CD@MB=%
\vbox{}\dp\CD@MB=\CD@uB\wd\CD@MB\CD@tB\advance\CD@MB1 }\def\CD@LJ#1,{\def
\CD@GK{#1}\ifx\CD@GK\CD@RD\else\advance\CD@tB\CD@GK\DiagramCellWidth\CD@XG
\expandafter\CD@LJ\fi}\def\CD@VJ#1,{\def\CD@GK{#1}\ifx\CD@GK\CD@RD\else\ifnum
\CD@NB>\CD@CA\CD@NF\CD@GK\DiagramCellHeight\advance\CD@NF-\dp\CD@NB\advance
\CD@NB\m@ne\ht\CD@NB\CD@NF\fi\expandafter\CD@VJ\fi}\def\CD@pJ{\CD@wE\CD@OA
\dimen0 \advance\CD@OA-\DiagramCellWidth\ifdim\CD@OA>2\MapShortFall\CD@KB{%
badly drawn diagonals (see manual)}\let\CD@pJ\empty\fi\else\let\CD@pJ\empty
\fi}\def\CD@KC{\CD@VA\CD@mA\CD@sF\ifnum\CD@VA<\CD@MB\dimen0\dp\CD@VA\advance
\dimen0\CD@NF\dp\CD@VA\dimen0 \advance\CD@VA1 \repeat}\def\CD@bH#1#2#3#4{%
\ifnum\CD@FA<\CD@LB\CD@OA=#1\relax\setbox\CD@FA=\hbox{\setbox0=#2\dimen7=#4%
\relax\dimen8=#3\relax\ifhbox\CD@FA\unhbox\CD@FA\advance\CD@OA-\lastkern
\unkern\fi\ifdim\CD@OA=\z@\else\kern-\CD@OA\fi\raise\dimen7\box0 \kern-\dimen
8 }\ifnum\CD@FA=\CD@lA\CD@V\CD@kF\CD@OA\fi\else\cd@shouldnt O\fi}\def\CD@w{%
\setbox\CD@NB=\hbox{\CD@FA\CD@lA\CD@VA\CD@mA\CD@PA\z@\relax\CD@sF\ifnum\CD@FA
<\CD@LB\CD@tB\wd\CD@VA\relax\CD@eI\advance\CD@FA1 \advance\CD@VA1 \repeat}%
\CD@V\CD@kI{\wd\CD@NB}\wd\CD@NB\z@}\def\CD@eI{\ifhbox\CD@FA\CD@OA\CD@tB\relax
\advance\CD@OA-\CD@PA\relax\ifdim\CD@OA=\z@\else\kern\CD@OA\fi\CD@PA\CD@tB
\advance\CD@PA\wd\CD@FA\relax\unhbox\CD@FA\advance\CD@PA-\lastkern\unkern\fi}%
\def\CD@ZI{\setbox\CD@sH=\box\voidb@x\CD@VA=\CD@MB\CD@FA\CD@LB\CD@VA\CD@mA
\advance\CD@VA\CD@FA\advance\CD@VA-\CD@lA\advance\CD@VA\m@ne\CD@tB\wd\CD@VA
\count@\CD@LB\advance\count@\m@ne\CD@hF.5\wd\count@\advance\CD@hF\CD@tB\CD@A
\m@ne\CD@gD\@m\CD@sF\ifnum\CD@FA>\CD@lA\advance\CD@FA\m@ne\advance\CD@hF-%
\CD@tB\CD@PI\wd\CD@VA\CD@tB\advance\CD@hF\CD@tB\advance\CD@VA\m@ne\CD@tB\wd
\CD@VA\repeat\CD@mF\CD@kF\CD@mI-\CD@mF\CD@vB}\newcount\CD@GB\def\CD@s{}\def
\CD@t{\CD@vK\mathsurround\z@\hsize\z@\rightskip\z@ plus1fil minus\maxdimen
\parfillskip\z@\linepenalty9000 \looseness0 \hfuzz\maxdimen\hbadness10000
\clubpenalty0 \widowpenalty0 \displaywidowpenalty0 \interlinepenalty0
\predisplaypenalty0 \postdisplaypenalty0 \interdisplaylinepenalty0
\interfootnotelinepenalty0 \floatingpenalty0 \brokenpenalty0 \everypar{}%
\leftskip\z@\parskip\z@\parindent\z@\pretolerance10000 \tolerance10000
\hyphenpenalty10000 \exhyphenpenalty10000 \binoppenalty10000 \relpenalty10000
\adjdemerits0 \doublehyphendemerits0 \finalhyphendemerits0 \baselineskip\z@
\CD@IA\def\parskip{\cd@shouldnt{PS}}\prevdepth\z@}\newbox\CD@KG\newbox\CD@IG
\def\CD@JG{\unhcopy\CD@KG}\def\CD@HG{\unhcopy\CD@IG}\def\CD@iJ{\hbox{}%
\penalty1\nointerlineskip}\def\CD@PI{\penalty5 \noindent\setbox\CD@MH=\null
\CD@mF\z@\CD@mI\z@\ifnum\CD@FA<\CD@LB\ht\CD@MH\ht\CD@FA\dp\CD@MH\dp\CD@FA
\unhbox\CD@FA\skip0=\lastskip\unskip\else\CD@OK\skip0=\z@\fi\endgraf\ifcase
\prevgraf\cd@shouldnt Y \or\cd@shouldnt Z \or\CD@RI\or\CD@XI\else\CD@QI\fi
\unskip\setbox0=\lastbox\unskip\unskip\unpenalty\noindent\unhbox0\setbox0%
\lastbox\unpenalty\unskip\unskip\unpenalty\setbox0\lastbox\CD@tF\CD@GB
\lastpenalty\unpenalty\ifnum\CD@GB>\z@\setbox\z@\lastbox\CD@lB\repeat\endgraf
\unskip\unskip\unpenalty}\def\CD@YJ{\CD@uA\CD@XB\advance\CD@uA-\CD@NB\CD@vA
\CD@FA\advance\CD@vA-\CD@lA\advance\CD@vA1 \expandafter\message{prevgraf=\the
\prevgraf at (\the\CD@uA,\the\CD@vA)}}\def\CD@XI{\CD@CE\setbox\CD@lI=\lastbox
\setbox\CD@lI=\hbox{\unhbox\CD@lI\unskip\unpenalty}\unskip\ifdim\ht\CD@lI>\ht
\CD@PC\setbox\CD@MH=\copy\CD@lI\else\ifdim\dp\CD@lI>\dp\CD@PC\setbox\CD@MH=%
\copy\CD@lI\else\CD@FG\CD@lI\fi\fi\advance\CD@mF.5\wd\CD@lI\advance\CD@mI.5%
\wd\CD@lI\setbox\CD@lI=\hbox{\unhbox\CD@lI\CD@HG}\CD@bH\CD@mF{\box\CD@lI}%
\CD@mI\z@\CD@yB\CD@vB}\def\CD@CE{\ifnum\CD@A>0 \advance\dimen0-\CD@tB\CD@iA-.%
5\dimen0 \CD@A-\CD@A\else\CD@A0 \CD@iA\z@\fi\setbox\CD@MH=\lastbox\setbox
\CD@MH=\hbox{\unhbox\CD@MH\unskip\unskip\unpenalty\setbox0=\lastbox\global
\CD@QA\lastkern\unkern}\advance\CD@iA-.5\CD@QA\unskip\setbox\CD@MH=\null
\CD@mI\CD@iA\CD@mF-\CD@iA}\def\CD@Z{\ht\CD@MH\CD@tI\dp\CD@MH\CD@sI}\def\CD@FG
#1{\setbox\CD@MH=\hbox{\CD@V{\ht\CD@MH}{\ht#1}\CD@V{\dp\CD@MH}{\dp#1}\CD@V{%
\wd\CD@MH}{\wd#1}\vrule height\ht\CD@MH depth\dp\CD@MH width\wd\CD@MH}}\def
\CD@QI{\CD@CE\CD@Z\setbox\CD@lI=\lastbox\unskip\setbox\CD@lF=\lastbox\unskip
\setbox\CD@lF=\hbox{\unhbox\CD@lF\unskip\global\CD@yA\lastpenalty\unpenalty}%
\advance\CD@yA9999 \ifcase\CD@yA\CD@VI\or\CD@YI\or\CD@TI\or\CD@dI\or\CD@cI\or
\CD@SI\else\cd@shouldnt9\fi}\def\CD@VI{\CD@FG\CD@lI\CD@UI\setbox\CD@sH=\box
\CD@lF\setbox\CD@tH=\box\CD@lI}\def\CD@YI{\CD@FG\CD@lF\setbox\CD@lI\hbox{%
\penalty8 \unhbox\CD@lI\unskip\unpenalty\ifnum\lastpenalty=8 \else\CD@xH\fi}%
\CD@UI\setbox\CD@lF=\hbox{\unhbox\CD@lF\unskip\unpenalty\global\setbox\CD@DA=%
\lastbox}\ifdim\wd\CD@lF=\z@\else\CD@xH\fi\setbox\CD@sH=\box\CD@DA}\def\CD@xH
{\CD@KB{extra material in \string\pile\space cell (lost)}}\def\CD@UI{\CD@yB
\ifvoid\CD@sH\else\CD@KB{Clashing horizontal arrows}\CD@mI.5\CD@hF\CD@mF-%
\CD@mI\CD@vB\CD@mI\z@\CD@mF\z@\fi\CD@hI\CD@hF\advance\CD@hI-\CD@mI\CD@hF-%
\CD@mF\CD@JC\CD@FA}\def\CD@RI{\setbox0\lastbox\unskip\CD@iA\z@\CD@Z\ifdim
\skip0>\z@\CD@tJ\CD@A0 \else\ifnum\CD@A<1 \CD@A0 \dimen0\CD@tB\fi\advance
\CD@A1 \fi}\def\VonH{\CD@MA46\VonH{.5\CD@LF}}\def\HonV{\CD@MA57\HonV{.5\CD@LF
}}\def\HmeetV{\CD@MA44\HmeetV{-\MapShortFall}}\def\CD@MA#1#2#3#4{\CD@pB34#1{%
\string#3}\CD@SD\CD@GB-999#2 \dimen0=#4\CD@tI\dimen0\advance\CD@tI\axisheight
\CD@sI\dimen0\advance\CD@sI-\axisheight\CD@CF\CD@HC\CD@ZD}\def\CD@HC#1{%
\setbox0=\hbox{\CD@k#1\CD@ND}\dimen0.5\wd0 \CD@tI\ht0 \CD@sI\dp0 \CD@ZD}\def
\CD@SD{\setbox0=\null\ht0=\CD@tI\dp0=\CD@sI\wd0=\dimen0 \copy0\penalty\CD@GB
\box0 }\def\CD@TI{\CD@GC\CD@yB}\def\CD@dI{\CD@GC\CD@vB}\def\CD@SI{\CD@GC
\CD@yB\CD@vB}\def\CD@GC{\setbox\CD@lI=\hbox{\unhbox\CD@lI}\setbox\CD@lF=\hbox
{\unhbox\CD@lF\global\setbox\CD@DA=\lastbox}\ht\CD@MH\ht\CD@DA\dp\CD@MH\dp
\CD@DA\advance\CD@mF\wd\CD@DA\advance\CD@mI\wd\CD@lI}\CD@tG
\ifPositiveGradient\CD@CI\CD@BI\CD@CI\CD@tG\ifClimbing\CD@rB\CD@qB\CD@rB
\newcount\DiagonalChoice\DiagonalChoice\m@ne\ifx\tenln\nullfont\CD@tJ\def
\CD@qF{\CD@KH\ifPositiveGradient/\else\CD@k\backslash\CD@ND\fi}\else\def
\CD@qF{\CD@rF\char\count@}\fi\let\CD@rF\tenln\def\Use@line@char#1{\hbox{#1%
\CD@rF\ifPositiveGradient\else\advance\count@64 \fi\char\count@}}\def\CD@cF{%
\Use@line@char{\count@\CD@TC\multiply\count@8\advance\count@-9\advance\count@
\CD@LH}}\def\CD@ZF{\Use@line@char{\ifcase\DiagonalChoice\CD@gF\or\CD@fF\or
\CD@fF\else\CD@gF\fi}}\def\CD@gF{\ifnum\CD@TC=\z@\count@'33 \else\count@
\CD@TC\multiply\count@\sixt@@n\advance\count@-9\advance\count@\CD@LH\advance
\count@\CD@LH\fi}\def\CD@fF{\count@'\ifcase\CD@LH55\or\ifcase\CD@TC66\or22\or
52\or61\or72\fi\or\ifcase\CD@TC66\or25\or22\or63\or52\fi\or\ifcase\CD@TC66\or
16\or36\or22\or76\fi\or\ifcase\CD@TC66\or27\or25\or67\or22\fi\fi\relax}\def
\CD@uC#1{\hbox{#1\setbox0=\Use@line@char{#1}\ifPositiveGradient\else\raise.3%
\ht0\fi\copy0 \kern-.7\wd0 \ifPositiveGradient\raise.3\ht0\fi\box0}}\def
\CD@jF#1{\hbox{\setbox0=#1\kern-.75\wd0 \vbox to.25\ht0{\ifPositiveGradient
\else\vss\fi\box0 \ifPositiveGradient\vss\fi}}}\def\CD@jI#1{\hbox{\setbox0=#1%
\dimen0=\wd0 \vbox to.25\ht0{\ifPositiveGradient\vss\fi\box0
\ifPositiveGradient\else\vss\fi}\kern-.75\dimen0 }}\CD@RC{+h:>}{%
\Use@line@char\CD@fF}\CD@RC{-h:>}{\Use@line@char\CD@gF}\CD@nF{+t:<}{-h:>}%
\CD@nF{-t:<}{+h:>}\CD@RC{+t:>}{\CD@jF{\Use@line@char\CD@fF}}\CD@RC{-t:>}{%
\CD@jI{\Use@line@char\CD@gF}}\CD@nF{+h:<}{-t:>}\CD@nF{-h:<}{+t:>}\CD@RC{+h:>>%
}{\CD@uC\CD@fF}\CD@RC{-h:>>}{\CD@uC\CD@gF}\CD@nF{+t:<<}{-h:>>}\CD@nF{-t:<<}{+%
h:>>}\CD@nF{+h:>->}{+h:>>}\CD@nF{-h:>->}{-h:>>}\CD@nF{+t:<-<}{-h:>>}\CD@nF{-t%
:<-<}{+h:>>}\CD@RC{+t:>>}{\CD@jF{\CD@uC\CD@fF}}\CD@RC{-t:>>}{\CD@jI{\CD@uC
\CD@gF}}\CD@nF{+h:<<}{-t:>>}\CD@nF{-h:<<}{+t:>>}\CD@nF{+t:>->}{+t:>>}\CD@nF{-%
t:>->}{-t:>>}\CD@nF{+h:<-<}{-t:>>}\CD@nF{-h:<-<}{+t:>>}\CD@RC{+f:-}{\CD@EF
\null\else\CD@cF\fi}\CD@nF{-f:-}{+f:-}\def\CD@tC#1#2{\vbox to#1{\vss\hbox to#%
2{\hss.\hss}\vss}}\def\hfdot{\CD@tC{2\axisheight}{.5em}}%
\def\vfdot{\CD@tC{1ex}\z@}%% % 1.46ex until 29.7.98
\def\CD@bF{\hbox{\dimen0=.3\CD@zC\dimen1\dimen0 \ifnum\CD@LH>\CD@TC\CD@iC{%
\dimen1}\else\CD@dG{\dimen0}\fi\CD@tC{\dimen0}{\dimen1}}}\newarrowfiller{.}%
\hfdot\hfdot\vfdot\vfdot\def\dfdot{\CD@bF\CD@CK}\CD@RC{+f:.}{\dfdot}\CD@RC{-f%
:.}{\dfdot}\def\CD@@K#1{\hbox\bgroup\def\CD@CH{#1\egroup}\afterassignment
\CD@CH%%
\count@='}\def\lnchar{\CD@@K\CD@qF}\def\CD@dF#1{\setbox#1=\hbox{\dimen5\dimen
#1 \setbox8=\box#1 \dimen1\wd8 \count@\dimen5 \divide\count@\dimen1 \ifnum
\count@=0 \box8 \ifdim\dimen5<.95\dimen1 \CD@gB{diagonal line too short}\fi
\else\dimen3=\dimen5 \advance\dimen3-\dimen1 \divide\dimen3\count@\dimen4%
\dimen3 \CD@dG{\dimen4}\ifPositiveGradient\multiply\dimen4\m@ne\fi\dimen6%
\dimen1 \advance\dimen6-\dimen3 \loop\raise\count@\dimen4\copy8 \ifnum\count@
>0 \kern-\dimen6 \advance\count@\m@ne\repeat\fi}}\def\CD@CG#1{\CD@EF\CD@xJ{#1%
}\else\CD@dF{#1}\fi}\def\CD@IH#1{}\newdimen\objectheight\objectheight1.8ex
\newdimen\objectwidth\objectwidth1em \def\CD@YD{\dimen6=\CD@aK
\DiagramCellHeight\dimen7=\CD@WK\DiagramCellWidth\CD@KJ\ifnum\CD@LH>0 \ifnum
\CD@TC>0 \CD@aF\else\aftergroup\CD@VC\fi\else\aftergroup\CD@UC\fi}\def\CD@VC{%
\CD@YA{diagonal map is nearly vertical}\CD@NA}\def\CD@UC{\CD@YA{diagonal map
is nearly horizontal}\CD@NA}\CD@rG\CD@NA{Use an orthogonal map instead}\def
\CD@aF{\CD@MJ\dimen3\dimen7\dimen7\dimen6\CD@iC{\dimen7}\advance\dimen3-%
\dimen7 \CD@MF\ifnum\CD@LH>\CD@TC\advance\dimen6-\dimen1\advance\dimen6-%
\dimen5 \CD@iC{\dimen1}\CD@iC{\dimen5}\else\dimen0\dimen1\advance\dimen0%
\dimen5\CD@dG{\dimen0}\advance\dimen6-\dimen0 \fi\dimen2.5\dimen7\advance
\dimen2-\dimen1 \dimen4.5\dimen7\advance\dimen4-\dimen5 \ifPositiveGradient
\dimen0\dimen5 \advance\dimen1-\CD@WK\DiagramCellWidth\advance\dimen1 \CD@ZK
\DiagramCellWidth\setbox6=\llap{\unhbox6\kern.1\ht2}\setbox7=\rlap{\kern.1\ht
2\unhbox7}\else\dimen0\dimen1 \advance\dimen1-\CD@ZK\DiagramCellWidth\setbox7%
=\llap{\unhbox7\kern.1\ht2}\setbox6=\rlap{\kern.1\ht2\unhbox6}\fi\setbox6=%
\vbox{\box6\kern.1\wd2}\setbox7=\vtop{\kern.1\wd2\box7}\CD@dG{\dimen0}%
\advance\dimen0-\axisheight\advance\dimen0-\CD@bK\DiagramCellHeight\dimen5-%
\dimen0 \advance\dimen0\dimen6 \advance\dimen1.5\dimen3 \ifdim\wd3>\z@\ifdim
\ht3>-\dp3\CD@TB\fi\fi\dimen3\dimen2 \dimen7\dimen2\advance\dimen7\dimen4
\ifvoid3 \else\CD@tE\else\dimen8\ht3\advance\dimen8-\axisheight\CD@iC{\dimen8%
}\CD@X{\dimen8}{.5\wd3}\dimen9\dp3\advance\dimen9\axisheight\CD@iC{\dimen9}%
\CD@X{\dimen9}{.5\wd3}\ifPositiveGradient\advance\dimen2-\dimen9\advance
\dimen4-\dimen8 \else\advance\dimen4-\dimen9\advance\dimen2-\dimen8 \fi\fi
\advance\dimen3-.5\wd3 \fi\dimen9=\CD@aK\DiagramCellHeight\advance\dimen9-2%
\DiagramCellHeight\CD@tE\advance\dimen2\dimen4 \CD@CG{2}\dimen2-\dimen0%
\advance\dimen2\dp2 \else\CD@CG{2}\CD@CG{4}\ifPositiveGradient\dimen2-\dimen0%
\advance\dimen2\dp2 \dimen4\dimen5\advance\dimen4-\ht4 \else\dimen4-\dimen0%
\advance\dimen4\dp4 \dimen2\dimen5\advance\dimen2-\ht2 \fi\fi\setbox0=\hbox to%
\z@{\kern\dimen1 \ifvoid1 \else\ifPositiveGradient\advance\dimen0-\dp1 \lower
\dimen0 \else\advance\dimen5-\ht1 \raise\dimen5 \fi\rlap{\unhbox1}\fi\raise
\dimen2\rlap{\unhbox2}\ifvoid3 \else\lower.5\dimen9\rlap{\kern\dimen3\unhbox3%
}\fi\kern.5\dimen7 \lower.5\dimen9\box6 \lower.5\dimen9\box7 \kern.5\dimen7
\CD@tE\else\raise\dimen4\llap{\unhbox4}\fi\ifvoid5 \else\ifPositiveGradient
\advance\dimen5-\ht5 \raise\dimen5 \else\advance\dimen0-\dp5 \lower\dimen0 \fi
\llap{\unhbox5}\fi\hss}\ht0=\axisheight\dp0=-\ht0\box0 }\def\NorthWest{\CD@BI
\CD@rB\DiagonalChoice0 }\def\NorthEast{\CD@CI\CD@rB\DiagonalChoice1 }\def
\SouthWest{\CD@CI\CD@qB\DiagonalChoice3 }\def\SouthEast{\CD@BI\CD@qB
\DiagonalChoice2 }\def\CD@aD{\vadjust{\CD@uA\CD@FA\advance\CD@uA
\ifPositiveGradient\else-\fi\CD@ZK\relax\CD@vA\CD@NB\advance\CD@vA-\CD@bK
\relax\hbox{\advance\CD@uA\ifPositiveGradient-\fi\CD@WK\advance\CD@vA\CD@aK
\hbox{\box6 \kern\CD@DC\kern\CD@eJ\penalty1 \box7 \box\z@}\penalty\CD@uA
\penalty\CD@vA}\penalty\CD@uA\penalty\CD@vA\penalty104}}\def\CD@eH#1{\relax
\vadjust{\hbox@maths{#1}\penalty\CD@FA\penalty\CD@NB\penalty\tw@}}\def\CD@lB{%
\ifcase\CD@GB\or\or\CD@bH{.5\wd0}{\box0}{.5\wd0}\z@\or\unhbox\z@\setbox\z@
\lastbox\CD@bH{.5\wd0}{\box0}{.5\wd0}\z@\unpenalty\unpenalty\setbox\z@
\lastbox\or\CD@TG\else\advance\CD@GB-100 \ifnum\CD@GB<\z@\cd@shouldnt B\fi
\setbox\z@\hbox{\kern\CD@mF\copy\CD@MH\kern\CD@mI\CD@uA\CD@XB\advance\CD@uA-%
\CD@NB\penalty\CD@uA\CD@uA\CD@FA\advance\CD@uA-\CD@lA\penalty\CD@uA\unhbox\z@
\global\CD@yA\lastpenalty\unpenalty\global\CD@zA\lastpenalty\unpenalty}\CD@uA
-\CD@yA\CD@vA\CD@zA\CD@fI\fi}\def\CD@TG{\unhbox\z@\setbox\z@\lastbox\CD@uA
\lastpenalty\unpenalty\advance\CD@uA\CD@mA\CD@vA\CD@XB\advance\CD@vA-%
\lastpenalty\unpenalty\dimen1\lastkern\unkern\setbox3\lastbox\dimen0\lastkern
\unkern\setbox0=\hbox to\z@{\unhbox0\setbox0\lastbox\setbox7\lastbox
\unpenalty\CD@eJ\lastkern\unkern\CD@DC\lastkern\unkern\setbox6\lastbox\dimen7%
\CD@tB\advance\dimen7-\wd\CD@uA\ifdim\dimen7<\z@\CD@CI\multiply\dimen7\m@ne
\let\mv\empty\else\CD@BI\def\mv{\raise\ht1}\kern-\dimen7 \fi\ifnum\CD@vA>%
\CD@NB\dimen6\CD@uB\advance\dimen6-\ht\CD@vA\else\dimen6\z@\fi\CD@jJ\CD@mK
\setbox1\null\ht1\dimen6\wd1\dimen7 \dimen7\dimen2 \dimen6\wd1 \CD@KJ\CD@uA
\CD@LH\CD@vA\CD@TC\dimen6\ht1 \CD@KJ\setbox2\null\divide\dimen2\tw@\advance
\dimen2\CD@eJ\CD@eG{\dimen2}\wd2\dimen2 \dimen0.5\dimen7 \advance\dimen0%
\ifPositiveGradient\else-\fi\CD@eJ\CD@dG{\dimen0}\advance\dimen0-\axisheight
\ht2\dimen0 \dimen0\CD@DC\CD@eG{\dimen0}\advance\dimen0\ht2\ht2\dimen0 \dimen
0\ifPositiveGradient-\fi\CD@DC\CD@dG{\dimen0}\advance\dimen0\wd2\wd2\dimen0
\setbox4\null\dimen0 .6\CD@zC\CD@eG{\dimen0}\ht4\dimen0 \dimen0 .2\CD@zC
\CD@dG{\dimen0}\wd4\dimen0 \dimen0\wd2 \ifvoid6\else\dimen1\ht4 \advance
\dimen1\ht2 \CD@CC6+-\raise\dimen1\rlap{\ifPositiveGradient\advance\dimen0-%
\wd6\advance\dimen0-\wd4 \else\advance\dimen0\wd4 \fi\kern\dimen0\box6}\fi
\dimen0\wd2 \ifvoid7\else\dimen1\ht4 \advance\dimen1-\ht2 \CD@CC7-+\lower
\dimen1\rlap{\ifPositiveGradient\advance\dimen0\wd4 \else\advance\dimen0-\wd7%
\advance\dimen0-\wd4 \fi\kern\dimen0\box7}\fi\mv\box0\hss}\ht0\z@\dp0\z@
\CD@bH{\z@}{\box\z@}{\z@}{\axisheight}}\def\CD@CC#1#2#3{\dimen4.5\wd#1 \ifdim
\dimen4>.25\dimen7\dimen4=.25\dimen7\fi\ifdim\dimen4>\CD@zC\dimen4.4\dimen4
\advance\dimen4.6\CD@zC\fi\CD@eG{\dimen4}\dimen5\axisheight\CD@dG{\dimen5}%
\advance\dimen4-\dimen5 \dimen5\dimen4\CD@eG{\dimen5}\advance\dimen0%
\ifPositiveGradient#2\else#3\fi\dimen5 \CD@dG{\dimen4}\advance\dimen1\dimen4 }
\def\CD@eD#1{\expandafter\CD@IK{#1}}\CD@ZA\CD@EK{output is PostScript
dependent}\def\CD@SC{\CD@IK{/bturn {gsave currentpoint currentpoint translate
4 2 roll neg exch atan rotate neg exch neg exch translate } def /eturn {%
currentpoint grestore moveto} def}}\def\CD@gK{\relax\CD@hK\CD@tK{Q}\else
\CD@IK{eturn}\fi} \def\CD@OJ#1{\count@#1\relax\multiply\count@7\advance
\count@16577\divide\count@33154 }\def\CD@fD#1{\expandafter\special{#1}} \def
\CD@xJ#1{\setbox#1=\hbox{\dimen0\dimen#1\CD@dG{\dimen0}\CD@OJ{\dimen0}\setbox
0=\null\ifPositiveGradient\count@-\count@\ht0\dimen0 \else\dp0\dimen0 \fi\box
0 \CD@uA\count@\CD@OJ\CD@LF\CD@fD{pn \the\count@}\CD@fD{pa 0 0}\CD@OJ{\dimen#%
1}\CD@fD{pa \the\count@\space\the\CD@uA}\CD@fD{fp}\kern\dimen#1}}\def\CD@JI{%
\CD@KJ\begingroup\ifdim\dimen7<\dimen6 \dimen2=\dimen6 \dimen6=\dimen7 \dimen
7=\dimen2 \count@\CD@LH\CD@LH\CD@TC\CD@TC\count@\else\dimen2=\dimen7 \fi
\ifdim\dimen6>.01\p@\CD@KI\global\CD@QA\dimen0 \else\global\CD@QA\dimen7 \fi
\endgroup\dimen2\CD@QA\CD@iK\CD@lK{\ifPositiveGradient\else-\fi\dimen6}\CD@iK
\CD@kK{\ifPositiveGradient-\fi\dimen6}\CD@iK\CD@eK{\dimen7}}\def\CD@KI{\CD@hJ
\ifdim\dimen7>1.73\dimen6 \divide\dimen2 4 \multiply\CD@TC2 \else\dimen2=0.%
353553\dimen2 \advance\CD@LH-\CD@TC\multiply\CD@TC4 \fi\dimen0=4\dimen2 \CD@ZG
4\CD@ZG{-2}\CD@ZG2\CD@ZG{-2.5}}\def\CD@AI{\begingroup\count@\dimen0 \dimen2 45%
pt \divide\count@\dimen2 \ifdim\dimen0<\z@\advance\count@\m@ne\fi\ifodd
\count@\advance\count@1\CD@@A\else\CD@y\fi\advance\dimen0-\count@\dimen2
\CD@gE\multiply\dimen0\m@ne\fi\ifnum\count@<0 \multiply\count@-7 \fi\dimen3%
\dimen1 \dimen6\dimen0 \dimen7 3754936sp \ifdim\dimen0<6\p@\def\CD@OG{4000}%
\fi\CD@KJ\dimen2\dimen3\CD@dG{\dimen2}\CD@hJ\multiply\CD@TC-6 \dimen0\dimen2
\CD@ZG1\CD@ZG{0.3}\dimen1\dimen0 \dimen2\dimen3 \dimen0\dimen3 \CD@ZG3\CD@ZG{%
1.5}\CD@ZG{0.3}\divide\count@2 \CD@gE\multiply\dimen1\m@ne\fi\ifodd\count@
\dimen2\dimen1\dimen1\dimen0\dimen0-\dimen2 \fi\divide\count@2 \ifodd\count@
\multiply\dimen0\m@ne\multiply\dimen1\m@ne\fi\global\CD@QA\dimen0\global
\CD@RA\dimen1\endgroup\dimen6\CD@QA\dimen7\CD@RA}\def\CD@OC{255}\let\CD@OG
\CD@OC\def\CD@KJ{\begingroup\ifdim\dimen7<\dimen6 \dimen9\dimen7\dimen7\dimen
6\dimen6\dimen9\CD@@A\else\CD@y\fi\dimen2\z@\dimen3\CD@XH\dimen4\CD@XH\dimen0%
\z@\dimen8=\CD@OG\CD@XH\CD@lC\global\CD@yA\dimen\CD@gE0\else3\fi\global\CD@zA
\dimen\CD@gE3\else0\fi\endgroup\CD@LH\CD@yA\CD@TC\CD@zA}\def\CD@lC{\count@
\dimen6 \divide\count@\dimen7 \advance\dimen6-\count@\dimen7 \dimen9\dimen4
\advance\dimen9\count@\dimen0 \ifdim\dimen9>\dimen8 \CD@@C\else\CD@AC\ifdim
\dimen6>\z@\dimen9\dimen6 \dimen6\dimen7 \dimen7\dimen9 \expandafter
\expandafter\expandafter\CD@lC\fi\fi}\def\CD@@C{\ifdim\dimen0=\z@\ifdim\dimen
9<2\dimen8 \dimen0\dimen8 \fi\else\advance\dimen8-\dimen4 \divide\dimen8%
\dimen0 \ifdim\count@\CD@XH<2\dimen8 \count@\dimen8 \dimen9\dimen4 \advance
\dimen9\count@\dimen0 \CD@AC\fi\fi}\def\CD@AC{\dimen4\dimen0 \dimen0\dimen9
\advance\dimen2\count@\dimen3 \dimen9\dimen2 \dimen2\dimen3 \dimen3\dimen9 }%
\def\CD@ZG#1{\CD@dG{\dimen2}\advance\dimen0 #1\dimen2 }\def\CD@dG#1{\divide#1%
\CD@TC\multiply#1\CD@LH}\def\CD@eG#1{\divide#1\CD@vA\multiply#1\CD@uA}\def
\CD@iC#1{\divide#1\CD@LH\multiply#1\CD@TC}\def\CD@hJ{\dimen6\CD@LH\CD@XH
\multiply\dimen6\CD@LH\dimen7\CD@TC\CD@XH\multiply\dimen7\CD@TC\CD@KJ}\def
\CD@iK#1#2{\begingroup\dimen@#2\relax\loop\ifdim\dimen2<.4\maxdimen\multiply
\dimen2\tw@\multiply\dimen@\tw@\repeat\divide\dimen2\@cclvi\divide\dimen@
\dimen2\relax\multiply\dimen@\@cclvi\expandafter\CD@jK\the\dimen@\endgroup
\let#1\CD@fK}{\catcode`p=12 \catcode`0=12 \catcode`.=12 \catcode`t=12 \gdef
\CD@jK#1pt{\gdef\CD@fK{#1}}}\ifx\errorcontextlines\CD@qK\CD@tJ\let\CD@GH
\relax\else\def\CD@GH{\errorcontextlines\m@ne}\fi\ifnum\inputlineno<0 \let
\CD@CD\empty\let\CD@W\empty\let\CD@mD\relax\let\CD@uI\relax\let\CD@vI\relax
\let\CD@zF\relax\else\def\CD@W{ at line \number\inputlineno}\def\CD@mD{ - first occurred%
}\def\CD@uI{\edef\CD@h{\the\inputlineno}\global\let\CD@jB\CD@h}\def\CD@h{9999%
}\def\CD@vI{\xdef\CD@jB{\the\inputlineno}}\def\CD@jB{\CD@h}\def\CD@zF{\ifnum
\CD@h<\inputlineno\edef\CD@CD{\space at lines \CD@h--\the\inputlineno}\else
\edef\CD@CD{\CD@W}\fi}\fi\let\CD@CD\empty\def\CD@YA#1#2{\CD@GH\errhelp=#2%
\expandafter\errmessage{\CD@tA: #1}}\def\CD@KB#1{\begingroup\expandafter
\message{! \CD@tA: #1\CD@CD}\ifnum\CD@XB>\CD@NB\ifnum\CD@CA>\CD@NB\else\ifnum
\CD@lA>\CD@FA\else\ifnum\CD@LB>\CD@FA\advance\CD@XB-\CD@NB\advance\CD@FA-%
\CD@lA\advance\CD@FA1\relax\expandafter\message{! (error detected at row \the
\CD@XB, column \the\CD@FA, but probably caused elsewhere)}\fi\fi\fi\fi
\endgroup}\def\CD@gB#1{{\expandafter\message{\CD@tA\space Warning: #1\CD@W}}}%
\def\CD@CB#1#2{\CD@gB{#1 \string#2 is obsolete\CD@mD}}\def\CD@AB#1{\CD@CB{%
Dimension}{#1}\CD@DE#1\CD@BB\CD@BB}\def\CD@BB{\CD@OA=}\def\CD@@B#1{\CD@CB{%
Count}{#1}\CD@DE#1\CD@OH\CD@OH}\def\CD@OH{\count@=}\def\HorizontalMapLength{%
\CD@AB\HorizontalMapLength}\def\VerticalMapHeight{\CD@AB\VerticalMapHeight}%
\def\VerticalMapDepth{\CD@AB\VerticalMapDepth}\def\VerticalMapExtraHeight{%
\CD@AB\VerticalMapExtraHeight}\def\VerticalMapExtraDepth{\CD@AB
\VerticalMapExtraDepth}\def\DiagonalLineSegments{\CD@@B\DiagonalLineSegments}%
\ifx\tenln\nullfont\CD@ZA\CD@KH{\CD@eF\space diagonal line and arrow font not
available}\else\let\CD@KH\relax\fi\def\CD@aG#1#2<#3:#4:#5#6{\begingroup\CD@PA
#3\relax\advance\CD@PA-#2\relax\ifdim.1em<\CD@PA\CD@uA#5\relax\CD@vA#6\relax
\ifnum\CD@uA<\CD@vA\count@\CD@vA\advance\count@-\CD@uA\CD@KB{#4 by \the\CD@PA
}\if#1v\let\CD@CH\CD@JK\edef\tmp{\the\CD@uA--\the\CD@vA,\the\CD@FA}\else
\advance\count@\count@\if#1l\advance\count@-\CD@A\else\if#1r\advance\count@
\CD@A\fi\fi\advance\CD@PA\CD@PA\let\CD@CH\CD@ZE\edef\tmp{\the\CD@NB,\the
\CD@uA--\the\CD@vA}\fi\divide\CD@PA\count@\ifdim\CD@CH<\CD@PA\global\CD@CH
\CD@PA\fi\fi\fi\endgroup}\CD@tG\CD@xE\CD@JD\CD@ID\CD@rG\CD@xI{See the message
above.}\CD@rG\CD@lH{Perhaps you've forgotten to end the diagram before
resuming the text, in\CD@uG which case some garbage may be added to the
diagram, but we should be ok now.\CD@uG Alternatively you've left a blank line
in the middle - TeX will now complain\CD@uG that the remaining \CD@S s are
misplaced - so please use comments for layout.}\CD@rG\CD@hD{You have already
closed too many brace pairs or environments; an \CD@HD\CD@uG command was (%
over)due.}\CD@rG\CD@hH{\CD@dC\space and \CD@HD\space commands must match.}%
\def\CD@jH{\ifnum\inputlineno=0 \else\expandafter\CD@iH\fi}\def\CD@iH{\CD@MD
\CD@GD\crcr\CD@YA{missing \CD@HD\space inserted before \CD@kH- type "h"}%
\CD@lH\enddiagram\CD@AG\CD@kH\par}\def\CD@AG#1{\edef\enddiagram{\noexpand
\CD@rD{#1\CD@W}}}\def\CD@rD#1{\CD@YA{\CD@HD\space(anticipated by #1) ignored}%
\CD@xI\let\enddiagram\CD@SG}\def\CD@SG{\CD@YA{misplaced \CD@HD\space ignored}%
\CD@hH}\def\CD@mC{\CD@YA{missing \CD@HD\space inserted.}\CD@hD\CD@AG{closing
group}}\ifx\ifx
\fi\fi

\catcode`\$=3 %% make sure that $ means maths-shift
\def\vboxtoz{\vbox to\z@}%% \z@ is in plain TeX and means 0pt

\def\scriptaxis#1{\@scriptaxis{$\scriptstyle#1$}}%%
\def\ssaxis#1{\@ssaxis{$\scriptscriptstyle#1$}}%%
\def\@scriptaxis#1{\dimen0\axisheight\advance\dimen0-\ss@axisheight\raise
\dimen0\hbox{#1}}\def\@ssaxis#1{\dimen0\axisheight\advance\dimen0-%
\ss@axisheight\raise\dimen0\hbox{#1}}

\ifx\boldmath\CD@qK%%
\let\boldscriptaxis\scriptaxis%%
\def\boldscript#1{\hbox{$\scriptstyle#1$}}%%
\else\def\boldscriptaxis#1{\@scriptaxis{\boldmath$\scriptstyle#1$}}%%
\def\boldscript#1{\hbox{\boldmath$\scriptstyle#1$}}%%
\fi

\def\raisehook#1#2#3{\hbox{\setbox3=\hbox{#1$\scriptscriptstyle#3$}%
\dimen0\ss@axisheight%% \scriptscriptstyle axis height
\dimen1\axisheight\advance\dimen1-\dimen0%% difference in axis heights
\dimen2\ht3\advance\dimen2-\dimen0%
\advance\dimen2-0.021em\advance\dimen1 #2\dimen2%
\raise\dimen1\box3}}%% print the character
\def\shifthook#1#2#3{\setbox1=\hbox{#1$\scriptscriptstyle#3$}\dimen0\wd1%
\divide\dimen0 12\CD@zH{\dimen0}%%  "u"
\dimen1\wd1\advance\dimen1-2\dimen0 \advance\dimen1-2\CD@oI\CD@zH{\dimen1}%
\kern#2\dimen1\box1}%% print

\def\@cmex{\mathchar"03}%%ascii double quote

\def\make@pbk#1{\setbox\tw@\hbox to\z@{#1}\ht\tw@\z@\dp\tw@\z@\box\tw@}\def
\CD@fH#1{\overprint{\hbox to\z@{#1}}}\def\CD@qH{\kern0.11em}\def\CD@pH{\kern0%
.35em}

\def\dblvert{\def\CD@rH{\kern.5\PileSpacing}}\def\CD@rH{}

\def\SEpbk{\make@pbk{\CD@qH\CD@rH\vrule depth 2.87ex height -2.75ex width 0.%
95em \vrule height -0.66ex depth 2.87ex width 0.05em \hss}}

\def\SWpbk{\make@pbk{\hss\vrule height -0.66ex depth 2.87ex width 0.05em
\vrule depth 2.87ex height -2.75ex width 0.95em \CD@qH\CD@rH}}

\def\NEpbk{\make@pbk{\CD@qH\CD@rH\vrule depth -3.81ex height 4.00ex width 0.%
95em \vrule height 4.00ex depth -1.72ex width 0.05em \hss}}

\def\NWpbk{\make@pbk{\hss\vrule height 4.00ex depth -1.72ex width 0.05em
\vrule depth -3.81ex height 4.00ex width 0.95em \CD@qH\CD@rH}}

\def\puncture{{\setbox0\hbox{A}\vrule height.53\ht0 depth-.47\ht0 width.35\ht
0 \kern.12\ht0 \vrule height\ht0 depth-.65\ht0 width.06\ht0 \kern-.06\ht0
\vrule height.35\ht0 depth0pt width.06\ht0 \kern.12\ht0 \vrule height.53\ht0
depth-.47\ht0 width.35\ht0 }}

\def\NEclck{\overprint{\raise2.5ex\rlap{ \CD@rH$\scriptstyle\searrow$}}}%%
\def\NEanti{\overprint{\raise2.5ex\rlap{ \CD@rH$\scriptstyle\nwarrow$}}}%%
\def\NWclck{\overprint{\raise2.5ex\llap{$\scriptstyle\nearrow$ \CD@rH}}}%%
\def\NWanti{\overprint{\raise2.5ex\llap{$\scriptstyle\swarrow$ \CD@rH}}}%%
\def\SEclck{\overprint{\lower1ex\rlap{ \CD@rH$\scriptstyle\swarrow$}}}%%
\def\SEanti{\overprint{\lower1ex\rlap{ \CD@rH$\scriptstyle\nearrow$}}}%%
\def\SWclck{\overprint{\lower1ex\llap{$\scriptstyle\nwarrow$ \CD@rH}}}%%
\def\SWanti{\overprint{\lower1ex\llap{$\scriptstyle\searrow$ \CD@rH}}}

\def\rhvee{\mkern-10mu\greaterthan}%%
\def\lhvee{\lessthan\mkern-10mu}%%
\def\dhvee{\vboxtoz{\vss\hbox{$\vee$}\kern0pt}}%%
\def\uhvee{\vboxtoz{\hbox{$\wedge$}\vss}}%%
\newarrowhead{vee}\rhvee\lhvee\dhvee\uhvee

\def\dhlvee{\vboxtoz{\vss\hbox{$\scriptstyle\vee$}\kern0pt}}%%
\def\uhlvee{\vboxtoz{\hbox{$\scriptstyle\wedge$}\vss}}%%
\newarrowhead{littlevee}{\mkern1mu\scriptaxis\rhvee}{\scriptaxis\lhvee}%
\dhlvee\uhlvee\ifx\boldmath\CD@qK%%
\newarrowhead{boldlittlevee}{\mkern1mu\scriptaxis\rhvee}{\scriptaxis\lhvee}%
\dhlvee\uhlvee\else%%
\def\dhblvee{\vboxtoz{\vss\boldscript\vee\kern0pt}}%%
\def\uhblvee{\vboxtoz{\boldscript\wedge\vss}}%%
\newarrowhead{boldlittlevee}{\mkern1mu\boldscriptaxis\rhvee}{\boldscriptaxis
\lhvee}\dhblvee\uhblvee%%
\fi

\def\rhcvee{\mkern-10mu\succ}%%
\def\lhcvee{\prec\mkern-10mu}%%
\def\dhcvee{\vboxtoz{\vss\hbox{$\curlyvee$}\kern0pt}}%%
\def\uhcvee{\vboxtoz{\hbox{$\curlywedge$}\vss}}%%
\newarrowhead{curlyvee}\rhcvee\lhcvee\dhcvee\uhcvee

\def\rhvvee{\mkern-13mu\gg}%% 24.8.92 changed 10mu to 13mu
\def\lhvvee{\ll\mkern-13mu}%% to make rule go through
\def\dhvvee{\vboxtoz{\vss\hbox{$\vee$}\kern-.6ex\hbox{$\vee$}\kern0pt}}%%
\def\uhvvee{\vboxtoz{\hbox{$\wedge$}\kern-.6ex \hbox{$\wedge$}\vss}}%%
\newarrowhead{doublevee}\rhvvee\lhvvee\dhvvee\uhvvee

\def\rhtriangle{\triangleright\mkern1.2mu}%% 29.1.93
\def\lhtriangle{\triangleleft\mkern.8mu}%%
\def\uhtriangle{\vbox{\kern-.2ex \hbox{$\scriptscriptstyle\bigtriangleup$}%
\kern-.25ex}}%%
\def\dhtriangle{\vbox{\kern-.28ex \hbox{$\scriptscriptstyle\bigtriangledown$}%
\kern-.1ex}}%% 15.1.93 from -.25ex
\def\dhblack{\vbox{\kern-.25ex\nointerlineskip\hbox{$\blacktriangledown$}}}%
\def\uhblack{\vbox{\kern-.25ex\nointerlineskip\hbox{$\blacktriangle$}}}%
\def\dhlblack{\vbox{\kern-.25ex\nointerlineskip\hbox{$\scriptstyle
\blacktriangledown$}}}%% AMS
\def\uhlblack{\vbox{\kern-.25ex\nointerlineskip\hbox{$\scriptstyle
\blacktriangle$}}}%% AMS
\newarrowhead{triangle}\rhtriangle\lhtriangle\dhtriangle\uhtriangle
\newarrowhead{blacktriangle}{\mkern-1mu\blacktriangleright\mkern.4mu}{%
\blacktriangleleft}\dhblack\uhblack\newarrowhead{littleblack}{\mkern-1mu%
\scriptaxis\blacktriangleright}{\scriptaxis\blacktriangleleft\mkern-2mu}%
\dhlblack\uhlblack

\def\rhla{\hbox{\setbox0=\lnchar55\dimen0=\wd0\kern-.6\dimen0\ht0\z@\raise
\axisheight\box0\kern.1\dimen0}}%%
\def\lhla{\hbox{\setbox0=\lnchar33\dimen0=\wd0\kern.05\dimen0\ht0\z@\raise
\axisheight\box0\kern-.5\dimen0}}%%
\def\dhla{\vboxtoz{\vss\rlap{\lnchar77}}}%%
\def\uhla{\vboxtoz{\setbox0=\lnchar66 \wd0\z@\kern-.15\ht0\box0\vss}}%% 1/93
\newarrowhead{LaTeX}\rhla\lhla\dhla\uhla

\def\lhlala{\lhla\kern.3em\lhla}%%
\def\rhlala{\rhla\kern.3em\rhla}%%
\def\uhlala{\hbox{\uhla\raise-.6ex\uhla}}%%
\def\dhlala{\hbox{\dhla\lower-.6ex\dhla}}%%
\newarrowhead{doubleLaTeX}\rhlala\lhlala\dhlala\uhlala

\def\hhO{\scriptaxis\bigcirc\mkern.4mu} \def\hho{{\circ}\mkern1.2mu}%
\newarrowhead{o}\hho\hho\circ\circ%%
\newarrowhead{O}\hhO\hhO{\scriptstyle\bigcirc}{\scriptstyle\bigcirc}%%

\def\rhtimes{\mkern-5mu{\times}\mkern-.8mu}\def\lhtimes{\mkern-.8mu{\times}%
\mkern-5mu}\def\uhtimes{\setbox0=\hbox{$\times$}\ht0\axisheight\dp0-\ht0%
\lower\ht0\box0 }\def\dhtimes{\setbox0=\hbox{$\times$}\ht0\axisheight\box0 }%
\newarrowhead{X}\rhtimes\lhtimes\dhtimes\uhtimes\newarrowhead+++++

\newarrowhead{Y}{\mkern-3mu\Yright}{\Yleft\mkern-3mu}\Ydown\Yup

\newarrowhead{->}\rightarrow\leftarrow\downarrow\uparrow

\newarrowhead{=>}\Rightarrow\Leftarrow{\@cmex7F}{\@cmex7E}

\newarrowhead{harpoon}\rightharpoonup\leftharpoonup\downharpoonleft
\upharpoonleft

\def\twoheaddownarrow{\rlap{$\downarrow$}\raise-.5ex\hbox{$\downarrow$}}%%
\def\twoheaduparrow{\rlap{$\uparrow$}\raise.5ex\hbox{$\uparrow$}}%%
\newarrowhead{->>}\twoheadrightarrow\twoheadleftarrow\twoheaddownarrow
\twoheaduparrow

\def\ltvee{\mkern-1mu{\lessthan}\mkern.4mu}%% \mkern added 15.1.93
\newarrowtail{vee}\greaterthan\ltvee\vee\wedge

\newarrowtail{littlevee}{\scriptaxis\greaterthan}{\mkern-1mu\scriptaxis
\lessthan}{\scriptstyle\vee}{\scriptstyle\wedge}\ifx\boldmath\CD@qK
\newarrowtail{boldlittlevee}{\scriptaxis\greaterthan}{\mkern-1mu\scriptaxis
\lessthan}{\scriptstyle\vee}{\scriptstyle\wedge}\else\newarrowtail{%
boldlittlevee}{\boldscriptaxis\greaterthan}{\mkern-1mu\boldscriptaxis
\lessthan}{\boldscript\vee}{\boldscript\wedge}\fi

\newarrowtail{curlyvee}\succ{\mkern-1mu{\prec}\mkern.4mu}\curlyvee\curlywedge

\def\rttriangle{\mkern1.2mu\triangleright}%% 29.1.93
\newarrowtail{triangle}\rttriangle\lhtriangle\dhtriangle\uhtriangle
\newarrowtail{blacktriangle}\blacktriangleright{\mkern-1mu\blacktriangleleft
\mkern.4mu}\dhblack\uhblack\newarrowtail{littleblack}{\scriptaxis
\blacktriangleright\mkern-2mu}{\mkern-1mu\scriptaxis\blacktriangleleft}%
\dhlblack\uhlblack

\def\rtla{\hbox{\setbox0=\lnchar55\dimen0=\wd0\kern-.5\dimen0\ht0\z@\raise
\axisheight\box0\kern-.2\dimen0}}%%
\def\ltla{\hbox{\setbox0=\lnchar33\dimen0=\wd0\kern-.15\dimen0\ht0\z@\raise
\axisheight\box0\kern-.5\dimen0}}%%
\def\dtla{\vbox{\setbox0=\rlap{\lnchar77}\dimen0=\ht0\kern-.7\dimen0\box0%
\kern-.1\dimen0}}%% 15.1.93 from -.6
\def\utla{\vbox{\setbox0=\rlap{\lnchar66}\dimen0=\ht0\kern-.1\dimen0\box0%
\kern-.6\dimen0}}%%
\newarrowtail{LaTeX}\rtla\ltla\dtla\utla

\def\rtvvee{\gg\mkern-3mu}%%
\def\ltvvee{\mkern-3mu\ll}%%
\def\dtvvee{\vbox{\hbox{$\vee$}\kern-.6ex \hbox{$\vee$}\vss}}%%
\def\utvvee{\vbox{\vss\hbox{$\wedge$}\kern-.6ex \hbox{$\wedge$}\kern\z@}}%%
\newarrowtail{doublevee}\rtvvee\ltvvee\dtvvee\utvvee

\def\ltlala{\ltla\kern.3em\ltla}%%
\def\rtlala{\rtla\kern.3em\rtla}%%
\def\utlala{\hbox{\utla\raise-.6ex\utla}}%%
\def\dtlala{\hbox{\dtla\lower-.6ex\dtla}}%%
\newarrowtail{doubleLaTeX}\rtlala\ltlala\dtlala\utlala

\def\utbar{\vrule height 0.093ex depth0pt width 0.4em}%%
\let\dtbar\utbar%%
\def\rtbar{\mkern1.5mu\vrule height 1.1ex depth.06ex width .04em\mkern1.5mu}%
\let\ltbar\rtbar%%
\newarrowtail{mapsto}\rtbar\ltbar\dtbar\utbar%%
\newarrowtail{|}\rtbar\ltbar\dtbar\utbar%%ascii vertical bar (|)

\def\rthooka{\raisehook{}+\subset\mkern-1mu}%%
\def\lthooka{\mkern-1mu\raisehook{}+\supset}%%
\def\rthookb{\raisehook{}-\subset\mkern-2mu}%%
\def\lthookb{\mkern-1mu\raisehook{}-\supset}%%

\def\dthooka{\shifthook{}+\cap}%%
\def\dthookb{\shifthook{}-\cap}%%
\def\uthooka{\shifthook{}+\cup}%%
\def\uthookb{\shifthook{}-\cup}%%

\newarrowtail{hooka}\rthooka\lthooka\dthooka\uthooka\newarrowtail{hookb}%
\rthookb\lthookb\dthookb\uthookb

\ifx\boldmath\CD@qK\newarrowtail{boldhooka}\rthooka\lthooka\dthooka\uthooka
\newarrowtail{boldhookb}\rthookb\lthookb\dthookb\uthookb\newarrowtail{%
boldhook}\rthooka\lthooka\dthookb\uthooka\else\def\rtbhooka{\raisehook
\boldmath+\subset\mkern-1mu}%%
\def\ltbhooka{\mkern-1mu\raisehook\boldmath+\supset}%%
\def\rtbhookb{\raisehook\boldmath-\subset\mkern-2mu}%%
\def\ltbhookb{\mkern-1mu\raisehook\boldmath-\supset}%%
\def\dtbhooka{\shifthook\boldmath+\cap}%%
\def\dtbhookb{\shifthook\boldmath-\cap}%%
\def\utbhooka{\shifthook\boldmath+\cup}%%
\def\utbhookb{\shifthook\boldmath-\cup}%%
\newarrowtail{boldhooka}\rtbhooka\ltbhooka\dtbhooka\utbhooka\newarrowtail{%
boldhookb}\rtbhookb\ltbhookb\dtbhookb\utbhookb\newarrowtail{boldhook}%
\rtbhooka\ltbhooka\dtbhooka\utbhooka\fi

\def\dtsqhooka{\shifthook{}+\sqcap}%%
\def\ltsqhooka{\mkern-1mu\raisehook{}+\sqsupset}%%
\def\rtsqhooka{\raisehook{}+\sqsubset\mkern-1mu}%%
\def\utsqhooka{\shifthook{}+\sqcup}%%
\newarrowtail{sqhook}\rtsqhooka\ltsqhooka\dtsqhooka\utsqhooka

\newarrowtail{hook}\rthooka\lthookb\dthooka\uthooka\newarrowtail{C}\rthooka
\lthookb\dthooka\uthooka

\newarrowtail{o}\hho\hho\circ\circ%%
\newarrowtail{O}\hhO\hhO{\scriptstyle\bigcirc}{\scriptstyle\bigcirc}%%

\newarrowtail{X}\lhtimes\rhtimes\uhtimes\dhtimes\newarrowtail+++++

\newarrowtail{Y}\Yright\Yleft\Ydown\Yup

\newarrowtail{harpoon}\leftharpoondown\rightharpoondown\upharpoonright
\downharpoonright

\newarrowtail{<=}\Leftarrow\Rightarrow{\@cmex7E}{\@cmex7F}

\newarrowfiller{=}=={\@cmex77}{\@cmex77}%% 16.1.93
\def\vfthree{\mid\!\!\!\mid\!\!\!\mid}%%ascii
\newarrowfiller{3}\equiv\equiv\vfthree\vfthree

\def\vfdashstrut{\vrule width0pt height1.3ex depth0.7ex}%%
\def\vfthedash{\vrule width\CD@LF height0.6ex depth 0pt}%%
\def\hfthedash{\CD@AJ\vrule\horizhtdp width 0.26em}%%
\def\hfdash{\mkern5.5mu\hfthedash\mkern5.5mu}%%
\def\vfdash{\vfdashstrut\vfthedash}%%
\newarrowfiller{dash}\hfdash\hfdash\vfdash\vfdash

\newarrowmiddle+++++

\iffalse%%
\newarrow{To}----{vee}%%
\newarrow{Arr}----{LaTeX}%%
\newarrow{Dotsto}....{vee}%%
\newarrow{Dotsarr}....{LaTeX}%%
\newarrow{Dashto}{}{dash}{}{dash}{vee}%%
\newarrow{Dasharr}{}{dash}{}{dash}{LaTeX}%%
\newarrow{Mapsto}{mapsto}---{vee}%%
\newarrow{Mapsarr}{mapsto}---{LaTeX}%%
\newarrow{IntoA}{hooka}---{vee}%%
\newarrow{IntoB}{hookb}---{vee}%%
\newarrow{Embed}{vee}---{vee}%%
\newarrow{Emarr}{LaTeX}---{LaTeX}%%
\newarrow{Onto}----{doublevee}%%
\newarrow{Dotsonarr}....{doubleLaTeX}%%
\newarrow{Dotsonto}....{doublevee}%%
\newarrow{Dotsonarr}....{doubleLaTeX}%%
\else%%
\newarrow{To}---->%%
\newarrow{Arr}---->%%
\newarrow{Dotsto}....>%%
\newarrow{Dotsarr}....>%%
\newarrow{Dashto}{}{dash}{}{dash}>%%
\newarrow{Dasharr}{}{dash}{}{dash}>%%
\newarrow{Mapsto}{mapsto}--->%%
\newarrow{Mapsarr}{mapsto}--->%%
\newarrow{IntoA}{hooka}--->%%
\newarrow{IntoB}{hookb}--->%%
\newarrow{Embed}>--->%%
\newarrow{Emarr}>--->%%
\newarrow{Onto}----{>>}%%
\newarrow{Dotsonarr}....{>>}%%
\newarrow{Dotsonto}....{>>}%%
\newarrow{Dotsonarr}....{>>}%%
\fi%%

\newarrow{Implies}===={=>}%% minimum cell height 9.5pt
\newarrow{Project}----{triangle}%%
\newarrow{Pto}----{harpoon}%% partial function
\newarrow{Relto}{harpoon}---{harpoon}%% binary relation

\newarrow{Eq}=====%%
\newarrow{Line}-----%%
\newarrow{Dots}.....%%
\newarrow{Dashes}{}{dash}{}{dash}{}%%

\newarrow{SquareInto}{sqhook}--->

\newarrowhead{cmexbra}{\@cmex7B}{\@cmex7C}{\@cmex3B}{\@cmex38}%%
\newarrowtail{cmexbra}{\@cmex7A}{\@cmex7D}{\@cmex39}{\@cmex3A}%%
\newarrowmiddle{cmexbra}{\braceru\bracelu}{\bracerd\braceld}{\vcenter{%
\hbox@maths{\@cmex3D\mkern-2mu}}}%% right
{\vcenter{\hbox@maths{\mkern2mu\@cmex3C}}}%% left
\newarrow{@brace}{cmexbra}-{cmexbra}-{cmexbra}%% braces
\newarrow{@parenth}{cmexbra}---{cmexbra}%% straight parentheses
\def\rightBrace{\d@brace[thick,cmex]}%%ASCII square brackets []
\def\leftBrace{\u@brace[thick,cmex]}%%ASCII square brackets []
\def\upperBrace{\r@brace[thick,cmex]}%%ASCII square brackets []
\def\lowerBrace{\l@brace[thick,cmex]}%%ASCII square brackets []
\def\rightParenth{\d@parenth[thick,cmex]}%%ASCII square brackets []
\def\leftParenth{\u@parenth[thick,cmex]}%%ASCII square brackets []
\def\upperParenth{\r@parenth[thick,cmex]}%%ASCII square brackets []
\def\lowerParenth{\l@parenth[thick,cmex]}%%ASCII square brackets []

\let\hEq\rEq%%
\let\vEq\uEq%%
\def\labelstyle{%%
\ifincommdiag%%
\textstyle%%
\else%%
\scriptstyle%%
\fi}%%
\let\objectstyle\displaystyle

\newdiagramgrid{pentagon}{0.618034,0.618034,1,1,1,1,0.618034,0.618034}{1.%
17557,1.17557,1.902113,1.902113}

\newdiagramgrid{perspective}{0.75,0.75,1.1,1.1,0.9,0.9,0.95,0.95,0.75,0.75}{0%
.75,0.75,1.1,1.1,0.9,0.9}

\diagramstyle[%%ascii open square bracket
dpi=300,%%              office laserwriters are usually 300 dots per inch
vmiddle,nobalance,%%    vertical and horizontal positioning
loose,%%                allow rows and columns to stretch
thin,%%                 line10 arrows; default rule thickness (TeXbook p447)
pilespacing=10pt,%
shortfall=4pt,%%        distance between arrowheads and their targets
]%%ascii close square bracket

\ifx\CD@qK\else\CD@PK\relax\CD@FF\CD@e\fi\fi

\CD@vE\CD@hK\else\fi\else\fi\cdrestoreat
\usepackage{amssymb}
\usepackage{mathrsfs}
\usepackage{cite}
\usepackage{tikz-cd}
\usepackage{MnSymbol}
\usepackage{a4wide}

\DeclareMathOperator\C{\mathbb C}
\DeclareMathOperator\Z{\mathbb Z}
\DeclareMathOperator\F{\mathbb F}

\newtheorem{theorem}{Theorem}[section]
\newtheorem{lemma}[theorem]{Lemma}
\newtheorem{cor}[theorem]{Corollary}
\newtheorem{prop}[theorem]{Proposition}
\newtheorem{claim}[theorem]{Claim}
\theoremstyle{definition}
\newtheorem{definition}[theorem]{Definition}

\theoremstyle{remark}
\newtheorem{remark}[theorem]{Remark}
\newtheorem{conjecture}[theorem]{Conjecture}

\numberwithin{equation}{section}

\newcommand{\dontprint}[1]\relax

\newcommand{\Id}{\operatorname{Id}}
\newcommand{\ind}{\operatorname{ind}}

\newcommand{\bSt}{{\bf St}}

\newcommand{\mC}{{\mathbb C}}

\newcommand{\bA}{{\mathbf A}}

\newcommand{\bO}{\Omega}

\newcommand{\Si}{\Sigma}

\newcommand{\mcC}{\mathcal C}

\newcommand{\mcR}{\mathcal R}

\newcommand{\ZZ}{\mathcal Z}

\newcommand{\fp}{\mathfrak p}

\newcommand{\ft}{\mathfrak t}
\newcommand{\ssl}{\mathfrak sl}

\newcommand{\strcusp}{\operatorname{str.cusp}}
\newcommand{\St}{\operatorname{St}}

\newcommand{\Q}{{\mathbb Q}}
\newcommand{\fg}{{\frak g}}
\newcommand{\fb}{{\frak b}}
\newcommand{\QBun}{\operatorname{QBun}}
\newcommand{\Ad}{\operatorname{Ad}}
\newcommand{\diag}{\operatorname{diag}}

\renewcommand{\div}{\operatorname{div}}
\newcommand{\coker}{\operatorname{coker}}

\newcommand{\De}{\Delta}

\newcommand{\Ga}{\Gamma}
\newcommand{\Aut}{\operatorname{Aut}}
\newcommand{\Div}{\operatorname{Div}}
\newcommand{\und}{\underline}
\newcommand{\Pic}{\operatorname{Pic}}
\newcommand{\hra}{\hookrightarrow}
\newcommand{\we}{\wedge}

\renewcommand{\P}{{\mathbb P}}
\newcommand{\A}{{\mathbb A}}
\renewcommand{\AA}{{\mathcal A}}
\newcommand{\wt}{\widetilde}
\newcommand{\ot}{\otimes}
\newcommand{\fu}{{\mathfrak u}}
\newcommand{\QTr}{\operatorname{QTr}}
\newcommand{\Hom}{\operatorname{Hom}}
\newcommand{\Ext}{\operatorname{Ext}}

\newcommand{\Om}{\Omega}
\newcommand{\TT}{{\mathcal T}}
\newcommand{\HH}{{\mathcal H}}

\newcommand{\NN}{{\mathcal N}}

\newcommand{\CC}{{\mathcal C}}

\renewcommand{\SS}{{\mathcal S}}

\newcommand{\LL}{{\mathcal L}}
\newcommand{\MM}{{\mathcal M}}
\newcommand{\OO}{{\mathcal O}}

\newcommand{\UU}{{\mathcal U}}

\newcommand{\si}{\sigma}
\newcommand{\de}{\delta}
\newcommand{\sub}{\subset}
\newcommand{\Spec}{\operatorname{Spec}}
\newcommand{\Res}{\operatorname{Res}}
\newcommand{\PGL}{\operatorname{PGL}}
\newcommand{\ov}{\overline}
\newcommand{\im}{\operatorname{im}}
\newcommand{\om}{\omega}

\renewcommand{\a}{\alpha}
\renewcommand{\b}{\beta}

\newcommand{\tr}{\operatorname{tr}}
\newcommand{\id}{\operatorname{id}}
\newcommand{\GL}{\operatorname{GL}}

\newcommand{\G}{{\mathbb G}}
\renewcommand{\th}{\theta}
\newcommand{\ga}{\gamma}
\newcommand{\lan}{\langle}
\newcommand{\ran}{\rangle}

\newcommand{\Tor}{{\operatorname{Tor}}}

\newcommand{\End}{{\operatorname{End}}}
\newcommand{\Bun}{{\operatorname{Bun}}}

\newcommand{\cusp}{{\operatorname{cusp}}}
\newcommand{\fm}{{\mathfrak m}}
\newcommand{\Dual}{{\mathbb D}}
\newcommand{\Hecke}{{\operatorname{Hecke}}}

\newcommand{\supp}{{\operatorname{supp}}}

\newcommand{\Tr}{{\operatorname{Tr}}}
\newcommand{\eps}{\epsilon}
\newcommand{\Gal}{{\operatorname{Gal}}}
\newcommand{\vol}{{\operatorname{vol}}}

\title{Automorphic functions for nilpotent extensions of curves over finite fields}

\author{Alexander Braverman}
\author{David Kazhdan}
\author{Alexander Polishchuk}
\thanks{D.K. is partially supported by the ERC grant No 101142781.
A.P. is partially supported by the NSF grant DMS-2349388, by the Simons Travel grant MPS-TSM-00002745,
and within the framework of the HSE University Basic Research Program}

\address{Deparment of Mathematics, University of Toronto, Ontario, Canada}
\email{braval@math.toronto.edu}
\address{Einstein Institute of Mathematics,
The Hebrew University of Jerusalem,
Jerusalem 91904, Israel}
\email{kazhdan@math.huji.ac.il}
\address{
    Department of Mathematics, 
    University of Oregon,     
    Eugene, OR 97403, USA; National Research University Higher School of Economics
  }
  \email{apolish@uoregon.edu}

\begin{document}

\maketitle

\begin{abstract}
We define and study the subspace of cuspidal functions for $G$-bundles (and its ramified analogs) on a class of nilpotent extensions of curves
over a finite field.
%We show that this subspace is preserved by the action of a certain noncommutative Hecke algebra $\HH_{G,C}$.  In the case $G=\GL_2$, 
%we construct a commutative subalgebra in $\HH_{G,C}$
%of Hecke operators associated with {\it simple divisors}. 
In the case of length $2$ extensions and of $G=\GL_2$, 
we prove that the space of cuspidal functions (for bundles with
a fixed determinant) is finite-dimensional and provide bounds on its dimension. In this case we also relate cuspidality to Hecke finiteness,
defined in terms of the action of some natural Hecke operators. Finally, we construct some Hecke eigenfunctions using the relation
to Higgs bundles over the corresponding reduced curve.
\end{abstract}

\section{Introduction}

\subsection{Langlands correspondence for curves over local fields}

The Langlands correspondence for the group $G=\GL_2$ and a smooth projective curve $C$ over a finite field $k$ was established by Drinfeld in \cite{Drinfeld-aut}.
%Let $G=\GL_2$ or $\SL_2$ (and $G^\vee=\GL_2$ or $\PGL_2$, respectively). 
The unramified case of the correspondence involves an action of a commutative Hecke algebra $\HH$ on the space $\C(\Bun_G(k))$ of 
functions on the set $\Bun_G(k)$ of isomorphism classes of $G$-bundles over $C$ (defined over $k$). 
There is a natural subspace $\C_{\cusp}(\Bun_G(k)) \sub \C(\Bun_G(k))$ of {\it cuspidal} functions, and Drinfeld's result implies that
$\C_{\cusp}(\Bun_G(k))$ has an $\HH$-eigenbasis parametrized by conjugacy classes of irreducible homomorphisms
$\pi_1(C)\to \GL_2(\ov{\Q}_l)$, continuous in $l$-adic topology (i.e., by irreducible rank $2$ local systems on $C$). Drinfeld's work has been generalized to $\GL_n$ and then (partially) to all reductive groups
(see \cite{Laf-L}, \cite{Laf-V}).
%with trivial determinant).

The works \cite{BK-Hecke}, \cite{EFK1} and \cite{EFK2} develop a conjectural analog of this picture with the finite field $k$ replaced by a local field $K$ (see also \cite{BK}).
In the case of parabolic bundles on $\P^1$ the Hecke operators over local fields were also considered by Kontsevich \cite{Ko}.
Essentially, the goal of this program, outlined in \cite{BK}, is to define an appropriate space of functions (or rather half-densities) on $\Bun_G(K)$ together with an action of commuting Hecke operators on it,
and to study the corresponding spectrum.

%Since the stack $\Bun_G$ is ind-proper, we may, 
For a curve $C$ over a non-archimedean local field $K$, following \cite{GK}
one can define the Schwartz space of half-densities on $\Bun_G(K)$ (where we consider only generically trivial $G$-bundles). 
In \cite{BKP} we show that in the case when $C$ has a smooth model over the ring of integers $O$ in $K$, 
all elements of this space come from
functions on $\Bun_G(O)$ depending on the reduction modulo some power of the maximal ideal $\fm^n\sub O$. 
%(see \cite[Sec.\ 5.8]{BK} for related constructions). 
This naturally leads to the problem of studying automorphic functions and Hecke operators for curves over finite rings such as $O/\fm^n$, or
more generally, over nilpotent extensions of curves over finite fields.
This is the main subject of the present paper.

\subsection{Special nilpotent extensions of curves: adeles and $G$-bundles}\label{nilp-adele-sec}

%Let $A$ be a finite commutative local ring such that the maximal ideal $\fm\sub A$ is principal, with $\eps\in \fm$ a generator, so that $k:=A/(\eps)$ is the
%finite residue field.
%For example, one can take $A$ to be a finite quotient of the ring of integers $O$ of a non-archimedean field $K$.
%a discrete valuation ring with finite residue field.
%Let $C$ be a smooth proper curve over $A$.

Let 
%$C$ be a proper scheme over $\Z$ such that the corresponding reduced scheme 
$\ov{C}$ be a geometrically connected smooth proper curve over a finite field $k$, and let
%Then the nilradical of $\OO_C$ is nilpotent, so we view 
$C$ be a nilpotent extension of $\ov{C}$, i.e., a scheme of finite type such that the corresponding reduced scheme is $C$.
We denote by $\NN\sub \OO_C$ the nilradical.
%We denote by $\ov{C}:=C\times_{\Spec A} \Spec k$ the corresponding curve over $k$, the residue field of $A$. 
%We always assume that $\ov{C}$ is geometrically irreducible.

%We distinguish the following classes of such schemes.

\begin{definition}
%(i) We say that $C$ is a {\it principal nilpotent extension} of $\ov{C}$ if the nilradical $\NN\sub\OO_C$ is locally principal, i.e., locally generated by one element.
%\noindent
%(ii) 
We say that $C$ is a {\it special nilpotent extension} of $\ov{C}$ if 
%it is principal and in addition 
there exists $n\ge 1$ such that $\NN^n=0$, and 
$\NN^i/\NN^{i+1}$ is a line bundle on $\ov{C}$ for $i=1,\ldots,n-1$. We refer to $n$ as the {\it length} of the nilpotent extension.
\end{definition}

For example, if $C$ is a smooth proper curve over a finite quotient of the ring of integers $O$ of a non-archimedean local field $K$ then it is a special nilpotent extension of $\ov{C}$.

From now on we always assume that $C$ is a special nilpotent extension of $\ov{C}$.
We denote by $F=F_C$ the stalk of the structure sheaf of $C$ at the general point.
Our assumption that $C$ is a special nilpotent extension implies that $F$ coincides with the total ring of fractions of the ring of functions on any nonempty affine open of $C$.
We also set $\ov{F}:=F_{\ov{C}}=k(\ov{C})$.

For every closed point $p\in \ov{C}$, we consider $\OO_p$, the completion of the local ring of $C$ at $p$ with respect to the maximal ideal.
Note that the quotient of $\OO_p$ by its nilradical is $\ov{\OO}_p$, the completion of the local ring of $\ov{C}$ at $p$.
We denote by $F_p$ the total ring of fractions of $\OO_p$, and we define the ring of adeles $\A=\A_C$ as the restricted product of $F_p$ over all points $p$,
with respect to the subrings $\OO_p\sub F_p$. We denote by $\OO\sub \A$ the ring of integer adeles. We have a natural embedding $F\sub \A$.
Let $\NN^i\A\sub \A$ (resp., $\NN^i_F\sub F$) be the $i$th power of the nilradical. 
The quotient $\ov{\A}:=\A/\NN\A$ is the usual ring of adeles for $\ov{C}$, and the quotient $\ov{\OO}$ of $\OO$ by its nilradical is the ring of integer adeles for $\ov{C}$.

%In the case when $\NN^2=0$, we will write elements of $\NN\A$ as $\eps x$, where $x$ is an $\NN/\NN^2$-valued adele on $\ov{C}$, so $\eps$ here is just a formal variable
%such that $\eps^2=0$.

Let $G$ be a connected group defined over $\Z$. We denote by $\Bun_G(C)$
the groupoid of $G$-bundles over $C$ trivial at the general point of $C$.
\footnote{If $G$ is connected split reductive, then any $G$-bundle is automatically trivial at the general point, as follows from \cite[Prop.\ 4.5]{BD}.}
Let us recall the adelic interpretation of $\Bun_G(C)$. 
Given a $G$-bundle $P$ on $C$ trivial at the general point (thought of as a right $G$-torsor), we fix a trivialization $e_\eta$ of $P$ at the general point and
trivializations $e_p$ of the restriction of $P$ to the formal completion of $C$ at each point $p$.
Writing $e_p=e_\eta\cdot g_p$, we get $g=(g_p)\in G(\A)$. Changing trivializations $e_p$ corresponds to
multiplying $g$ on the right with $G(\OO)$, while changing $e_\eta$ corresponds to multiplying $g$ on the left with $G(F)$.
The following result is well known in the case $C=\ov{C}$ and is proved similarly in general.

\begin{prop} If $C$ is a special nilpotent extension of $\ov{C}$ then
the above construction gives an equivalence of groupoids
\begin{equation}\label{BunG-double-cosets-eq}
\Bun_G(C)\simeq G(F)\backslash G(\A)/G(\OO),
\end{equation}
where the set of double cosets is naturally viewed as a groupoid, see Sec.\ \ref{double-coset-sec}.
\end{prop}

For every open compact subgroup $K\sub G(\A)$, we set
$$\Bun_G(C,K):=G(F)\backslash G(\A)/K.$$

In the case of $G=\GL_n$, we can consider vector bundles instead of $G$-bundles. In this case we associate with $g\in \GL_n(\A)$
the vector bundle $V=V(g)$ where one 
has $V_\eta=F^n$ and $V_p=g_p\OO_p^n\sub F_p^n$ (here $V_p$ is the completion of the stalk of $V$ at $p$).

As in the case of curves over a finite field, we have the following interpretation of the cohomology $H^*(C,V)$ for $V=V(g)$ with $g\in \GL_n(\A)$:
\begin{equation}\label{cohomology-eq}
H^0(C,V)\simeq F^n\cap g\OO^n, \ \ H^1(C,V)\simeq \A^n/(F^n+g\OO^n).
\end{equation}

When integrating functions over $H(\A)$, where $H$ is an algebraic subgroup of $G$ defined over $\Z$, we always normalize the Haar measure so that $\vol(H(\OO))=1$.
For a vector bundle $V$ on $C$ we set  
$$\ov{V}:=V|_{\ov{C}}=V/\NN V.$$     
    
\subsection{Cuspidal functions and Hecke operators: definitions and conjectures}\label{gen-conj-sec}

%Let $\Bun_G(C)$ denote the set of isomorphism classes of $G$-bundles over $C$.
%We define the ring of adeles $\A=\A_C$ associated with $C$, using appropriate localizations of the completed local rings of $C$ at all closed points.
%We denote by $\OO\sub\A$ the subring of integer adeles, and by $F=F_C$ principal adeles, i.e., the stalk of the structure sheaf of $C$ at the general point.
%As in the case of a finite field, one has a natural bijection
%identification of groupoids
%(where we use a natural groupoid structure on double cosets).

For a small groupoid $X$ we denote by $\SS(X)$ (resp., $\C(X)$) the space of finitely supported (resp., all) $\C$-valued 
functions on the set of isomorphism classes of $X$.

Below we define the subspace $\C_{\cusp}(\Bun_G(C))\sub\C(\Bun_G(C))$ of cuspidal functions invariant under the action of a certain algebra of Hecke operators on $\Bun_G(C)$.
More generally, for each compact open subgroup $K\sub G(\A)$, we can define a similar subspace $\C_{\cusp}(\Bun_G(C,K))\sub\C(\Bun_G(C,K))$.

Let $G$ be a connected reductive group over $\Z$.
One could define cuspidal functions for an arbitrary such $G$, but to simplify notations, we assume from now on that
$G$ is split, and we denote by $T\sub G$ a split maximal torus (defined over $\Z$).

In analogy with \cite{K-YD}, which dealt with the local case, for each compact open subgroup $K\sub G(\A)$
we consider the {\it Hecke algebra}, $\HH_{G,C,K}$ of compactly supported, $K$-biinvariant measures on $G(\A)$.
For $K=G(\OO)$ we denote this algebra simply as $\HH_{G,C}$
(note that this algebra is not commutative when $C$ is not a curve over a finite field).
The algebra $\HH_{G,C,K}$ acts by convolution on the space $\C(\Bun_G(C,K))$,
%$\C(\Bun_G(C))$, 
preserving the subspace $\SS(\Bun_G(C,K))$.
%$\SS(\Bun_G(C))$.  
We denote this action as $h*f$, where $h\in \HH_{G,C,K}$, $f\in \C(\Bun_G(C,K))$.

Let $B=TU$ be a Borel subgroup. A parabolic subgroup is {\it standard} if it contains $B$.

\begin{definition}\label{QBun-L-def}
For a standard parabolic subgroup $P\supset B$ (defined over $\Z$) with the unipotent radical $U$ and a Levi subgroup $L\sub P$, we set
$$\QBun_L(C,K):=(L(F)U(\A))\backslash G(\A)/K,$$
$\QBun_L(C):=\QBun_L(C,G(\OO))$.
\end{definition}

%Note that the set $\QBun_L(\ov{C})$ can be identified with the set of isomorphism classes of $L$-bundles on $\ov{C}$.

\begin{definition}\label{cusp-def}
\begin{enumerate}
\item For a standard parabolic subgroup $P$, the constant term operator 
$$C_P:\C(\Bun_G(C,K))\to \C(\QBun_L(C,K))$$
is given by the formula $C_Pf(g)=\int_{u\in U_P(F)\backslash U_P(\A)} f(ug)du$.
\footnote{This integral can be rewritten as a finite sum over $U_P(F)\backslash U_P(\A)/(U_P(\A)\cap gKg^{-1})$.}

\item
More generally, for each $i\ge 0$, we define the subgroup $U_P(\NN^i\A)\sub U_P(\A)$ (resp., $U_P(\NN_F^i)\sub U_P(F)$) as the kernel
of the projection $U_P(\A)\to U_P(\A/\NN^i\A)$ (resp., $U_P(F)\to U_P(F/\NN^i_F)$). We define the corresponding constant term operator by
$$C^i_Pf(g)=\int_{u\in U_P(\NN^i_F)\backslash U_P(\NN^i\A)} f(ug)du,$$ 
%which shows that it is well defined.

\item
 We say that a function $f$ on $\Bun_G(C,K)$ is {\it cuspidal} if $C_Pf\equiv 0$ for every proper standard parabolic subgroup $P$ containing $B$.
%on $\QBun_L(C)$.

\item
$\C_{\cusp}(\Bun_G(C,K))\sub \C(\Bun_G(C,K))$ denotes the subspace of cuspidal functions.

\item
%We denote by $\SS_{\cusp}(\Bun_\LL(C))\sub \C(\Bun_\LL(C))$ the subspace of cuspidal functions supported on $\Bun_\LL(C)$.
For $i\ge 0$, we say that $f$ is {\it $i$-cuspidal} if $C^i_Pf\equiv 0$ for every proper standard parabolic subgroup $P$ containing $B$. 
In the case $n=2$, we also call $1$-cuspidal functions {\it strongly cuspidal}.

\item
$\C^i_{\cusp}(\Bun_G(C,K))\sub \C(\Bun_G(C,K))$ denotes the subspace of $i$-cuspidal functions.
\end{enumerate}
\end{definition}

For a closed embedding $C'\sub C$ (where $C$ and $C'$ are special nilpotent extensions of the reduced curve $\ov{C}$),
%proper ideal $I\sub A$, set $C_{A/I}:=C\times_{\Spec A}\Spec (A/I)$.
we have a natural reduction map $\Bun_G(C)\to \Bun_G(C')$, which induces 
the embedding $\C(\Bun_G(C'))\hra \C(\Bun_G(C))$ as the subspace of functions depending only on the reduction of a $G$-bundle to $C'$.

\bigskip

\noindent
{\bf Proposition A}. {\it (1) The action of the Hecke algebra $\HH_{G,C,K}$ on $\C(\Bun_G(C,K))$ preserves the subspace of cuspidal (resp., $i$-cuspidal) functions.

\noindent
(2) For a closed embedding $C'\sub C$ of special nilpotent extensions, we have
$$\C_{\cusp}(\Bun_G(C'))=\C(\Bun_G(C'))\cap \C_{\cusp}(\Bun_G(C)).$$
}

\bigskip

Let $G'\sub G$
%Consider the quotient $G/G'$ by 
be the commutator subgroup. 
Then we have a natural projection $\Bun_G(C)\to \Bun_{G/G'}(C)$.
For each $G/G'$-bundle $L$ over $C$, we denote by $\Bun^L_{G}(C)\sub \Bun_G(C)$ the subset of isomorphism classes of 
$G$-bundles whose associated $G/G'$-bundle is isomorphic to $L$.
For every open compact subgroup $K\sub G(\OO)$, we denote by $\Bun^L_G(C,K)\sub \Bun_G(C,K)$ the preimage of $\Bun^L_G(C)$ under the
natural projection $\Bun_G(C,K)\to \Bun_G(C)$.
For $L\in \Bun_{G/G'}(C)$, we always view functions on $\Bun^L_{G}(C,K)$ as functions on $\Bun_G(C,K)$ (extended by zero).

%Note that $\Bun_{G/G'}(C)$ has a natural group structure, and the algebra
%$\HH_{G,C}$ is naturally $\Bun_{G/G'}(C)$-graded.

On the other hand, let $Z\sub G$ be the center of $G$, and let $Z_0\sub Z$ be the connected component of $1$. Then $Z_0\to G/G'$ is an isogeny.
The group $\Bun_{Z_0}(C,K\cap Z_0)$ acts naturally on $\Bun_G(C,K)$.
%and there is an induced action on $\C(\Bun_G(C))$.
We denote this action as $M+P$, where $M\in \Bun_{Z_0}(C,K\cap Z_0)$ and $P\in \Bun_G(C,K)$. 
%is a $Z_0$-bundle and $P$ is a $G$-bundle.

\begin{definition}
A function $f\in \C(\Bun_G(C,K))$ is called {\it Hecke-bounded} if there exists a finite set $S\sub \Bun_G(C,K)$, such that 
 the support of $h*f$ is contained in $\Bun_{Z_0}(C,K\cap Z_0)+S$ for every $h\in \HH_{G,C,K}$.
%We denote the subspace of Hecke-bounded functions as $\SS_b(\Bun_\LL(C))\sub \SS(\Bun_\LL(C))$.
\end{definition}

%\bigskip

%\noindent
\begin{conjecture}\label{main-cusp-conj} 
Assume that $C$ is a special nilpotent extension of $\ov{C}$ of length $n$, and that the genus $g$ of $\ov{C}$ is $\ge 2$.
Let $K\sub G(\OO)$ be an  open subgroup.

\begin{enumerate}

\item
 The $\C$-vector space
$$\C_{\cusp}(\Bun^L_{G}(C,K)):=\C_{\cusp}(\Bun_G(C,K))\cap \C(\Bun^L_{G}(C,K))$$
is finite-dimensional for any $L\in \Bun_{G/G'}(C)$. 

\item
If $\deg(\NN/\NN^2)=0$ then there exists a constant $c>0$ and positive integer $r$ (both depending on the genus of $\ov{C}$, $G$, $L$ and $n$, but not on $q$) such that
$$|\frac{\dim \C_{\cusp}(\Bun^L_{G}(C))}{q^{n\dim(G')(g-1)}}-r| \le c q^{-1/2},$$
where 
%$n$ is the length of the ring $A$, $g$ is the genus of $\ov{C}$ and 
$q=|k|$.

\item
A function $f$ on $\Bun^L_{G}(C,K)$ is cuspidal if and only if it is Hecke-bounded.

\end{enumerate}
\end{conjecture}
%}
%\bigskip

Note that part (3) implies that any cuspidal function in $\Bun^L_G(C,K)$ has finite support. On the other hand, the ``only if" part in (3) follows from part (1) together with the statement that any cuspidal function in $\Bun^L_G(C,K)$ has
finite support.

In the case when $C$ is a curve over a finite field and $G=\PGL_2$ or $G=\GL_2$, the finite dimensionality is well known, 
while the dimension estimate is computed by Schleich in \cite[Sec.\ 3.3]{Schleich}.

\subsection{The case of $G=\GL_2$: unramified commuting Hecke operators}

From now till the end of this section we assume that $G=\GL_2$ and write $\Bun(C,K)$ and $\Bun^L(C,K)$ instead of $\Bun_G(C,K)$ and $\Bun^L_{G}(C,K)$.
% (for $K\sub G(\OO)$).
In this subsection we assume that $K=G(\OO)$ and
define a commuting family of Hecke operators associated with certain divisors on $C$.
We always assume that $C$ is a special nilpotent extension of $\ov{C}$.

\begin{definition}\label{irr-et-div-def}
A {\it simple divisor} $c\sub C$ is an effective Cartier divisor $c\sub C$
such that $\ov{c}:=c\cap \ov{C}$ (the scheme-theoretic intersection in $C$)
%=c\times_{\Spec A}\Spec k$ 
is a reduced point in $\ov{C}$.
\end{definition}  

Every point $\ov{c}\in \ov{C}$ can be lifted to a simple divisor (non uniquely).
%We will show that $c\sub C$ is a simple divisor if and only if it is an irreducible affine closed subscheme, \'etale over $\Spec(A)$ (see Lemma \ref{irr-et-div-lem}).

For a simple divisor $c\sub C$, we denote by $f_c$ an idele given by a local generator of the ideal of $c$ at the place $\ov{c}$ and trivial at all other places. 

\begin{definition}
We denote by $h_c\in \HH_{G,C}$ the characteristic measure of the double class 
\begin{equation}\label{Tc-double-coset}
G(\OO)\left(\begin{matrix} f_c^{-1} & 0 \\ 0 & 1\end{matrix}\right)G(\OO),
\end{equation}
where the Haar measure is normalized so that $G(\OO)$ has measure $1$. 
%and say that $h_c$ is a simple Hecke operator.
We define the corresponding Hecke operator on $\Bun(C)$ by
$$T_c:\C(\Bun(C))\to \C(\Bun(C)): f\mapsto h_c*f.$$
\end{definition}

Note that for an open compact subgroup $K\sub G(\OO)$ the same construction of $h_c\in \HH_{G,C,K}$ (resp., operator $T_c$ on $\C(\Bun(C,K))$) still makes sense
provided $K$ is unramifed at $\ov{c}$ (i.e., $K$ contains $G(\OO_{\ov{c}})$).

\bigskip

\noindent
{\bf Theorem B}. {\it Assume that $C$ is a special nilpotent extension of $\ov{C}$. Then for any pair of simple divisors $c$ and $c'$, one has
$h_ch_{c'}=h_{c'}h_c$ in $\HH_{G,C}$.
}

\bigskip

For a line bundle $L$ on $C$ we denote by $t_L^*$ the operator on $\C(\Bun(C))$ given by
$$t_L^*f(V)=f(V\ot L).$$
These operators come from central elements $h_L$ of $\HH_{G,C}$. Thus, the elements $(h_L)$ and $(h_c)$ generate a commutative subalgebra of $\HH_{G,C}$.
%Combining both types of operators we can get a family of commuting operators acting on a conjecturally finite-dimensional space 
%$\Bun^L_{G}(C)$ for a fixed line bundle $L$.

\begin{remark}
Our proof of Theorem B is purely local. 
In \cite{BKP} we give a global proof of a similar result for any split connected reductive group $G$ such that $[G,G]$ is simple and simply connected
%we will generalize the construction of this commutative subalgebra to other reductive groups 
%in the case when $C$ is a smooth curve over $O/\fm^n$, where $O$ is the ring of integers in
%a local nonarchimedean field, under the assumption that the characteristic of $O/\fm$ 
assuming that the residue characteristic is large enough (e.g., for $G=\GL_n$ there is no restriction on characteristic). 
In \cite{BKP} we also relate our Hecke operators over $O/\fm^n$
with those arising in the theory of representations of the ind-pro-group $G(K(\!(t)\!))$ in pro-vector spaces (the latter Hecke operators were studied in \cite{BK-Hecke}).
\end{remark}

We also define a weaker notion of Hecke-boundedness, that uses only commuting operators $T_c$ corresponding to simple divisors,
as opposed to the entire Hecke algebra $\HH_{G,C}$.

\begin{definition}
A function $f\in \C(\Bun(C))$ is called {\it weakly Hecke-bounded} if there exists a finite set $S\sub \Bun(C)$, such that for every 
collection of simple divisors $c_1,\ldots,c_n$ on $C$ (not necessarily distinct),
the support of $T_{c_1}\ldots T_{c_n}f$ is contained in $S\ot \Pic(C)$.
\end{definition}

There is also an analogous notion for $f\in \C(\Bun(C,K))$ for an arbitrary open compact subgroup $K\sub G(\OO)$ (see \ref{weakly-Hecke-b-def}),
defined using simple divisors at places where $K$ is unramified. 
%It is clear from this definition that the subspace of Hecke-bounded functions
%$$\SS_b(\Bun(C)):=\bigoplus_{L} \SS_b(\Bun^L(C))\sub \SS(\Bun(C))$$ 
%is preserved by Hecke operators $(T_c)$.

%\bigskip

%\noindent
%{\bf Conjecture.} {\it 

\begin{conjecture}
Fix a line bundle $L$ on $C$ and an open compact subgroup $K\sub G(\OO)$.
Any weakly Hecke-bounded function on $\Bun^L(C,K)$ is cuspidal.
\end{conjecture}

%}
%\bigskip

Note that together with Conjecture \ref{main-cusp-conj}(3) this implies 
that for a function on $\Bun^L(C,K)$ being ``weakly Hecke-bounded" is equivalent to ``Hecke-bounded" and is equivalent to being cuspidal.

\subsection{The case of $G=\GL_2$ and an extension of length $2$}

This paper provides the proof of most of the above conjectures in the case when $G=\GL_2$ and the special nilpotent extension $C$ has length $2$.
Until the end of this section we restrict our analysis to this case. In this case we also construct cuspidal Hecke eigenfunctions coming from smooth Hitchin fibers
of the reduced curve $\ov{C}$ (see Theorem F and Theorem \ref{smooth-hitchin-thm}).

We say that functions from the subspace $\C^1_{\cusp}(\Bun^L(C))=\C^1_{\cusp}(\Bun(C))\cap \C(\Bun^L(C))$
% $\SS^1_{\cusp}(\Bun(C))\sub \SS_{\cusp}(\Bun(C))$ of $1$-cuspidal functions.
%We will also call these functions 
are {\it strongly cuspidal}.
%We set 

\bigskip

\noindent
{\bf Theorem C}. {\it Assume that $C$ is a special nilpotent extension of length $2$ and the genus $g$ of $\ov{C}$ is $\ge 2$.
Let $L$ be a line bundle on $C$, and let $K\sub G(\OO)$ be an open compact subgroup. 
%and $\ov{C}$ has a $k$-rational point.

\noindent
(1) The space $\C_{\cusp}(\Bun^L(C,K))$ is finite-dimensional. 

\noindent
(2) A function $f$ on $\Bun^{L}(C,K)$ is cuspidal if and only if it is weakly Hecke-bounded. 
%i.e., we have $$\SS_{\cusp}(\Bun^L(C))=\SS_b(\Bun^L(C)).$$ 
}

\bigskip

The main technical result used in the proof of Theorem C(1) is Proposition \ref{cusp-vanishing-thm}
showing that every cuspidal function on $\Bun_G(C)$ (resp., $\Bun_G(C,K)$) vanishes on vector bundles $V$ (resp., elements of $\Bun_G(C,K)$ projecting to $V$)
such that $\ov{V}\simeq L_0\oplus M_0$, where $L_0$ and $M_0$
are line bundles on $\ov{C}$ with $\deg(M_0)-\deg(L_0)\ge \max(2g+2+|\ell|,6g-\ell-1)$, where $\ell=\deg(\NN/\NN^2)$ 
(resp., $\deg(M_0)-\deg(L_0)\ge N(K)$ for some constant $N(K)$).

%Note that the upper bound in Theorem C(1) is weaker than the conjectural one.

\subsection{The space of cuspidal functions for $\PGL_2$}

To simplify notations we replace in this subsection the group $G=\GL_2$ by $\PGL_2$. 
%(and $C$ a nilpotent extension of length $2$). 
By Proposition A, 
the Hecke algebra $\HH_K:=\HH_{\PGL_2,C,K}$ acts on the space of cuspidal functions 
$$V(C,K):=\C_{\cusp}(\Bun_{\PGL_2}(C,K)).$$
For each simple divisor $c\sub C$ we denote by $T_c$ the Hecke operator associated with the double coset of the image of 
\eqref{Tc-double-coset}.

\medskip

\noindent
{\bf Corollary D}. {\it Let $K\sub \PGL_2(\OO)$ be an open compact subgroup. 
\begin{enumerate}
\item
The space $V(C,K)$ is finite-dimensional and is contained in $\SS(\Bun_{\PGL_2}(C,K))$.
\item
A function $f\in \SS(\Bun_{\PGL_2}(C,K))$ is cuspidal if and only if for some (equivalently, for all) finite set of places $\Si$ such that $K$ is unramified outside $\Si$,
 the space spanned by $T_{c_1}\ldots T_{c_n}f$ with $\ov{c}_i\not\in\Si$ (resp., the space $\HH_{\PGL_2,C}\cdot f$) is finite-dimensional. 
 \end{enumerate}
 }
 
 \medskip
 
% We also prove a similar characterization of cuspidal function in $\SS(\Bun_{\PGL_2}(C,K))$ for an arbitrary open subgroup $K\sub \PGL_2(\OO)$ (see Corollary \ref{hecke-fin-cusp-cor}).
 
%and it follows from Theorem C(1) that the latter space is finite-dimensional (see Corollary \ref{PGL2-cor}).

Recall (see Definition \ref{cusp-def}(5)) that we also have the subspace of strongly cuspidal functions
$$V^1(C,K):=\C^1_{\cusp}(\Bun_{\PGL_2}(\C,K))\sub V(C,K),$$
preserved by $\HH_K$. Let $\ov{K}\sub \PGL_2(\A)$ be the image of $K$.
Then we can view $V(\ov{C},\ov{K})$ as a subspace in  $V(C,K)$.

Using tools from representation theory of $G(\A)$, we get a direct sum decomposition
$$V(C,K)=V^1(C,K)\oplus V(C,K)_n \oplus V(\ov{C},\ov{K}),$$
where $V(C,K)_n$ is a certain piece corresponding to the nilpotent adjoint orbit in $\fg(\ov{F})$ (see Corollary \ref{main-dec-cor}).
%, while $V^1(C,K)$ further decomposes
%into pieces corresponding to elliptic semisimple orbits (see Theorem \ref{main-dec-thm}).
%Furthermore, we decompose each of the spaces $V^1(C,K)$ and $V(C,K)_n$ into the direct sum, where the summands are
%given by $K$-invariants in some irreducible admissible representations of $G(\A)$. The corresponding local representations and the action of
%the Hecke operators on $K$-invariants are discussed in \cite{K-YD}.

We also relate $V^1(C):=V^1(C,G(\OO))$ to the moduli space of Higgs bundles over $\ov{C}$ (see Proposition \ref{Higgs-prop}) and use this to study the dimension of
$V^1(C)$ and of $V(C):=V(C,G(\OO))$.

\medskip

\noindent
{\bf Theorem E}. {\it
Assume that $C$ is a special nilpotent extension of length $2$ with $L=\NN/\NN^2$ of degree $0$, the genus $g$ of $\ov{C}$ is $\ge 2$, and the characteristic of $k$ is $\neq 2$.
%Let $r$ denote the number of irreducible components of the moduli space of stable $L^{-1}$-twisted Higgs $\PGL_2$-bundles on $\ov{C}_{\ov{k}}$.
Then 
$$|\frac{\dim V^1(C)}{q^{6g-6}}-2|\le a(g)q^{-1/2}, \ \ \dim(V(C)/V^1(C))\le b(g)q^{3g-3},$$
for some constants $a(g),b(g)$ depending only on the genus $g$.
}

\medskip

For simplicity of notation we will formulate our last result for $C=\ov{C}\times \Spec(k[\eps]/(\eps)^2)$, where $\ov{C}$ is a curve over $k$
(a generalization to an arbitrary special nilpotent extension of length $2$  is given by Theorem \ref{smooth-hitchin-thm}).
Set $A:=H^0(\ov{C},\om_{\ov{C}}^2)$, and let $A'\sub A$ be the complement to the set of elements $\eta^2$, where $\eta\in H^0(\ov{C},\om_{\ov{C}})$.
There is a decomposition
$$V^1(C)\simeq \bigoplus_{\a\in A'}\SS(HF_\a(k)),$$
whose summands are preserved by the big Hecke algebra $\HH_C$ (see \eqref{HF-decomposition-eq}).
Here $HF_\a$ are Hitchin fibers in the moduli stack of Higgs $\PGL_2$-bundles.
In the case when $\a\in A'$ has only simple zeros, $HF_\a(k)$ can be identified with the finite group $\Pic(\ov{C}_\a)/\pi^*\Pic(\ov{C})$ for a double cover $\pi:\ov{C}_\a\to \ov{C}$,
hence, it makes sense to consider characters of $HF_\a(k)$.

\medskip

\noindent
{\bf Theorem F.} {\it
Assume that $C=\ov{C}\times \Spec(k[\eps]/(\eps)^2)$, and the characteristic of $k$ is $\neq 2$. Then for every $\a\in A'$ having only simple zeros on $\ov{C}$, the characters
of the group $HF_\a(k)$ form a basis in the space $\SS(HF_\a(k))$ consisting of $\HH_{\PGL_2,C}$-eigenvectors.
}

\medskip

Thus, the big Hecke algebra $\HH_{\PGL_2,C}$ acts on the summands in $V^1(C)$ corresponding to smooth spectral curves through a commutative quotient. However, we expect
to see noncommuting operators from this algebra when considering the summands corresponding to singular spectral curves.
One can ask whether Hecke eigenfunctions corresponding to singular spectral curves come from some perverse sheaves on Hitchin fibers.

The paper is organized as follows. 
%In Section \ref{O-K-sec} we explain our motivation to study automorphic functions for nilpotent extensions of curves. 
%Namely, we show that nilpotent extensions of curves appear in the context of studying Hecke operators on the space of half-densities on the stack of bundles over a curve over a 
%non-archimedean local field. 
In Section \ref{gen-res-sec} we prove some simple general results, which hold for nilpotent extensions $C$ of arbitrary length.
In Sec.\ \ref{const-term-gen-sec}, we prove Proposition A (for any $G$), and in Sec.\ \ref{hecke-GL2-sec} and \ref{hecke-commute-sec} we study Hecke operators for
$\GL_2$, in particular, proving Theorem B. 

Starting from Section \ref{aut-rep-sec} we consider only special nilpotent extensions of length $2$. 
In Section \ref{aut-rep-sec}, we study the representation of the group $G(\A)$ on the space $\SS(G(F)\backslash G(\A))$ of
 smooth functions with compact support on $G(F)\backslash G(\A)$ for an arbitrary split group $G$. 
In Sec.\ \ref{orbit-dec-sec}, we decompose this space $\SS(G(F)\backslash G(\A))$ into the direct sum
of invariant subspaces parametrized by coadjoint orbits of the group $G(\ov{F})$. We describe explicitly the pieces corresponding to elliptic orbits (i.e., regular semisimple with anisotropic stabilizer) in terms of certain induced representations (see Sec. \ref{reg-ss-sec} and \ref{reg-ell-sec}), and study admissible subrepresentations in the pieces corresponding to regular semisimple orbits (see Sec.\ \ref{adm-rep-sec}).
In Sec.\ \ref{Hitchin-sec}, we establish the relation of these pieces with the moduli stack
of Higgs bundles on $\ov{C}$. In Sec.\ \ref{finitary-sec} we specialize to the case $G=\PGL_2$ 
and describe completely the subspace of {\it finitary functions} in $\SS(G(F)\backslash G(\A))$, i.e., those contained in an admissible $G(\A)$-subrepresentation.
\footnote{We choose to work with $\PGL_2$ rather than $\GL_2$ in representation-theoretic part of the paper, since the results are somewhat easier to formulate and since
the corresponding local statements are readily available in \cite{K-YD}.}
%In Sec.\ \ref{nilp-orbit-sec}, we study the piece in $\SS(\PGL_2(F)\backslash \PGL_2(\A))$
%corresponding to the nonzero nilpotent orbit. 

In Section \ref{const-GL2-sec} we study the constant term operator in the case of $\GL_2$.
In Sec.\ \ref{Iwasawa-sec}, we introduce an important technical tool, an analog of the Iwasawa decomposition for nilpotent extensions of length $2$.
In Sec.\ \ref{GL2-const-term-Hecke-sec}, we establish a useful compatibility of the constant term with the Hecke operators. In Sec.\ \ref{cusp-GL2-sec} we combine
this result with an analog of the reduction theory to prove Theorem C. 

In Section \ref{cusp-PGL2-sec} we combine the previous results to study cuspidal functions for $\PGL_2$.
In Sec.\ \ref{CorC-sec}, we prove Corollary D on equivalence of cuspidality with Hecke finiteness for $\PGL_2$. Then we combine this with the orbit decomposition
from Sec.\ \ref{orbit-dec-sec} to prove that the subspaces of cuspidal and finitary functions in $\SS(\PGL_2(F)\backslash \PGL_2(\A))$ coincide (see Theorem \ref{main-dec-thm}). 
%In Sec. \ref{str-cusp-sec}, we describe explicitly the subspace of strongly cuspidal functions as the sum of pieces corresponding
%to elliptic orbits. 
%In Sec.\ \ref{orbit-cusp-sec} we give a representation theoretic description of the space of cuspidal functions in $\SS(\PGL_2(F)\backslash \PGL_2(\A))$ (see Theorem 
Finally, in Sec.\ \ref{ThmE-sec} we prove Theorems E and F.
% giving the estimate of the dimension of the space of cuspidal functions for $\PGL_2$. 
%In fact, we give precise formulas for the dimensions
%of the new pieces in this space (that do not come from the similar space for $\ov{C}$). 
In Appendix \ref{group-app} we describe some basic results on groupoids given by double cosets.
Appendix \ref{loc-app} contains a calculation with certain local induced representations of $\PGL_2$ (over a square-zero extension of a local field), 
generalized a calculation done in \cite{K-YD}. In Appendix we give an example of nontrivial action of automorphism groups for the natural $k$-torsor of the groupoid
of Higgs $\PGL_2$-bundles over $\ov{C}$ associated with a square-zero extension $C$, in the case the characteristic of $k$ is $2$.
In Appendix \ref{geom-const-term-sec} we give a geometric interpretation of the constant
term operator for $G=\GL_2$. 

%\medskip

\section{Hecke operators for special nilpotent extensions of curves over finite fields}\label{gen-res-sec}

Let $C$ be a nilpotent extension of a curve $\ov{C}$.
In Sec.\ \ref{const-term-gen-sec} we prove Proposition A (which works for any reductive group $G$ and any $C$). Then in Sec.\ \ref{hecke-GL2-sec},
we restrict to the case when $G=\GL_2$ and compare geometric and adelic definitions of the Hecke operators associated with effective Cartier divisors. 
In Sec.\ \ref{hecke-commute-sec} we assume that $C$ is a special nilpotent extension and prove Theorem B.

\subsection{Compatibility of the constant term with the Hecke algebra}\label{const-term-gen-sec}

For each parabolic subgroup $P=LU$ and a function $f$ on $\Bun_G(C,K)$, we define the constant term $C_P(f)$ as a function on $\QBun_L(C,K)$ (see Def.\ \ref{QBun-L-def}) given by
$$C_Pf(g)=\int_{u\in U(F)\backslash U(\A)} f(ug)du,$$
where $g\in G(\A)$ and $du$ is the Haar measure on $U(\A)$ normalized so that $\vol(U(\OO))=1$ 
(we will check that the function on $\QBun_L(C,K)$ is well defined in Lemma \ref{EPf-lem} below).
Consider the compact open subgroup $U_{g,K}:=U(\A)\cap gKg^{-1}$. Since the function $u\mapsto f(ug)$ is right $U_{g,K}$-invariant,
and the set of double cosets $U(F)\backslash U(\A)/U_{g,K}$ is finite, the integral defining $C_Pf$ can be rewritten as a finite sum.
%Thus, we can rewrite the integral defining $E_Pf$ as a finite sum:
%$$E_Pf(g)=\vol(U_g)\cdot \sum_{u\in U(F)\backslash U(\A)/U_g} f(ug).$$

\begin{lemma}\label{EPf-lem} 
The map $g\mapsto C_Pf(g)$ descends to a well defined function on $\QBun_L(C,K)$.
\end{lemma} 

\begin{proof}
It is clear that $g\mapsto C_Pf(g)$ is right $K$-invariant and left $U(\A)$-invariant. It remains to
check left $L(F)$-invariance. For $l\in L(F)$, we have
$$C_Pf(lg)=\int_{u\in U(F)\backslash U(\A)} f(ulg)du=\int_{u\in U(F)\backslash U(\A)} f(\Ad(l^{-1})(u)g)du=
\int_{u\in U(F)\backslash U(\A)} f(ug)du,$$
since for $l\in L(F)$, the action of $Ad(l)$ on $U(\A)$ preserves the Haar measure.
\end{proof}

For each $g_0\in G(\A)$, let $h_{g_0}$ denote the characteristic measure of the double coset $Kg_0K$, where
the Haar measure on $G(\A)$ is normalized so that the measure of $K$ is $1$. 
Note that the measure of $Kg_0K$ is equal to the number of left (equivalently, right) $K$-cosets contained in it. 
Using the partition of $Kg_0K$ into right $K$-cosets, we get for any $f$ on $G(\A)/K$,
\begin{equation}\label{T-g0-eq}
T_{g_0,K}f(g):=(h_{g_0}*f)(g)=\mu(Kg_0K)\cdot \int_{h\in K}f(ghg_0) dh.
\end{equation}
%where the Haar measure is normalized so that $\vol(K)=1$.

\noindent
{\it Proof of Proposition A}.
(1) We have to show that each operator $T_{g_0,K}$ preserves
the subspace of cuspidal functions. It is enough to construct an operator $T_{g_0,K}^P$ on $\C(\QBun_L(C,K))$, such that
$$C_P\circ T_{g_0,K}=T_{g_0,K}^P\circ C_P.$$
%Note that if we change $h$ to $hh_0$, where $h_0\in H_{g_0}$ then the coset $hg_0G(\OO)$ will not change.
%Hence, 
Note that for $h,h_0\in K$ one has $hh_0g_0K=hg_0K$ for $h_0\in H_{g_0}$, where
$$H_{g_0}=H_{g_0,K}:=K\cap g_0Kg^{-1}_0 \sub K.$$
Hence, in the formula \eqref{T-g0-eq}
we can replace the integration over $h\in K$ by a finite sum over $K/H_{g_0}$. Furthermore, we have
$$\vol(H_{g_0})=|K/H_{g_0}|^{-1}=\mu(Kg_0K)^{-1}.$$ 
Hence, we get
\begin{equation}\label{hecke-gen-formula-eq}
T_{g_0,K}f(g)=\sum_{h\in K/H_{g_0}}f(ghg_0).
\end{equation}

Let $P$ be a parabolic subgroup, $U\sub P$ its unipotent radical. We have
%$$E_PT_{g_0}f(g)=\int_{h\in G(\OO)}\int_{u\in U(F)\backslash U(\A} f(ughg_0) du dh.$$
$$C_PT_{g_0,K}f(g)=\sum_{h\in K/H_{g_0}}\int_{u\in U(F)\backslash U(\A)} f(ughg_0) du.$$
%Now we notice that the right-hand side can be expressed in terms of $E_Pf$, so we obtain
This formula can be rewritten as
\begin{equation}\label{EPT-formula}
C_PT_{g_0,K}f(g)=\sum_{h\in K/H_{g_0}}C_Pf(ghg_0),
\end{equation}
so the right-hand side has the required form $T_{g_0,K}^PC_P(f)$ where 
$$T_{g_0,K}^P\phi(g)=\sum_{h\in K/H_{g_0}}\phi(ghg_0).$$
%Hence, $T_{g_0}$ preserves the kernel of $E_P$.

The proof that $T_{g_0,K}$ preserves the kernel of $C^i_P$ is analogous for each $i\ge 0$.

\noindent 
(2) It is enough to consider the case when the ideal sheaf $I_{C'}\sub \OO_C$ defining $C'\sub C$ satisfies $I_{C'}^2=0$.  Let
$$G(\A_C)\to G(\A_{C'}): g\mapsto \ov{g}$$ 
denote the natural reduction map. Suppose we are given a function $f_0$ on $\Bun_G(C')=G(F_{C'})\backslash G(\A_{C'})/G(\OO_{C'})$.
Consider the corresponding function $f$ on $\Bun_G(C)$ given by $f(g)=f_0(\ov{g})$. Then we have
$$C_Pf(g)=\int_{u\in U(F_C)\backslash U(\A_C)} f_0(\ov{u}\ov{g})du=\int_{u_0\in U(F_{C'})\backslash U(\A_{C'})}\vol(\pi^{-11}(u_0)) f_0(u_0\ov{g}) du_0.$$
where $\pi:U(F_C)\backslash U(\A_C)\to U(F_{C'})\backslash U(\A_{C'})$ is the reduction map.
It remains to show that $\vol(\pi^{-1}(u_0))$ does not depend on $u_0$. 
Since $I^2=0$, we can identify the kernel of the map $U(\A_C)\to U(\A_{C'})$ with $1+\fu\ot (I\cdot \A_{C'})$. This leads to the identification
$$\pi^{-1}\pi(u)=(1+\fu\ot (I\cdot \A_{C'})/\fu\ot (I\cdot F_{C'}))\cdot u\sub U(F_C)\backslash U(\A_C).$$
Thus,
$$\vol(\pi^{-1}\pi(u))=\vol(1+\fu\ot (I\cdot \A_{C'})/\fu\ot (I\cdot F_{C'})),$$
which does not depend on $\pi(u)$. Thus, up to a constant factor, $C_Pf(g)$ is given by the constant term operator applied to $f_0$.
\qed

%\begin{remark}
%We did not use our assumptions on the ring $A$ in the above proof, so Proposition A holds for any finite commutative ring $A$.
%\end{remark}

\subsection{Hecke operators for $\GL_2$: adelic and modular definitions}\label{hecke-GL2-sec}

In this subsection we assume that $G=\GL_2$ and write $\Bun(C)=\Bun_{\GL_2}(C)$.

\begin{definition}\label{hecke-def} Set $a(V)=|\Aut(V)|$.

\noindent

(i) For any effective Cartier divisor $c\sub C$ and a function $f$ on $\Bun(C)$, we set
$$T_cf(V):=\sum_{V'}\frac{i(V,V';c)}{a(V')} f(V'),$$
where $i(V,V';c)$ is the number of embeddings $V\to V'$ such that $V'/V\simeq \OO_c$.
%every vector bundle $F$, given a surjective map $\ell:F|_c\to A$, viewed up to the action of $A^*$,
%we set $F_\ell:=\ker(V\to F|_c\rTo{\ell} A)$. We define the Hecke operator $T_c$ on $\SS(\Bun(C))$ by setting
%$$T_c\de_F:=\sum_{\ell\in \P(V^\vee|_c)} \de_{F_\ell},$$
%where $\de_F$ is the delta-function of a point $F\in \Bun(C)$.

\noindent
(ii) We denote by $\Dual:\C(\Bun(C))\to \C(\Bun(C))$ the operator induced by duality:
$$\Dual f(V)=f(V^\vee).$$
We also consider another Hecke operator $T'_c$ given by
$$T'_c=\Dual\circ T_c\circ \Dual.$$
Equivalently, using the equality $i(V,W,c)=i(W^\vee,V^\vee,c)$, we get
\begin{equation}\label{T'c-formula}
T'_cf(V')=\sum_V \frac{i(V,V',c)}{a(V)} f(V)=\sum_{V\sub V': V'/V\simeq \OO_c} f(V),
\end{equation}
where in the last formula we sum over subsheaves $V$ of $V'$ which are vector bundles.
%$$T'_c\de_F:=\sum_{\ell\in \P(F|_c)} \de_{V'_\ell},$$
%where $V'_\ell=\ker(V^\vee\to V^\vee|_c)^\vee$.
%$$T_c(f)(F)=\sum_{V'}\sum_{S\sub V': S\simeq F, V'/V\simeq \OO_c}f(V').$$
\end{definition}

%It is easy to see that the operators $T_c$ and $T'_c$ are well defined and preserve the subspaces of finitely supported functions.
We have
$$T_c(\C(\Bun^L(C)))\sub \C(\Bun^{L(-c)}(C)), \ \ T'_c:\C(\Bun^{L(-c)}(C))\to \C(\Bun^{L}(C)).$$
The operators $(T_c)$ (resp., $T'_c$) in general do not commute (see Remark \ref{noncomm-rem}).
However, in the next section we will show that in the case when $C$ is a special nilpotent extension of $\ov{C}$, the operators associated with {\it simple} divisors do commute.

Now we connect the above definition to the adelic Hecke operators $T_g$ given by \eqref{hecke-gen-formula-eq}. 
Let $c\sub C$ be an effective Cartier divisor. 
We denote by $f_c=(f_{p,c})$ an idele such that $f_{p,c}\in \OO_p$ is a local equation of $c$ for each point $p$.

\begin{lemma}\label{hecke-mod-adelic-lem}
Set $A(c)=H^0(\OO_C/I_c)$, where $I_c$ is the ideal of $c\sub C$.
%be the extension of $A$ associated with $c$. 
Consider the element 
\begin{equation}\label{gc-eq}
g_c:=\left(\begin{matrix} f_c^{-1} & 0 \\ 0 & 1\end{matrix}\right)\in T(\A_C).
\end{equation}
Then one has 
$$T_{g_c}=T_c.$$
%, where 
%$$T_{g_c}f(g)=\int_{h\in G(\OO)}f(ghg_c)dh.$$
\end{lemma}

\begin{proof}
The subgroup $H_{g_c}\sub G(\OO)$ consists of matrices $(a_{ij})\in G(\OO)$ such that $a_{21}\in I_c:=f_c\cdot \OO$. Thus, the quotient $G(\OO)/H_{g_c}$ can be identified
with $G(\OO/I_c)/B(\OO/I_c)$, i.e., with the projective line $\P^1(A(c))$. Hence,
% $\vol(H_{g_c})=1/|\P^1(A(c))|$, and
the general formula \eqref{hecke-gen-formula-eq} can be rewritten as
$$T_{g_c}f(g)=\sum_{[a_1:a_2]\in \P^1(A(c))} f(gh_{[a_1:a_2]}g_c),$$
 where $h_{[a_1:a_2]}$ is any element of $G(\OO)$ such that 
 $$h_{[a_1:a_2]} \left(\begin{matrix} 1 \\ 0 \end{matrix}\right)=\left(\begin{matrix} a_1 \\ a_2 \end{matrix}\right).$$
Note that since $g_c^{-1}$ has entries in $\OO_C$, for the bundles $V(g)$ and $V(gh_{[a_1:a_2]}g_c)$ we have an
inclusion of sheaves
$$V:=V(g)\hra V(gh_{[a_1:a_2]}g_c)=:V_{[a_1:a_2]}$$
which induces at every point the inclusion
$$V_p=g\OO_p^2\hra gh_{[a_1:a_2]}g_c\OO_p^2=(V_{[a_1:a_2]})_p.$$
%Note that the quotient is isomorphic to 
%$$\coker(\OO_p^2\to h_{[a_1:a_2]}g_c\OO_p^2)\simeq \coker(\OO_p^2\to g_c\OO_p^2)\simeq (\OO/I_c)_p.$$
%Hence, $V_{[a_1:a_2]}/V\simeq \OO_c$.
%It remains to show that the inclusions $V\to V_{[a_1:a_2]}$ ???

On the other hand, surjections $V^\vee=V((g^t)^{-1})\to \OO_c$ up to automorphisms of $\OO_c$, correspond exactly to the points $[a_1:a_2]$ of the projective line $\P^1(\A(c))$,
and their kernels are easily seen to be given by subsheaves $V^\vee_{[a_1:a_2]}$. Taking into account the relation
$T_c=\Dual\circ T'_c\circ\Dual$ and the formula \eqref{T'c-formula} for $T'_c$ we get
$$T_cf(g)=\sum_{[a_1:a_2]\in \P^1(A(c))} f(gh_{[a_1:a_2]}g_c).$$
\end{proof}

\begin{lemma}\label{Hecke-tensor-lem}
(i) For every line bundle $L$ on $C$ the operator $t_L^*$ on $\C(\Bun(C))$ given by $t_L^*f(V)=f(V\ot L)$ commutes with all the operators $T_c$ and $T'_c$
associated with effective Cartier divisors.

\noindent
(ii) For any effective Cartier divisor $c\sub C$, one has
$$T'_c=t^*_{\OO(-c)}\circ T_c,$$
i.e.,
$$T'_cf(V)=(T_cf)(V(-c)).$$
\end{lemma}

\begin{proof} (i) is straightforward. For (ii) we use that for every pair of vector bundles $V$ and $V'$, there is a bijection between the set of embeddings $i:V\to V'$ with the cokernel
isomorphic to $\OO_c$ and the set of embeddings $V'(-c)\to V$ with the cokernel isomorphic to $\OO_c$. Namely, starting with $i:V\to V'$ we observe that the natural embedding
$V'(-c)\to V'$ factors uniquely through an embedding $j:V'(-c)\to V$. Furthermore, $\coker(j)$ is isomorphic to the kernel of the induced map $V'|_c\to \OO_c$, hence, it is isomorphic to 
$\OO_c$.
\end{proof}

Now we rewrite Definition \eqref{hecke-def} in terms of the natural groupoid structure on the Hecke correspondence where we take into
account automorphisms of objects.
 
For each effective Cartier divisor $c\sub C$, let us consider the Hecke groupoid
$$\Hecke(C,c):=
%\sqcup_{(V,V')\in \Bun_\LL(C)\times \Bun_{\LL(c)}(C)} 
\{V\hra V' \ |\ V'/V\simeq \OO_c\},$$
where $V$ and $V'$ are vector bundles over $C$,
and let
%$$\Hecke'(C):=\sqcup_{(F,V')\in \Bun(C)^2} \{F\hra V' \ |\ V'/V\simeq \OO_c\}/\Aut(V').$$
$p_1,p_2$ denote the natural projections from $\Hecke(C,c)$ to $\Bun(C)$ sending $V\hra V'$ to $V$ and $V'$, respectively.
We use the pull-backs and push-forwards with respect to these projections (see Appendix \ref{group-app}).

\begin{lemma}
One has
$$T_c=p_{1*}p_2^*, \ \ T'_c=p_{2*}(p_1)^*.$$
\end{lemma}

\begin{proof}
Let us consider the case of $T'_c$ (the case of $T_c$ is similar, and also follows by duality).
Note that for a vector bundle $V$, the fiber groupoid $p_2^{-1}(V)$ is the groupoid of embeddings $V'\hra V$ (with $V$ fixed), such that $V/V'\simeq \OO_c$.
This groupoid has trivial automorphisms, so by Lemma \ref{groupoid-push-for-lem}, we get
$$p_{2*}(p_1)^*f(V)=\sum_{V'\sub V: V/V'\simeq \OO_c} f(V')=T'_cf(V).$$
\end{proof}

\subsection{Proof of Theorem B}\label{hecke-commute-sec}

%We still assume that $C$ is a smooth proper curve over a finite ring $A$.

%Here we assume that $C$ is a smooth proper curve over a finite local ring $A$ with the residue field $k$. 
%By an {\it effective divisor} on $C$ we mean a relative effective Cartier divisor $D\sub C$, so $D$ is locally given by one
%equation, which is a nonzerodivisor, and $D$ is flat over $A$.

For every closed point $p$ of $C$ let us denote by $\HH_p=\HH_{G(F_p),G(\OO_p)}$ the {\it local Hecke algebra} of compactly supported,
$G(\OO_p)$-biinvariant measures on $G(F_p)$. 
%Note that for every double coset $S=G(\OO_p)gG(\OO_p)$ there is a unique {\it characteristic measure} $\mu_S$ supported on $S$, which is $G(\OO_p)$-biinvariant
%and satisfies $\mu_S(1)=1$. It can be defined by
%$$\mu_S(f)=\int_{G(\OO_p)\times G(\OO_p)}f(k_1gk_2)dk_1dk_2,$$
%where $dk$ is the Haar measure on $G(\OO_p)$ (left and right invariant since $G(\OO_p)$ is compact).

For any open compact subgroup $K\sub G(\A)$ of the form $K=G(\OO_p)\times K'$, where $K'\sub G(\A')$, and $\A'$ is the adeles
without the place $p$, we have a natural injective homomorphism 
$$\HH_p\to \HH_{G,C,K}.$$
In particular, for $K=G(\OO)$ we get homomorphisms $\HH_p\to \HH_{G,C}$,
and the images of these homomorphisms for different points commute.

Recall that the element $h_c$ associated with a simple divisor $c$
is defined using the double coset of the element $g_c$ in $G(\A)$ which is trivial away from the place $p=\ov{c}$ (see \eqref{gc-eq}).
Hence, $h_c$ can be viewed as an element of the local Hecke algebra $\HH_p$ corresponding to the characteristic measure of $G(\OO_p)g_cG(\OO_p)$.
%it is clear that the Hecke operators supported at different places commute. 
%Thus, we are reduced to proving a simiar result about the local Hecke algebra.
%Namely,  we have a 
%The operators $T_c$ come from a similarly defined elements of $\HH_{G(F_p),G(\OO_p)}$ associated with simple divisors supported at $p\in \ov{C}$.

Now let $G=\GL_2$. Theorem B follows from the following local result.

\begin{theorem}\label{Hecke-commute-thm} 
%Assume that $A$ is a finite local ring with the principal maximal ideal $(\eps)$ 
Assume that $C$ is a special nilpotent extension of $\ov{C}$, and let $c$ and $d$ be simple divisors in $C$ supported at a point $p\in \ov{C}$.
Then $h_ch_d=h_dh_c$ in the local Hecke algebra $\HH_p$.
%there exists an element $\pi$ in the maximal ideal of $A$ such that every ideal in $A$ has form $(\pi^m)$ for some $m\ge 0$.
%for any extension $R$ of $\OO_d$ by $\OO_c$, one has $R^\circ\simeq R$. Hence, in this case $T'_c$ and $T'_d$ commute.
\end{theorem}

For the rest of this subsection we will write $\OO$ instead of $\OO_p$. 

As in \cite{K-YD}, we use the Gelfand's trick. Let $\th:G(F_p)\to G(F_p)$ denote the anti-involution given by the transpose of a matrix.
Then $\th$ 
%sends any $G(\OO)$-biinvariant measure to a $G(\OO)$-biinvariant measure, so it 
induces an anti-involution of the Hecke algebra $\HH_p$.
%Note that $\th$ sends the characteristic measure of a double coset $G(\OO)gG(\OO)$ to
%the characteristic measure of $G(\OO)\th(g)G(\OO)$.

\begin{lemma}\label{Gelf-trick-lem} 
Assume that for $g_1,g_2\in \GL_2(F_p)$, the double cosets $G(\OO)g_1G(\OO)$ and $G(\OO)g_2G(\OO)$ are $\th$-invariant, and
the subset $G(\OO)g_1G(\OO)g_2G(\OO)\sub\GL_2(F_p)$
is the union of $\th$-invariant double $G(\OO)$-cosets. Then the characteristic measures of $G(\OO)g_1G(\OO)$ and $G(\OO)g_2G(\OO)$ commute in $\HH_{G(F_p)}$.
\end{lemma}

\begin{proof}  Since the group $G(\OO)$ is compact, the Haar measure on it is preserved by $\th$. Hence, $\th$ maps the characteristic measure of a double coset $G(\OO)gG(\OO)$
to the characteristic measure of a double coset $G(\OO)\th(g)G(\OO)$. Hence, any element of the Hecke algebra supported on
the union of a finite number of $\th$-invariant double $G(\OO)$-cosets is $\th$-invariant. Let $\chi_1$ and $\chi_2$ be the characteristic measures of 
$G(\OO)g_1G(\OO)$ and $G(\OO)g_2G(\OO)$. Then $\chi_1*\chi_2$ is supported on $G(\OO)g_1G(\OO)g_2G(\OO)$.
Hence, we have
$$\th(\chi_1)=\chi_1, \ \ \th(\chi_2)=\chi_2, \ \ \th(\chi_1*\chi_2)=\chi_1*\chi_2.$$
Since $\th$ is an anti-involution, we deduce
$$\chi_1*\chi_2=\th(\chi_1*\chi_2)=\th(\chi_2)*\th(\chi_1)=\chi_2*\chi_1.$$
\end{proof}

Note that if a matrix $g$ is symmetric then $G(\OO)gG(\OO)$ is invariant under $\th$. Hence, this condition holds for the elements $g_c$
%$$g_c=\left(\begin{matrix} f_c^{-1} & 0 \\ 0 & 1\end{matrix}\right),$$
associated with generators $f_c$ of the ideals of simple divisors $c$ supported at $p$.

Let $\eps\in \OO=\OO_p$ denote a generator of the ideal $\NN_p$.

\begin{lemma}\label{upper-tr-double-classes-lem}
We have an inclusion
$$\left(\begin{matrix} f_c & 0 \\ 0 & 1\end{matrix}\right) G(\OO) \left(\begin{matrix} 1 & 0 \\ 0 & f_d\end{matrix}\right)\sub
G(\OO)\left(\begin{matrix} 1 & 0 \\ 0 & f_cf_d\end{matrix}\right)G(\OO)\cup \bigcup_{x\in \eps\OO} G(\OO)\left(\begin{matrix} f_c & 0 \\ x & f_d\end{matrix}\right)G(\OO).$$
\end{lemma}

\begin{proof}
Recall that we denote by $H_{g_c}\sub G(\OO)$ the subgroup of matrices such that $a_{21}\in f_c\cdot \OO$.
For any $h\in H_{g_c}$ we have
$$\left(\begin{matrix} f_c & 0 \\ 0 & 1\end{matrix}\right) h\in G(\OO)\left(\begin{matrix} f_c & 0 \\ 0 & 1\end{matrix}\right).$$
Hence, it is enough to study the double $G(\OO)$-cosets of
$$\left(\begin{matrix} f_c & 0 \\ 0 & 1\end{matrix}\right)g\left(\begin{matrix} 1 & 0 \\ 0 & f_d\end{matrix}\right),$$
where $g$ runs through some set of representatives of right $H_{g_c}$-cosets in $G(\OO)$.
It is easy to see that we can take as such representatives
some matrices of the form
$$\left(\begin{matrix} 0 & 1 \\ 1 & a\end{matrix}\right), \ \ \left(\begin{matrix} 1 & 0 \\ \eps b & 1\end{matrix}\right),$$
with $a,b\in\OO$. Now we observe that
$$\left(\begin{matrix} f_c & 0 \\ 0 & 1\end{matrix}\right)\left(\begin{matrix} 1 & 0 \\ \eps b & 1\end{matrix}\right)\left(\begin{matrix} 1 & 0 \\ 0 & f_d\end{matrix}\right)=
\left(\begin{matrix} f_c & 0 \\ x & f_d\end{matrix}\right)$$
for some $x\in \eps \OO$, and 
$$\left(\begin{matrix} f_c & 0 \\ 0 & 1\end{matrix}\right)\left(\begin{matrix} 0 & 1 \\ 1 & a\end{matrix}\right)
\left(\begin{matrix} 1 & 0 \\ 0 & f_d\end{matrix}\right)=\left(\begin{matrix} 0 & f_cf_d \\ 1 & af_d\end{matrix}\right)\in G(\OO)\left(\begin{matrix} 1 & 0 \\ 0 & f_cf_d\end{matrix}\right)G(\OO).$$
\end{proof}

\begin{proof}[Proof of Theorem \ref{Hecke-commute-thm}] 
The case $c=d$ is clear so we can assume $c\neq d$.
By Lemma \ref{Gelf-trick-lem}, it is enough to show that 
the subset $G(\OO)g_c G(\OO)g_d G(\OO)\sub G(F_p)$ is the union of $\th$-invariant double $G(\OO)$-cosets. 

%In the case $c=d$, the subset $G(\OO)g_c G(\OO)g_c G(\OO)$ is clearly invariant under $\th$. Thus, 
By Lemma \ref{upper-tr-double-classes-lem}, it is enough to check that for every $x\in \eps\OO$ the subset 
$$G(\OO)\left(\begin{matrix} f_c & 0 \\ x & f_d\end{matrix}\right)G(\OO)=G(\OO)\left(\begin{matrix} f_d & x \\ 0 & f_c\end{matrix}\right)G(\OO)$$
is invariant under $\th$. Note that multiplying by elementary matrices on the left and on the right we can add to $x$ any element in $(f_c,f_d)$.
%Hence, in the case when the underlying points $\ov{c}$ and $\ov{d}$ of $\ov{C}$ are different, we can reduce to $x=0$, and the assertion is clear.

%Now, assume that $\ov{c}=\ov{d}$.
Set $f=f_d$.
%Let $f$ be a local equation of $c$ on $C$, so that $\OO_c=\OO/(f)$. 
Since $\ov{c}=\ov{d}$, rescaling $f_c$ by a unit we can assume that $f_c=f+g$, where $g\in \eps\OO$. Thus, we are interested in the double $G(\OO)$-coset of the matrix
$$M=\left(\begin{matrix} f & x \\ 0 & f+g\end{matrix}\right).$$
Note that 
$$\OO/(f,\eps)=k(\ov{d})$$
is the residue field of the corresponding reduced point.
%$x$ reduces to a generator of the maximal ideal of $\ov{c}=\ov{d}$ on $\ov{C}$.

%By Lemma \ref{Hecke-commute-lem}(iii), the module $R$ associated with an extension class $e$ is given by
%$$R=\coker\left(\begin{matrix} f & \wt{e} \\ 0 & f+g\end{matrix}\right),$$
%where $\wt{e}\in \OO$ is a representative of $e$,
%so multiplying the matrix on the left or on the right with a matrix in $G(\OO)$ will not change the isomorphism class of $R$.
The case $g=0$ is clear, so we can assume that $g\neq 0$.
Let $n\ge 1$ be maximal such that $g=\eps^ng_0$. If $g_0\equiv 0\mod (\eps,f)$, $g_0=fa+\eps b$, then 
$$f+g=f+\eps^naf+\eps^{n+1}b=(1+\eps^na)(f+g'),$$
where $g'\in (\eps^{n+1})$. Hence, multiplying the second row with $(1+\eps^na)^{-1}$ we can replace $n$ by $n+1$.
Continuing like this, we will either reduce to the case $g=0$, or to the case where
$g_0\not\in (\eps,f)$. 

%If $\wt{e}\mod f$ is invertible then we are done by Lemma \ref{Hecke-commute-lem}(ii). Hence, we can assume that $\wt{e}\in (\eps,f)$.
Let $m\ge 1$ be maximal such that $x\in (\eps^m,f)$. Since we can add a multiple of $f$ to $x$, we can assume $x=\eps^me_0$,
where $e_0\not\in(\eps,f)$. If $m\ge n$ then since $(\eps,f,g_0)=1$, we obtain 
$$\eps^n\in (\eps^{n+1},f,\eps^ng_0).$$
Hence, the ideals $I=(f,\eps^n)$ and $J=(f,\eps^ng_0)\sub I$ satisfy $I\sub \eps I+J$, which implies that $I=J$.
Thus, $x\in (\eps^n)\in J=(f,g)$, so this reduces to the case $x=0$.
%$\wt{e}\in (\eps^{m+1},f,g)$.
%Hence, changing a representative $\wt{e}$ we can replace $m$ with $m+1$, and

It remains to consider the case $m<n$. 
We claim that there exist $a$ and $b$ such that
$$(f+\eps^ng_0)(1+\eps^n a)=f+\eps^n e_0 b.$$
This can be checked by the descending induction in $n$, using the fact that $(\eps,f,e_0)=(1)$.
Indeed, we can find $a$ and $b$ such that
$$(f+\eps^ng_0)(1+\eps^n a)=f+\eps^n e_0 b+\eps^{n+1}g'_0.$$
Then we can apply the induction assumption to $n+1$ instead of $n$, $f'=f+\eps^n e_0b$ instead of $f$
(which still satisfies $(\eps,f',e_0)=(1)$) and $g'_0$ instead of $g_0$.

Hence, we reduce to the case $g=bx$ for some $b$.
%Hence, we have $\OO_d=\OO/(x+a')$ for some $a'\in A$. Let $e\equiv a\mod (x,a')$ for some $a\in A$. We can assume that $e$ is nontrivial, so $a\not\in (a')$. 
%By our assumption on $A$, this implies that $a'=ab$ for some $b\in A$.
%\noindent
%(iv) It is enough to consider the cases $O_c=O_d=\OO/x$ and $\OO_c=\OO/x$ and $\OO_d=\OO/(x+t)$.
%In the former case the assertion follows from part (iii). In the latter case the quotient $\OO/(x,x+t)$ is one-dimensional, so $e$ is either $0$ or invertible, hence we are done by
%part (ii).
%\noindent
%(v) It is enough to consider the case $\OO_c=\OO/x$ and $\OO_d=\OO/(x+at+bt^2)$, with $a,b\in k$. If $a\neq 0$ or $a=b=0$ then the proof is the same as in part (iv).
%Thus, it remains to consider the case $\OO_d=\OO/(x+t^2)$ and $e=t$, so $R=\coker(A)$, where 
%$$A=\left(\begin{matrix} x & t \\ 0 & x+t^2\end{matrix}\right).$$ 
Thus, 
$$M=\left(\begin{matrix} f & x \\ 0 & f+bx\end{matrix}\right).$$ 
Now the assertion follows from the identity
$$MB=BM^t, \ \text{where } B=\left(\begin{matrix} 0 & 1 \\ 1 & b\end{matrix}\right).$$ 
\end{proof}

\begin{cor}\label{Hecke-commute-cor}
Under the assumption of the theorem, all the Hecke operators $(T'_c)$ and $(T_c)$ commute.
\end{cor}

\begin{proof} The case of $(T'_c)$ follows from Theorem \ref{Hecke-commute-thm}. The rest follows from Lemma \ref{Hecke-tensor-lem}.
%Now the case of $(T_c)$ follows using the duality $T_c=\Dual T'_c \Dual$.
\end{proof}

\begin{remark}\label{noncomm-rem}
If the ground ring $A$ is a finite field then the Hecke operators corresponding to arbitrary effective Cartier divisors still commute. 
However, already for $A=k[\eps]/(\eps^2)$ the Hecke operators corresponding to non-simple divisors do not necessarily commute. 
Namely, for any simple divisor $c$, the characteristic measures of
$G(\OO)g_cG(\OO)$ and $G(\OO)g_c^2G(\OO)$ do not commute.
%Namely, if we take the extension $R$ of $\OO/x$ by $\OO/x^2$ with the class $e=\eps$, then $R^\circ$ is not isomorphic to $R$. We expect that the corresponding Hecke operators
%on $\C(\Bun(C))$ also do not commute.
In fact, we have
$$G(\OO)g_cG(\OO)g_c^2G(\OO)\neq G(\OO)g_c^2G(\OO)g_cG(\OO).$$
Indeed, we have
$$\left(\begin{matrix} f_c & 0 \\ \eps & f_c^2\end{matrix}\right)=\left(\begin{matrix} f_c & 0 \\ 0 & 1\end{matrix}\right)\left(\begin{matrix} 1 & 0 \\ \eps & 1\end{matrix}\right)
\left(\begin{matrix} 1 & 0 \\ 0 & f_c^2\end{matrix}\right)\in G(\OO)g_c^{-1}G(\OO)g_c^{-2}G(\OO),
$$
but it is easy to see that this matrix is not in a $G(\OO)$-double coset of any matrix of the form $\left(\begin{matrix} f_c & \eps b \\ 0 & f_c^2\end{matrix}\right)$, and so
it is not in $G(\OO)g_c^{-2}G(\OO)g_c^{-1}G(\OO)$.
\end{remark}

In the case when $C$ is a smooth curve over a finite local ring $A$, we have the following characterization of simple divisors on $C$.

\begin{lemma}\label{irr-et-div-lem} Assume $C$ is a smooth curve over a finite local ring $A$. 
A subscheme $c\sub C$ is a simple divisor if and only if
it is an irreducible affine closed subscheme, \'etale over $\Spec(A)$.
\end{lemma}

\begin{proof}
Assume that $c\sub C$ is a simple divisor. If $f_c\in \OO$ is a local equation of $c$ then its reduction $\ov{f}_c=f_{\ov{c}}$ is an equation of $\ov{c}$. 
Note that $c$ is flat over $A$. Indeed, the resolution $[\OO\rTo{f_c}\OO]$ for $\OO_c$ shows that 
$$\Tor_1^A(\OO_c,k)=\ker(\ov{\OO}\rTo{f_{\ov{c}}}\ov{\OO})=0.$$
Also, $\OO_c/\fm\OO_c$ is $k(\ov{c})$, so $c$ is unramified over $\Spec(A)$, hence it is \'etale.

Conversely, if $c=\Spec(A(c))\sub C$ is an irreducible affine closed subscheme, \'etale over $\Spec(A)$,
then its reduction $\ov{c}$ is a reduced point, so $A(c)/\fm A(c)=k(\ov{c})$ 
Furthermore, by flatness of $A(c)$, we have $I_c/\fm I_c=I_{\ov{c}}$, where $I_c\sub \OO$ (resp., $I_{\ov{c}}\sub \ov{\OO}$)
is the ideal of $c$ (resp., $\ov{c}$). This implies that a local generator of $I_{\ov{c}}$ can be lifted to a generator of $I_c$, so
$c\sub C$ is an effective Cartier divisor. 
\end{proof}

\section{Automorphic representations for an extension of length $2$}\label{aut-rep-sec}
%\section{Cuspidal functions for $\PGL_2(\A)$}

From now on we always assume that $C$ is a special nilpotent extension of $\ov{C}$ of length $2$.  

In this section we consider a general split reductive group $G$. Our goal is to describe a decomposition of the space $\SS(G(F)\backslash G(\A))$ (see Sec.\ \ref{orbit-dec-sec})
into a direct sum
of $G(\A)$-representations numbered by coadjoint orbits in $\fg^\vee(\ov{F})$ and to analyze the summands corresponding
to elliptic regular orbits. We start with a local representation theory in Sec.\ \ref{loc-rep-sec}. As in \cite{K-YD} where the case of $\PGL_2$
was considered, the main idea is to use the Mackey theory. In Sec.\ \ref{orbit-dec-sec}, we establish the orbit decomposition of $\SS(G(F)\backslash G(\A))$.
In Sec.\ \ref{reg-ss-sec}, \ref{reg-ell-sec} and \ref{adm-rep-sec} we give more details on the pieces of this decomposition corresponding to regular semisimple and elliptic orbits.
In Sec.\ \ref{Hitchin-sec} we establish the relation between the spaces of $G(\OO)$-invariants and the moduli spaces of Higgs bundles.
In Sec.\ \ref{finitary-sec} we determine the subspaces of {\it finitary functions} in $\SS(G(F)\backslash G(\A))$, i.e., those contained in admissible $G(\A)$-subrepresentations,
in the case $G=\PGL_2$.

\subsection{Local theory}\label{loc-rep-sec}

For a point $p\in C$, we have the square-zero extension $E:=F_p$ of the corresponding local field $\ov{E}:=\ov{F}_p$.
We have an element $\eps\in E$ such that $\eps^2=0$, $E/\eps E=\ov{E}$ and $\eps E\simeq \ov{E}$.

Let $N\sub G(E)$ denote the kernel of the (surjective) reduction homomorphism 
$$r:G(E)\to G(\ov{E}).$$
Then we have a natural isomorphism $\fg(\ov{E})\simeq N$.
%If the characteristic of $k$ is sufficiently large 
If there is a ring isomorphism $E\simeq \ov{E}[\eps]/(\eps^2)$ then $G(E)\simeq G(\ov{E})\ltimes \fg(\ov{E})$.
%: X\mapsto 1+\eps X$.
%We also have the compact subgroup $G(\OO_p)\sub G(E)$

%We assume that the Killing form $\lan\cdot,\cdot\ran$ on $\fg(\ov{E})=N$ is non-degenerate.
Let us fix a non-trivial additive character $\psi:\ov{E}\to U(1)$.

\begin{definition} \begin{enumerate} 

\item $ \Xi$ is the group of smooth unitary characters $ \chi:N\to \mC ^\ast$. 
 %and $\Xi (\mcO)\subset \Xi$ is the subgroup of characters $\chi$ such that $ chi _{|U(\mcO)}\equiv 1 $.
\item We identify $ \Xi$  with $\fg^\vee(\ov{E})$ through 
the isomorphism
 $a:\fg^\vee(\ov{E})\to \Xi$ given by $x\to \chi _x$ where $\chi _x(y)= \psi (\lan x,y\ran)$. 

%\item $I$ is the set of $\bar G$-orbits in $\Xi$.
\item For $\chi \in \Xi$ we denote by $\bO _\chi$  the $G(\ov{E})$-orbit of $\chi$, by   $St (\chi)\subset G(E)$ the stabilizer of $\chi$ 
in $G(E)$, and by $\ov {St} (\chi)= St (\chi)/N \subset G(\ov{E})$ the stabilizer 
of $\chi$ in $G(\ov{E})$. 
%\item An open compact subgroup $K\subset G(E)$ is {\it admissible} if it contains $\fg(\fm^n)\sub U$ and $\ov{K}:=r(K)$ contains a congruence subgroup in $G(\ov{\OO}_p)\sub G(\ov{E})$. 
%is of the containes a subgroup of the form 
%$r^{-1}(\bar K)\cap K_0$ where $\bar K\subset \bar G$ is a congruence subgroup.
\item For an open compact subgroup $K\sub G(E)$ we denote by $\Xi_K\sub \Xi$ the subgroup of characters $\chi$ such that $\chi|_{N\cap K}\equiv 1$.
\item For an orbit $\bO \subset \Xi$ and an open compact subgroup $K\subset G(E)$
we define $\De _\bO (K): = \bO \cap \Xi_K/\ov{K}$, where $\ov{K}:=r(K)$.

\end{enumerate} 
\end{definition}

%\begin{remark} \begin{enumerate}
%\item The subgroup $\Xi(\mcO)$ and the notion of admissibility depend on a choice of an additive character $\psi$. 
%\item There exist an additive character $\psi$ such that  $\Xi (\mcO)= a(\fg (\mcO))$.
%\item Any open compact subgroup of $G$ contains an  admissible subgrouo for an appropriate choice of an additive character $\psi$. 
%\end{enumerate} 
%\end{remark} 

\begin{remark} It is easy to see that the set $\De _\bO (K)$ is finite for every regular semisimple $\bO$ 
(see the proof of Proposition \ref{K-inv-prop}(1) below).
%shown in \cite{K-YD} that in the case $G=\PGL_2$ 
\end{remark}

\begin{definition} 
\begin{enumerate}

\item $\mcR$ denotes the category of smooth representations of $G(E)$.
%\item  $Ir _{G(E)}$ denotes the set of equivalence classes of smooth irreducible representations of $G(E)$.

\item For $\chi \in \Xi$, $\mcC _\chi$ denotes the category of smooth representations $\rho$ of $ St (\chi) $ such that $\rho|_{N}= \chi Id$. 

\item  $Ir _\chi $ denotes the set of equivalence classes of irreducible representations in $\mcC_\chi$.
%smooth representations of $ St (\chi) $.

\item $\ind _\chi: \mcC _\chi \to \mcR: (\rho ,W)\to (\pi(\chi ,\rho) ,V):= \ind _ { St (\chi)}^{G(E)}(\rho ,W)$ 
is the compactly supported induction functor.
\item $ J_\chi : \mcR \to \mcC _\chi $ is the functor given by 
$J_\chi (\pi,V)= V_{N,\chi}$ where $V_{N,\chi}$ is the space of   $\chi$-coinvariants of $V$ with respect to $N$.
\item For a $G(\ov{E})$-orbit $\bO \subset \Xi$ we denote by 
$\mcR _\bO$ the full subcategory of representations $\pi$ such that $ J_{\chi } (\pi) =\{0\}$ for $\chi  \not \in \bO$.

\end{enumerate} 
\end{definition}

\begin{lemma}\label{mackey-lem}
\begin{enumerate}

\item The functors $J_\chi$ and $\ind_\chi$ are exact, $\ind_\chi$ lands in $\mcR_\bO$, and $J_\chi|_{\mcR_\bO}$ is left adjoint to $\ind_\chi$.

\item For any $\chi \in \Xi,  \rho  \in  Ir_\chi $ the representation $\ind_\chi(\rho)$ of $G$ is irreducible.
Furthermore, $\ind_\chi(\rho)$ is isomorphic to $\ind_{\chi'}(\rho')$ if and only if the pairs $(\chi,\rho)$ and $(\chi',\rho')$ are conjugate.
%\item $  \mcR _{\bO _\chi} =  \mcR _{\bO _\chi '} $ if 
%$\bO _\chi = \bO _{\chi '}$. 

%We will write $ \mcR _\bO := \mcR _{\bO _\chi} ,\chi \in \bO$.

\item Any smooth irreducible representation of $G$  is equivalent to $\ind_\chi(\rho) $ for some pair
 $ (\chi ,\rho) $, with $\chi\in \Xi$, $\rho\in Ir_\chi$.

\item Let $\bO_\chi$ be the orbit of $\chi$. If $V\in \mcR_{\bO_\chi}$ satisfies $J_\chi(V)=0$ then $V=0$.

\end{enumerate} 
\end{lemma}

\begin{proof}
Part (1) is well known, and parts (2) and (3) follow from the Mackey's theory. For part (4), we use the fact that a nonzero representation $V$ has an irreducible
subquotient $V_0$. Then we would get $J_\chi(V_0)=0$. But we also have $V_0\in \mcR_{\bO_\chi}$, so $J_{\chi'}(V_0)=0$ for any character $\chi'$.
Since $V_0=\ind_{\chi'}(\rho)$ for some $\chi'$ and $\rho\in Ir_{\chi'}$ we get a contradiction.
\end{proof}

\begin{prop}\label{ind}  
\begin{enumerate}
\item For any $(\rho ,W)\in \mcC _\chi $ and an open compact subgroup $K\subset G(E)$, for $(\pi(\chi ,\rho),V)=\ind_\chi(\rho,W)$, one has
$$V^K= \bigoplus  _{\delta=g_\delta\cdot\chi \in \De _{\bO _\chi}(K)} 
W ^{g_\delta ^{-1}K g_\delta \cap St _\chi}$$
\item $J_\chi \circ \ind _\chi \simeq \Id _{ \mcC _\chi}$.
%\item $ \mcR _\bO \cap \mcR _{\bO '}= 0$ for $\bO \neq \bO '$. 
\item The functors $ \ind _\chi : \mcC _\chi \to \mcR _{\bO _\chi} $ and $ J_\chi : \mcR _{\bO _\chi} \to \mcC _\chi $ are equivalences of categories.
\end{enumerate} 
\end{prop} 

\begin{proof} (1) Recall that $V$ is the space of smooth functions $f:G\to W$
such that $f(hg)=\rho(h)f(g)$ for $h\in \St_\chi$, such that the support of $f$ is compact modulo left shifts by $\St_\chi$. In particular,
$f(ug)=\chi(u)f(g)$ for $u\in N$.
The space $V^K$ consists of $f\in V$ such that $f(gk)=f(g)$ for $k\in K$. Note that if $gkg^{-1}\in N$ then we get
$$f(g)=f(gk)=\chi(gkg^{-1})f(g),$$
so $f(g)=0$ unless $\chi|_{gKg^{-1}\cap N}\equiv 1$, or equivalently, $g^{-1}\chi\in \Xi_K$.
Similarly, we see that for any $g$ one has $f(g)\in W^{gKg^{-1}\cap\St_\chi}$.
Furthermore, since the support of $f$ is compact modulo left shifts by $\St_\chi$, there are only finitely many
double cosets $\St_\chi gK\in \St_\chi\backslash G/K$ such that $f(g)\neq 0$. Thus, we get an embedding
$$V^K\to \bigoplus  _{\delta=g_\delta\cdot\chi \in \De _{\bO _\chi}(K)} 
W ^{g_\delta ^{-1}K g_\delta \cap St _\chi}: f \mapsto (f(g_\delta^{-1})).$$
Conversely, for each $g$ and each $w\in W^{gKg^{-1}\cap\St_\chi}$, we have a well defined
function $f_{g,K,w}\in V^K$ given by 
$$f_{g,K,w}(g_1)=\begin{cases} \rho(h)w, & g_1=hgk, h\in \St_\chi, k\in K,\\ 0, g_1\not\in \St_\chi gK.\end{cases}$$
This proves the claimed decomposition.

\noindent
(2) We use the same notation as in the proof of (1). We need to show that the map 
$$V\to W: f\mapsto f(1)$$
induces an isomorphism $J_\chi(V)\rTo{\sim} W$. The surjectivity is clear (e.g., using functions $f_{1,K,w}$ for sufficiently small $K$).

Let $V_0\sub V$ denote the subspace of $f$ such that $f(1)=0$.
We need to show that functions of the form $uf-\chi(u)f$, where $u\in N$, span $V_0$.
Note that for any $f\in V$ and $u\in N$, we have
$$(uf-\chi(u)f)(g)=f(gu)-\chi(u)f(g)=[\chi(gug^{-1})-\chi(u)]\cdot f(g).$$
Now for any $g\not\in\St_\chi$, let $u\in N$ be such that $\chi(gug^{-1})\neq \chi(u)$. Then for
any sufficiently small $K$ we have $\chi((gk)u(gk)^{-1})=\chi(gug^{-1})$, hence for $w\in W^{gKg^{-1}\cap \St_\chi}$,
the function $f_{g,K,w}$ is proportional to $uf_{gK,w}-\chi(u)f_{g,K,w}$. It remains to note that such functions span $V_0$.

\noindent
(3) We have to check that for any $V\in\mcR _{\bO _\chi}$ the adjunction map 
$$c_V:V\to \ind_\chi J_\chi(V)$$ 
is an isomorphism. Let $K=\ker(c_V)$, $C=\coker(c_V)$. By part (2), $J_\chi(c_V)$ is an isomorphism, hence $J_\chi(K)=J_\chi(C)=0$.
By Lemma \ref{mackey-lem}(4), this implies that $K=C=0$, so $c_V$ is an isomorphism.
\end{proof}

%Let $(U\cap K)^\perp$ denote the set of $x\in U$ such that $\chi_x\in \Xi_K$, i.e., $\psi(\lan x, U\cap K\ran)\equiv 1$.

%Recall that a smooth representation $(\pi,V)$ of $G(E)$ is {\it admissible} if $V^K$ is finite dimensional for every open compact subgroup $K\sub G(E)$.

For $m\in \fg^\vee({\ov{E}})$, let $\lan m\ran^\perp\sub N$ denote the subgroup of $u$ such that $\lan m,u\ran=0$.
We denote by $H_m$ the quotient of $\St_\chi$, where $\chi=\chi_m$, by the normal subgroup $\lan m\ran^\perp$,
so that we have a commutative diagram
\begin{diagram}
1&\rTo{}& N&\rTo{}& \St_\chi&\rTo{}& \ov{\St}_\chi&\rTo{}&1\\
&&\dTo{\lan m,\cdot\ran}&&\dTo{}&&\dTo{\id}\\
1&\rTo{}& \ov{E} &\rTo{}& H_m&\rTo{}& \ov{\St}&\rTo{}&1
\end{diagram}
and $H_m$ is a central extension of $\ov{\St}_\chi$ by $\ov{E}$.

\begin{lemma}\label{H-commutative-local-lem} 
Assume that $m$ is regular semisimple. 
Then $H_m$ is commutative.
\end{lemma}

\begin{proof}
Note that in this case we have the maximal torus $T_m\sub G$ defined over $\ov{E}$ such that $\ov{\St}_\chi=T_m(\ov{E})$.
If $T_m$ is defined over $\Z$ then we have the commutative subgroup $T_m(E)\sub \St_{\chi}$ surjecting onto $T_m(\ov{E})$.
Hence, in this case $H_m$ is commutative (since it is generated by $T_m(E)$ and the center).
Since the isomorphism class of the group $H_m$ depends only on the conjugacy class of $m$ in $\fg^\vee(\ov{E})$, using the fact that any split maximal
torus is conjugate to a torus defined over $\Z$, we see that $H_m$ is commutative whenever $T_m$ is split.

In the general case, there exists a finite separable field extension $\ov{E}\sub \ov{E}'$ such that $T_m$ splits over $\ov{E}'$. Thus, it is enough to prove that for any such extension 
the nilpotent extension $E\to \ov{E}$ extends to a nilpotent extension $E'\to \ov{E}'$ (since then $H_m$ is contained in a similar extension of $T_m(\ov{E}')$
defined using $G(E')$). 
%It is enough to consider the case when the finite field extension is of the form
We have
$\ov{E}'=\ov{E}[x]/(f(x))$, for some monic polynomial $f(x)\in \ov{E}[x]$. But then we can take
$E'=E[x]/(\wt{f}(x))$ as the required nilpotent extension of $\ov{E}'$, where $\wt{f}(x)$ is any monic lift of $f(x)$ to $E[x]$.
\end{proof}

Thus, in the case when $m$ is regular semisimple, all irreducible representations in $\CC_\chi$, where $\chi=\chi_m$, are $1$-dimensional.

\begin{prop}\label{K-inv-prop}
Assume that $\chi=\chi_m$, where $m$ is regular semisimple, and $\rho$ is a smooth character of $\St_\chi$, such that $\rho|_N=\chi$.
Let $V=\ind_\chi(\rho)$.

\noindent
(i) The space $V^K$ is finite-dimensional, and
$V^K=0$ unless there exists $g\in G(\ov{E})$ such that $\rho\circ \Ad(g)^{-1}|_{K\cap \St_{g\chi}}\equiv 1$
(where we use the fact that $\Ad(g)(\St_\chi)=\St_{g\chi}$).

\noindent
(ii) Assume $\fg_{\ov{\OO}}$ is equipped with a $G_{\ov{\OO}}$-invariant perfect symmetric bilinear pairing $\fg_{\ov{\OO}}\times \fg_{\ov{\OO}}\to \ov{\OO}$. We use
this pairing to identify $\fg(\ov{E})$ with $\fg^\vee(\ov{E})$. 
Assume $K=G(\OO)$, $N\cap K=\fg(\ov{\OO})$, $\psi|_{\OO}\equiv 1$, $\psi|_{\fm^{-1}}\not\equiv 1$,  $m\in \fg(\ov{\OO})$ is such that $\ov{m}\in \fg(\ov{\OO}/\fm)$ is still regular semisimple, and $\rho|_{G(\OO)\cap \St_\chi}\equiv 1$. Assume also that the torus $T_m$ (defined over $\ov{E}$) splits over an unramified extension of $\ov{E}$.
Then $V^K$ is $1$-dimensional.
\end{prop}

\begin{proof}
(i) Set $\Om=\Om_\chi$. By Proposition \ref{ind}(1), it is enough to check that the set $\Delta_{\bO}(K)$ is finite. Since
$\Om\cap \Xi_K$ is compact (as $\Om$ is closed), this follows from the fact that $\ov{K}$-orbits on $\Om$ are open. 

\noindent
(ii) It is enough to check that $\Om\cap \fg(\ov{\OO})$ is a single $G(\ov{\OO})$-orbit.

\noindent
{\bf Step 1. Split case}. Let $B$ be the Borel subgroup containing the split torus $T_m$, $U\sub B$ its unipotent radical. Using the Iwasawa decomposition,
we are reduced to checking that if $u\in U(\ov{E})$ satisfies $\Ad(u)m\in \fg(\ov{\OO})$ then $u\in U(\ov{\OO})$. 

Let us consider the map
$$\th: U\to \fu=\fb/\ft: u\mapsto \Ad(u)m \mod \ft\in \fb/\ft.$$
Let  $\fu_1\supset \fu_2\supset\ldots$ denote the root-height filtration, and let
$U=U_1\supset U_2\supset\ldots$ denote the corresponding group filtration.
%be the lower central series, 
%so $U_{n+1}=[U_n,U]$, and let be the corresponding Lie subalgebras. 
Then we have an identification
$$U_n/U_{n+1}\simeq \fu_n/\fu_{n+1}\simeq \bigoplus_{\a=\sum m_i\a_i: \sum m_i=n} \fg_{\a},$$
where $(\a_i)$ are simple roots. Furthermore, the map $U_n/U_{n+1}\to \fu_n/\fu_{n+1}$, induced by $\th$, is given by the multiplication with $\a(m)$ on the
summand $\fg_{\a}$.
%For each positive root $\a$, we have a subgroup $U_{\ge \a}\sub U$ (resp., $U_{>\a}$) and the corresponding subalgebra $\fu_{\ge \a}$ (resp., $\fu_{>\a}$).
%We have $\th(U_{\ge \a})\sub \fg_{\ge \a}$, and the induced map of additive groups
%$$U_{\ge \a}/U_{>\a}\to \fu_{\ge \a}/\fu_{>\a}$$ 
%is given by the multiplication with $\a(m)$. 
Note that our assumption on $m$ means that $\a(m)\in\ov{\OO}^*$ for all positive roots $\a$.

Now we can prove our statement by the descending induction: assume we know it for elements of $U_{n+1}$, and let $u\in U_n(\ov{E})$ be
such that $\Ad(u)m\in \fg(\ov{\OO})$. Since $\a(m)\in \ov{\OO}^*$ for all $\a$, we obtain that the reduction of $u$ modulo $U_{n+1}$ belongs to $(U_n/U_{n+1})(\ov{\OO})$.
Hence, we can find $u_0\in U_n(\ov{\OO})$ and $u_1\in U_{n+1}(\ov{E})$ such that $u=u_0u_1$. But then $\Ad(u_1)m\in \fg(\ov{\OO})$, so by the induction assumption 
$u_1\in U_{n+1}(\ov{\OO})$, and so, $u\in U_n(\ov{\OO})$.

\noindent
{\bf Step 2. Reduction to finite rings}. We claim that it is enough to check the same statement with $\ov{\OO}$ replaced by $\ov{\OO}/\fm^n$ for any $n\ge 1$.
Indeed, suppose $\Om\cap \fg(\ov{\OO})$ contains two distinct $G(\ov{\OO})$-orbits, $O_1$ and $O_2$. Let $x\in O_1\setminus O_2$.
Since $G(\ov{\OO})$-orbits are closed in $\Om\cap \fg(\ov{\OO})$, this means that there exists $n\ge 1$ such that $(x+\fg(\fm^n))\cap O_2=\emptyset$, i.e.,
$x\not\in O_2+\fg(\fm^n)$. But this contradicts transitivity over $\ov{\OO}/\fm^n$.

\noindent
{\bf Step 3. General case}. By Step 2, it is enough to prove transitivity of the action of $G(\ov{\OO}/\fm^n)$ on $\Om(\ov{\OO}/\fm^n)$. Using the Weil restriction functor,
we can consider this as an action of the group of $k$-points of an algebraic group over $k$ on the set of $k$-points of a variety over $k$. 
By Step 1, we know that this action becomes transitive over a finite extension of the finite field $k$. Furthermore, the stabilizer subgroup is obtained by
applying the Weil restriction functor to $T_m(\ov{\OO}/\fm^n)$, so it is connected. Using triviality of the Galois cohomology $H^1$ of $k$ with coefficients in a connected algebraic group, 
we deduce the required transitivity over $k$.
\end{proof}

\begin{remark}
For an open compact subgroup $K\sub G(\ov{E})$ and $\chi\in \Xi$, let 
$$G_{m,K}:=\{g\in G(\ov{E}) \ |\  \chi|_{\Ad(g)(N\cap K)}\equiv 1\}.$$  
This set is invariant under right shifts by $\ov{K}$ and by left shifts by $\ov{\St}_\chi$.
The quotient $A_{\chi}:=G_{m,K}/\ov{K}$ is an analog of the affine Springer fiber.
The space of $K$-invariants in $\ind_\chi(\rho)$, where $\rho$ is a character of $\St_\chi$ extending
$\chi$, can be interpreted as a space of $\ov{\St}_\chi$-invariant sections of a $\ov{\St}_\chi$-equivariant line bundle
over $A_{\chi}$. In the case of $G=PGL_2$, $E=\ov{E}[\eps]/(\eps^2)$, and $K=G(\OO)$, the orbits of $\St_\chi$ on $A_\chi$
and the resulting spaces of $K$-invariants are described explicitly in \cite{K-YD}.
\end{remark}

\subsection{Adelic case: orbit decomposition}\label{orbit-dec-sec}

We use the notations of Sec.\ \ref{nilp-adele-sec}.

Let us fix a nontrivial additive character $\psi$ of the finite field $k$.

For each line bundle $M$ over $\ov{C}$, we have the corresponding module $M(\ov{\A})$  over the ring of adeles $\A$, obtained as the restricted product of $M_p\ot \ov{F}_p$.
The corresponding twisted principal adeles $M(\ov{F})$ are identified with the space of rational sections of $M$.
For example, $\om_{\ov{C}}(\ov{F})$ is the $1$-dimensional $\ov{F}$-linear space of rational $1$-forms on $\ov{C}$.
The character $\psi$ induces the character
$$\psi_{\ov{C}}:\om_{\ov{C}}(\ov{\A})/\om_{\ov{C}}(\ov{F})\to U(1): \a\mapsto \psi(\sum_p \Tr_{k(p)/k}\Res_p \a).$$

Recall that for our nilpotent extension $C$ of $\ov{C}$, the kernel $L:=\NN\sub \OO_C$ of the projection to $\OO_{\ov{C}}$ is a line bundle on $\ov{C}$.
Thus, we have an exact sequence
$$0\to L(\ov{\A})\to \A\to \ov{\A}\to 0.$$
Let $N_{\A}\sub G(\A)$ denote the kernel of the reduction homomorphism 
$$r:G(\A)\to G(\ov{\A}),$$
and let $N_F=N_{\A}\cap G(F)$.
Then we have a natural isomorphism $\fg\ot L(\ov{\A})\rTo{\sim} N_{\A}: X\mapsto 1+ X$, inducing an isomorphism $\fg\ot L(\ov{F})\simeq N_F$. 

%We assume that the Killing form $\lan\cdot,\cdot\ran$ on $\fg(\ov{F})$ is non-degenerate.
The group $\Xi$ of characters of $\fg\ot L(\ov{\A})/\fg\ot L(\ov{F})$ is naturally identified with $\fg^\vee\ot L^{-1}\om_{\ov{C}}(\ov{F})$. Namely, with 
$\eta\in \fg^\vee \ot L^{-1}\om_{\ov{C}}(\ov{F})$ we associate the character
$$X\mapsto \psi_\eta(X):=\psi_{\ov{C}}(\lan \eta,X\ran).$$

Thus, we can
view characters in $\Xi$ as characters of $N_{\A}/N_F$.
%We also have the compact subgroup $G(\OO_p)\sub G(E)$

%Furthermore, the group of characters of $\fg(\ov{\A})/\fg(\ov{F})$ is naturally identified
%Wwith $\om_{\ov{F}}\ot_{\ov{F}} \fg(\ov{F})$, where $\om_{\ov{F}}$ is the space of rational $1$-forms on $\ov{C}$. Namely, with 
%$\eta\in\om_{\ov{F}}\ot_{\ov{F}} \fg(\ov{F})$ we associate the character
%$$X\mapsto \psi_\eta(X):=\psi(\sum_p \Res_p\lan \eta,X\ran),$$

%We denote by $\SS(G(F)\backslash G(\A))$ the space of 

Let $\Ga\sub G(\A)$ be a discrete subgroup such that $\Ga\cap N_{\A}=N_F$ and the image of $\Ga$ in $G(\ov{\A})$ is $G(\ov{F})$, so $\Ga/N_F\simeq G(\ov{F})$. 
We denote by $\SS(\Ga\backslash G(\A))$ the space
of compactly supported functions on $\Ga\backslash G(\A)$ which are $K$-invariant under the right shifts by some open compact subgroup $K\sub G(\A)$.
The group $G(\A)$ acts on $\SS(\Ga\backslash G(\A))$ by $gf(g')=f(g'g)$. 

We consider $\SS(\Ga\backslash G(\A))$ as a subspace of the space $\C_{lc}(N_F\backslash G(\A))$
of locally constant functions on $N_F\backslash G(\A)$.
%{\color{red} Definition of the latter space??? Compact support modulo $G(F)$???}
The latter space has, in addition to the standard $G(\A)$-action, an action of the compact
abelian group $N_{\A}/N_F$ given by $uf(g)=f(ug)$. 

For every $\eta\in \fg^\vee \ot L^{-1}\om_{\ov{C}}(\ov{F})$, we have the corresponding $G(\A)$-invariant subspace
$$\C_{lc}(N_F\backslash G(\A))_{\eta}:=
\{f\in \C_{lc}(N_F\backslash G(\A)) \ |\  f(ug)=\psi_\eta(u)\cdot f(g),  u\in N_{\A}\}.$$
%{\color{red} In what sense do we have a direct sum decomposition
%$$\CC_{lc}((1+\eps\fg(\ov{F}))\backslash G(\A))=\bigoplus_{\eta\in\om_{\ov{F}}\ot_{\ov{F}} \fg(\ov{F})} \CC_{lc}((1+\eps\fg(\ov{F}))\backslash G(\A))_{\eta}$$
%}

We also have a collection of mutually orthogonal projectors compatible with the $G(\A)$-action, 
%\begin{align*}
$$\Pi_\eta:\C_{lc}(N_F\backslash G(\A))\to \C_{lc}(N_F\backslash G(\A))_\eta: 
\Pi_\eta f(g)=\int_{N_{\A}/N_F} \psi_\eta(u)^{-1}\cdot f(ug)du,$$
%\end{align*}
where the measure $du$ is normalized so that the volume of $N_{\A}/N_F$ is $1$.

\begin{lemma}\label{P-eta-lem} For every $f\in \C_{lc}(N_F\backslash G(\A))$ and a compact subset $B\sub G(\A)$, there exists a finite set of characters $S$,
such that 
$\Pi_\eta f(g)=0$ for all $g\in B$ and $\eta\not\in S$. Furthermore,
one has $f(g)=\sum_{\eta} \Pi_\eta f(g)$.
\end{lemma}

\begin{proof}
Since $N_F\backslash N_{\A}$ is compact, increasing $B$ we can assume that the image of $B$ in $N_F\backslash G(\A)$ is closed under left multiplication with 
$N_F\backslash N_{\A}$.
Since $f$ is locally constant, there exists an open compact subgroup $K\sub G(\A)$ such that $f(kg)=f(g)$ for every $g\in B$ and $k\in K$.
Let $K_0=K\cap N_\A$, and let $\ov{K}_0$ denote the image of $K_0$ in $N_{\A}/N_F$.
In the integration defining $\Pi_{\eta} f(g)$ we can first integrate over $K_0$-cosets. If the restriction $\psi_{\eta}|_{K_0}$ is nontrivial, we get zero. Hence,
we can take as $S$ the finite set of $\eta$ such that $\psi_{\eta}|_{K_0}$ is trivial.
\end{proof}

Let us consider the $G(\ov{F})$-action on $\C(N_F\backslash G(\A))$ given by $g_0f(g)=f(\wt{g}_0^{-1}g)$, where $\wt{g}_0\in \Ga$ is a lift of
$g_0$. Then for $g_0\in G(\ov{F})$, we have
$$g_0\Pi_\eta f=\Pi_{g_0(\eta)} g_0f,$$
%$$g_0\SS((1+\eps\fg(\ov{F}))\backslash G(\A))_{\eta}=\SS((1+\eps\fg(\ov{F}))\backslash G(\A))_{g_0\eta},$$
where we use the coadjoint action of $G(\ov{F})$ on $\fg^\vee(\ov{F})$.

Now for a $G(\ov{F})$-orbit $\Om\sub \fg^\vee \ot L^{-1}\om_{\ov{C}}(\ov{F})$, 
we can consider the sum of projectors
$$\Pi_{\Om}(f):=\sum_{\eta\in\Om} \Pi_{\eta}(f).$$
%=\sum_{g_0\in G(\ov{F})/\St_{\eta_0}} g_0\Pi_{\eta} g_0^{-1}f.$$
%{\color{red} On which functions is it well defined???} 
By Lemma \ref{P-eta-lem}, for $f\in \C_{lc}(N_F\backslash G(\A))$, $\Pi_\Om(f)$ is a well defined function on $N_F\backslash G(\A)$.
Furthermore, if $f$ is $G(\ov{F})$-invariant then $\Pi_\Om(f)$ is also $G(\ov{F})$-invariant, so we get a well defined operator
$$\Pi_{\Om}:\SS(G(F)\backslash G(\A))\to \C_{lc}(G(F)\backslash G(\A)).$$
%The projector $\Pi_{\Om}$ commutes with $G(\ov{F})$-action, hence, it induces a projector on the space of $G(\ov{F})$-invariants, 
%$$\Pi_{\Om}:L^2(G(F)\backslash G(\A))\to L^2(G(F)\backslash G(\A))_\Om: f\mapsto \sum_{\eta\in \Om}\int_{\fg(\ov{\A})/\fg(\ov{F})} \psi_\eta(-X)\cdot f((1+\eps X)g)dX.$$
%$$\Pi_{\Om}:L^2(G(F)\backslash G(\A))\to L^2(G(F)\backslash G(\A))_\Om: \Pi_{\Om}f(g)=\Pi_\eta f(g) \text{ for any } \eta\in \Om.$$ 

\begin{lemma}
For every $f\in \SS(\Ga\backslash G(\A))$, one has 
\begin{enumerate}
\item $\Pi_{\Om}(f)\in \SS(\Ga\backslash G(\A))$,
\item $\Pi_{\Om}(f)=0$ for all $\Om$ except for a finite number, and
\item $\sum_{\Om} \Pi_{\Om}(f)=f$.
\end{enumerate}
\end{lemma}

\begin{proof}
Since $f$ has compact support we can choose an $N_{\A}/N_F$-invariant compact subset $K\sub N_F\backslash G(\A)$ such that $f(g)=0$ for $g\not\in \Ga K$.
Now applying Lemma \ref{P-eta-lem}, we see that there exists a finite set of characters $S$ such that $\Pi_\eta f(g)=0$ for all $\eta\not\in G(\ov{F})S$ and all $g$.
%Here is another way to write the projector $\Pi_{\Om}$.
%Consider the Fourier transforms between functions on $\fg(\ov{\A})/\fg(\ov{F})$ and $\fg(\ov{F})$. Then we have
%$$\Pi_{\Om}f(g)=\int_{\fg(\ov{\A})/\fg(\ov{F})}\widehat{\de_{\Om}}\cdot f((1+\eps X)g) dX=\sum_{Y\in \fg(\ov{F})}\de_{\Om}(Y)\hat{f_g}=\sum_{\eta\in\Om}\hat{f_g}(\eta),$$
%where $f_g$ is the function on $\fg(\ov{\A})/\fg(\ov{F})$ given by $X\mapsto f((1+\eps X)g)$. Since the function $f_g$ is ??? {\color{red} What property do we have (something
%like locally constant modulo the adjoint action of $G(\ov{F})$)?},
%its Fourier transform $\hat{f_g}$ is supported on finitely many orbits $\Om$. This easily implies all our assertions. 
\end{proof}

The above Lemma shows that we have a direct sum decomposition of $G(\A)$-representations (in algebraic sense),
\begin{equation}
\SS(\Ga\backslash G(\A))=\bigoplus_{\Om}\SS(\Ga\backslash G(\A))_{\Om}.
\end{equation}
It is clear that the term corresponding to $\Om=0$ is the subspace $\SS(G(\ov{F})\backslash G(\ov{\A}))\sub \SS(\Ga\backslash G(\A))$.

%\begin{lemma}\label{str-cusp-orbit-decomp}
%One has the induced decomposition
%$$\SS_{\strcusp}(G(F)\backslash G(\A))=\bigoplus_{\Om\neq 0}\SS_{\strcusp}(G(F)\backslash G(\A))_{\Om},$$
%where $\SS_{\strcusp}(G(F)\backslash G(\A))_{\Om}:=\SS_{\strcusp}(G(F)\backslash G(\A))\cap \SS(G(F)\backslash G(\A))_{\Om}$.
%\end{lemma}

%\begin{proof} 
% \end{proof}

Next, we will describe more explicitly the pieces $\SS(\Ga\backslash G(\A))_{\Om}$.
For $\eta\in \fg^\vee\ot L^{-1}\om_{\ov{C}}(\ov{F})$, let $\bSt_\eta\sub G_{\ov{F}}$ denote the stabilizer of $\eta$, viewed as an algebraic subgroup defined
over $\ov{F}$. Then the stabilizer $\St_\eta$ of $\eta$ in $G(\A)$ is the preimage of $\bSt_\eta(\ov{\A})$ under the projection $G(\A)\to G(\ov{\A})$.
It is clear that $\St_\eta$ contains $N_{\A}$, $\St_{\eta}/N_{\A}\simeq \bSt_\eta(\ov{\A})$, and$(\St_\eta\cap \Ga)/N_F\simeq \bSt_\eta(\ov{F})$.
Note that $(\Ga\cap\St_\eta)\backslash\Ga\simeq \bSt_\eta(\ov{F})\backslash G(\ov{F})$.

We denote by $L_\eta\sub \St_\eta$ the preimage of $\bSt_\eta(\ov{F})$ under the homomorphism $\St_\eta\to \bSt_\eta(\ov{\A})$.
We have $L_\eta=N_{\A}\cdot (\St_\eta\cap \Ga)$.

\begin{claim} There exists a unique character $\wt{\psi}_\eta$ of $L_\eta$, trivial on $\St_\eta\cap \Ga$ and equal to $\psi_\eta$ on $N_{\A}$.
%The adjoint action of $\St_\eta$ on $L_\eta$ preserves $\wt{\psi}_\eta$.
\end{claim}

\begin{proof} We have $L_\eta=N_{\A}\cdot(\St_{\eta}\cap \Ga)$, and $N_{\A}\cap (\St_{\eta}\cap \Ga)=N_F$. 
\end{proof}

\begin{definition} We denote by $\wt{\SS}_\eta\sub \C_{lc}(N_F\backslash G(\A))_\eta$ the space of functions $f$ on $G(\A)$, such that $f(lg)=\wt{\psi}_{\eta}(l)f(g)$ for $l\in L_\eta$,
with compact support modulo left translations by $L_\eta$.
\end{definition}

Let us denote by $\Om_\eta$ the orbit of $\eta$.

\begin{lemma}\label{kappa-lem}
The map 
$$\kappa_\eta f(g)=\sum_{\ga\in (\Ga\cap\St_\eta)\backslash\Ga}f(\ga g).$$
is well defined and gives an isomorphism of $G(\A)$-representations
$$\kappa_\eta:\wt{\SS}_\eta\rTo{\sim} \SS(\Ga\backslash G(\A))_{\Om_\eta}.$$
\end{lemma}

\begin{proof} 
{\bf Step 1.}
We claim that the sum in the definition of $\kappa_\eta$ is finite.
It is enough to check that for any $f\in \wt{\SS}_{\eta}$ one has $f(\ga)\neq 0$ for finitely many $\ga\in (\Ga\cap\St_\eta)\backslash\Ga$.
Given $f\in\wt{\SS}_{\eta}$, let $K\sub G(\A)$ be an open compact subgroup such that $f(gk)=f(g)$ for $g\in G(\A)$, $k\in K$.
%$\La\sub \fg(\ov{\A})=U_{\A}$ be a lattice such that $\psi_{\eta}|_{\La}\equiv 1$. Then
We claim that $f$ is supported on the set $G_{\eta,K}$ of $g$ such that $\psi_{\Ad(g^{-1})\eta}|_{N_{\A}\cap K}\equiv 1$.
Indeed, for $u\in N_{\A}\cap K$ we have
$$f(g)=f(gu)=f(gug^{-1}\cdot g)=\psi_{\eta}(gug^{-1})f(g)=\psi_{\Ad(g^{-1})(\eta)}(u)f(g).$$
Thus, it is enough to check the set $(\Ga\cap\St_\eta)\backslash(\Ga\cap G_{\eta,K})$ is finite.
The map $g\mapsto \Ad(g^{-1})\eta$ identifies this set with the set $G(\ov{F})\eta\cap K'$, where 
$K'\sub \fg^\vee\ot(L^{-1}\om)(\ov{\A})$ is the orthogonal to $N_{\A}\cap K$ (with respect
to the pairing induced by $\psi$). 
But the latter set is contained in $\fg^\vee \ot(L^{-1}\om)(\ov{F})\cap K'$ which is finite.
%$N_F$ is discrete and $\La$ is compact, so $N_F\cap \La$ is finite. 
%Now, since $\eta$ is semisimple, the corresponding orbit $\Om\sub \fg$ is closed. Hence, intersection $\Om(\ov{F})
%The set $G_{\eta,K}$ is a union of $(\St_\eta,K)$ double cosets, and claim that
%$\St_\eta\backslash G_{\eta,K}/K$ is finite.???

\noindent
{\bf Step 2.}
For $f\in \SS(\Ga\backslash G(\A))$, we have
\begin{align*}
&\Pi_{\Om_\eta}f(g)=\sum_{\ga\in \Ga/\Ga\cap \St_\eta}\int_{N_{\A}/N_F}\psi_{\eta}(\ga^{-1} u\ga)^{-1}f(ug)du=\\
&\sum_{\ga\in \Ga/\Ga\cap \St_\eta}\int_{N_{\A}/N_F}\psi_\eta(u)^{-1}f(u\ga^{-1}g) du=
\kappa_\eta(\Pi_\eta(f)).
\end{align*}
Note that $\Pi_\eta(f)$ has compact support modulo $L_\eta$, so it belongs to $\wt{\SS}_\eta$.
This shows that the image of $\kappa_\eta$ contains $\SS(\Ga\backslash G(\A))_{\Om_\eta}$.

\noindent
{\bf Step 3.}
For $f\in \wt{\SS}_\eta$, we have
\begin{align*}
&\Pi_\eta\kappa_\eta f(g)=\int_{N_{\A}/N_F}\psi_{\eta}(u)^{-1}\sum_{\ga\in (\Ga\cap\St_\eta)\backslash\Ga}f(\ga ug)du=\\
&\int_{N_{\A}/N_F}\sum_{\ga\in (\Ga\cap\St_\eta)\backslash\Ga}\psi_\eta^{-1}(u)\psi_\eta(\ga u\ga^{-1}) f(\ga g).
\end{align*}
Integrating first over $N_{\A}/N_F$ we get zero unless $\ga$ preserves $\psi_\eta$. Hence, we deduce that
$$\Pi_\eta\kappa_\eta f=f,$$
which implies that $\kappa_\eta$ is injective.

\noindent
{\bf Step 4.}
Finally, for $f\in \wt{\SS}_{\eta}$, we have
\begin{align*}
&\Pi_{\Om_\eta}\kappa_\eta(f)(g)=\sum_{\ga\in\Ga/\Ga\cap\St_{\eta}}\int_{N_{\A}/N_F}\psi_{\eta}(\ga^{-1}u\ga)^{-1}
\sum_{\ga_1\in\Ga/\Ga\cap\St_\eta}f(\ga_1^{-1}ug)du=\\
&\sum_{\ga,\ga_1\in\Ga/\Ga\cap\St_{\eta}}\int_{N_{\A}/N_F}\psi_{\eta}(\ga^{-1}u\ga)^{-1}\psi_{\eta}(\ga_1^{-1}u\ga_1)f(\ga_1^{-1}g)du=
\sum_{\ga\in\Ga/\Ga\cap\St_{\eta}}f(\ga^{-1}g)=\kappa_\eta(f)(g),
\end{align*}
so we conclude that the image of $\kappa_\eta$ is exactly $\SS(\Ga\backslash G(\A))_{\Om_\eta}$.
\end{proof}

%Since the induced representations $\wt{\SS}_{\eta,\chi}$ are irreducible and pairwise non-isomorphic, it remains to show
%that $\kappa_\eta$ is nonzero on each of them.
%Let $K$ be a sufficiently small open compact subgroup in $G(\A$), so that the restriction of $\chi$ to $\St_\eta\cap K$ is trivial.
%There there is a unique function
%$f\in \wt{\SS}_{\eta,\chi}$ given by 
%$$f(g)=\begin{cases} \chi(h), & g=hk\in \St_\eta\cdot K,\\ 0, &g\not\in \St_\eta\cdot K.\end{cases}$$ 
%We want to choose $K$ small enough so that $f(\ga)=0$ for $\ga\in \Ga\setminus \Ga\cap \St_\eta$.
%Since the support of $f$ is $\St_\eta\cdot K$, this is equivalent to the condition $G(\ov{F})\cdot\eta\cap K\cdot\eta=\{\eta\}$.
%Since $\fg^\vee\ot L^{-1}\om_{\ov{C}}(\ov{F})$ is discrete in $\fg^\vee\ot L^{-1}\om_{\ov{C}}(\ov{\A})$, there exists an open neighborhood $V$ of $\eta$ in $\fg^\vee\ot L^{-1}\om_{\ov{C}}(\ov{\A})$
%such that $V\cap \fg(\ov{F})=\{\eta\}$.
%Now it is enough to choose $K$ such that $K\cdot \eta\sub V$.

\subsection{Regular semisimple case}\label{reg-ss-sec}

From now on we assume that $\Ga=G(F)$.

We say that $\eta$ is regular semisimple if the stabilizer subgroup $T_\eta:=\bSt_\eta\sub G_{\ov{F}}$ is a maximal torus in $G_{\ov{F}}$.
Let us fix a regular semisimple orbit $\Om=G(\ov{F})\cdot\eta\sub \fg^\vee \ot L^{-1}\om_{\ov{C}}(\ov{F})$.
%We denote by $T_\eta\sub G$ the stabilizer of $\eta$, viewed as an algebraic subgroup (a torus) in $G$ (defined over $\ov{F}$), and 
%by $\St_\eta$ the stabilizer of $\eta$ in $G(\A)$.

Let $\lan\eta\ran^\perp\sub N_\A$ denote the subgroup of $u\in N_\A$ such that $\lan \eta,u\ran=0$, and let
$H_\eta$ denote the quotient of $\St_\eta$ by the normal subgroup $\lan\eta\ran^\perp$,
so that we have a commutative diagram
\begin{diagram}
1&\rTo{}& N_\A&\rTo{}& \St_\eta&\rTo{}& T_\eta(\ov{\A})&\rTo{}&1\\
&&\dTo{\lan\eta,\cdot\ran}&&\dTo{}&&\dTo{\id}\\
1&\rTo{}& \om_{\ov{C}}(\ov{\A})&\rTo{}& H_\eta&\rTo{}& T_\eta(\ov{\A})&\rTo{}&1
\end{diagram}
and $H_\eta$ is a central extension of $T_\eta(\ov{\A})$ by $\om_{\ov{C}}(\ov{\A})$.
We have a canonical splitting of $H_\eta/\om_{\ov{C}}(\ov{F})\to T_\eta(\ov{\A})$ over $T_\eta(\ov{F})$:
$$\si:T_\eta(\ov{F})\simeq (\Ga\cap \St_\eta)/N_F\to H_\eta/\om_{\ov{C}}(\ov{F}).$$

\begin{lemma}
The group $H_\eta/\om_{\ov{C}}(\ov{F})$ is commutative. 
\end{lemma}

\begin{proof}
%It is enough to prove the corresponding local extension $H_{\eta,p}$ of $T_\eta(\ov{F}_p)$ by $\om_{\ov{C},p}\ot \ov{F}_p$ is commutative.
This immediately follows from the corresponding local statement, Lemma \ref{H-commutative-local-lem}.
\end{proof}

Let us now consider the commutative group 
$$\ov{H}_\eta:=H_{\eta}/(\om_{\ov{C}}(\ov{F})\cdot T_\eta(\ov{F})),$$
an extension of $T_{\eta}(\ov{\A})/T_{\eta}(\ov{F})$ by $\om_{\ov{C}}(\ov{\A})/\om_{\ov{C}}(\ov{F})$.
%The commutativity of $H_\eta/\om_{\ov{C}}(\ov{F})$ 
%implies that 
There is a well defined action of $\ov{H}_\eta$
%$H_\eta/\om_{\ov{C}}(\ov{F})$ 
on $\wt{\SS}_\eta$ given by $hf(g)=f(\wt{h}^{-1}g)$, where $\wt{h}\in \St_\eta$ is
any lifting of $h\in \ov{H}_\eta$.
%H_\eta/\om_{\ov{C}}(\ov{F})$. 
Under this action the subgroup $\om_{\ov{C}}(\ov{\A})/\om_{\ov{C}}(\ov{F})$ acts through the character $\psi_{\ov{C}}^{-1}$.
%The group $\St_\eta$ acts on $\wt{\SS}_\eta$ by $hf(g)=f(h^{-1}g)$, so that the subgroup $L_\eta$ acts by $(\wt{\psi}_\eta)^{-1}$.
%The action of $\St_\eta$ on $\wt{\SS}_\eta$ factors through an action of $H_\eta$, so that $\om_{\ov{C}}(\ov{\A})$ acts through the character associated with $\psi$
%and $T_\eta(\ov{F})$ acts trivially.

\begin{definition}\label{S-eta-chi-def}
For a smooth character $\chi$ of $\ov{H}_\eta$
%$H_\eta/(\om_{\ov{C}}(\ov{F})T_{\eta}(\ov{F}))$, 
restricting to $\psi_{\ov{C}}$
on $\om_{\ov{C}}(\ov{\A})/\om_{\ov{C}}(\ov{F})$, let $\wt{\SS}_{\eta,\chi}$ denote
the compactly induced representation from $\chi$ viewed as a character of $\St_\eta$,
i.e., $\wt{\SS}_{\eta,\chi}$ is the space of smooth functions $f$ on $G(\A)$ such that $f(hg)=\chi(h)f(g)$ for $h\in \St_\eta$, with support compact modulo
left shifts by $\St_\eta$.
\end{definition}

\begin{lemma}\label{reg-sem-adm-lem} 
For any open compact subgroup $K\sub G(\OO)$, the space of invariants
$\wt{\SS}_{\eta,\chi}^K$ is finite-dimensional.
\end{lemma}

\begin{proof} This follows immediately from the local Proposition \ref{K-inv-prop}.
\end{proof}

\subsection{Regular elliptic case}\label{reg-ell-sec}

Now assume in addition that $\eta$ is elliptic, i.e., the torus $T_\eta$ is anisotropic. Then the group $T_{\eta}(\ov{\A})/T_{\eta}(\ov{F})$ is compact.
Hence, the commutative group $\ov{H}_\eta$
%H_{\eta}/(\om_{\ov{F}}(\ov{F})\cdot T_\eta(\ov{F}))$$
is also compact.
%so we have a decomposition

Consider the set $\Pi_\eta$ of characters of $\ov{H}_\eta$, restricting to the
character $\psi_{\ov{C}}$ of $\om_{\ov{C}}(\ov{\A})/\om_{\ov{C}}(\ov{F})$. 
This is a torsor over the set of characters of  $T_{\eta}(\ov{\A})/T_{\eta}(\ov{F})$. 

The action of $H_\eta$ on $\wt{\SS}_{\eta}$ factors through the action $\ov{H}_\eta$, so that 
$\om_{\ov{C}}(\ov{\A})/\om_{\ov{C}}(\ov{F})$ acts by $\psi_{\ov{C}}^{-1}$. 
Hence, $\wt{\SS}_\eta$ splits into a direct sum of $\chi$-isotypic components, where $\chi$ runs through $\Pi$.
Finally, we observe that the $\chi^{-1}$-isotypic component in $\wt{\SS}_\eta$ coincides with the induced representation
$\wt{\SS}_{\eta,\chi}$ (see Definition \ref{S-eta-chi-def}), so we get a decomposition
%$T_\eta(\ov{A})/T_\eta(\ov{F})$:
\begin{equation}\label{S-eta-decomposition}
\SS(G(F)\backslash G(\A))_{\Om_\eta}\simeq \wt{\SS}_{\eta}=\bigoplus_{\chi\in \Pi_\eta} \wt{\SS}_{\eta,\chi}.
\end{equation}
%where $\wt{\SS}_{\eta,\chi}$ is ???

Combining the decomposition \eqref{S-eta-decomposition} with Lemma \ref{kappa-lem}, we arrive at the following result.

\begin{prop}\label{reg-ell-prop}
Let $\eta$ be regular elliptic. Then one has a decomposition of $G(\A)$-representations
$$\SS(\Ga\backslash G(\A))_{\Om_\eta}\simeq \bigoplus_{\chi\in \Pi_\eta} \wt{\SS}_{\eta,\chi}.$$
\end{prop}

\subsection{Admissible representations and finitary functions}\label{adm-rep-sec}

A smooth representation $V$ of $G(\A)$ is called admissible if $V^K$ is finite dimensional for every compact open subgroup $K\sub G(\A)$.
Since the functors $V\mapsto V^K$ are exact, any subquotient of an admissible representation is admissible.

\begin{definition}
We say that a function $f\in \SS(G(F)\backslash G(\A))$ is {\it $G(\A)$-finitary} (or simply, {\it finitary}) if it is contained in an admissible $G(\A)$-subrepresentation of $\SS(G(F)\backslash G(\A))$. 
We denote by $\SS_f(G(F)\backslash G(\A))\sub \SS(G(F)\backslash G(\A))$ the subspace of finitary functions.
\end{definition}

By Lemma \ref{reg-sem-adm-lem} and Proposition \ref{reg-ell-prop}, for $\eta$ regular elliptic, 
the $G(\A)$-representation $\SS(G(F)\backslash G(\A))_{\Om_\eta}$ is a direct sum of admissible representations. 
Hence, the regular elliptic part of $\SS(G(F)\backslash G(\A))$ consists of finitary functions.

Now let us consider the case of other semisimple orbits.

\begin{lemma}\label{split-non-adm-lem}
Let $\Om=\Om_\eta$ be a regular semisimple orbit such that the corresponding maximal torus admits a nontrivial character $T_\eta\to \G_m$ (defined over $\ov{F}$). 
Then $\SS(G(F)\backslash G(\A))_{\Om}$ does not contain admissible subrepresentations.
\end{lemma}

\begin{proof}
We use notation of Sec.\ \ref{reg-ss-sec}.
By Lemma \ref{kappa-lem}, $\SS(G(F)\backslash G(\A))_{\Om}\simeq \wt{\SS}_{\eta}$, where $\eta$ is regular semisimple. We have an isomorphism
$$\wt{\SS}_{\eta}\simeq \ind_{\St_\eta}^{G(\A)} \SS_{\St_\eta},$$ 
where $\SS_{\St_\eta}$ is the subspace of $f\in \C_{lc}(N_F\backslash \St_\eta)$ such that $f(lh)=\wt{\psi}_{\eta}(l)f(h)$ for $l\in L_\eta$, $h\in \St_\eta$,
with compact support modulo left translations by $L_\eta$. Note that since $T_\eta=\bSt_\eta$ is commutative, $L_\eta$ is a normal subgroup in $\St_\eta$. 

Note that as in Prop.\ \ref{ind}, the induction functor gives an equivalence of the category of smooth representations $\rho$ of $\St_\eta$ with 
$\rho|_{N_\A}=\psi_{\eta}$, with a subcategory of the category of representations of $G(\A)$, closed under subobjects.
Hence, any subrepresentation of $\wt{\SS}_{\eta}$ is induced by a subrepresentation of $\SS_{\St_\eta}$.
Thus, it is enough to prove that $\SS_{\St_\eta}$ does not have any admissible subrepresentations.
Note that $\St_\eta$ acts on $\SS_{\St_\eta}$ through its commutative quotient $H_\eta$. Thus, for any open compact subgroup $K$,
the subspace of $K$-invariants is an $H_\eta$-subrepresentation. Thus, it is enough to check that $\SS_{\St_\eta}$ does not contain finite-dimensional
subrepresentations. Using our homomorphism $T_\eta\to \G_m$ we get a nonzero homomorphism
$$\deg:\St_\eta/L_{\eta}\to T_\eta(\ov{\A})/T_\eta(\ov{F})\to \Pic(\ov{C})\rTo{\deg}\Z.$$
Now for any nonzero $f\in \SS_{\St_\eta}$, let $S\sub \St_\eta/L_\eta$ be the compact support of $f$. Then $\deg(S)$ is finite.
Hence, for $h\in \St_\eta$ such that $\deg(h)\gg 0$, the functions $(h^nf)_{n\in \Z}$ have disjoint supports, hence, span an infinite-dimensional
subspace.
\end{proof}

\subsection{Connection with the Hitchin fibers}\label{Hitchin-sec}

Now let us 
%assume that $\Ga=G(F)$ and
consider the invariants in $\wt{\SS}_\eta$ with respect to $G(\OO)\sub G(\A)$. As in Lemma \ref{kappa-lem}, one checks that
every $f\in\wt{\SS}_\eta^{G(\OO)}$ is supported on the subset $G(\A)_\eta\sub G(\A)$ consisting of $g\in G(\A)$ such that 
$$\Ad(g^{-1})(\eta)\in (\fg^\vee\ot L^{-1}\om_{\ov{C}})(\ov{\OO})\sub (\fg^\vee\ot L^{-1}\om_{\ov{C}})(\ov{\A}).$$
Note that there is a similarly defined subset $G(\ov{\A})_\eta\sub G(\ov{\A})$.

Let $\MM^{Higgs,L^{-1}}(\ov{C})$ denote the stack of $L^{-1}$-twisted Higgs $G$-bundles $(P,\phi)$ on $\ov{C}$. By definition, here $P$ is a $G$-bundle,
and $\phi$ is a section of $H^0(\ov{C},\fg^\vee_P\ot L^{-1}\om_{\ov{C}})$, where $\fg_P$ is the vector bundle associated with $P$ and the adjoint representation of $G$.

We have a natural $\G_a$-torsor $\TT^C=\TT^{G,C}$ over the groupoid $\MM^{Higgs,L^{-1}}(\ov{C})$ (depending on the square zero extension $C$ of $\ov{C}$),
defined as follows. For simplicity, let us describe the corresponding $k$-torsor $\TT^C(k)$ on the groupoid of $k$-points.
For every $L^{-1}$-twisted Higgs $G$-bundle $(P,\phi)$ on $\ov{C}$, we have
a natural $H^1(\ov{C},\fg_P\ot L)$-torsor of all liftings of $P$ to a $G$-bundle $\wt{P}$ over $C$ (with a fixed isomorphism of the reduction of $\wt{P}$ with $P$). 
By Serre duality, we have a natural pairing between
$H^1(\ov{C},\fg_P\ot L)$ and $H^0(\ov{C},\fg_P^\vee\ot L^{-1}\om_{\ov{C}})$. 
Now we define $\TT^C(k)|_{P,\phi}$ as the push-forward of our $H^1(\ov{C},\fg_P\ot L)$-torsor with respect
to the functional  $H^1(\ov{C},\fg_P\ot L)\to k$ given by the pairing with $\phi$.

Let us denote by $\LL_\psi$ the associated $\C^*$-torsor on the groupoid of $k$-points, where we use the additive character $\psi:k\to \C^*$.

For a $\C^*$-torsor $\LL$ over a groupoid $X$ we denote by $\SS(X,\LL)$ the space of finitely supported sections of $\LL\times_{\C^*} \C$.
Note that that by definition such sections should respect isomorphisms $\LL|_x\to \LL|_{x'}$ associated with isomorphisms $x\to x'$ in $X$.

\begin{prop}\label{Higgs-prop}
\begin{enumerate}
\item
There is a natural identification of groupoids
$$\bSt_\eta(\ov{F})\backslash G(\ov{\A})_\eta/G(\ov{\OO})\simeq \MM^{Higgs,L^{-1}}_\eta(\ov{C})(k),$$
where $\MM^{Higgs,L^{-1}}_\eta(\ov{C})$ is the subgroupoid of $L^{-1}$-twisted Higgs $G$-bundles $(V,\phi)$ such that the orbit of $\phi$ at the general point coincides with $\Om_\eta$.

\item
We have a natural identification
\begin{equation}\label{Higgs-identification-eq}
\wt{\SS}_\eta^{G(\OO)}\simeq \SS(\MM^{Higgs,L^{-1}}_\eta(\ov{C})(k),\LL_\psi).
\end{equation}
In particular, 
$$\dim \wt{\SS}_\eta^{G(\OO)}\le |\MM^{Higgs,L^{-1}}_\eta(\ov{C})(k)|,$$
with the equality provided for every $(P,\phi)\in \MM^{Higgs,L^{-1}}_\eta(\ov{C})(k)$, the action of $\Aut(P,\phi)$ on
the fiber of $\LL_\psi$ at $(P,\phi)$ is trivial.

In the case $C$ is split, i.e., $C=\ov{C}\times_{\Spec(k)}\Spec(k[\eps]/(\eps^2))$, the $k$-torsor $\TT^C(k)$ and the $\C^*$-torsor $\LL_\psi$ on 
$\MM^{Higgs,L^{-1}}_\eta(\ov{C})(k)$ are trivial.

\item 
Assume in addition that $\eta$ is regular semisimple. Then there is an action of $\ov{H}_\eta$ on $\LL_{\psi}$ compatible with the action of $T_\eta(\ov{\A})$ on
$\MM^{Higgs,L^{-1}}_\eta(\ov{C})(k)$ such that the subgroup $\om_{\ov{C}}(\ov{\A})/\om_{\ov{C}}(\ov{F})$ acts via $\psi_{\ov{C}}^{-1}$, and such that the isomorphism
\eqref{Higgs-identification-eq} is compatible with the $\ov{H}_\eta$-action. In this case we have
$$\dim \wt{\SS}_\eta^{G(\OO)}= |\MM^{Higgs,L^{-1}}_\eta(\ov{C})(k)|.$$
\end{enumerate}
\end{prop}

\begin{proof}
\begin{enumerate}
\item
Let $P(g)$ be the $G$-bundle on $\ov{C}$ associated with $g=(g_p)\in G(\ov{\A})$ (recall that we assume $G$ to be split connected reductive, so every $G$-bundle on $\ov{C}$ has
this form). By definition, $P(g)$ is equipped with trivializations $e_\eta$, $(e_p)$, such that
$e_p=e_\eta g_p$. The associated vector bundle $\fg^\vee_{P(g)}$ has 
$$\fg^\vee_{P(g),\eta}=\fg^\vee(\ov{F}), \ \ \fg^\vee_{P(g),p}=\Ad(g_p)\fg^\vee(\ov{\OO}_p)\sub \fg(\ov{F}_p).$$
Hence, assuming that $g\in G(\ov{\A})_\eta$, we can view $\eta\in \fg^{\vee}\ot L^{-1}\om_{\ov{C}}(\ov{F})$ as an $L^{-1}$-twisted Higgs field for $P(g)$.
If $g'=hgk$, with $h\in \bSt_\eta(\ov{F})$ and $k\in G(\ov{\OO})$, then we have the isomorphism between the Higgs bundles $(P(g),\eta)$ and $(P(g'),\eta)$ sending
$e_\eta$ to $e_\eta h$.

\item
The space $\wt{\SS}_{\eta}$ is the space of functions on $(\St_\eta\cap G(F))\backslash G(\A)_\eta/G(\OO)$ transforming according to the character $\psi_\eta$
with respect to the left shifts by $N_\A\simeq \fg\ot L(\ov{\A})$. It remains to observe that these shifts act transitively on the groupoid fibers of the projection
$$(\St_\eta\cap G(F))\backslash G(\A)_\eta/G(\OO)\to \bSt_\eta(\ov{F})\backslash G(\ov{\A})_\eta/G(\ov{\OO})\simeq  \MM^{Higgs,L^{-1}}_\eta(\ov{C})(k),$$
and that these groupoid fibers are identified with $H^1(\ov{C},\fg_P\ot L)$.

The last assertion follows from the fact that in the split case the $H^1(\ov{C},\fg_P\ot L)$-torsor of liftings of a $G$-bundle $P$ on $\ov{C}$ to a $G$-bundle on $C$ is canonically trivial
(since we can pull back $P$ via the projection $C\to \ov{C}$).

\item
The action of $\ov{H}_\eta$ on $\LL_\psi$ is induced by the action of $\St_\eta$ on $(\St_\eta\cap G(F))\backslash G(\A)_\eta/G(\OO)$ by left shifts, which is compatible
with the action of $T_\eta(\ov{\A})=\bSt_\eta(\ov{\A})$ on $T_\eta(\ov{F})\backslash G(\A)_\eta/G(\OO)$. 

For the last assertion, we have to check that for every $(P,\phi)\in \MM^{Higgs,L^{-1}}_\eta(\ov{C})(k)$, the action of $\Aut(P,\phi)$ on
the fiber of $\LL_\psi$ at $(P,\phi)$ is trivial. Note that the group $\Aut(P)$ is finite (since $\ov{C}$ is a curve over a finite field $k$).
Restriction to the generic point gives an injection
$$\Aut(P,\phi)\sub T_\eta(\ov{F}).$$
Let $p$ denote the characteristic of $k$. Since $\ov{F}$ is a field of characteristic $p$, every element of finite order in $T_\eta(\ov{F})$ has order prime to $p$.
On the other hand, the action of $\Aut(P,\phi)$ on the fiber of $\LL_\psi$ factors through a homomorphism $\Aut(P,\phi)\to k$, which must be trivial
since all elements of $k$ have order $p$.
\end{enumerate}
\end{proof}

Proposition \ref{Higgs-prop}(2) shows that it is important to know when 
the action of the automorphism group $\Aut(P,\phi)$ of a Higgs bundle $(P,\phi)$ on the fiber of $\LL_\psi$ is trivial.
Note that $\Aut(P,\phi)$ acts on the fiber $\TT^C(k)|_{P,\phi}$ via a homomorphism
\begin{equation}\label{c-P-phi-eq}
c_{P,\phi}:\Aut(P,\phi)\to k,
\end{equation}
and on the fiber of $\LL_\psi$ by the character $\psi(c_{P,\phi})$. Thus, our question is when this character is trivial. 
We have already seen in the proof of Proposition \ref{Higgs-prop}(3) that this is always the case when $(P,\phi)\in \MM^{Higgs,L^{-1}}_\eta(\ov{C})(k)$
with $\eta$ regular semisimple.

\begin{prop} The homomoprhism $c_{P,\phi}$ is trivial in the following cases:
\begin{enumerate}
\item $G=\GL_n$;
\item $G=\PGL_2$ if characteristic of $k$ is $\neq 2$;
\item if there exists a representation $\rho:G\to \GL_n$ defined over $\Z$ such that the form $(x,y)_\rho=\tr(d\rho(x)d\rho(y))$ is nondegenerate over $k$.
\end{enumerate}
\end{prop}

\begin{proof} (1) Let $(V,\phi)$ be a Higgs bundle of rank $n$. The group $\Aut(V,\phi)$ is finite, so it suffices to study the restriction 
$c_{V,\phi}$ to automorphisms $a$ such that $a^{p^l}=1$, where $p$ is the characteristic of $k$, or equivalently $(a-1)^{p^l}=0$.
Consider the filtration 
$$F_i:=\ker(a-1)^i$$ 
by the subbundles in $V$. Note that the twisted Higgs field $\phi$ preserves this filtration (since it commutes with $a$), while $a$ preserves this filtration
and acts as $1$ on associated graded quotients. Let $\fp\sub \fg$ denote the corresponding parabolic subalgebra with the unipotent radical $\fu\sub \fp$.
Then $c_{V,\phi}(a)$ is obtained by the pairing of an element in $H^1(\ov{C},\fu_V\ot L)\sub H^1(\ov{C},\fg_V\ot L)$ (determined by $a$) with 
$$\phi\in H^0(\ov{C},\fp_V\ot L^{-1}\om_{\ov{C}})\sub H^0(\ov{C},\fg_V\ot L^{-1}\om_{\ov{C}}).$$
Since $\fu$ and $\fp$ are orthogonal with respect to the trace form, this pairing is zero.

\noindent
(2) Let $(P,\phi)$ be an $L^{-1}$-twisted Higgs $\PGL_2$-bundle. We already know that $c_{V,\phi}=0$ if either $\phi=0$ or $\phi$ is generically semisimple,
so assume $\phi\neq 0$ and $\phi$ is generically nilpotent. Then we have a natural reduction $(P_B,\phi)$ of $(P,\phi)$ to the Borel subgroup $B\sub \PGL_2$, given 
by the line subbundle $\ker(\phi)$. Thus, similarly to part (1), $c_{V,\phi}(a)$ is obtained by the pairing of an element in 
$H^1(\ov{C},\fb_{P_B}\ot L)\sub H^1(\ov{C},\fg_P\ot L)$ with 
$\phi\in H^0(\ov{C},\fu_{P_B}\ot L^{-1}\om_{\ov{C}})$, where $\fu\sub \fb$ is the unipotent radical in the Borel subalgebra. Now the assertion follows by
the orthogonality between $\fu$ and $\fb$.
%\wub H^0(\ov{C},\fg_V\ot L^{-1}\om_{\ov{C}}).$$ 

\noindent
(3) We can use the form $(\cdot,\cdot)_\fg$ to identify $\fg$ with $\fg^\vee$. Then the homomorphisms $\rho$ and $d\rho$ give a morphism of groupoids
$$\varphi:\MM^{Higgs,G,L^{-1}}(\ov{C})\to \MM^{Higgs,\GL_n,L^{-1}}(\ov{C}),$$
such that we have isomorphicm of torsors $\varphi^*\TT^{\GL_n,C}\simeq \TT^{G,C}$. Thus, the required statement for $G$ follows from the one for $\GL_n$.
\end{proof}

In Appendix \ref{Aut-ex-sec} we will give an example showing that the homomorphism $c_{P,\phi}$ can be nontrivial for $G=\PGL_2$ in the case characteristic of $k$ is $2$
(for any genus $g\ge 1$).

\subsection{Finitary functions for $\PGL_2$}\label{finitary-sec}

Now we specialize to the case $G=\PGL_2$, so everywhere in this subsection we assume that $G=\PGL_2$.
We also assume that the characteristic of $k$ is $\neq 2$, and use the identification $\fg^\vee\simeq \fg$ given by the trace form.

Our goal is to determine explicitly the subspace of finitary functions in $\SS(G(F)\backslash G(\A))$
using the orbit decomposition obtained in Section \ref{aut-rep-sec} (see Theorem \ref{fin-dec-thm}).

%\subsection{Case of the nonzero nilpotent orbit}\label{nilp-orbit-sec}

In addition to our previous results concerning semisimple orbits, we need to study the space $\SS(G(F)\backslash G(\A))_{\Om_n}$, where $\Om_n$ is the orbit of the matrix
\begin{equation}\label{eta0-eq}
\eta_0:=\left(\begin{matrix} 0 & \a_0 \\ 0 & 0\end{matrix}\right)\in \fg^\vee\ot L^{-1}\om_{\ov{C}}(\ov{F}), 
\end{equation}
where we take any nonzero $\a_0\in L^{-1}\om_{\ov{C}}(\ov{F})$.
This case turns out to be similar to that of a regular semisimple orbit. 

The stabilizer of $\eta_0$ in $\PGL_2$ is exactly the group $U$ of upper triangular unipotent matrices. Thus, $L_{\eta_0}=N_{\ov{\A}}U(F)$.
By Lemma \ref{kappa-lem}, we have $\SS(G(F)\backslash G(\A))_{\Om_n}\simeq \wt{\SS}_{\eta_0}$.

As in Sec. \ref{reg-ss-sec}, we define the central extension $H_{\eta_0}$ of $U(\ov{\A})$ by $\om_{\ov{C}}(\ov{\A})$ by setting $H_{\eta_0}=\St_{\eta_0}/\lan \eta_0\ran^\perp$
where $\lan\eta_0\ran^\perp\sub N_\A$ is the kernel of $\lan \eta_0,\cdot\ran$. Note that $\lan\eta_0\ran^\perp=\fb(\ov{\A})\ot L$. 
Thus, we have a splitting 
$$\si:U(\ov{\A})\simeq U(\A)/(\fu(\ov{\A})\ot L)\to H_{\eta_0},$$
so we have a canonical decomposition
$$H_{\eta_0}\simeq \om_{\ov{C}}(\ov{\A})\times U(\ov{\A}).$$

The action of $\St_{\eta_0}$ on $\wt{\SS}_{\eta_0}$ factors through an action of $H_{\eta_0}$, such that the subgroup $\om_{\ov{C}}(\ov{F})\times U(\ov{F})$ acts trivially.
Since the group $U(\ov{\A})/U(\ov{F})\simeq \ov{\A}/\ov{F}$ is compact, we can decompose $\wt{\SS}_{\eta_0}$ into isotypical components corresponding to characters
of $\ov{\A}/\ov{F}$. These characters are identified with $\om_{\ov{C}}(\ov{F})$: with every $\a\in \om_{\ov{C}}(\ov{F})$, one associates the character 
$\chi_\a(x)=\psi_{\ov{C}}(x\a)$. Thus, we get
\begin{equation}\label{nilp-decomposition-eq}
\wt{\SS}_{\eta_0}\simeq\bigoplus_{\a\in \om_{\ov{C}}(\ov{F})} \wt{\SS}_{\eta_0,\chi_\a},
\end{equation}
where 
$$\wt{\SS}_{\eta_0,\chi_\a}=\ind_{\St_{\eta_0}}^{G(\A)}(\psi_{\ov{C}},\chi_\a).$$
Here we use the homomorphism $\St_{\eta_0}\to H_{\eta_0}=\om_{\ov{C}}(\ov{\A})\times U(\ov{\A})$ to view $(\psi_{\ov{C}},\chi_\a)$
as a character of $\St_{\eta_0}$. Note that the $G(\A)$-representations in the right-hand side of \eqref{nilp-decomposition-eq} are irreducible and pairwise non-isomorphic.

\begin{lemma}\label{nilp-adm-lem}
(i) For any nonzero $\a\in \om_{\ov{C}}(\ov{F})$, the representation $\wt{\SS}_{\eta_0,\chi_\a}$ is admissible.
%the space of invariants $\wt{\SS}_{\eta_0,\chi_a}^K$ is finite-dimensional.

\noindent
(ii) The representation $\wt{\SS}_{\eta_0,\chi_0}$ of $G(\A)$, where $\chi_0$ is the trivial character of $\ov{\A}/\ov{F}$, is not admissible.
\end{lemma}

\begin{proof} This follows from the corresponding local results, similar to \cite[Sec.\ 4.4]{K-YD},
%Prop. \ref{K-inv-prop}, 
established in Appendix \ref{loc-app}.
\end{proof}

%\subsection{Orbit/character decomposition and finitary functions}

\begin{theorem}\label{fin-dec-thm}
For $G=\PGL_2$ (assuming $\operatorname{char}(k)\neq 2$), one has a decomposition
\begin{equation}\label{main-fin-dec}
\SS_f(G(F)\backslash G(\A))=\SS_f(G(\ov{F})\backslash G(\ov{\A}))\oplus \bigoplus_{\a\neq 0} \wt{\SS}_{\eta_0,\chi_\a}\oplus 
\bigoplus_{\Om \ \text{reg. ell.}}\SS(G(F)\backslash G(\A))_{\Om}.
\end{equation}
\begin{equation}\label{Om-d-dec}
\SS(G(F)\backslash G(\A))_{\Om_\eta}=\bigoplus_{\chi\in\Pi_{\eta}}\wt{\SS}_{\eta,\chi} \ \ \text{  for } \eta \text{ regular elliptic}.
\end{equation}
In \eqref{main-fin-dec} the summation is over regular elliptic $G(\ov{F})$-orbits in $\fg^\vee\ot L^{-1}\om_{\ov{C}}(\ov{F})$. 
In \eqref{Om-d-dec}, $\eta$ is a representative of the orbit $\Om_\eta$ and 
$\Pi_{\eta}$ is the set of characters of the commutative compact group $\ov{H}_{\eta}$ extending $\psi_{\ov{C}}$ 
(see Sec.\ \ref{reg-ell-sec}).
\end{theorem}

\begin{proof} We already know the decomposition \eqref{Om-d-dec} (see \eqref{S-eta-decomposition}).
By Lemma \ref{reg-sem-adm-lem}, Proposition \ref{reg-ell-prop} and Lemma \ref{nilp-adm-lem}, 
the right-hand side of \eqref{main-fin-dec} is contained in $\SS_f(G(F)\backslash G(\A))$, so it remains to establish the opposite inclusion.

Let $\SS(G(\ov{F})\backslash G(\ov{\A}))_{\neq 0}=\bigoplus_{\Om\neq 0}\SS(G(\ov{F})\backslash G(\ov{\A}))_{\Om}$.
We have to prove that every admissible subrepresentation
$V\sub  \SS(G(\ov{F})\backslash G(\ov{\A}))_{\neq 0}$ is contained in the right-hand side of \eqref{main-fin-dec}.

By Lemma \ref{split-non-adm-lem}, we obtain $\Pi_\Om V=0$ for every split semisimple orbit $\Om$.
It remains to prove that $\Pi_{\Om_n} V$ is contained in $\bigoplus_{\a\neq 0} \wt{\SS}_{\eta_0,\chi_\a}$. But this follows immediately from 
Lemma \ref{nilp-adm-lem}(ii).
\end{proof}

\section{Constant term for $G=\GL_2$}\label{const-GL2-sec}

In this section we assume that $G=\GL_2$, and $C$ is a special nilpotent extension of length $2$.
%, and $\eps$ is a generator of the maximal ideal in $A$.
We will give more precise information on the constant term operator in this case. We start by partitioning the set $\QBun_T(C)$ into subsets
numbered by effective divisors in $\ov{C}$ (this is based on an analog of the Iwasawa decomposition developed in Sec.\ \ref{Iwasawa-sec}).
Then we establish a more explicit form of compatibility between the constant term operator and the Hecke operators 
(see Lemma \ref{hecke-qbun-formula-lem}). This analysis will be used in Section \ref{cusp-GL2-sec} to show that Hecke finiteness implies cuspidality.
%We also prove a technical result, Proposition \ref{const-term-red-prop}, proving a sufficient criterion for the cuspidality of a function.

\subsection{Analog of the Iwasawa decomposition}\label{Iwasawa-sec}

Let $\ov{K}$ be a local non-archimedean field, $\ov{\OO}\sub \ov{K}$ the ring of integers, $t\in \ov{\OO}$ a uniformizer.
Let $\OO$ be a commutative ring, which is a square zero extension of $\ov{\OO}$ by $\ov{\OO}$.
We fix an element $\eps\in \OO$ such that $\eps^2=0$, $\OO/\eps \OO\simeq \ov{\OO}$ and the map
$\ov{\OO}\to \eps \OO: x\mapsto \eps x$ is an isomorphism. 
%$\OO=\ov{\OO}[\eps]/(\eps^2)$.
Let $K$ be the total ring of fractions of $\OO$. Then $K/\eps K\simeq \ov{K}$, and $\eps K\simeq\ov{K}$.

Let $G=\GL_2$, $B\sub G$ the subgroup of upper-triangular matrices. Let $\fg$ and $\fb$ denote the Lie algebras of $G$ and $B$.

For every $n>0$, consider the matrix
$$\varphi_n:=\left(\begin{matrix} 0 & 0 \\ t^{-n} & 0\end{matrix}\right)\in \fg(\ov{K}),$$
and set $\varphi_0=0$.

\begin{lemma}\label{Lie-Iwasawa-lem}
We have a decomposition into $\Ad(B(\ov{\OO}))$-orbits,
$$\fg(\ov{K})/(\fb(\ov{K})+\fg(\ov{\OO}))=\sqcup_{n\ge 0} \Ad(B(\ov{\OO}))\cdot (\varphi_n+\fb(\ov{K})+\fg(\ov{\OO})).$$
\end{lemma}

\begin{proof} The class of a matrix $X=\left(\begin{matrix} a_{11} & a_{12} \\ a_{21} & a_{22} \end{matrix}\right)\in \fg(\ov{K})$ in
$\fg(\ov{K})/(\fb(\ov{K})+\fg(\ov{\OO}))$ is determined by the class of $a_{21}$ in $\ov{K}/\ov{\OO}$.
The adjoint action of $U(\ov{\OO})$ doesn't change this class, while the action of $T(\ov{\OO})$ corresponds to rescalings by $\ov{\OO}^*$.
This immediately gives the result.
\end{proof}

\begin{prop}\label{Iwasawa-prop}
We have the following decomposition into open subsets
$$G(K)=\sqcup_{n\ge 0} B(K)\cdot (1+\eps\cdot \varphi_n)\cdot G(\OO).$$
\end{prop}

\begin{proof}
The Iwasawa decomposition $G(\ov{K})=B(\ov{K})\cdot G(\ov{\OO})$ implies that every $B(K)-G(\OO)$ double coset in $G(K)$ has a representative of the form
$1+\eps X$, where $X\in \fg(\ov{K})$. Furthermore, assume that
$$1+\eps X'=b\cdot (1+\eps X)\cdot g_\OO,$$
with $b\in B(K)$, $g_\OO\in G(\OO)$. Then the reductions $\ov{b}\in B(\ov{K})$ and $\ov{g}_\OO\in G(\ov{\OO})$ satisfy $\ov{b}\ov{g}_\OO=1$. Hence, we can write
$$b=(1+\eps Y)b_0,  \ \ g_{\OO}=b_0^{-1}(1+\eps Z),$$
with $b_0\in B(\ov{\OO})$, $Y\in \fb(\ov{K})$, $Z\in \fg(\ov{\OO})$. In other words, we get
$$X'=Y+\Ad(b_0)X+Z.$$
Now the result follows from Lemma \ref{Lie-Iwasawa-lem}.
\end{proof}

Next, let us return to the picture with the special nilpotent extension $C$ of length $2$. Set 
$$N_0:=\NN/\NN^2$$ 
(this is a line bundle on $\ov{C}$).
%$\eps$ is a generator of the maximal ideal of $A$.
Then for every point $p$ we can consider the corresponding square-zero extension $F_p$ of the local field $\ov{F}_p$ by $N_{0,p}\ot_{\ov{\OO}_p} \ov{F}_p$. 
%(the completion of the stalk of $N_0$ at $p$).
%have the corresponding ring $\OO_p$, which is a square-zero extension of $\ov{\OO}_p$ by
%We have the following analog of Proposition \ref{Iwasawa-prop} for adeles.

We pick a generator $\eps_p\in N_{0,p}$ at every point and consider the corresponding generator of the nilradical,
$$\eps:=(\eps_p)\in \NN\A\sub \A.$$
Note that if the line bundle $N_0$ is a nontrivial then $\eps$ is not a principal adele (see \eqref{eps-F-eq} below). 
Let us also pick a uniformizer $t_p\in \ov{\OO}_p$ at every point, and for every effective divisor $D=\sum_p n_p p$ set
\begin{equation}\label{fD-eq}
f_D:=(t_p^{n_p})\in \ov{\A}^*,
\end{equation}
$$\varphi_D:=\left(\begin{matrix} 0 & 0 \\ f_D^{-1} & 0\end{matrix}\right)\in \fg(\ov{\A}), \ \ g_D:=1+\eps \varphi_D\in G(\A).$$

Let $\Div(\ov{C})_{\ge 0}$ denote the set of (finite) linear combinations of points of $\ov{C}$ with integer nonnegative coefficients.

\begin{cor}\label{Iwasawa-cor}
We have the following decomposition into (open) double cosets:
$$G(\A)=\sqcup_{D\in\Div(\ov{C})_{\ge 0}} B(\A)\cdot g_D\cdot G(\OO).$$
\end{cor}

Recall that the target of the constant term map $C=C_B$ associated with the Borel subgroup $B\sub G=\GL_2$ is the space of functions
on 
$$\QBun_T(C):=(T(F)U(\A))\backslash G(\A)/G(\OO),$$
where $U\sub B$ is the subgroup of strictly upper-triangular matrices.

%From the above Corollary we get the induced decomposition of $\QBun_T$.

\begin{definition} We set $B(\OO)[D]:=B(\A)\cap g_DG(\OO)g_D^{-1}$
and denote by $T(\OO)[D]\sub T(\A)$ the image of $B(\OO)[D]$ under the natural projection
$B(\A)\to T(\A)$.
\end{definition}

\begin{cor}\label{QBunT-decomposition-cor}
We have a decomposition
$$\QBun_T(C)=\sqcup_{D\in \Div(\ov{C})_{\ge 0}} \QBun_T(C,D),$$
where 
$$\QBun_T(C,D):=T(F)\backslash T(\A)/T(\OO)[D]\cdot g_D.$$
\end{cor}

\begin{proof} We need to check that the map $t\mapsto t\cdot g_D$ induces an identification
$$T(\A)/(T(F)\cdot T(\OO)[D])\rTo{\sim} (T(F)U(\A))\backslash(B(\A)\cdot g_D\cdot G(\OO))/G(\OO).$$
First, to see that the map is well defined, suppose we have $t_0\in T(\OO)[D]$. We need to check that $tt_0g_D$ is in the same double coset as $tg_D$.
By definition, we can write $t_0=ub_0$, with $u\in U(\A)$ and $b_0\in B(\A)\cap g_DG(\OO)g_D^{-1}$. Hence,
$$tt_0g_D=tub_0g_D=(tut^{-1})(tg_D)(g_D^{-1}b_0g_D)\in U(\A) tg_D G(\OO).$$

Conversely, suppose
$$t'g_D=ut_Ftg_Dg_\OO,$$
where $t,t'\in T(\AA)$, $u\in U(\A)$, $g_\OO\in G(\OO)$. Then 
$$b_0:=g_Dg_\OO g_D^{-1}\in B(\A)\cap g_D G(\OO)g_D^{-1}.$$
Projecting the relation $t'=ut_Ftb_0$ to $T(\A)$, we deduce that $t'\in t T(F)T(\OO)[D]$.
\end{proof}

\begin{remark}
The double cosets presentations of $\QBun_T(C)$ and $\QBun_T(C,D)$ mean that these sets can be identified with isomorphism classes
of certain groupoids (see Appendix \ref{group-app}). 
%For now we do not need to consider the groupoid structure. 
In Appendix \ref{geom-const-term-sec} we will give a geometric interpretation of the double quotient groupoids corresponding to $\QBun_T(C,D)$.
Note that the decomposition of Corollary \ref{QBunT-decomposition-cor} does not lift to the level of groupoids: the double quotient groupoids corresponding
to $\QBun_T(C,D)$ have finite automorphism groups, whereas the automorphism groups of the double quotient groupoid corresponding to $\QBun_T(C)$ contain infinite compact groups, such as $U(\OO)$.
\end{remark}

Later we will need the following explicit description of $T(\OO)[D]$.

\begin{lemma}\label{TOD-lem}
One has $t=\diag(a_1,a_2)\in T(\OO)[D]$ if and only if
\begin{equation}\label{TOD-eq}
a_1,a_2\in \OO[D]:=\OO+\eps\cdot \ov{\OO}(D), \ \ a_1a_2\in \OO^*, \ \ (a_1-a_2)\mod\NN\in \ov{\OO}(-D).
\end{equation}
\end{lemma}

\begin{proof} 
%Set $f_D=(t_p^{n_p})$, where $n_p$ is the multiplicity of $D$ at $p$.
One has $\left(\begin{matrix} a_1 & b \\ 0 & a_2\end{matrix}\right)\in B(\A)\cap g_DG(\OO)g_D^{-1}$ 
if and only if 
\begin{equation}\label{BOD-eq}
b\in \OO, \ \ a_1+\eps f_D^{-1}b\in \OO, \ \ a_2-\eps f_D^{-1}b\in \OO, \ \ \eps f_D^{-1}(a_1-a_2)\in \OO, \ \ a_1a_2\in \OO^*,
\end{equation}
where $f_D\in \ov{\A}^*$ is given by \eqref{fD-eq}
(the last condition is obtained by looking at the invertibility of the determinant). Thus, $t=\diag(a_1,a_2)\in T(\OO)[D]$ if and only if there exists $b$ such that
the conditions \eqref{BOD-eq} are satisfied. This implies that for such $(a_1,a_2)$ conditions \eqref{TOD-eq} are satisfied. Conversely, assume conditions
\eqref{TOD-eq} are satisfied. Then we can write 
$$a_i=c_i+\eps d_i, \ i=1,2, \text{ with } c_i\in \OO^* \text{ and } d_i\in f_D^{-1}\ov{\OO}.$$
By assumption, we have $c_1-c_2\mod\NN \in f_D\ov{\OO}$ and $c_1d_2+c_2d_1\mod \NN\in \ov{\OO}$.
Hence, $d_1(c_1-c_2)\mod\NN\in \ov{\OO}$, and so
$$c_1(d_1+d_2)\mod\NN=d_1(c_1-c_2)+(c_1d_2+c_2d_1)\mod\NN\in \ov{\OO}.$$
%$$a_1=c-\eps f_D^{-1}b, \ \ a_1-a_2=f_Dx+\eps f_D^{-1}y,$$
%with $b,c,x,y\in \OO$. This implies that $(a_1-a_2)^2\in \OO$, and hence,
%$$(a_1+a_2)^2=(a_1-a_2)^2+4a_1a_2\in \OO,$$
It follows that $a_1+a_2\in \OO$. Therefore, we have $a_2=c'+\eps f_D^{-1}b$ for some $c'\in \OO$, which shows the existence of $b$ such that \eqref{BOD-eq} is satisfied.
\end{proof}

Thus, the constant term operator $C:\C(G(F)\backslash G(\A)/G(\OO))\to \C(\QBun_T(C))$ 
can be viewed as a collection of operators
$$C_D:\C(G(F)\backslash G(\A)/G(\OO))\to \C(T(\A)/(T(F)\cdot T(\OO)[D])), \text{ where}$$
$$C_Df(t)=\int_{U(F)\backslash U(\A)} f(utg_D) du.$$

Setting $U_{t,D}:=U(\A)\cap (tg_D)G(\OO)(tg_D)^{-1}$, and noticing that the fibers of the surjective map
$$U(\A)/U(F)\to U(\A)/(U(F)\cdot U_{t,D})$$
are $U_{t,D}/(U(F)\cap U_{t,D})$-cosets, we can rewrite the integral defining $C_D$
as the following finite sum:
\begin{equation}\label{ED-sum-def-eq}
C_Df(t)=\vol(U_{t,D}/(U(F)\cap U_{t,D}))\cdot \sum_{u\in U(\A)/(U(F)\cdot U_{t,D})} f(utg_D).
\end{equation}

%Similarly, we have the refined constant term operator 
%$$E^1=E^1_B:\C(G(F)\backslash G(\A)/G(\OO))\to \C((T(F)U(\eps\A))\backslash G(\A)/G(\OO)),$$
%given by integration along right $U(\eps\A)/U(\eps F)$-cosets. Using the analog of the Iwasawa decomposition
%as in Corollary \ref{QBunT-decomposition-cor}, we get
%$$(T(F)U(\eps\A))\backslash G(\A)/G(\OO)=\sqcup_D (T(F)U(\eps\A))\backslash B(\A)/B(\OO)[D]\cdot g_D.$$
%Thus, setting $U^1_{t,D}:=U(\eps\A)U_{t,D}$ we can rewrite $E^1$ as a collection of operators
%$$E^1_D:\C(G(F)\backslash G(\A)/G(\OO))\to \C((T(F)U(\eps\A))\backslash B(\A)/B(\OO)[D]),$$
%\begin{equation}\label{E1D-sum-def-eq}
%E^1_Df(ut)=\vol(U^1_{t,D}/(U(\eps F)\cap U^1_{t,D}))\cdot \sum_{u_1\in U(\eps\A)/(U(\eps F)\cdot U^1_{t,D})} f(u_1utg_D).
%\end{equation}

\begin{lemma}\label{UtD-lem}
For $t=\diag(a_1,a_2)$, one has
$$U_{t,D}=\{\left(\begin{matrix} 1 & x \\ 0 & 1\end{matrix}\right) \ |\ x\in a_1a_2^{-1}(\eps\OO+f_D\OO)\}.$$
%$$U^1_{t,D}=\{\left(\begin{matrix} 1 & \eps x \\ 0 & 1\end{matrix}\right) \ |\ x\in \ov{a}_1\ov{a}_2^{-1}\ov{\OO}\}.$$
\end{lemma}

\begin{proof}
Indeed, $U_{t,D}=U(\A)\cap (tg_D)G(\OO)(tg_D)^{-1}$, so the condition on $x$ is that
$$g_D^{-1}t^{-1}\left(\begin{matrix} 1 & x \\ 0 & 1\end{matrix}\right)tg_D=\left(\begin{matrix} 1+\eps f_D^{-1}a_1^{-1}a_2x & a_1^{-1}a_2x \\ 0 & 1-\eps f_D^{-1}a_1^{-1}a_2x\end{matrix}\right)$$
is in $G(\OO)$. In other words, the condition is that $a_1^{-1}a_2x\in \OO$ and $\eps f_D^{-1}a_1^{-1}a_2x\in \OO$, which is equivalent to
$x\in a_1a_2^{-1}(\eps\OO+f_D\OO)$. 
%The identification of $U^1_{t,D}=U(\eps\A)\cap U_{t,D}$ follows.
\end{proof}

\subsection{Auxiliary results involving reduction}\label{const-term-reduction-sec}

We continue to work with $G=\GL_2$.
The projection $G(\A)\to G(\ov{\A})$ induces a well defined map
\begin{align*}
&\QBun_T(C)=(T(F)U(\A))\backslash G(\A)/G(\OO)\to (T(\ov{F})U(\ov{\A}))\backslash G(\ov{\A})/G(\ov{\OO})\simeq \\
&T(\ov{\A})/(T(F)\cdot T(\ov{\OO}))\simeq \Pic(\ov{C})\times \Pic(\ov{C}).
\end{align*}
Concretely, the double coset represented by $t\cdot g_D$, where 
$t=\left(\begin{matrix} a_1 & 0 \\ 0 & a_2\end{matrix}\right)\in T(\A)$ is mapped to the pair of line bundles
$(L_0,M_0)$, corresponding to the reductions of the ideles $a_1$ and $a_2$.

\begin{lemma}\label{support-surjection-lem}
If $t\in \QBun_T(C,D)$ is in the support of $C_D(f)$
% (resp., $ut\in (T(F)U(\eps\A))\backslash B(\A)/B(\OO)[D]$ is in the support of $E^1_D(f)$), 
and $(L_0,M_0)$ are the line bundles on $\ov{C}$ associated with $t$,
then there exists a vector bundle $F$ in the support of $f$, and a surjection $\ov{F}\to M_0$.
\end{lemma}

\begin{proof} 
By definition of $C_D(f)$, 
%(resp., $E^1_D(f)$), 
there exists an element $g\in G(\A)$ in the support of $f$, such that $g=u_1tg_D$ for some $u_1\in U(\A)$.
This gives $\ov{g}=\ov{u}_1\ov{t}\in B(\ov{\A})$. Hence the vector bundle $\ov{F}$ is an extension of $M_0$ by $L_0$.
\end{proof}

Let us denote by
$$\pi:\Bun(C)\to \Bun(\ov{C})$$  
the natural functor of groupoids given by reduction. 
We denote by $\pi^{-1}(V_0)$ the groupoid fiber of $\pi$ over $V_0$. Note that the set of isomorphism classes of $\pi^{-1}(V_0)$
surjects onto the usual set-theoretic fibers of $\pi$ (as a map on isomorphism classes of objects).
%The following description of the groupoid fibers of $\pi$ is well known.

Let us denote by $\fg_0\sub \fg$ the subalgebra ${\mathfrak sl}_2$.
Recall that we denote $N_0$ the line bundle on $\ov{C}$ corresponding to $\NN$, and $\eps=(\eps_p)\sub\A$, where $\eps_p$ are local generators of $N_0$.
Let us also denote by $\eps_F\in\A$ the principal adele corresponding to a generator of $(N_0)_F$. Then we have
\begin{equation}\label{eps-F-eq}
\eps=\eps_F a_0,
\end{equation}
where $a_0\in\ov{\A}^*$, so $N_0$ is the line bundle associated with $a_0$.

\begin{lemma}\label{reduction-fibre-lem}
%(i) 
For every $V_0\in \Bun(\ov{C})$ there is a simply transitive action of
$H^1(\ov{C},\und{\End}(V_0)\ot N_0)$ on the isomorphism classes of objects of the groupoid fiber $\pi^{-1}(V_0)$.
If $V$ is represented by the adele-valued element $g\in G(\A)$, then this action is induced by $g\mapsto (1+\eps_F X)g$, 
where $X\in \fg \ot N_0(\ov{\A})$. 
Assume in addition that the map $H^0(C,\OO_C)\to H^0(\ov{C},\OO_{\ov{C}})$ is surjective.
Then for $V_0\in \Bun^{\ov{L}}(\ov{C})$, there is a transitive action of $H^1(\ov{C},\und{\End}_0(V_0)\ot N_0)$ on 
the isomorphism classes of objects of $\pi^{-1}(V_0)\cap \Bun^L(C)$,
 where $\und{\End}_0(V_0)$ is the sheaf of endomorphisms of $V_0$ with zero trace. 
%
%\noindent 
%(ii) Let us denote the result of the action of $e\in H^1(\ov{C},\und{\End}_0(V_0)\ot N_0)$ on $V\in \pi^{-1}(V_0)$ by $e+F$.
%Then for any automorphism $\a_0$ of $V_0$, one has 
%$$\a_0(e)+V=e+V,$$
%where we use the adjoint action of $\Aut(V_0)$ on $\und{\End}_0(V_0)$. 
\end{lemma}

\begin{proof}
%(i) 
Let $g\in G(\A)$ be an element with the reduction $\ov{g}\in G(\ov{A})$. By definition, the isomorphism classes of the groupoid fiber over $\ov{g}$ is the set of double cosets
$$(1+\eps_F \fg\ot \ov{F})\backslash (1+\eps_F \fg\ot \ov{\A})g/(1+\eps \fg\ot \ov{\OO}).$$ 
%Since the maps $G(F)\to G(\ov{F})$ and $G(\OO)\to G(\ov{\OO})$
%are surjective, if $g'\in G(\A)$ and $g$ project to the same element in 
%$G(\ov{F})\backslash G(\ov{\A})/G(\ov{\OO})$ then we can change $g'$ to another representative in the double coset of $g'$,
%so that $g'=(1+\eps_F X)g$, where $X\in \fg(\ov{\A})$. 
%Then the double coset of $(1+\eps_F X)g$ in $\Bun(C)$ depends only on the class of $X$ in
Since $\eps=\eps_F a_0$, we can identify this set with
$$\fg(\ov{\A})/(\fg(\ov{F})+a_0\cdot g\fg(\ov{\OO})g^{-1})\simeq H^1(\ov{C},\und{\End}(V_0)\ot N_0).$$

For the version with fixed determinants, we observe that 
if $\det((1+\eps_F X)g)$ represents the same line bundle
as $\det(g)$ then $\det(1+\eps_F X)=1+\eps_F\tr(X)$ is of the form $f\cdot u$ with $f\in F^*$, $u\in\OO^*$. 
This implies that $\ov{f}\ov{u}=1$, so $\ov{f}\in \ov{F}^*\cap \ov{\OO}^*$ can be lifted to an element of $F^*\cap \OO^*$ (using our assumption).
Hence, modifying $f$ and $u$ we can assume that $\ov{f}=\ov{u}=1$. Hence, we can modify $g'$ by appropriate diagonal elements, $1$ modulo $\eps_F$,
so that $\tr(X)=0$. 
%\noindent
%(ii) We can lift $\a_0$ to an automorphism $\a$ of $V$. The latter is represented by an element $\a\in G(F)$ such that $\a gG(\OO)=gG(\OO)$.
%If $e$ is represented by $X\in \fg_0(\ov{A})$ then $e+V$ is represented by $(1+\eps X)g$, which is equivalent to
%$$\a(1+\eps X)g=(1+\eps \Ad(\a_0)(X))\a g\simeq (1+\eps \Ad(\a_0)(X))g.$$
%The latter element represents $\a_0(e)+V$. 
\end{proof}

\subsection{Constant term and the Hecke operators}\label{GL2-const-term-Hecke-sec}

We continue to assume that $G=\GL_2$ and that $C$ is a special extension of length $2$.
Let $K\sub G(\OO)$ be an open subgroup.
For a divisor $D\ge 0$ on $\ov{C}$, $r\in G(\OO)$ and $t\in T(\A)$, and $f\in \C(\Bun_G(C,K))$ we set
\begin{equation}\label{Const-term-D-K-eq}
C_{D,rK}f(t)=\int_{U(F)\backslash U(\A)} f(utg_Dr)du.
\end{equation}
Since $f$ is right $K$-invariant, the right-hand side depends only on $rK$. 
%Furthermore, if $K$ is a normal subgroup in $C^i_P$ then we have
%$$C_{D,rK}f(t)=\vol(U_{t,K,D}/(U(F)\cap U_{t,K,D}))\cdot
%\sum_{u\in U(\A)/(U(F)\cdot U_{t,K,D})}f(utg_Dr),$$
%where $U_{t,K,D}:=U(\A)\cap (tg_D)K(tg_D)^{-1}$.

By Corollary \ref{Iwasawa-cor}, every $g\in G(\A)$ can be written as
$$g=utg_Dr, \ \ u\in U(\A), \ t\in T(\A), \ r\in G(\OO),$$
for some divisor $D\ge 0$ on $\ov{C}$. Thus, the operators \eqref{Const-term-D-K-eq}, with $r$ running over representatives of the finite set
$K/G(\OO)$, determines the constant term. In particular, 
$$f \text{ is cuspidal} \iff \ C_{D,rK}f(t)=0 \ \text{ for all } D, r, t.$$

Recall that for each element $g_0\in G(\A)$, we have the corresponding Hecke operator
$T_{g_0,K}$ given by formula \eqref{hecke-gen-formula-eq}.
%f(g)=\int_{h\in G(\OO)}f(ghg_0)dh.$$
We want to rewrite the general formula \eqref{EPT-formula} from the proof of Proposition A more concretely, in terms of the operators $C_D$, where
$D\in \Div(\ov{C})_{\ge 0}$.

%As before, we consider the subgroup of finite index in $G(\OO)$,
%$$H_{g_0}:=G(\OO)\cap g_0G(\OO)g^{-1}_0.$$
Fix $g_0\in G(\A)$, $D\in \Div(\ov{C})_{\ge 0}$ and $r\in G(\OO)$. For each representative $h\in G(\OO)/H_{g_0,K}$ let us choose a decomposition 
%using the analog of Iwasawa decomposition we can write
\begin{equation}\label{Iw-gDhg0-eq}
g_Drhg_0=u_ht_hg_{D_h}r_h
%g_\OO,
\end{equation}
with $r_h\in G(\OO)$, $u_h\in U(\A)$, $t_h\in T(\A)$, and some $D_h\in \Div(\ov{C})_{\ge 0}$.

\begin{lemma}\label{hecke-ED-comp-lem} (i) For $f\in \C(\Bun_G(C,K))$
one has
$$T_{g_0,K}f(utg_Dr)=\sum_{h\in G(\OO)/H_{g_0,K}} f(u\Ad(t)(u_h)tt_hg_{D_h}r_h),$$
%\vol(H_{g_0})
%There exists a collection of operators $T^{D,D'}_{g_0}$ such that
$$(C_{D,rK} T_{g_0,K})f(t)=\sum_{h\in G(\OO)/H_{g_0,K}} C_{D_h,r_hK}f(tt_h).$$
%\sum_{D'} T^{D,D'}_{g_0} E_{D'}f,$$
%where the sum on the right is finite.

\noindent
(ii) Assume in addition that $K$ is unramified at a place $p$, i.e., $K=K^p\times G(\OO_p)$, and that $g_0\in G(F_p)$ and $r\in G(\OO)$ has $r_p=1$.
Then $G(\OO)/H_{g_0,K}\simeq G(\OO_p)/H_{g_0,p}$, where $H_{g_0,p}=G(\OO_p)\cap g_0G(\OO_p)g_0^{-1}$. For each representative 
$h\in G(\OO_p)/H_{g_0,p}$, we can write
$$g_Dhg_0=u_ht_hg_{D_h}k_h, \ \text{ with } u_h\in U(F_p), t_h\in T(F_p), k_h\in G(\OO_p),$$
where $D_h\in \Div(\ov{C})_{\ge 0}$ can differ from $D$ only at $p$.
Then for $f\in \C(\Bun_G(C,K))$, one has
$$(C_{D,rK}T_{g_0,K}f)(t)=\sum_{h\in G(\OO_p)/H_{g_0,p}} C_{D_h,rK}f(tt_h).$$
\end{lemma}

\begin{proof}
(i) Using \eqref{Iw-gDhg0-eq}, we get for $f\in\C(\Bun(C))$,
$$f(utg_Drhg_0)=f(utu_ht_hg_{D_h}r_h)=f(u\Ad(t)(u_h)tt_hg_{D_h}r_h).$$
Hence,
$$T_{g_0}f(t)=\sum_{h\in G(\OO)/H_{g_0}}f(utg_Drhg_0)=
\sum_{h\in G(\OO)/H_{g_0}}f(u\Ad(t)(u_h)tt_hg_{D_h}r_h),$$
which gives the first formula. Integrating over $u\in U(F)\backslash U(\A)$ we get the second formula.
%$\int_{u\in U(F)\backslash U(\A)} $
%In other words,
%$$E_D T_{g_0}f(t)=\sum_{D'} T^{D,D'}_{g_0} E_{D'}f(t),$$
%where
%$$T^{D,D'}_{g_0}\varphi(t)=\sum_{h\in G(\OO)/H_{g_0}: D(h)=D'} \varphi(tt(h))dh.$$

\noindent
(ii) Since $r$ commutes with $G(F_p)$, we have
$$g_Drhg_0=g_Dhg_0r=u_ht_hg_{D_h}k_hr=u_ht_hg_{D_h}rk_h.$$
Now the assertion follows from part (i).
%$$(C_{D,rK} T_{g_0,K}f)(t)=\sum_{h\in G(\OO_p)/H_{g_0,p}} C_{D_h,rK}
\end{proof}

%We want to establish a more precise version of Lemma \ref{hecke-ED-comp-lem} in 
Now we specialize further, and consider the Hecke operator $T_c=T_{g_c,K}$
%associated with $$g_c:=\left(\begin{matrix} f_c^{-1} & 0 \\ 0 & 1\end{matrix}\right),$$
%where $f_c\in \A$ is a generator of the ideal corresponding to 
associated with a simple divisor $c$ of $C$, such that $K$ is unramified at $p=\ov{c}$. 
Here $(g_c)_p=\diag(f_c^{-1},1)$ and $(g_c)_v=1$ for $v\neq p$,
for some local equation of $c$, $f_c\in \OO_p$ (see \eqref{gc-eq}).

\begin{lemma}\label{hecke-qbun-formula-lem} 
Assume that $c\sub C$ is a simple divisor, such that $K$ is unramified at $p=\ov{c}$, and let $D\in\Div(\ov{C})_{\ge 0}$. 
Let $n$ denote the multiplicity of $\ov{c}$ in $D$, and set $A(c)=\OO_p/(f_c)$. Fix $r\in G(\OO)$ with $r_p=1$.
Then for $n>0$, one has
%\begin{align*}
%&E^1_D T_cf(t)=
%\begin{align*}
$$C_{D,rK} T_cf(t)=
%\begin{cases} 
\sum_{a\in A(c)}C_{D-p,rK}f(t\left(\begin{matrix} 1-\eps_p f_c^{-n}a & 0 \\ 0 & (1+\eps_p f_c^{-n}a)f_c^{-1}\end{matrix}\right))+
\sum_{b\in k(p)}C_{D+p,rK}f(t \left(\begin{matrix} f_c^{-1} & 0 \\ 0 & 1+bf_c^n\end{matrix}\right)),$$
%\end{align*}
while for $n=0$, one has
%\begin{align*}
$$C_{D,rK} T_cf(t)=
|A(c)|\cdot C_{D,rK}f(t\left(\begin{matrix} 1 & 0 \\ 0 & f_c^{-1}\end{matrix}\right))+
\sum_{b\in k(p)^*}C_{D+p,rK}f(t \left(\begin{matrix} f_c^{-1} & 0 \\ 0 & b\end{matrix}\right))+C_{D,rK}f(t \left(\begin{matrix}f_c^{-1} & 0 \\ 0 & 1\end{matrix}\right)),$$
%\end{align*} 
where $t\in T(\A)$.
%, $u\in U(\ov{\A})$
%\end{cases}
%$$
%&\begin{cases} |k(\ov{c})|\cdot (E_{D-\ov{c}}f)(t\cdot \left(\begin{matrix} 1 & 0 \\ 0 & f_c^{-1}\end{matrix}\right))+
%\sum_{\a\in k(\ov{c})} (E_{D+\ov{c}}f)(t\cdot \left(\begin{matrix} f_c^{-1}(1-\eps f_c^{-n}\a) & 0 \\ 0 & 1+\eps f_c^{-n}\a\end{matrix}\right)), & D-\ov{c}\ge 0,\\
%|k(\ov{c})|\cdot (E_Df)(t\cdot \left(\begin{matrix} 1 & 0 \\ 0 & f_c^{-1}\end{matrix}\right))+
%\sum_{\a\in k(\ov{c})} (E_{D+\ov{c}}f)(t\cdot \left(\begin{matrix} f_c^{-1}(1-\eps f_c^{-n}\a) & 0 \\ 0 & 1+\eps f_c^{-n}\a\end{matrix}\right)), & \text{otherwise}.
%\end{cases}
%\end{align*}
%$T^{D,D'}_{g_c}=0$ unless $D'=D\pm c$. Furthermore,
%$$T^{D,D-c}_{g_c}\varphi(t)=\vol(H_{g_c})\cdot \varphi(t\cdot \left(\begin{matrix} 1 & 0 \\ 0 & t_c^{-1}\end{matrix}\right)),$$
%$$T^{D,D+c}_{g_c}\varphi(t)=\vol(H_{g_c})\cdot \sum_{\a\in k}
%\varphi(t\cdot \left(\begin{matrix} t_c^{-1}-et_c^{-n}\a & 0 \\ 0 & 1+et_c^{-n}\a\end{matrix}\right)),$$
%where $n$ is the multiplicity of $c$ in $D$.
\end{lemma}

\begin{proof}
%The subgroup $H_{g_c}\sub G(\OO)$ consists of matrices $(a_{ij})\in G(\OO)$ such that $a_{12}\in I_c:=f_c\cdot \OO$. Thus, 
Recall (see the proof of Lemma \ref{hecke-mod-adelic-lem}) 
that the set of cosets $G(\OO)/H_{g_c,K}\simeq G(\OO_p)/H_{g_c,p}$ can be identified with $G(A(c))/B(A(c)),$
which itself is identified with the projective line over $A(c)$.
%, the \'etale extension of $A$ associated with $c$.
From this we see that as representatives of $G(\OO_p)/H_{g_c,p}$ we can take the matrices
$$h_a:=\left(\begin{matrix} a & 1 \\ 1 & 0\end{matrix}\right), \ \ h'_b:=\left(\begin{matrix} 1 & 0 \\ \eps_p b & 1\end{matrix}\right), \ \ $$
where $a$ runs over representatives of $\OO_p/(f_c)$, $b$ runs over representatives of $\ov{\OO}/(\ov{f}_c)$.

Now, as in Lemma \ref{hecke-ED-comp-lem}(ii), we apply the analog of Iwasawa decompositions to the elements $g_Dh_ag_c$, $g_Dh'_bg_c$,
i.e., we should find $u_a,u'_b\in U(F_p)$, $t_a,t'_b\in T(F_p)$ and divisors $D_a,D'_b\in \Div(\ov{C})_{\ge 0}$ (that can only differ from $D$ at $p$), such that
$$g_Dh_ag_c\in u_at_ag_{D_a}G(\OO_p), \ \  g_Dh'_bg_c\in u'_bt'_bg_{D'_b}G(\OO_p).$$
Then as we have seen in Lemma \ref{hecke-ED-comp-lem}(ii), we will have
%$$T_cf(utg_D)=\sum_{a\in \OO/I_c}f(u\Ad(t)(u_a)tt_ag_{D_a})+\sum_{b\in \ov{\OO}/\ov{I}_c}f(u\Ad(t)(u_b)tt_bg_{D_b}),$$ 
$$C_{D,rK}T_cf(t)=\sum_{a\in \OO_p/(f_c)} C_{D_a,rK}f(tt_a)+\sum_{b\in \ov{\OO}_p/(\ov{f}_c)} C_{D'_b,rK}f(tt'_b).$$
It remains to find formulas for $t_a$, $t'_b$, $D_a$ and $D'_b$.

%Since $(g_c)_p$ belongs to $G(\OO_p)$ for $p\neq \ov{c}$, this is a local computation at the point $\ov{c}$
%(so $u_b,u_a,t_b,t_a$ have trivial components away from $\ov{c}$ and $D_b,D_a$ differ from $D$ only at $\ov{c}$).
%Let $n\ge 0$ denote the multiplicity of $\ov{c}$ in $D$. 
Let us take $\ov{f}_c$ as a uniformizer on $\ov{C}$ at $p$,
and use it when defining the matrix $\varphi_n$ at $p$, so 
$$(g_D)_p=v_n:=\left(\begin{matrix} 1 & 0 \\ \eps_p f_c^{-n} & 1\end{matrix}\right).$$ 

Assume first that $n>0$.
We have
\begin{equation}\label{vn-ha-id}
v_nh_ag_c=\left(\begin{matrix} 1 & a(1-\eps_p a f_c^{-n}) \\ 0 & 1\end{matrix}\right)
\left(\begin{matrix} 1-\eps_p f_c^{-n}a & 0  \\ 0 & (1+\eps_p f_c^{-n}a) f_c^{-1}\end{matrix}\right) 
v_{n-1}s, 
\end{equation}
where $s=\left(\begin{matrix} 0 & 1 \\ 1 & 0\end{matrix}\right)$. On the other hand,
$$v_nh'_bg_c=
\left(\begin{matrix} f_c^{-1} & 0  \\ 0 & 1+bf_c^n \end{matrix}\right)\cdot v_{n+1}\cdot
\left(\begin{matrix} 1 & 0 \\ 0 & (1+bf_c^n)^{-1}\end{matrix}\right),$$
Hence, we get $D_a=D-p$, $D'_b=D+p$,
$$u_a=\left(\begin{matrix} 1 & a(1-\eps_p a f_c^{-n}) \\ 0 & 1\end{matrix}\right), \ u_b=1, \ 
t_a=\diag(1-\eps_p f_c^{-n}a,(1+\eps_p f_c^{-n}a)f_c^{-1}), \ t'_b=\diag(f_c^{-1},1+bf_c^n).$$
This gives our formula for $n>0$.

Next, assume that $n=0$, so $(g_D)_p=1$. Then the identity 
$$h_ag_c=\left(\begin{matrix} 1 & a \\ 0 & 1\end{matrix}\right)\left(\begin{matrix} 1 & 0 \\ 0 & f_c^{-1}\end{matrix}\right)s$$
shows that $D_a=D$, $u_a=\left(\begin{matrix} 1 & a \\ 0 & 1\end{matrix}\right)$ and $t_a=\diag(1,f_c^{-1})$.
This gives the first term in the formula for $C_{D,rK} T_cf(t)$ for $n=0$.
%Note that for $n=0$, we have 
%$$t_a\equiv \left(\begin{matrix} 1 & 0  \\ 0 & f_c^{-1}\end{matrix}\right) \mod T(\OO)[D].$$

To calculate $(D'_b,u'_b,t'_b)$ we consider two subcases (still for $n=0$). For $b=0$ we have
$$h'_0g_c=g_c=\diag(f_c^{-1},1)$$
so $D'_0=D$, $u'_0=1$, $t'_0=\diag(f_c^{-1},1)$, which gives the last term.
For $b\neq 0$, choose a lift $b\in \OO_p^*$. Then we have
$$h'_bg_c=\diag(f_c^{-1},b)v_1\diag(1,b^{-1}),$$
%$$(1+\eps_p \varphi_0)h_bg_c=\left(\begin{matrix} f_c^{-1} &0  \\ 0 & b+1\end{matrix}\right)(1+\eps_p\varphi_{1})\left(\begin{matrix} 1 &0  \\ 0 & (b+1)^{-1}\end{matrix}\right),$$
%$$t_bc\left(\begin{matrix} 1 &0  \\ 0 & f_c^{-1}\end{matrix}\right),$$
so $D_b=D+p$, $u_b=1$, $t_b=\diag(f_c^{-1},b)$. This gives the second term in the formula for $C_{D,rK} T_cf(t)$ for $n=0$.
%This finishes the proof of Lemma.
\end{proof}

\subsection{Constant term and the Hecke operators over a finite field}\label{const-term-finite-field-sec}

Here we consider (a much simpler) classical case of the reduced curve $\ov{C}$ over the finite field $k$ (we continue to assume $G=\GL_2$).
We can consider the constant term and the Hecke operators for functions on $\Bun_G(\ov{C},K)$, where $K\sub G(\ov{\OO})$ is an open compact subgroup.

The Hecke operators $T_p=T_{g_p}$ associated with points $p\in \ov{C}$ and the 
constant term operator 
$$C:\C(\Bun_G(\ov{C},K))\to \C(T(\ov{F})U(\ov{\A})\backslash G(\ov{\A})/K))$$ 
are defined as before.

Assume now that for some point $p\in \ov{C}$, $K$ is unramified at $p$, i.e., $K=G(\ov{\OO}_p)\times K'$, where $K'\sub G(\ov{\A}')$, where $\ov{\A}'$ are adeles for 
$\ov{C}\setminus p$.
Then using the Iwasawa decomposition $G(\ov{F}_p)=B(\ov{F}_p)G(\ov{\OO}_p)$, we can write
$$T(\ov{F})U(\ov{\A})\backslash G(\ov{\A})/K)=T(\ov{F})U(\ov{\A}')\backslash T(\ov{F}_p)G(\ov{\A}')/T(\ov{\OO}_p)K',$$
so we can use representatives of the form $tg'$, where $t\in T(\ov{F}_p)$, $g'\in G(\ov{\A}')$, in the argument of the constant term operator
$Cf$.

In this situation we have the following simple analog of Lemma \ref{hecke-qbun-formula-lem}.

\begin{lemma}\label{hecke-finite-formula-lem} Let $f_p$ be a generator of the maximal ideal of $\ov{\OO}_p$.
One has 
$$CT_pf(tg')=|k(p)|\cdot Cf(t\left(\begin{matrix} 1 & 0 \\ 0 & f_p^{-1}\end{matrix}\right)g')+Cf(t\left(\begin{matrix}f_p^{-1} & 0 \\ 0 & 1\end{matrix}\right)g').$$
\end{lemma}

We omit the proof since it is similar to (but simpler than) that of Lemma \ref{hecke-qbun-formula-lem}.

\section{Cuspidal functions for $\GL_2$: bounds on support}\label{cusp-GL2-sec}

In this section we continue to assume that $G=\GL_2$. The main result of this section is the proof of Theorem C.
% about the space of cuspidal functions for ($G=\GL_2$ and $C$ of length $2$).
In the proof we use the adelic representatives for
vector bundles $V$ on $C$ with split reduction $\ov{V}\simeq L_0\oplus M_0$ such that $\deg(M_0)-\deg(L_0)\ge 2g+2$
%If the line bundles $L_0$ and $M_0$ on $\ov{C}$ are
%such that $\Ext^1(L_0,M_0)=0$ then the bundle $V$ (which we can view as a deformation of $L_0\oplus M_0$) is determined by line bundles $L$ and $M$ on $C$,
%extending $L_0$ and $M_0$, and by an extension class $e\in\Ext^1(M_0,L_0)=H^1(\ov{C},M_0^{-1}L_0)$ 
%(see Lemma\ \ref{split-reduction-lem}). 
%By fixing a minimal effective divisor $D$ on $\ov{C}$ such that $e$ goes to zero under $H^1(\ov{C},M_0^{-1}L_0)\to H^1(\ov{C},M_0^{-1}L_0(D))$,
%we then translate this into adelic language by finding some distinguished representatives for double cosets in $G(\A)$ corresponding to such bundles $V$
(see Lemma \ref{split-red-representatives-lem}). 
%The key technical result, Lemma \ref{adelic-const-term-ex-lem}, 
%implies the vanishing of certain
%sums of values of a cuspidal function on bundles with split reduction.
We then show in Proposition \ref{cusp-vanishing-thm} the existence of a constant $N(K)$ for an open subgroup $K\sub G(\OO)$ (given explicitly for $K=G(\OO)$) such
that in the case when $\deg(M_0)-\deg(L_0)\ge N(K)$ any cuspidal function vanishes on an element of $\Bun_G(C,K)$ with reduction $L_0\oplus M_0$.
This vanishing immediately implies that the space of cuspidal functions on $\Bun_G^L(C,K)$ is finite-dimensional and that cuspidal functions are Hecke-bounded.
To show the converse (see Proposition \ref{hecke-bound-prop}) we use a precise information on the compatibility of the Hecke operators with the constant term operator given by
Lemma \ref{hecke-qbun-formula-lem}.

%To give a lower bound for the dimension of the space of cuspidal functions (see Theorem \ref{cusp-fin-dim-thm}) we consider functions on $\Bun^L(C)$ supported on bundles such that $\ov{V}%$ is stable and ``typical",
%i.e., we remove a finite number of bad loci from the moduli space of stable bundles on $\ov{C}$. Then for every allowed $\ov{V}$ we use the identification of the set of bundles with given reduction $\ov{V}$ with a $k$-vector space and impose equations that the sums of values over certain affine subspaces are zero. Using Proposition \ref{const-term-red-prop} we show
%that these equations imply that the function is cuspidal. Finally, we estimate from above the number of equations imposed. 

%In Proposition \ref{upper-bound-prop} we establish the upper bound on the dimension of the space of cuspidal functions (based on the vanishing proved in Proposition \ref{cusp-vanishing-thm}).

In Sec.\ \ref{cusp-hb-sec}
we prove Theorem C(2) that cuspidality is equivalent to (weak) Hecke-boundedness (see Proposition \ref{hecke-bound-prop}).
%use the connection with representation theory of $G(\A)$ to describe the space of cuspidal functions for $G=\PGL_2$???
%consider the case of $G=\PGL_2$, proving Corollary D.

\subsection{Elements of $G(\A)$ with split reduction}
%Henceforward, we always assume that $G=\SL_2$, $A$ is of length $2$, $\eps$ is a generator of the maximal ideal in $A$.
For $g\in G(\A)$ we denote its reduction modulo $\eps$ by $\ov{g}\in G(\ov{\A})$.
Here we will study $g\in G(\A)$ such that $\ov{g}$ is diagonal.

For an idele $a\in \ov{\A}^*$, we define the degree $\deg(a)$ as the degree of the corresponding line bundle on $\ov{C}$.
For example, for the idele $f_D$ associated with an effective divisor $D\sub \ov{C}$ (see \eqref{fD-eq}) we have $\deg(f_D)=-\deg(D)$.

\begin{lemma}\label{split-degree-div-ext-lem}
(i) Let $N_0$ be a line bundle of degree $\ge 2g+2$ on $\ov{C}$.
Then for any class $e\in H^1(N_0^{-1})$ 
%Let $E$ be an extension of $M_0$ by $L_0$ on $\ov{C}$. Assume that $\deg(M_0)-\deg(L_0)\ge 2g+2$.
%Then 
there exists an effective divisor $D\sub \ov{C}$ of degree $\le \frac{\deg(N_0)+1}{2}+g$, 
%$\le \frac{\deg(M_0)-\deg(L_0)}{2}+g+1$, 
defined over $k$, such that the class $e$ goes to zero under the map $H^1(N_0^{-1})\to H^1(N_0^{-1}(D))$.
%this extension splits over $M_0(-D)\sub M_0$.

\noindent
(ii) Let $a\in \ov{\A}^*$ be an element with $\deg(a)\ge 2g+2$. Then any class in $\ov{\A}/(\ov{F}+a^{-1}\ov{\OO})$ can be represented by an element
of the form $ua^{-1}f_D^{-1}$ 
for some $u\in \ov{\OO}^*$ and some effective divisor $D\sub \ov{C}$ of degree $\le\frac{\deg(a)+1}{2}+g$.
\end{lemma}

\begin{proof}
(i) Let 
$$0\to \OO\to E\to N_0\to 0$$
be an extension on $\ov{C}$ representing the class $e\in H^1(N_0^{-1})\simeq \Ext^1(N_0,\OO)$.
Take any line bundle $P_0$ of degree $\lfloor \frac{\deg(N_0)}{2}\rfloor-g$ (it is well known that there exists a line bundle of every degree on $C$ defined over $k$, 
see e.g. \cite[Sec.\ 1.5.3, Lem.\ 1]{Soule}). Then $\mu(P_0^\vee\ot E)\ge g$. Hence, $H^0(P_0^\vee\ot E)\neq 0$, in other words, there exists a nonzero
morphism $P_0\to E$. Note that by assumption,
$$\deg(P_0)=\lfloor \frac{\deg(N_0)}{2}\rfloor-g\ge 1,$$
so $\Hom(P_0,\OO)=0$. Hence, the composition $P_0\to E\to N_0$ is nonzero, so $P_0=N_0(-D)$, where 
$$\deg(D)=\deg(N_0)-\deg(P_0)\le \deg(N_0)-\frac{\deg(N_0)-1}{2}+g=\frac{\deg(N_0)+1}{2}+g.$$

\noindent
(ii) Let $N_0$ be the line bundle on $\ov{C}$ associated with $a$. Since we have an identification $H^1(N_0^{-1})=\ov{\A}/(\ov{F}+a^{-1}\ov{\OO})$,
by part (i), for any class $e\in \ov{\A}/(\ov{F}+a^{-1}\ov{\OO})$
there exists an effective divisor $D\sub \ov{C}$ of degree $\le\frac{\deg(a)+1}{2}+g$, such that $e$ goes to zero under the map
$$\ov{\A}/(\ov{F}+a^{-1}\ov{\OO})\to \ov{\A}/(\ov{F}+a^{-1}f_D^{-1}\ov{\OO}).$$
In other words, $e$ is in the image of the map
$$f_D^{-1}\ov{\OO}/\ov{\OO}\rTo{a^{-1}} \ov{\A}/(\ov{F}+a^{-1}\ov{\OO}).$$
But any element of $f_D^{-1}\ov{\OO}/\ov{\OO}$ is represented by $uf_{D'}^{-1}$ for some $u\in \ov{\OO}^*$ and some subdivisor $D'\sub D$,
and our assertion follows.
\end{proof}

Recall that $\eps=\eps_F\cdot a_0$, where $a_0\in\ov{\A}^*$ is the idele corresponding to the line bundle $N_0$ on $\ov{C}$ (corresponding to $\NN\sub \OO_C$).
Set
$$\ell=\deg(N_0)=\deg(a_0).$$

\begin{lemma}\label{split-red-representatives-lem}
%For an effective divisor $D$, let $f_D=(t_p^{n_p})$, where $n_p$ is the multiplicity of $D$ at $p$.
Suppose the reduction modulo $\eps$ of a double coset $G(F)gG(\OO)$ is $G(\ov{F})\diag(\ov{a}_1,\ov{a}_2)G(\ov{\OO})$,
where 
$$\deg(\ov{a}_2)-\deg(\ov{a}_1)\ge 2g+2+|\ell|.$$ 
Then there exists an effective divisor $D\sub \ov{C}$ of degree $\le\frac{\deg(\ov{a}_2)-\deg(\ov{a}_1)-\ell+1}{2}+g$,
such that
$$G(F)gG(\OO)=G(F)\left(\begin{matrix} a_1 & \eps\cdot a_1f_D^{-1} \\ 0 & a_2\end{matrix}\right)G(\OO)$$
for some ideles $a_1,a_2$ whose reductions represent the same line bundles as $\ov{a}_1$ and $\ov{a}_2$.
%such that $\ov{\A}_F=\ov{F}+a_1^{-1}a_2\ov{\OO}$ and 
%such the line bundle $L_0$ (resp., $M_0$) corresponds to $a_1\mod \eps$ (resp., $a_2\mod \eps$).
\end{lemma}

\begin{proof}
First, it is clear that we can choose a representative $g$ for our double coset of the form
$g=(1+\eps X)\cdot t_a$, where $t_a=\diag(a_1,a_2)$, where $a_1,a_2\in \A^*$ are some liftings of $\ov{a}_1,\ov{a}_2\in \ov{\A}^*$.
Furthermore, we claim that we can assume $X$ to be upper-triangular. Indeed, setting $\varphi(x):=\left(\begin{matrix} 1 & 0 \\ x & 1\end{matrix}\right)$,
we have
$$\Ad(t_a)\varphi(x)=\varphi(\ov{a}_2\ov{a}_1^{-1}x).$$
Thus, multiplying $g$ on the right with elements from $\varphi(\eps \ov{\OO})$ we can add to $X$ elements of $\fu_-(\ov{a}_2\ov{a}_1^{-1}\ov{\OO})$.
On the other hand, multiplying $g$ on the left with elements from $(1+\eps_F \fu_-(\ov{F})=1+\eps a_0^{-1} \fu_-(\ov{F}))$, we can add to $X$ elements of $a_0^{-1}\fu_-(\ov{F})$.
But $a_0^{-1}\ov{F}+\ov{a}_2\ov{a}_1^{-1}\ov{\OO}=\ov{\A}$ (since $\deg(a_0\ov{a}_2\ov{a}_1^{-1})\ge 2g+2$), so we can make $X$ to be upper-triangular.

Next, modifying $a_1$ and $a_2$, we can assume that $X$ is strictly upper triangular, so 
$$1+\eps X=\left(\begin{matrix} 1 & \eps x \\ 0 & 1\end{matrix}\right),$$
where $x$ can be modified to any element in its $a_0^{-1}\ov{F}+\ov{a}_1\ov{a}_2^{-1}\ov{\OO}$-coset. Thus, applying Lemma \ref{split-degree-div-ext-lem}(ii) to
$a_0x$ and $a=a_0^{-1}\ov{a}_2\ov{a}_1^{-1}$ (which is possible since $\deg(a_0^{-1}\ov{a}_2\ov{a}_1^{-1})\ge 2g+2$), we can assume that
$$x=u\ov{a}_1\ov{a}_2^{-1}f_D^{-1}$$
with $u\in \ov{\OO}^*$, for some effective divisor $D$ of degree $\le\frac{\deg(\ov{a}_2)-\deg(\ov{a}_1)-\ell+1}{2}+g$. Thus, we get a representative of the form
$$g=\left(\begin{matrix} 1 & \eps\cdot ua_1a_2^{-1}f_D^{-1}\\ 0 & 1\end{matrix}\right)\cdot t_a=\left(\begin{matrix} a_1 & \eps\cdot ua_1f_D^{-1}\\ 0 & a_2\end{matrix}\right).$$
Finally, we observe that
$$\left(\begin{matrix} a_1 & \eps\cdot ua_1f_D^{-1} \\ 0 & a_2\end{matrix}\right)\cdot
\left(\begin{matrix} u & 0\\ 0 & 1\end{matrix}\right)=
\left(\begin{matrix} ua_1 & \eps\cdot ua_1f_D^{-1} \\ 0 & a_2\end{matrix}\right).$$
Hence, replacing $a_1$ by $\wt{a}_1=ua_1$, we get a representative of the claimed form.
\end{proof}
%More convenient formula???

\subsection{Constant term and bundles with split reduction}

%Let $K\sub G(\OO)$ be a normal open compact subgroup, and let $(g_i)_{1\le i\le r}$ be representatives of $G(\OO)/K$.
%Then for any $f\in \C(G(F)\backslash G(\A)/K)$, the corresponding constant term is determined by the operators
%$$C_{D,g_iK}(f)(t)=\int_{U(\A)/U(F)} f(utg_Dg_i)du=\vol(U_{t,K,D}/(U(F)\cap U_{t,K,D}))\cdot
%\sum_{u\in U(\A)/(U(F)\cdot U_{t,K,D})}f(utg_Dg_i),$$
%where $t\in T(\A)$, $U_{t,K,D}:=U(\A)\cap (tg_D)K(tg_D)^{-1}$.

Recall that for an open subgroup $K\sub G(\OO)$ the corresponding constant term on $\C(\Bun_G(C,K))$ is determined by the operators
$C_{D,g_iK}$, where $(g_iK)$ are all left $K$-cosets in $G(\OO)$, and $D\ge 0$ is a divisor on $\ov{C}$ (see \eqref{Const-term-D-K-eq}).
Furthermore, if $K$ is a normal subgroup in $C^i_P$ then we have
$$C_{D,rK}f(t)=\vol(U_{t,K,D}/(U(F)\cap U_{t,K,D}))\cdot
\sum_{u\in U(\A)/(U(F)\cdot U_{t,K,D})}f(utg_Dr),$$
where $U_{t,K,D}:=U(\A)\cap (tg_D)K(tg_D)^{-1}$.

\begin{lemma}\label{adelic-const-term-ex-lem}
(i) Let $a_1,a_2$ be a pair of ideles such that 
\begin{equation}\label{a0-assumption-eq}
\ov{\A}=\ov{F}+\ov{a}_1^{-1}\ov{a}_2\ov{\OO}
%\ov{F}+a_0^{-1}\ov{a}_1^{-1}\ov{a}_2\ov{\OO}.
=\ov{F}+a_0\ov{a}_1^{-1}\ov{a}_2\ov{\OO}.
\end{equation}
Set $t=\left(\begin{matrix} a_2 & 0 \\ 0 & a_1\end{matrix}\right)$.
Then for $f\in \C(G(F)\backslash G(\A)/G(\OO))$, one has
\begin{equation}\label{E-D-split-formula}
\begin{array}{l}
C_D(f)(t)=\vol(U_{t,D}/(U(F)\cap U_{t,D}))\times 
\\ \sum_{x\in a_1^{-1}a_2\ov{\OO}/(\ov{F}\cap a_1^{-1}a_2\ov{\OO}+a_1^{-1}a_2f_D\ov{\OO})} 
f(\left(\begin{matrix} a_1(1-\eps f_D^{-1}a_1a_2^{-1}x) & \eps f_D^{-1}a_1 \\ 0 & a_2(1+\eps f_D^{-1}a_1a_2^{-1}x)\end{matrix}\right)).
\end{array}
\end{equation}

\noindent
(ii) Let $K\sub G(\OO)$ be a normal open compact subgroup. Then there exists an effective divisor $D_0$ on $C$ (depending on $K$) such that for any pair of ideles $a_1,a_2$
with 
\begin{equation}\label{a0-D0-assumption-eq}
\ov{\A}=\ov{F}+\ov{a}_1^{-1}\ov{a}_2f_{D_0}\ov{\OO}
=\ov{F}+a_0\ov{a}_1^{-1}\ov{a}_2f_{D_0}\ov{\OO},
\end{equation} 
and any $f\in  \C(G(F)\backslash G(\A)/K)$, one has
\begin{equation}\label{E-D-i-split-formula}
\begin{array}{l}
C_{D,sg_iK}(f)(t)=c(t,K,D)\times 
\\ \sum_{x\in a_1^{-1}a_2f_{D_0}\ov{\OO}/(\ov{F}\cap a_1^{-1}a_2f_{D_0}\ov{\OO}+a_1^{-1}a_2f_Df_{D_0}\ov{\OO})} 
f(\left(\begin{matrix} a_1(1-\eps f_D^{-1}a_1a_2^{-1}x) & \eps f_D^{-1}a_1 \\ 0 & a_2(1+\eps f_D^{-1}a_1a_2^{-1}x)\end{matrix}\right)g_i),
\end{array}
\end{equation}
where $t=\left(\begin{matrix} a_2 & 0 \\ 0 & a_1\end{matrix}\right)$, $s=\left(\begin{matrix} 0 & 1 \\ 1 & 0\end{matrix}\right)$, and
$c(t,K,D)\in\C^*$ is a constant.
\end{lemma}

\begin{proof}
(i) 
By Lemma \ref{UtD-lem}, $\left(\begin{matrix} 1 & x \\ 0 & 1\end{matrix}\right)\in U_{t,D}$ if and only if $x\in a_1^{-1}a_2(\eps\OO+f_D\OO)$.
Hence,
$$C_D(f)(t)=\vol(U_{t,D}/(U(F)\cap U_{t,D}))\cdot\sum_{x\in \A_C/(F+a_1^{-1}a_2(\eps\OO+f_D\OO))} f(\left(\begin{matrix} 1 & x \\ 0 & 1\end{matrix}\right)\cdot t\cdot g_D).$$
Now
$$\left(\begin{matrix} 1 & x \\ 0 & 1\end{matrix}\right)\cdot t\cdot g_D=
\left(\begin{matrix} a_2+\eps f_D^{-1} a_1 x & x a_1\\ \eps f_D^{-1} a_1 & a_1\end{matrix}\right).$$

We claim that \eqref{a0-assumption-eq} implies that $\A_C=F+a_1^{-1}a_2\OO$. Indeed, one has to combine the inclusion
$$\eps\ov{\A}=\eps(a_0^{-1}\ov{F}+\ov{a}_1^{-1}\ov{a}_2\ov{\OO})=\eps_F \ov{F}+a_1^{-1}a_2\eps\ov{\OO}\sub F+a_1^{-1}a_2\OO$$
with $\ov{\A}= \ov{F}+\ov{a}_1^{-1}\ov{a}_2\ov{\OO}$. 
Hence, in the formula for $C_D(f)(t)$ we can restrict the summation to $x\in a_1^{-1}a_2\OO$ (modulo the intersection with $F+a_1^{-1}a_2(\eps\OO+f_D\OO)$).
Then we get
\begin{equation}\label{x-y-gD-eq1}
\left(\begin{matrix} 1 & x \\ 0 & 1\end{matrix}\right)\cdot t\cdot g_D\cdot\left(\begin{matrix} 1 & -a_1a_2^{-1}x \\ 0 & 1\end{matrix}\right)=
\left(\begin{matrix} a_2\a^{-1} & \eps y \\ \eps f_D^{-1}a_1 & a_1\a\end{matrix}\right),
\end{equation}
for some $y\in \ov{\A}$, where $\a=1-\eps f_D^{-1}a_1a_2^{-1}x$.
Using \eqref{a0-assumption-eq} we can write $y=a_0^{-1}\ov{a}_1\phi+\ov{a}_2b$ with $\phi\in \ov{F}$ and $b\in \ov{\OO}$. Hence, we get
$$\eps y=\eps_F \phi a_1+\eps a_2b,$$
and hence,
\begin{equation}\label{x-y-gD-eq2}
\left(\begin{matrix} 0 & 1 \\ 1 & 0\end{matrix}\right)\cdot \left(\begin{matrix} 1 & -\eps_F \phi \\ 0 & 1\end{matrix}\right)\cdot
\left(\begin{matrix} a_2\a^{-1} & \eps y \\ \eps f_D^{-1}a_1 & a_1\a\end{matrix}\right)\cdot
\left(\begin{matrix} 1 & -\eps b \\ 0 & 1\end{matrix}\right)\cdot \left(\begin{matrix} 0 & 1 \\ 1 & 0\end{matrix}\right)=
\left(\begin{matrix} a_1\a & \eps f_D^{-1}a_1 \\ 0 & a_2\a^{-1}\end{matrix}\right).
\end{equation}
%\left(\begin{matrix} a_2\a^{-1} & 0 \\ \eps f_D^{-1}a_1 & a_1\a\end{matrix}\right).$$
Hence,
$$f(\left(\begin{matrix} 1 & x \\ 0 & 1\end{matrix}\right)\cdot t\cdot g_D)
%=f\left(\begin{matrix} a_2\a^{-1} & 0 \\ \eps f_D^{-1}a_1 & a_1\a\end{matrix}\right)
=f(\left(\begin{matrix} a_1\a & \eps f_D^{-1}a_1 \\ 0 & a_2\a^{-1}\end{matrix}\right)).$$

\noindent
(ii) We make appropriate modifications in the proof of part (i).
There exists an effective divisor $D_0$ such that 
%\begin{itemize}
%\item
$$\left(\begin{matrix} 1+\eps a & b \\ 0 & 1+\eps c\end{matrix}\right)\in K$$ 
provided $a,b,c\in f_{D_0}\OO$.
%\item
%$s\cdot\left(\begin{matrix} 1 & \eps b \\ 0 & 1\end{matrix}\right)\cdot s\in K$ provided $b\in f_{D_0}\ov{\OO}$.
%\end{itemize}
It follows that for $x\in a_1^{-1}a_2f_{D_0}(f_D\OO+\eps \OO)$ one has $\left(\begin{matrix} 1 & x \\ 0 & 1\end{matrix}\right)\in U_{t,K,D}$. 
Hence, there exists a constant $c(t,K,D)$ such that
$$C_{D,g_iK}(f)=c(t,K,D)\sum_{x\in \A/(F+a_1^{-1}a_2f_{D_0}(f_D\OO+\eps \OO))}f(\left(\begin{matrix} 1 & x \\ 0 & 1\end{matrix}\right)tg_Dg_i).$$
Assume now that $a_1$ and $a_2$ satisfy \eqref{a0-D0-assumption-eq}. The similarly to (i) we deduce that 
$\A=F+a_1^{-1}a_2f_{D_0}\OO$. Hence, our summations is over 
$$\A/(F+a_1^{-1}a_2f_{D_0}(f_D\OO+\eps \OO)\simeq a_1^{-1}a_2f_{D_0}\OO/(a_1^{-1}a_2f_{D_0}\OO\cap F+a_1^{-1}a_2f_{D_0}(f_D\OO+\eps \OO)),$$

Now for $x\in a_1^{-1}a_2f_{D_0}\OO$, the equations
%repeating the manipulations in part (i) and 
\eqref{x-y-gD-eq1} and \eqref{x-y-gD-eq2} hold, with 
$$\left(\begin{matrix} 1 & -a_1a_2^{-1}x \\ 0 & 1\end{matrix}\right),\left(\begin{matrix} 1 & -\eps b \\ 0 & 1\end{matrix}\right)\in K$$
(since \eqref{a0-D0-assumption-eq} allows $b$ to be chosen in $f_{D_0}\ov{\OO}$).
Hence, we get
$$\left(\begin{matrix} 1 & x \\ 0 & 1\end{matrix}\right)\cdot t\cdot g_D\in G(F)\left(\begin{matrix} a_1\a & \eps f_D^{-1}a_1 \\ 0 & a_2\a^{-1}\end{matrix}\right)sK,$$
which gives the claimed formula (using normality of $K$ in $G(\OO)$).
\end{proof}

\begin{prop}\label{cusp-vanishing-thm} Recall that we denote by $N_0$ the line bundle on $\ov{C}$ corresponding to $\NN$ and we set $\ell=\deg(N_0)$.

\noindent
(i) Let $f$ be a cuspidal function on $\Bun(C)$. Then $f(V)=0$ for any vector bundle $V$ such that $\ov{V}\simeq L_0\oplus M_0$, where 
$$\deg(M_0)-\deg(L_0)\ge \max(2g+2+|\ell|,6g-\ell-1),$$
where $g$ is the genus of $\ov{C}$.

\noindent
(ii) Let $K\sub G(\OO)$ be an open compact subgroup, and let $\Bun_G(C,K)\to \Bun_G(C):g\mapsto V(g)$ be the natural projection. 
Then there exists a constant $N(K)>0$, such that $f(g)=0$ for any cuspidal function $f$ on $\Bun_G(C,K)$ and any $g$ such that $\ov{V(g)}\simeq L_0\oplus M_0$ where $\deg(M_0)-\deg(L_0)\ge N(K)$.
\end{prop}

\begin{proof} (i) Set $N=\deg(M_0)-\deg(L_0)$.
Let $V$ be a vector bundle such that $\ov{V}=L_0\oplus M_0$, where $\Ext^1(L_0,M_0)=0$. By Lemma \ref{split-red-representatives-lem}, 
if $N\ge 2g+2+|\ell|$ then we can find a representative
of the corresponding double coset of the form 
$$g=\left(\begin{matrix} a_1 & \eps\cdot a_1f_D^{-1} \\ 0 & a_2\end{matrix}\right),$$
%Let $E$ be the extension of $M_0$ by $L_0$ corresponding to $V$ (see Lemma \ref{split-reduction-lem}).
%First, we claim that if $E$ splits over $M_0(-D)$, where $\deg(D)\le N-(2g-1)$, then $f(V)=0$.
%Note that by Lemma \ref{split-red-representatives-lem}, this reduces to checking that 
%$$f(\left(\begin{matrix} a_1 & \eps\cdot a_1t_{D}^{-1} \\ 0 & a_2\end{matrix}\right))=0$$
%whenever $\deg(D)\le N-(2g-1)$, 
where $a_1$ and $a_2$ are ideles representing some line bundles reducing to $L_0$ and $M_0$, and $D\sub \ov{C}$ is an effective divisor of degree 
$\le g+(N-\ell+1)/2$.

We have 
$$\deg(L_0^{-1}M_0(-D))=N-\deg(D)\ge (N+\ell-1)/2-g\ge 2g-1,$$ 
so $H^1(L_0^{-1}M_0(-D))=0$. Hence, the restriction map
$$H^0(L_0^{-1}M_0)\to H^0(L_0^{-1}M_0|_D)$$
is surjective. Now we observe that the summation in formula \eqref{E-D-split-formula} is precisely over the cokernel of this map, hence, the sum reduces
to the single term with $x=0$. Since $f$ is cuspidal, this implies the vanishing
$$f(V)=f(\left(\begin{matrix} a_1 & \eps\cdot a_1f_D^{-1} \\ 0 & a_2\end{matrix}\right))=0.$$
%By Lemma \ref{split-degree-div-ext-lem}, any extension $E$ splits over $M_0(-D)$ for some $D$ of degree $\le \frac{N}{2}+g+1$.
%It remains to note that $\frac{N}{2}+g+1\le N-(2g-1)$ since $N\ge 6g$.

\noindent
(ii) We can assume that $K\sub G(\OO)$ is normal.
Let $(g_i)_{1\le i\le r}$ denote representatives of $G(\OO)/K$.
Set $N=\deg(M_0)-\deg(L_0)$. Assuming that $N\ge 2g+2+|\ell|$, as in part (i), we can find a representative of the form
$$g=\left(\begin{matrix} a_1 & \eps\cdot a_1f_D^{-1} \\ 0 & a_2\end{matrix}\right)g_i,$$
where $\deg(D)\le g+(N-\ell+1)/2$. Let us choose a divisor $D_0$ as in Lemma \ref{adelic-const-term-ex-lem}(ii) (it depends only on the subgroup $K$), and
let $\ov{D}_0$ be the corresponding divisor on $\ov{C}$.
We set 
$$N(K)=\max(2g+2+|\deg(N_0)|,2(3g-1+\deg(\ov{D}_0))-\ell+1).$$
Then for $N\ge N(K)$, we have
$$\deg(L_0^{-1}M_0(-D-\ov{D}_0))=N-\deg(D)-\deg(\ov{D}_0)\ge (N+\ell-1)/2-g-\deg(\ov{D}_0)\ge 2g-1.$$
Hence, $H^1(L_0^{-1}M_0(-D_0))=H^1(L_0^{-1}M_0(-D-D_0))=0$.
Therefore, we can apply formula \eqref{E-D-i-split-formula}. Furthermore, the summation in this formula is over the cokernel of the map
$$H^0(L_0^{-1}M_0(-D_0))\to H^0(L_0^{-1}M_0(-D_0)|_D).$$
The vanishing of $H^1(L_0^{-1}M_0(-D-D_0))$ implies that this map is surjective, so the sum reduces to the single term, and we get the required vanishing.
\end{proof}

\begin{cor}\label{finite-supp-cor} 
Every cuspidal function on $\Bun^L(C,K)$ has finite support.
\end{cor}
%Part 1 of Theorem C follows from the following result.

Now we will prove part (1) of Theorem C.
%, except for the upper bound on the dimension of $\SS_{\cusp}(\Bun^L(C))$.

\begin{theorem}\label{cusp-fin-dim-thm}
For any open compact subgroup $K\sub G(\A)$, the space $\C_{\cusp}(\Bun^L(C,K))$ is finite dimensional.
%(ii)  Assume that $\deg(N_0)=0$. Then there exists a function $c(g)$ such that
%$$\dim \SS^1_{\cusp}(\Bun^L(C))\ge q^{6g-6}(1-c(g)q^{-1/2}).$$
\end{theorem}

\begin{proof} 
Without loss of generality we can assume that $K\sub G(\OO)$.
%For $K=G(\OO)$, this follows from Proposition \ref{cusp-vanishing-thm} and Lemma \ref{reduction-fibre-lem}. To prove this for general $K$,
Note that the projection $\Bun^L(C,K)\to \Bun^L(C)$ has finite fibers since $G(\OO)/K$ is finite, while the projection $\Bun^L(C)\to \Bun^{\ov{L}}(\ov{C})$
has finite fibers by Lemma \ref{reduction-fibre-lem}. Now Proposition \ref{cusp-vanishing-thm} implies that any cuspidal function on $\Bun^L(C,K)$ is
supported on the preimage of a finite subset in $\Bun^{\ov{L}}(\ov{C})$, and the assertion follows.
%For each unstable bundle $\ov{F}$ on $\ov{C}$, let $\ov{M}(\ov{F})$ denote the uniquely defined line bundle of minimal degree $<\mu(\ov{F})$, 
%such that there exists a surjection $\ov{F}\to \ov{M}(\ov{F})$.
%It is enough to show that there exists a constant $N$ (depending on the degree of $\LL$ and on the genus $g$ of $\ov{C}$), such that $\deg(\ov{M}(\ov{F}))\ge N$ for any $F$ in the support
%of any $f\in \SS_{\cusp}(\Bun_\LL(C))$. Indeed, we observe that if $\deg(\ov{M}(\ov{F}))\ll 0$ then 
%$$\ov{F}\simeq \ov{\LL}\ot \ov{M}(\ov{F})^{-1}\oplus \ov
\end{proof}

\subsection{Cuspidality $=$ weak Hecke-boundedness}\label{cusp-hb-sec}

For a rank $2$ bundle $\ov{V}$ on $\ov{C}$ we denote by $\mu(\ov{V})=\deg(\ov{V})/2$ its slope and by
$\phi_-(\ov{V})$ the smallest integer $n$ such that there exists a surjection $\ov{V}\to \ov{L}$,
where $\ov{L}$ is a line bundle of degree $n$ on $\ov{C}$. Furthermore, we set
$$b(\ov{V}):=\mu(\ov{V})-\phi_-(\ov{V}).$$
Note that this quantity is unchanged by tensoring with a line bundle.

\begin{definition}\label{weakly-Hecke-b-def}
(i) Let $\Si\sub \ov{C}$ be a finite set of points, and let $K\sub G(\OO)$ be an open compact subgroup unramified away from $\Si$, so
$$K=K_\Si\times\prod_{p\not\in\Si} G(\OO_p).$$
A function $f\in \C(\Bun(C,K))$ is called {\it weakly Hecke-bounded away from $\Si$} if there exists a finite set $S\sub \Bun(C,K)$, such that for every 
collection of simple divisors $c_1,\ldots,c_n$ on $C$ (not necessarily distinct) with $\ov{c}_i\not\in \Si$,
the support of $T_{c_1}\ldots T_{c_n}f$ is contained in $S\ot \Pic_K(C)$, where
$$\Pic_K(C)=\A^*/(F^*\cdot (Z(\A)\cap K)),$$
where $\G_m\simeq Z\sub G=\GL_2$ is the center.

\noindent
(ii) We say that a function $f\in \C(\Bun(C,K)$ is {\it weakly Hecke-bounded} if there exists some finite $\Si\sub \ov{C}$ such that $K$ is unramified away from $\Si$ and
$f$ is weakly Hecke-bounded away from $\Si$.
\end{definition}

\begin{lemma}\label{Hecke-b-lem} Let $\Si\sub \ov{C}$ be a finite set of points, and let $K\sub G(\OO)$ be an open compact subgroup unramified outside $\Si$.
For $x\in \Bun(C,K)$ we denote by $V_x\in \Bun(C)$ the vector bundle associated with $x$, and by $\ov{V}_x\in \Bun(\ov{C})$ its reduction to $\ov{C}$.

\noindent
(i) For a constant $B>0$, let $S_B(K)$ denote the set of all $x\in \Bun(C,K)$ such that $b(\ov{V}_x)\le B$. Then 
there exists a finite set $S\sub \Bun(C,K)$ such that  
$$S_B(K)\sub \Pic_K(C)\ot S.$$

\noindent
(ii) Fix $L\in \Pic(C)$.
A function $f\in \C(\Bun^L(C,K))$ is weakly Hecke-bounded away from $\Si$ if and only if there exists $B>0$ such that $b(\ov{V}_x)\le B$
%$$\mu(\ov{V}_x)-\phi_-(\ov{V}_x)\le B$$
for every $x$ in the support of $T_{c_1}\ldots T_{c_n}f$,
for every collection of simple divisors $c_1,\ldots,c_n$ on $C$ with $\ov{c}_i\not\in\Si$. 
%It is easy to see that such a function automatically has finite support.
%We denote the subspace of Hecke-bounded functions as $\SS_b(\Bun_\LL(C))\sub \SS(\Bun_\LL(C))$.
\end{lemma}

\begin{proof} (i) First, we claim that for any $L\in\Pic(C)$,
the set $S_B(L,K)$ of all $x\in \Bun^L(C,K)$ such that $b(\ov{V}_x)\le B$ is finite. 
Indeed, this follows immediately from the fact that the set of $\ov{V}\in\Bun^{\ov{L}}(\ov{C})$ such that $b(\ov{V})\le B$ is finite 
(where $\ov{L}\in \Pic(\ov{C})$ is the reduction of $L$), and the fibers of $\Bun(C,K)\to \Bun(C)\to \Bun(\ov{C})$ are finite. 

Now let $L_1,\ldots,L_m$ be representatives of $\Pic(C)/2\Pic(C)$. Then we claim that we can take
$$S:=S_B(L_1,K)\cup\ldots\cup S_B(L_m,K),$$
which is finite by part (i).
Indeed, it suffices to show that any $x\in \Bun(C,K)$ with $b(x)\le B$ is contained in
$\Pic_K(C)\ot S$. But we can find $M\in \Pic(C)$ such that $\det(V_x)\simeq M^{-2}\ot L_i$, so $M\ot V_x\in \Bun^{L_i}(C)$.
Hence, lifting $M$ to $\wt{M}\in\Pic_K(C)$, we get $\wt{M}\ot x\in \Bun^{L_i}(C,K)$, so $\wt{M}\ot x\in S_B(L_i,K)$.
It follows that $x\in \wt{M}^{-1}\ot S_B(L_i,K)\sub \Pic_K(C)\ot S$.

\noindent
(ii) The ``if" part follows immediately from part (i).  

Now assume that $f$ is weakly Hecke-bounded away from $\Si$, and let $S\sub \Bun(C,K)$ be a finite set, such that the support of 
$T_{c_1}\ldots T_{c_n}f$ is in $\Pic_K(C)\ot S$. 
For $x\in \Bun(C,K)$, set $b(x)=b(\ov{V}_x)$.
Then $b(x)$ does not change under tensoring with $\Pic_K(C)$.
%$$\phi_-(S):=\min_{x\in S} \phi_-(\ov{V}_x),$$
%$$\mu_+(S):=\max_{x\in S} \mu(\ov{V}_x),$$
%$$B=\mu_+(S)-\phi_-(S).$$
%we have $\phi_-(L_0\ot V_0)=\phi_-(V_0)+\deg(L_0)$,
%$\mu(L_0\ot V_0)=\mu(V_0)+\deg(L_0)$.
Let us set
$$B:=\max_{x\in S}b(x).$$
Then for every $x\in S$, we have $b(x)\le B$. Hence, the same is true for $x\in \Pic_K(C)\ot S$. This prove the ``only if" part.
\end{proof}
%It is clear from this definition that the subspace of Hecke-bounded functions
%$$\SS_b(\Bun(C)):=\bigoplus_{\LL} \SS_b(\Bun_\LL(C))\sub \SS(\Bun(C))$$ 
%is preserved by Hecke operators $(T_c)$.
%for any point $c$ one has
%$$T_c\SS_{\cusp}(\Bun_\LL(C))\sub \SS_{\cusp}(\Bun_{\LL(-c)}(C)), \ \ T'_c\SS_{\cusp}(\Bun_{\LL(-c)}(C))\sub \SS_{\cusp}(\Bun_\LL(C)).$$

%Let us define the duality operator $\Dual:\SS(\Bun(C))\to \SS(\Bun(C))$ by
%$$\Dual f(V)=f(V^\vee),$$
%also consider another type of Hecke operator $T'_D$ given by
%$$T'_D=\Dual\circ T_D\circ \Dual.$$
%Note that $\Dual$ acts from $\SS(\Bun_\LL(C))$ to $\SS(\Bun_{\LL^{-1}}(C))$, and
%$T'_D$ acts from $\SS(\Bun_\LL(C))$ to $\SS(\Bun_{\LL(D)}(C))$.

%\bigskip

%\noindent
%{\bf Conjecture}. {\it For each line bundle $\LL$ on $C$, the space $\SS_b(\Bun_\LL(C))$ is finite-dimensional. There exists a function $c(g)>0$ such that
%$$\dim \SS_b(\Bun_\LL(C))\ge q^{n(3g-3)}[1-c(g)q^{-1/2}],$$
%where $g$ is the genus of $\ov{C}$ and $q=|k|$.
%\noindent
%2. One has $\Dual \SS_{\cusp}(\Bun_\LL(C))\sub\SS_{\cusp}(\Bun_{\LL^{-1}}(C))$.
%\noindent
%3. All the operators $T_D$ and $T'_D$ commute.
%}

%\bigskip

Now we can prove the following strengthening of part (2) of Theorem C.

\begin{prop}\label{hecke-bound-prop} 
Assume $\ov{C}$ has genus $g\ge 2$. Let $K\sub \GL_2(\OO)$ be open, and let $\Si\sub \ov{C}$ be a finite set such that $K$ is unramified away from $\Si$.
Let also $L$ be a line bundle on $C$.
Then for a function $f$ on $\Bun^{L}(C,K)$ the following are equivalent:
\begin{enumerate}
\item $f$ is cuspidal;
\item $f$ is weakly Hecke-bounded away from $\Si$;
\item $f$ is Hecke-bounded.
\end{enumerate}
% if and only if it is weakly Hecke-bounded. 
%Assume that $\ov{C}(k)\neq\emptyset$.
%For each function $f\in \SS(\Bun_\LL(C))$, let us denote by $T'-\sspan(f)$ the subspace in $\SS(\Bun(C))$ generated by all iterated Hecke transforms $T'_{c_1}\ldots T'_{c_n}f$,
%where $c_i\in C(A)$. Then 
%One has 
%$$\C_{\cusp}(\Bun^L(C))=\SS_b(\Bun^L(C)).$$
%\{f\in \SS(\Bun^L(C)) \ | \ \exists N>0 \text{ such that } \forall D: \ 
%E_{L(-D)}T_D(f)\in F_{-\frac{\deg D}{2}-N} \}.$$
\end{prop}

\begin{proof} (1)$\implies$(3). 
By Proposition \ref{cusp-vanishing-thm}(ii), there exists a constant $B(K)$, such that for any $x\in \Bun_G(C,K)$ in the support of a cuspidal function on $\Bun_G(C,K)$,
one has $b(\ov{V}_x)\le B(K)$. Hence, by Lemma \ref{Hecke-b-lem}(i), the set of all such $x$ is contained in $\Pic_K(C)\ot S$
for a fixed finite set $S\sub \Bun(C,K)$. Since the action of $\HH_{G,C,K}$ preserves cuspidality (by Proposition A), this shows that any cuspidal function is Hecke bounded.

\noindent
(3)$\implies$(2). This is clear.
 
%Assume that $f$ is cuspidal. Then, since Hecke operators $T_c$ preserve $\C_{\cusp}(\Bun(C))$ (by Proposition A), 
%we have $T_{c_1}\ldots T_{c_n}(f)\in \SS_{\cusp}(\Bun^L(-c_1-\ldots-c_n)(C))$, and the assertion follows 
%from Proposition \ref{cusp-vanishing-thm}. 

\noindent
(2)$\implies$(1).
Assume that $f$ is weakly bounded away from $\Si$. Let $W$ denote the subspace of functions on $\Bun_G(C,K)$, spanned by $T_{c_1}\ldots T_{c_n}f$, where
$\ov{c}_i\not\in\Si$. Let $S\sub \Bun_G(C,K)$ be a finite set such that $\supp(w)\sub S\ot \Pic_K(C)$ for any $w\in W$, and let $B=\max_{x\in S} b(\ov{V}_x)$,
so 
$$b(\ov{V}_x)\le B \ \text{ for any } x\in\supp(w), \ \text{ with } w\in W.$$

Let us choose representatives $r_i$ of $G(\OO)/K$, with components equal to $1$ outside $\Si$. Recall that the constant term of $w\in W$ is determined by the values
of $C_{D,r_iK}w(t)$, where $t\in T(\A)$, given by \eqref{Const-term-D-K-eq}.

For $t=\diag(a_1,a_2)$, set $\de(t)=\deg(\ov{a}_1)-\deg(\ov{a}_2)$. If $C_{D,r_iK}w(t)\neq 0$ for some $t=\diag(a_1,a_2)$, then there exists $x\in \supp(w)$, such that
the reduced vector bundle $\ov{V}_x$ on $\ov{C}$ fits into an exact sequence
$$0\to L_0\to \ov{V}_x\to M_0\to 0$$
with line bundles $L_0$ and $M_0$ corresponding to the ideles $\ov{a}_1$ and $\ov{a}_2$. Hence, in this case
$$\de(t)=\deg(L_0)-\deg(M_0)\le 2b(\ov{V}_x)\le 2B.$$

Assume $f$ is not cuspidal. Then the set of integers 
$$\{\de(t) \ |\ t\in \supp(C_{D,r_iK}w) \text{ for some } w\in W, D r_i\}$$
is nonempty and bounded above. Choose $w,D,r_i,t=\diag(a,b)$ such that $\de(t)=\deg(a)-\deg(b)$ is maximal.

Pick a closed point $p\not\in\Si\cup\supp(D)$, and let $c$ be a simple divisor such that $\ov{c}=p$.
Set  $t';=\diag(a,bf_c)$, where $f_c\in\OO_p$ is an equation of $c$. Then the case $p\not\in\supp(D)$ of Lemma \ref{hecke-qbun-formula-lem} gives
$$C_{D,r_iK}T_cf(t')=|A(c)|\cdot C_{D,rk}f(t)+\sum_j C_{D_j,rk}f(\diag(af_c^{-1},bf_cu_j)),$$
for some $u_j\in \OO^*_p$. Since $\de(\diag(af_c^{-1},bf_cu_j))=\de(t)+2\deg(p)>\de(t)$ and $\de(t')=\de(t)+\deg(p)>\de(t)$, by maximality of $\de(t)$, the 
left-hand side and the sum in the right-hand side vanish. Hence, $C_{D,rkf}(t)$ should also a vanish, which is a contradiction.
Therefore, $f$ is in fact cuspidal.
\end{proof}

%Let $(\LL_1,\ldots,\LL_n)$ be the set of representatives in $\Pic(C)/2\Pic(C)$.

%\begin{lemma}
%The space $\SS^{\inv}_{\cusp}(\Bun(C))$ is finite-dimensional.
%\end{lemma}

\section{Cuspidal functions for $\PGL_2$: dimension estimates and Hecke eigenfunctions}\label{cusp-PGL2-sec}

Throughout this section we work with rank $2$ bundles and we assume that the characteristic of $k$ is $\neq 2$.
In Sec.\ \ref{CorC-sec}, we will prove Corollary D on equivalence of cuspidality with Hecke finiteness for $\PGL_2$ (using Theorem C proved in Section \ref{cusp-GL2-sec}).
Using this we will prove that $\SS_\cusp(G(F)\backslash G(\A))$ coincides with the subspace of finitary functions described in Section \ref{finitary-sec}.
Next, we will give an explicit description of the space $\C_{\cusp}(G(F)\backslash G(\A))$ of cuspidal functions in $\SS(G(F)\backslash G(\A))$ for $G=\PGL_2$
(see Theorem \ref{main-dec-thm}).
%This shows that cuspidal functions are finitary. 
%To prove the converse we show directly that the finitary pieces of the orbit
%decomposition of $\SS(G(F)\backslash G(\A))$ are cuspidal. 
Finally, in Sec.\ \ref{ThmE-sec} we will prove Theorems E and F. For both results the crucial role is played by the relation with the moduli space of Higgs bundles established in Prop. \ref{Higgs-prop}. 
%We use the identification $\fg^\vee\simeq \fg$ given by the Killing form.

\subsection{Cuspidality and Hecke-finiteness}\label{CorC-sec}

%\begin{cor}\label{PGL2-cor}
%The space $\SS_{\cusp}(\Bun_{\PGL_2}(C))$ is finite-dimensional. 
%\end{cor}

The relation between cuspidality and Hecke-finiteness (resp., Hecke-boundedness) for $\PGL_2$ is an easy consequence of the relation between cuspidality and 
weak Hecke boundedness for $\GL_2$ established in Sec.\ \ref{cusp-GL2-sec}.
%Now let us go back to the case nilpotent extension $C$ of $\ov{C}$ of length $2$.

\medskip

\noindent
{\it Proof of Corollary D(1)}.
%\begin[Proof of Corollary D]{proof}
%One can deduce from finite-dimensionality of $\SS_{\cusp}(\Bun_{\GL_2}^L(C))$ for every line bundle $L$ that for the group $\PGL_2$ Indeed,
Let $K\sub \GL_2(\OO)$ be a compact open subgroup, $\ov{K}\sub \PGL_2(\A)$ its image.
It is easy to see that the embedding $\pi^*\C(\Bun_{\PGL_2}(C,\ov{K}))\hra \C(\Bun_{\GL_2}(C,K))$,
associated with the natural projection $\pi:\Bun_{\GL_2}(C,K)\to \Bun_{\PGL_2}(C,\ov{K})$ has the property that $f\in \C(\Bun_{\PGL_2}(C,\ov{K}))$ is cuspidal if and only if $\pi^*f$ is cuspidal.
Furthermore, its image is contained in functions invariant under tensoring with line bundles. 

Let $Z\sub \GL_2$ denote the center (isomorphic to $\G_m$). Then the image of the composed map 
$$Z(\A)/(Z(F)\cdot Z(\A)\cap K)\to \Bun_{GL_2}(C,K)\to \Pic(C)$$
is exactly $2\Pic(C)$, which has finite index in $\Pic(C)$. 
Hence, there exist
finitely many line bundles $L_1,\ldots,L_s$, such that the pull-back of a function on $\Bun_{\PGL_2}(C,\ov{K})$ is determined by its restrictions
to $\Bun_{\GL_2}^{L_i}(C,K)$, $i=1,\ldots,s$. 

Note that a function on $\Bun_{\GL_2}(C,K)$ is cuspidal if and only if its restriction to each $\Bun_{\GL_2}^L(C,K)$
is cuspidal. Indeed, this is clear from the form of the constant term operator: $Ef(t)$ depends only on values $f(g)$ with $\det(g)=\det(t)$.
Thus, we get an embedding 
$$V(C,\ov{K})=\C_{\cusp}(\Bun_{\PGL_2}(C,\ov{K}))\hra \bigoplus_{i=1}^s \C_{\cusp}(\Bun_{\GL_2}^{L_i}(C,K))\sub \bigoplus_{i=1}^s \SS(\Bun_{\GL_2}^{L_i}(C,K))$$
(the last inclusion is due to Corollary \ref{finite-supp-cor}), which implies that $V(C,\ov{K})$ is finite-dimensional and is contained in $\SS(\Bun_{\PGL_2}(C,\ov{K}))$.
\qed

\begin{definition}\label{PGL2-weakly-Hecke-b-def} 
Let $\Si\sub \ov{C}$ be a finite set of points, and let $K\sub \PGL_2(\OO)$ be an open compact subgroup unramified away from $\Si$, so
$$K=K_\Si\times\prod_{p\not\in\Si} \PGL_2(\OO_p).$$
A function $f\in \C(\Bun_{\PGL_2}(C,K))$ is called {\it weakly Hecke-bounded away from $\Si$} if there exists a finite set $S\sub \Bun_{\PGL_2}(C,K)$, such that for every 
collection of simple divisors $c_1,\ldots,c_n$ on $C$ with $\ov{c}_i\not\in \Si$,
the support of $T_{c_1}\ldots T_{c_n}f$ is contained in $S$.
\end{definition}

We have the following analog of Proposition \ref{hecke-bound-prop}, which subsumes Corollary D(2).

\begin{prop}\label{PGL2-hecke-bound-prop} 
Assume $\ov{C}$ has genus $g\ge 2$. Let $K\sub \PGL_2(\OO)$ be open, and let $\Si\sub \ov{C}$ be a finite set such that $K$ is unramified away from $\Si$.
Then for a function $f$ on $\Bun_{\PGL_2}(C,K)$ the following are equivalent:
\begin{enumerate}
\item $f$ is cuspidal;
\item $f$ is weakly Hecke-bounded away from $\Si$;
\item the space spanned by $T_{c_1}\ldots T_{c_n}f$, where $\ov{c}_i\not\in \Si$, is finite-dimensional;
\item $f$ is Hecke-bounded;
\item the space $\HH_{\PGL_2,C,K}\cdot f$ is finite-dimensional.
\end{enumerate}
\end{prop}

\begin{proof}
(1)$\implies$(5) follows from the fact that the space $V(C,K)$ is finite-dimensional, and is preserved by $\HH_{\PGL_2,C,K}$.

Since every cuspidal function has finite support (by Corollary D(1)), we see that (2)$\iff$(3) and (4)$\iff$(5).
It is clear that (4)$\implies$(2).

Thus, it remains to show that (2)$\implies$(1), so let $f$ be a weakly Hecke-bounded function on $\Bun_{\PGL_2}(C,K)$.
Then the pull-back $\pi^*f$ to $\Bun_{\GL_2}(C,K)$ has the property that all functions $T_{c_1}\ldots T_{c_n}(\pi^*f)$, where $\ov{c}_i\not\in\Si$ are supported on
$\pi^{-1}(S)$, where $S\sub \Bun_{\PGL_2}(C)$ is a fixed finite set. This easily implies that there exists a finite set $\wt{S}\sub \Bun_{\GL_2}(C,K)$,
such that all functions $T_{c_1}\ldots T_{c_n}(\pi^*f)$ are supported on $S\ot \Pic_K(C)$, so $\pi^*f$ is weakly Hecke-bounded away from $\Si$.
By Proposition \ref{hecke-bound-prop}, this implies that $\pi^*f$ is cuspidal. Therefore, $f$ is cuspidal. 
\end{proof}

%Recall that a smooth representation $V$ of $G(\A)$ is called admissible if $V^K$ is finite dimensional for every compact open subgroup $K\sub G(\A)$.
%Since the functors $V\mapsto V^K$ are exact, any subquotient of an admissible representation is admissible.

\begin{definition}
We define the subspace of {\it cuspidal} functions $\SS_{\cusp}(G(F)\backslash G(\A))\sub \SS(G(F)\backslash G(\A))$
by the condition
$$\int_{u\in U(\A)/U(F)}f(ug)du=0.$$
%\noindent
%(ii) We say that a function $f\in \SS(G(F)\backslash G(\A))$ is {\it $G(\A)$-finite} if it is contained in an admissible $G(\A)$-subrepresentation of $\SS(G(F)\backslash G(\A))$. 
\end{definition}

Recall that we call a function $f\in\SS(G(F)\backslash G(\A))$ {\it finitary} if it is contained in an admissible $G(\A)$-subrepresentation.
We deduce the following corollary from Proposition \ref{PGL2-hecke-bound-prop}.

\begin{cor}\label{cusp-adm-cor} 
Let $f$ be in $\SS(G(F)\backslash G(\A))$. 
Then $f$ is cuspidal if and only if it is finitary.
\end{cor}

\begin{proof} By Corollary D(1), the subspace $\SS_{\cusp}(G(F)\backslash G(\A))$ is an admissible $G(\A)$-representation. Hence, any cuspidal function is finitary.
Conversely, suppose $f$ is contained in an admissible $G(\A)$-subrepresentation $V$. Let $K\sub G(\OO)$ be such that $f\in V^K$.
Then $\HH_{\PGL_2,C,K}\cdot f\sub V^K$, so it is finite-dimensional. Hence, $f$ is cuspidal by Proposition \ref{PGL2-hecke-bound-prop}.
\end{proof}

%\noindent
%(ii) If $f$ is admissible and $G(\OO)$-invariant then it is cuspidal.

%\noindent
%(ii) In this case $f$ is contained in $\SS(\Bun_{PGL_2}(C))$, so it is cuspidal by Corollary D.

\subsection{Orbit decomposition and strongly cuspidal functions}\label{str-cusp-sec}
From now on $G=\PGL_2$.
Since the characteristic of $k$ is $\neq 2$ every element of $\fg(\ov{F})$ with nonzero determinant is regular semisimple.
Thus, in addition to the zero orbit $\Om_0$ and the orbit $\Om_n$ of a nilpotent element,
we have orbits in $\fg\ot (L^{-1}\om_{\ov{C}})(\ov{F})$ parametrized by nonzero values of the determinant in the $1$-dimensional $\ov{F}$-vector space
$$H:=L^{-2}\om_{\ov{C}}^{\ot 2}(\ov{F}).$$
For $d\in H\setminus \{0\}$, we denote by $\Om_d$ the orbit of elements $x$ such that $\det(x)=-d$.
Note that $d\neq 0$ is a square if and only if $\Om_d$ has a diagonal representative.
Let us denote by $H'\sub H$ the elements which are not squares.

\begin{definition}
Let us define the subspace of {\it strongly cuspidal} functions 
$\SS_{\strcusp}(G(F)\backslash G(\A))\sub \SS(G(F)\backslash G(\A))$ by the condition
$$\int_{u\in U(\NN\A)/U(\NN_F)}f(ug)du=0.$$
\end{definition}

Note that the subspace $\SS_{\strcusp}(G(F)\backslash G(\A))$ is preserved by $G(\A)$ and is contained in $\SS_{\cusp}(G(F)\backslash G(\A))$.

%Let $\SS_{\strcusp}(G(F)\backslash G(\A))\sub \SS(G(F)\backslash G(\A))$ denote the subspace of strongly cuspidal functions.

\begin{prop}\label{str-cusp-prop}
One has
$$\SS_{\strcusp}(G(F)\backslash G(\A))=\bigoplus_{d\in H'}\SS(G(F)\backslash G(\A))_{\Om_d}.$$ 
\end{prop}

\begin{proof}
Note that since $(1+\fu(\NN\A))$ commutes with $N_\A=(1+\fg(\NN\A))$, the projectors $\Pi_\Om$ preserve the subspace of strongly cuspidal functions.
Thus, we have a decomposition
$$\SS_{\strcusp}(G(F)\backslash G(\A))=\bigoplus_{\Om}(\SS_{\strcusp}(G(F)\backslash G(\A))\cap\SS(G(F)\backslash G(\A))_{\Om}),$$

%It remains to note that for $f\in \SS(G(\ov{F})\backslash G(\ov{\A}))$ the integration over $(1+\eps \fu(\ov{\A})/\fu(\ov{F}))$ does not change
%$f$, so for such $f$ strong cuspidality implies that $f=0$.

Assume first that $\Om=\Om_d$, where $d\in H$ is not a square. Then we claim that any function $f$ in $\SS(G(F)\backslash G(\A))_{\Om}$ is strongly cuspidal.
It is enough to prove that one has
$$\int_{a\in \fu(\NN\A)/\fu(\NN_F)}\Pi_{\Om}f((1+a)g)=0.$$
We can rewrite this expression as
\begin{align*}
&\sum_{\eta\in \Om}\int_{X\in\fg(\NN\A)/\fg(\NN_F)}\int_{a\in \fu(\NN\A)/\fu(\NN_F)} \psi_\eta(-X)f((1+X+a)g)dadX=\\
&\sum_{\eta\in \Om}\int_{X\in\fg(\NN\A)/\fg(\NN_F)} \int_{a\in \fu(\NN\A)/\fu(\NN_F)}\psi_{\eta}(a)\psi_\eta(-X)f((1+X)g)dadX.
\end{align*}
Now we observe that $\int_{a\in \fu(\NN\A)/\fu(\NN_F)}\psi_{\eta}(a)da=0$ unless the restriction $\psi_{\eta}|_{\fu(\NN\A)}$ is trivial,
i.e., $\eta$ is upper-triangular. Since $d$ is not a square, this never happens, and our claim follows.

Next, assume that either $\Om=0$ or $\Om=\Om_n$ or $\Om=\Om_d$, where $d$ is a square. 
Then we claim that for any strongly cuspidal function $f\in \SS(G(F)\backslash G(\A))$, one has $\Pi_{\Om}(f)=0$.
Indeed, let $\eta\in \Om$ be an upper-triangular representative. Then for any $g_0\in G(F)$, we have
\begin{align*}
&\Pi_{g_0\eta}f(g)=\int_{\fg(\NN\A)/\fg(\NN_F)}\psi_{g_0\eta}(-X)f((1+X)g)dX=\\
&\int_{\fg(\NN\A)/\fg(\NN_F)}\psi_\eta(-\Ad(g_0^{-1})(X))f((1+X)g)dX=
\int_{\fg(\NN\A)/\fg(\NN_F)}\psi_\eta(-X)f(g_0(1+ X)g_0^{-1}g)dX=\\
&\int_{\fg(\NN\A)/(\fg(\NN_F)+\fu(\NN\A))}\int_{\fu(\NN\A)/\fu(\NN_F)}\psi_{\eta}(-X)f((1+X+a)g_0^{-1}g)dadX,
\end{align*}
where we used the fact that $\psi_\eta$ is trivial on $\fu(\NN\A)$. Now integrating over $a$ gives zero, since $f$ is strongly cuspidal. 
Hence, $\Pi_{g_0\eta}f(g)=0$, and so $\Pi_{\Om}f=0$.
\end{proof}

\subsection{Orbit/character decomposition and cuspidal functions}\label{orbit-cusp-sec}

%Recall that smooth irreducible representations of $G(F_p)$ are classified by pairs ???

%Every admissible irreducible representation of $G(\A)$ is a tensor product of local factors???

%\begin{lemma}\label{irr-global-lem}
%Let $\pi=\bigotimes_p \pi_p$
%be an admissible irreducible $G(\A)$-subrepresentation occurring as a subquotient in $\SS(G(F)\backslash G(\A))_{\Om}$, where $\Om\neq 0$.
%Then: 

%\noindent
%(i) For each $v$, the local parameter of $\pi_v$ is the orbit corresponding to $d$.
%Either $m_v=0$ for all $v$, or $m_v\neq 0$ for all $v$.

%\noindent
%(i) $\Om=\Om_d$, where $d\in H'$, i.e. $d$ is not a square.

%\noindent
%(ii) If $\pi^{G(\OO)}\neq 0$ then $\Om=G(\ov{F})\eta$, where $\eta\in H^0(\ov{C},\om_{\ov{C}}\ot \fg)$  is a regular Higgs field. 
%$\pi$ is contained in $\SS^1_{\cusp}(G(F)\backslash G(\A))$ if and only if $\eta$ is semisimple???
%\noindent
%(iv) $\pi$ is contained in $\SS(G(\ov{F})\backslash G(\ov{\A}))$ if and only if $\eta=0$???
%\end{lemma}

%\begin{proof}
%???
%\end{proof}

We can now combine our results to give a description of the subspace of cuspidal functions on $G(F)\backslash G(\A)$.

\begin{theorem}\label{main-dec-thm}
%A function $f\in \SS(G(F)\backslash G(\A))$ is cuspidal if and only if it is finitary.
One has a decomposition
\begin{equation}\label{main-cusp-dec}
\SS_{\cusp}(G(F)\backslash G(\A))
%=\SS_f(G(F)\backslash G(\A))
=\SS_{\cusp}(G(\ov{F})\backslash G(\ov{\A}))\oplus \bigoplus_{\a\neq 0} \wt{\SS}_{\eta_0,\chi_\a}\oplus \SS_{\strcusp}(G(F)\backslash G(\A)),
\end{equation}
where 
\begin{equation}\label{str-cusp-dec}
\SS_{\strcusp}(G(F)\backslash G(\A))=\bigoplus_{d\in H'}\SS(G(F)\backslash G(\A))_{\Om_d},
\end{equation}
%\begin{equation}\label{Om-d-dec}
%\SS(G(F)\backslash G(\A))_{\Om_d}=\bigoplus_{\chi\in\Pi_{\eta_d}}\wt{\SS}_{\eta_d,\chi}.
%\end{equation}
%Here $\eta_d$ is a representative of the orbit $\Om_d$ and $\Pi_{\eta_d}$ is the set of characters of the commutative compact group $\ov{H}_{\eta_d}$ extending $\psi_{\ov{C}}$ 
%(see Sec.\ \ref{reg-ell-sec}).
\end{theorem}

\begin{proof} This follows by combining Theorem \ref{fin-dec-thm}, Corollary \ref{cusp-adm-cor} and Prop.\ \ref{str-cusp-prop}.
% the right-hand side of \eqref{main-cusp-dec} is contained in $\SS_{\cusp}(G(F)\backslash G(\A))$, 
%so it remains to establish the opposite inclusion.
%Since the projector $\Pi_0$ preserves $\SS_{\cusp}(G(F)\backslash G(\A))$, we have a decomposition
%$$\SS_{\cusp}(G(F)\backslash G(\A))=\SS_{\cusp}(G(\ov{F})\backslash G(\ov{\A}))\oplus \SS_{\cusp}(G(\ov{F})\backslash G(\ov{\A}))\cap \SS(G(F)\backslash G(\A))$$
%where $\SS(G(\ov{F})\backslash G(\ov{\A}))_{\neq 0}=\bigoplus_{\Om\neq 0}\SS(G(\ov{F})\backslash G(\ov{\A}))_{\Om}$.
%Since every cuspidal function is $G(\A)$-finite (see Corollary \ref{cusp-adm-cor}), it is enough to prove that every admissible subsrepresentation
%$V\sub  \SS(G(\ov{F})\backslash G(\ov{\A}))_{\neq 0}$ is contained in the right-hand side of \eqref{main-cusp-dec}.
%By Lemma \ref{split-non-adm-lem}, we obtain $\Pi_\Om V=0$ for every split semisimple orbit $\Om$.
%It remains to prove that $\Pi_{\Om_n} V$ is contained in $\bigoplus_{\a\neq 0} \wt{\SS}_{\eta_0,\chi_\a}$. But this follows immediately from 
%Lemma \ref{nilp-adm-lem}(ii).
\end{proof}

%\begin{cor}\label{adm-cusp-cor}
%Let $f\in \SS(G(\ov{F})\backslash G(\ov{\A}))$. If $f$ is $G(\A)$-finite then it is cuspidal.
%\end{cor}

%\begin{proof}
%Let $V\sub \SS(G(F)\backslash G(\A))$ be an admissible subrepresentation. Then as in the proof of Theorem \ref{main-dec-thm}, we get
%$\Pi_\Om V=0$ for split semisimple $\Om$, and $\Pi_{\Om_n} V$ is contained in $\bigoplus_{\a\neq 0} \wt{\SS}_{\eta_0,\chi_\a}$.
%Also, by Prop. \ref{hecke-fin-cusp-class-prop}, $\Pi_0 V$ is contained in $\SS_{\cusp}(G(\ov{F})\backslash G(\ov{\A}))$. Hence, $V$ is contained in
%the right-hand-side of \eqref{main-cusp-dec}.
%\end{proof}

By passing to $K$-invariants in \eqref{main-cusp-dec}, where $K\sub G(\A)$ is an open compact subgroup, we obtain a decomposition of the space $V(C,K)$ of
cuspidal functions on $\Bun_{\PGL_2}(C,K)$.

\begin{cor}\label{main-dec-cor}
One has a direct sum decomposition preserved by the Hecke algebra $\HH_K$,
$$V(C,K)=V(\ov{C},\ov{K})\oplus V(C,K)_n\oplus V^1(C,K),$$
where $V^1(C,K)\sub V(C,K)$ is the space of strongly cuspidal functions and
$V(C,K)_n=\bigoplus_{\a\neq 0}\wt{\SS}_{\eta_0,\chi_\a}^K$.
\end{cor}

%\begin{lemma}\label{nilp-cusp-lem} 

\begin{remark}
It is instructive to see directly that for any nonzero $\a\in \om_{\ov{C}}(\ov{F})$, the subrepresentation $\wt{\SS}_{\eta_0,\chi_\a}\sub\SS(G(F)\backslash G(\A))_{\Om_n}$
is cuspidal, i.e., contained in $\SS_{\cusp}(G(F)\backslash G(\A))$.
Recall that an embedding of $\wt{\SS}_{\eta_0,\chi_\a}$ into $\SS(G(F)\backslash G(\A))$ is given by the map 
$$\kappa=\kappa_{\eta_0}: f\mapsto \kappa f(g)=\sum_{\ga\in U(F)\backslash G(F)} f(\ga g).$$
For any $f\in \wt{\SS}_{\eta,\chi_\a}$ one has
$$\int_{N_\A\cap U(\A)/(N_F\cap U(F))}(\kappa f)(ug) du=\sum_{\ga\in U(F)\backslash G(F)}\int_{N_\A\cap U(\A)/(N_F\cap U(F))}f(\ga ug)du.$$
Note that if $\Ad(\ga^{-1})\eta_0$ is not upper-triangular, then the restriction $\psi_{\Ad(\ga^{-1})\eta_0}|_{N_\A\cap U(\A)}$ is nontrivial.
Hence in this case we get
$$\int_{N_\A\cap U(\A)/(N_F\cap U(F))}f(\ga ug)du=\int \psi_{\eta_0}(\ga u\ga^{-1})f(\ga g)du=0.$$
%and therefore the integral over $U(\A)/U(F)$ of such term is also zero.
Note that $\Ad(\ga^{-1})(\eta_0)$ is upper-triangular if and only if $\ga$ itself is upper-triangular. Such 
%Thus, in the above sum we can restrict to $\ga$ such that .
%But then 
$\ga$ normalize $U(\A)$, so for $\ga\in B(F)$ we have
$$\int_{U(\A)/U(F)}f(\ga ug)du=\int_{U(\A)/U(F)}\chi_\a(\ga u\ga^{-1})f(\ga g)du=0,$$
since the character $\chi_\a$ is nontrivial. Therefore,
\begin{align*}
&\int_{U(\A)/U(F)}(\kappa f)(ug) du=\int_{U(\A)/(N_\A\cap U(\A)} d\ov{u} \int_{N_\A\cap U(\A)/(N_F\cap U(F))}(\kappa f)(ug) du\\
&=\int_{U(\A)/(N_\A\cap U(\A)}d\ov{u} \sum_{\ga\in U(F)\backslash B(F)}\int_{N_\A\cap U(\A)/(N_F\cap U(F))}f(\ga ug)du\\
&=\sum_{\ga\in U(F)\backslash B(F)}\int_{U(\A)/U(F)}f(\ga ug)du=0.
\end{align*}
\end{remark}

\subsection{Proofs of Theorems E and F}\label{ThmE-sec}

Recall that we are studying the spaces 
$$V(C):=\SS_{\cusp}(G(F)\backslash G(\A)/G(\OO))  \text{ and  } V^1(C):=\SS_{\strcusp}(G(F)\backslash G(\A)/G(\OO)),$$
where $G=\PGL_2$.

We start with determining the space $V^1(C)$ of strongly cuspidal functions. Note that for $PGL_2$ semisimple adjoint orbits are classified by the value of the determinant.
Thus, using \eqref{str-cusp-dec} and Lemma \ref{kappa-lem}, we can write 
$$V^1(C)=\bigoplus_{d\in H'}\wt{\SS}_{\Om_d}^{G(\OO)},$$
where $H'$ is the set of non-squares in $H=\om_{\ov{C}}^2L^{-2}(\ov{F})$ , and $\Om_d$ is the orbit corresponding to $d\in H'$.
Furthermore, by Proposition \ref{Higgs-prop}, we have an isomorphism
$$\wt{\SS}_{\Om_d}^{G(\OO)}\simeq \SS(\MM_d^{Higgs,L^{-1}}(\ov{C})(k),\LL_\psi),$$
where $\MM_d^{Higgs,L^{-1}}(\ov{C})$ is the moduli stack of $L^{-1}$-twisted Higgs $\PGL_2$-bundles $(V,\phi)$ with $-\det(\phi)=d$
(and $\LL_\psi$ is a certain $\C^*$-torsor).
Thus, $\wt{\SS}_{\Om_d}\neq 0$ only when $d$ comes from a global section $\a$ of 
$$A:=H^0(\ov{C},\om_{\ov{C}}^2L^{-2}),$$
and the condition $d\in H'$ is equivalent to the condition $\a\in A'\sub A$,
the complement to the image of map 
$$H^0(\ov{C},\om_{\ov{C}}L^{-1})\to A:\a\mapsto \a^2.$$
In other words, we have a decomposition
\begin{equation}\label{HF-decomposition-eq}
V^1(C)=\bigoplus_{\a\in A'} \SS(HF_\a(k),\LL_\psi),
\end{equation}
where $HF_\a$ is the Hitchin fiber over $\a$, i.e., the fiber of 
the Hitchin map 
\begin{equation}\label{Hitchin-map}
h:\MM^{Higgs,L^{-1}}(\ov{C})\to A:=H^0(\ov{C},\om_{\ov{C}}^2L^{-2}): (V,\phi)\mapsto -\det(\phi).
\end{equation}
%Note that $A'$ is a nonempty Zariski open subset.
In particular, to estimate the dimension of $V^1(C)$, we need to count $(L^{-1}$-twisted) Higgs $\PGL_2$-bundles over $A'\sub A$.

In what follows we will use the well known fact that every $\PGL_2$-bundle over $\ov{C}$ comes from a $\GL_2$-bundle, and that the
$\PGL_2$-bundles associated with $V_1$ and $V_2$ are isomorphic if and only if there exists a line bundle $M$ over $\ov{C}$ such that $V_1\simeq V_2\ot M$. 
All of this follows e.g. from the adelic descriptions and the fact that any $\PGL_2$-bundle over $\ov{C}$ is trivial at the general point (see \cite[Prop.\ 4.5]{BD}).
%, and in particular, is trivial at the general point of $\ov{C}$ (see ???). 

\begin{lemma}\label{V1C-dim-lem}
Let $(V,\phi)$ be an $L^{-1}$-twisted Higgs $\PGL_2$-bundle. Then $-\det(\phi)_\eta\in H'$ if and only if 
\begin{itemize}
\item either 
$(V,\phi)$ is stable and belongs to $h^{-1}(A')$,
\item or
$(V,\phi)=(L,\phi_1)\oplus (\si(L),-\phi_1)$, where $L$ is a line bundle defined over the quadratic extension $k'$ over $k$, $\si$ is the nontrivial involution in $\Gal(k'/k)$, and
$\phi_1\in H^0(\ov{C}_{k'},\om\ot L^{-1})\setminus \{0\}$ is such that $\si(\phi_1)=-\phi_1$.
\end{itemize}
%Hence, 
%$$\dim V^1(C)=|h^{-1}(A')(k)|+???.$$
\end{lemma}

\begin{proof}
Assume that $(V,\phi)$ is such that $-\det(\phi)_\eta\in H'$. Then we claim that there are no line subbundles $L\sub V$ defined over $k$ (where we represent $V$ by a rank $2$ bundle)
preserved by $\phi$. Indeed, we have $\tr(\phi)=\phi|_L+\phi_{V/L}=0$, so
this would imply that 
$$\det(\phi)=\det(\phi|_L)\cdot \det(\phi|_{V/L})=-\det(\phi|_L)^2,$$
a contradiction. Hence, $(V,\phi)$ is semistable (since a destablizing subbundle would be defined over $k$) and belongs to $h^{-1}(A')$. Conversely, if $(V,\phi)$ is in $h^{-1}(A')$ then 
$-\det(\phi)\in H^0(\ov{C},\om_{\ov{C}}^2L^{-2})$ is not a square at generic point, so $-\det(\phi)_\eta$ is in $H'$.

Now assume that $(V,\phi)$ is strictly semistable and $\det(\phi)_\eta\in H'$. Then there exists a line subbundle $L\sub V$ over an algebraic closure of $k$, preserved by $\phi$,
such that $\deg(V/L)=\deg(L)=\deg(V)/2$. As we have seen above $L\sub V$ cannot be defined over $k$.
If $M\sub V$ is a different line subbundle (defined over an algebraic closure)
preserved by $\phi$, with $\deg(M)=\deg(L)$, then the composed map $M\to V\to V/L$ is an isomorphism, so we get a decomposition of
Higgs bundles
$$(V,\phi)=(L,\phi|_L)\oplus (M,\phi|_M)$$
Since $L\sub V$ has at least one Galois conjugate different from it, we do have such a decomposition.
Furthermore, since $\tr(\phi)=0$, we have $\phi|_L=-\phi|_M$. This implies that there are no other line subbundles preserved by $\phi$. 
Hence, $L\sub V$ has only two conjugates under the Galois group, so it is defined over the quadratic extension $k'$ of $k$,
and $(M,\phi|_M)=\si(L,\phi|_L)$. 

Note that $\phi|_L=\phi_1\in H^0(\ov{C}_{k'},\om\ot L^{-1})\setminus \{0\}$, and $\phi|_{\si(L)}=\si(\phi_1)=-\phi_1$.
Finally, we observe that $-\det(\phi)=-\phi_1\si(\phi_1)=\phi_1^2$ is not a square of an element in $H^0(\ov{C},\om_{\ov{C}}L^{-1})$ (otherwise $\phi_1$ would be defined over $k$). 
%The involution $\si$ acts on the space $H^0(\ov{C}_{k'},\om\ot L^{-1})$, and if we view it as $ $k$-vector space then it decomposes into $1$ and $-1$ eigenspace,
%with the former one being $H^0(\ov{C},\om\ot L^{-1})$. Thus, there are $q^{h^0(\om\ot L^{-1})}-1$ elements $\phi_1\in H^0(\ov{C}_{k'},\om\ot L^{-1})\setminus \{0\}$
%such that $\si(\phi_1)=-\phi_1$. ???
\end{proof}

Next, let us consider the piece corresponding to the nonzero nilpotent orbit, i.e., the space
$$V(C)_n:=\bigoplus_{\chi\neq 1}\wt{\SS}_{\eta_0,\chi}^{G(\OO)},$$
where $\eta_0$ is a nonzero nilpotent element in $\fg\ot \om_{\ov{C}}L^{-1}(\ov{F})$,
and $\chi$ runs over nontrivial characters of $U(\ov{\A})/U(\ov{F})\simeq \ov{\A}/\ov{F}$.

\begin{prop}\label{VCn-dim-prop}
One has 
$$\dim V(C)_n=(q-1)\cdot \sum_{D\in |\om_{\ov{C}}^2L^{-1}|} N(D),$$
where for an effective divisor $D=\sum n_ip_i$, $N(D)=\prod_i (n_i+1)$. 
\end{prop}

\begin{proof}
%Using ???, we get a decomposition
%$$V(C)_n=\bigoplus_{\chi\neq 1} \wt{\SS}_{\eta_0,\chi}^{G(\OO)},$$
Let us consider the support of a function $f\in \wt{\SS}_{\eta_0,\chi}^{G(\OO)}$ in $U(\A)\backslash G(\ov{\A})_{\eta_0}/G(\ov{\OO})$.
By Iwasawa decomposition, it is enough to consider diagonal representatives $t=\diag(a,1)$, where $a\in \ov{\A}^*$.

Let us fix a nonzero $\a_0\in (\om_{\ov{C}}L^{-1})(\ov{F})$, and take as $\eta_0$ the corresponding element \eqref{eta0-eq}.
Then the condition $t\in G(\ov{\A})_{\eta_0}$ is equivalent to $a^{-1}\a_0\in \fg\ot \om_{\ov{C}}L^{-1}(\ov{\OO})$.

Note that for $u\in U(\ov{\A})$, the condition $t^{-1}ut\in G(\ov{\OO})$ is equivalent to $u\in U(a\OO)$.
Thus, the condition 
$$f(t)=f(t(t^{-1}ut))=f(ut)=\chi(u)f(t)$$ 
for such $u$ implies that $f(t)\neq 0$ only if $\chi|_{a\OO}\equiv 1$.

Conversely, for every $t=\diag(a,1)\in G(\ov{\A})_{\eta_0}$ such that $\chi|_{a\OO}\equiv 1$, there is a unique function $f_t\in \wt{\SS}_{\eta_0,\chi}^{G(\OO)}$, supported
on $U(\A)tG(\OO)$ and such that $f_t(t)=1$. We can also rescale $a$ by integer ideles,
so we can assume that $a_p=t_p^{n_p}$, where $t_p$ are fixed generators of the maximal ideals at points $p\in\ov{C}$. 

Let us also fix nonzero $\b_0\in \om_{\ov{C}}(\ov{F})$.
We can write every nontrivial character of $U(\ov{\A})/U(\ov{F})=\ov{\A}/\ov{F}$ as 
$$\chi(x)=\chi_{\phi\b_0}(x)=\psi_{\ov{C}}(x\phi\b_0)$$
for a unique $\phi\in \ov{F}^*$. Then $\chi|_{a\OO}\equiv 1$ means that $a\phi\b_0\in \om_{\ov{C}}(\ov{\OO})$.

The conditions that $a^{-1}\a_0$ and $af\b_0$ are regular imply that $\phi\a_0\b_0$ is regular, hence $\phi\a_0\b_0$ is a nonzero global section of $\om_{\ov{C}}^2L^{-1}$.
Once we fix $\phi$ modulo $k^*$, the conditions on $a$ are equivalent to
$$-v_p(\phi\b_0)\le n_p\le v_p(\a_0).$$
Thus, for every point $p$, we have $v_p(\phi\a_0\b_0)+1$ choices for $n_p$. Multiplying over all points $p$, gives $N(D)$ choices, where $D$ is the divisor of zeros of $f\a_0\b_0$.
We also have the extra factor $q-1=|k^*|$, since the sum is over the projective space.
\end{proof}

\medskip

\begin{proof}[Proof of Theorem E]
By Theorem \ref{main-dec-thm}, we have
$$\dim V(C)=\dim V(\ov{C})+\dim V(C)_n + \dim V^1(C),$$
where $V(\ov{C})=\SS_{\cusp}(G(\ov{F})\backslash G(\ov{\A})/G(\ov{\OO}))$.
We will show that $\dim V^1(C)$ has the required asymptotics, while two other dimensions are of order of magnitude $q^{3g-3}$.

For the first term we have $\dim V(\ov{C})\le a(g)\cdot q^{3g-3}$ for some constant $a(g)$
(see \cite[Sec.\ 3.3]{Schleich}).

Next, we can apply Lemma \ref{V1C-dim-lem} to estimate the dimension of $V^1(C)$.
Namely, it shows that
$$\dim V^1(C)=|h^{-1}(A')_{st}|+|X(k)|,$$
where $h$ is the Hitchin map \eqref{Hitchin-map}, $h^{-1}(A')_{st}\sub h^{-1}(A')$ is the subset corresponding to stable Higgs bundles, and
$X$ is the set of equivalence classes of pairs $(L,\a)$ over $\ov{C}_{k'}$, where $k'$ is the quadratic extension of $k$, with $\si(\a)=-\a$, $\a\neq 0$
(the pairs $(L,\a)$ and $(L',\a')$ are equivalent if either $(L',\a')\simeq (L\ot M,\a)$ or $(L',\a')\simeq (\si(L)\ot M,-\a)$, where $M\in \Pic(\ov{C})$). 
The number of $\a\in H^0(\ov{C}_{k'},\om\ot L^{-1})$ such that $\si(\a)=-\a$ is $\le q^g$. Also, $|J_{\ov{C}}(k')|/|J_{\ov{C}}(k)|\le c\cdot q^g$ (where $c=c(g)$). Hence,
$$|X(k)|\le c\cdot q^{2g}.$$

%Now let $\UU\sub h^{-1}(A')_{st}$ denote the open substack of stable Higgs $\PGL_2$-bundles with trivial automorphisms. 
%Then $|\UU(k)|$ is equal to the number of $k$-points in the corresponding open subset in the coarse moduli space of stable $L^{-1}$-twisted Higgs bundles,  
%which is a variety of dimension $6g-6$ with two connected components (both geometrically irreducible).
%Hence, by Lang-Weil estimate \cite{LW}, we get 
%$$|\frac{|\UU(k)|}{q^{6g-6}}-2|\le b(g)q^{-1/2}.$$

Let $\bA^{sm}\sub A=H^0(\ov{C},\om_{\ov{C}}^2L^{-2})$ be the open subset of sections with only simple zeros, and let $\UU^{sm}=h^{-1}(\bA^{sm})$.
We claim that all Higgs bundles in this locus are stable and have trivial automorphism group. Indeed, 
recall that for $\a\in \bA^{sm}$ the corresponding spectral curve $\pi:C_\a\to \ov{C}$ is smooth and geometrically irreducible, 
and Higgs $\PGL_2$-bundles in $h^{-1}(\a)$ correspond to
line bundles on $C_\a$ up to tensoring with $\pi^*\Pic(\ov{C})$. Stability of corresponding Higgs bundles follows immediately from geometrical irreducibility (similarly to
the proof of Lemma \ref{V1C-dim-lem}). A projective automorphism of a point in $h^{-1}(\a)$ represented by a line bundle $\LL$ on $C_\a$ would give a line bundle $M$ of degree $0$ on $\ov{C}$ (defined over $\ov{k}$) such that $\LL\ot \pi^*M\simeq \LL$, i.e., such that $\pi^*M\simeq \OO$. Since $\pi_*\OO_{C_\a}=\OO\oplus Q^{-1}$, where $\deg(Q)=g-1>0$,
we see that $H^0(C_\a,\pi^*M)=0$ unless $M$ is trivial. In the latter case the only possible automorphisms are rescalings (by stability), which disappear in $\PGL_2$.

By the correspondence with $\Pic(C_\a)/\pi^*\Pic(\ov{C})$, $\UU^{sm}$ has two geometrically irreducible components, defined over \(k\) and 
indexed by the degree modulo \(2\); each has dimension \(6g-6\). As $\ov{C}$ and $L$ vary with $g$ fixed and $\deg(L)=0$, these spaces form a bounded family.
Therefore, the Lang-Weil estimate \cite{LW} gives 
$$|\frac{|\UU^{sm}(k)|}{q^{6g-6}}-2|\le b(g)q^{-1/2}.$$

Let $A^{\times}=A\setminus\{0\}$ and let $\De^{\times}\sub A^{\times}$ denote the discriminant locus, i.e., the set of $\a$ with a non-simple zero.
Let $\MM^s$ denote the stack of geometrically stable $L^{-1}$-twisted Higgs $\PGL_2$-bundles, and set 
$$\ZZ=\MM^s\times_A \De^{\times}.$$
Let $Z$ denote the coarse moduli space of $\ZZ$. We still have a morphism $h:Z\to \De^{\times}$, and we denote by $Z_\a$ the fiber over $\a$. 
For every $\a\in\De^{\times}$, we have an identification of $Z_\a$ with the quotient of the compactified Jacobian of the spectral curve $C_\a$ by $\Pic(\ov{C})$. 
Since $C_\a$ is reduced and locally planar, the usual Jacobian is dense in the compactified Jacobian, which therefore has dimension $p_a(C_\a)=4g-3$ (see \cite[Thm.\ 2.3]{MRV}), so $Z_\a$ has dimension $3g-3$. 
It follows that $Z$ has dimension $\le 6g-7$, so we get 
\begin{equation}\label{Z-k-number}
|Z(k)|\le b'(g)q^{6g-7}
\end{equation}

Any projective automorphism of a stable Higgs bundle $(V,\phi)$ is determined by a line bundle $M$ on $\ov{C}$ such that $V\ot M\simeq V$, which implies that $M^2\simeq \OO$.
Hence, 
$$|\Aut(V,\phi)(\ov{k})|\le 2^{2g}.$$
Hence, $|H^1(k,\Aut(V,\phi)|\le 2^{2g}$, which implies the bound of the form \eqref{Z-k-number} for $|\ZZ(k)|$ (with the constant $2^{2g}b'(g)$).

%Next, we need to estimate the number of stable Higgs $\PGL_2$-bundles with nontrivial automorphisms. We observe that if $(V,\phi)$ is such a Higgs bundle
%then there is an isomorphism $V\simeq V\ot M$ for some line bundle $M$ such that $M^{\ot 2}\simeq \OO$. But this implies that $V$ comes from a line bundle 
%on the corresponding unramified double covering $\wt{C}\to \ov{C}$, so the number of possible $\PGL_2$-bundles appearing like this is $\le c\cdot q^{g-1}$. Hence, the number
%of stable Higgs $\PGL_2$-bundles with nontrivial automorphisms is $\le c\cdot q^{4g-4}$.
%Combining the steps above, we conclude that 
%$$|\frac{\dim V^1(C)}{q^{6g-6}}-2|\le c(g)q^{-1/2},$$
%for some constant $c(g)$.

Finally, we use Proposition \ref{VCn-dim-prop} to estimate the dimension of $V(C)_n$.
We claim that there exists a function $c(d)$, such that $N(D)\le c(\deg(D))$ for any effective divisor $D$.
Indeed, let $D=\sum_{i=1}^s n_i p_i$, where all $n_i>0$. Then we have $\sum n_i \deg(p_i)=d$. Therefore,
$$\prod (n_i+1)\le (\frac{1}{s}\sum (n_i+1))^s \le (d+1)^d.$$
Since $\dim H^0(\ov{C},\om_{\ov{C}}^2L^{-1})=3g-3$, Proposition \ref{VCn-dim-prop} gives 
$$\dim V(C)_n\le c(g)q^{3g-3},$$
for some constant $c(g)$.
\end{proof}

Now let us consider the case of a Hitchin fiber $HF_\a$ in the moduli stack of $L^{-1}$-twisted Higgs $\PGL_2$-bundles over $\ov{C}$, associated with
$\a\in H^0(\ov{C},\om_{\ov{C}}^2L^{-2})$ that has only simple zeros.
As is well-known, in this case we have an identification 
$$HF_\a(k)\simeq \Pic(C_\a)/\pi^*\Pic(\ov{C}),$$ 
where $\pi:C_\a\to \ov{C}$ is the corresponding spectral curve (which is smooth), so $HF_\a(k)$ has a group structure.
As before, we will use an identification 
$$\wt{\SS}_{\eta}^{G(\OO)}\simeq \SS(HF_\a(k),\LL_\psi),$$
for some natural $\C^*$-torsor $\LL_\psi$, where $\eta$ is an element in $\ssl_2(\ov{F})$ with $\det(\eta)=-\a$.

Theorem F is implied by the following result.

\begin{theorem}\label{smooth-hitchin-thm}
Assume the characteristic of $k$ is $\neq 2$, and that $\a\in A'\sub H^0(\ov{C},\om_{\ov{C}}^2L^{-2})$ has only simple zeros.
Then there exists a commutative group extension 
\begin{equation}\label{Hitchin-fiber-ext}
1\to U(1)\to H_{U(1)}\to HF_\a(k)\to 1,
\end{equation}
and an action of $H_{U(1)}$ on the $\C^*$-torsor $\LL_\psi$ over $HF_\a(k)$ (where $U(1)$ acts naturally), compatible with the action of $HF_\a(k)$ on itself by shifts.
Furthermore, there is an $\HH_{\PGL_2,C}$-eigenbasis $(f_\chi)$ in $\SS(HF_\a(k),\LL_\psi)$ numbered by characters $\chi$ of $H_{U(1)}$ extending the identity character of $U(1)$,
where 
$$hf_\chi=\chi(h)f_\chi \ \text{ for } h\in H_{U(1)}.$$ 
In the split case $C=\ov{C}\times_{\Spec(k)}\Spec(k[\eps]/(\eps^2))$ the extension \eqref{Hitchin-fiber-ext} has a natural splitting, 
so $\LL_\psi$ is trivialized and the $\HH_{\PGL_2,C}$-eigenbasis can be indexed by characters of $HF_\a(k)\simeq \Pic(C_\a)/\pi^*\Pic(\ov{C})$.
\end{theorem}

\begin{proof}
We can identify $\pi_*\OO_{C_\a}$ with $\AA=\OO_{\ov{C}}\oplus \om^{-1}_{\ov{C}}L\cdot t$, where $t$ is a formal variable such that $t^2=\a$. 
This induces an identification of the group of units $\AA^*(\ov{F}_p)$ and $\AA^*(\ov{F})$ with the groups of $\ov{F}_p$ or $\ov{F}$-points of 
the stabilizer $\wt{T}_\eta$ of $\eta$ in $\GL_2$, where we realize $\GL_2(F)$ as $F$-automorphisms of the generic stalk $F\oplus F\cdot t$ of $\AA$.
This induces an identification of $\Pic(C_\a)$ with $\wt{T}_\eta(\ov{A})/\wt{T}_\eta(\ov{\OO})\wt{T}_\eta(\ov{F})$, and of
$\Pic(C_\a)/\pi^*\Pic(\ov{C})$ with the similar group associated with $T_\eta\sub \PGL_2$.

Recall that by Proposition \ref{Higgs-prop}(3) we have a commutative extension $\ov{H}_\eta$ of $T_\eta(\ov{\A})/T_\eta(\ov{F})$ by $\om_{\ov{C}}(\ov{\A})/\om_{\ov{C}}(\ov{F})$ 
acting on 
the $\C^*$-torsor $\LL_\psi$ over $HF_\a(k)$, such that  $\om_{\ov{C}}(\ov{\A})/\om_{\ov{C}}(\ov{F})$ acts by $\psi_{\ov{C}}^{-1}$. It is easy to see that this action on $\LL_\psi$ factors
through the action of the induced extension $H_{U(1)}$ of $T_\eta(\ov{\A})/T_\eta(\ov{\OO})T_\eta(\ov{F})\simeq \Pic(C_\a)/\pi^*\Pic(\ov{C})$ by $U(1)$.

Now we observe that since $C_\a$ is smooth, the action of $\Pic(C_\a)/\pi^*\Pic(\ov{C})$ on $HF_\a(k)$ is simply transitive. Thus, we have a basis $(f_\chi)$ with the claimed properties.
The fact that the subspaces $\C\cdot f_\chi$ are preserved by $\HH_{\PGL_2,C}$-action follows from the fact they can be identified with the summands of the decomposition
%Let $\a\in A'\sub H^0(\ov{C},\om_{\ov{C}}^2L^{-2})$ be a section with only simple zeros.
%By Proposition \ref{reg-ell-prop}, we have a decomposition
$$\wt{\SS}_{\eta}^{G(\OO)}=\bigoplus_{\chi} \wt{\SS}_{\eta,\chi}^{G(\OO)}$$
(see Proposition \ref{reg-ell-prop}).
%where $\chi$ runs over the characters of ???
%the finite group $T_\a(\ov{\A})/(T_\a(\ov{F})T_\a(\ov{\OO}))$ 
%(note that since $C=C_0[\eps]$, the extension $H_\a$ of $T_\a(\ov{\A})$ splits canonically).
%Note that $\wt{\SS}_{\a,\chi}$ decomposes into the restricted tensored product of the corresponding local induced representations $I_p$ of $G(F_p)$, and each local Hecke algebra 
%$\HH_p$ acts on $\wt{\SS}_{\a,\chi}$ via its action on the $G(\OO_p)$-invariants of $I_p$. Thus, it remains to prove that each space of invariants $I_p^{G(\OO_p)}$ is at most $1$-dimensional.
%Assume first that $\a$ is a square over $F_p$, i.e., the corresponding conjugacy class is represented by a diagonal element. Then $\a$ does not vanish at $p$ (otherwise,
%it would have a double zero at $p$), hence the fact that the space of invariants is at most $1$-dimensional follows from \cite[Cor.\ 4.5]{K-YD}.
%Next, assume that $\a$ is not a square over $F_p$, so the corresponding conjugacy class is represented by a matrix of the form 
%$\left(\begin{matrix} 0 & d \\ 1 & 0\end{matrix}\right)$. Then $d$ has at most zero of order $1$ at $p$, so   
%the fact that the space of invariants is at most $1$-dimensional follows from \cite[Cor.\ 4.13]{K-YD}.

In the case when $C$ is split, the natural splitting $T_\eta(\ov{\A})\to \St_\eta$ induces a splitting of $\ov{H}_\eta$, and hence of \eqref{Hitchin-fiber-ext}.
\end{proof}

\appendix
\section{Some results on groupoids.}\label{group-app}

\subsection{Push-forward}\label{push-forward-sec}

For a small groupoid $\Ga$ we denote by $\C(\Ga)$ the space of $\C$-valued functions on the set of isomorphism classes of $\Ga$.
For a functor of groupoids $\Phi:\Ga_1\to \Ga_2$ the pullback map $\Phi^*:\C(\Ga_2)\to \C(\Ga_1)$ is given by $\Phi^*f(\ga_1)=f(\Phi(\ga_1))$.

Assume that all objects of $\Ga_1$ and $\Ga_2$ have finite groups of automorphisms, and that 
for every $\ga_2\in \Ga_2$ there are finitely many isomorphism classes of $\ga_1\in \Ga_1$ such that $\Phi(\ga_1)\simeq \ga_2$.
Then we define the push-forward
\begin{equation}\label{groupoid-push-for-def}
\Phi_*f(\ga_2)=\sum_{\ga_1: \Phi(\ga_1)\simeq \ga_2} \frac{|\Aut(\ga_2)|}{|\Aut(\ga_1)|} f(\ga_1).
\end{equation}

Recall that for each object $\ga_2$, one defines the fiber groupoid $\Phi^{-1}(\ga_2)$, whose objects are
pairs $(\ga_1,\phi)$, where $\ga_1$ is an object of $\Ga_1$ and $\phi:\Phi(\ga_1)\to \ga_2$ is an isomorphism.
This fiber shows up when computing the push-forward map on functions $\Phi_*:\C(\Ga_1)\to \C(\Ga_2)$ (when it is defined).

\begin{lemma}\label{groupoid-push-for-lem}
Let $\Phi:\Ga_1\to\Ga_2$ be a functor, such that for every $\ga_2\in \Ga_2$ there are finitely many isomorphism classes of $\ga_1\in \Ga_1$ such that $\Phi(\ga_1)\simeq \ga_2$.
Assume also that all objects of $\Ga_1$ and $\Ga_2$ have finite groups of automorphisms.
Then for a function $f\in \C(\Ga_1)$, one has 
\begin{equation}\label{push-for-fiber-eq}
\Phi_*f(\ga_2)=\sum_{(\ga_1,\phi)\in \Phi^{-1}(\ga_2)} \frac{1}{|\Aut(\ga_1,\phi)|} f(\ga_1).
\end{equation}
\end{lemma}

\begin{proof}
For fixed objects $\ga_1$ and $\ga_2$, let us consider the finite set $H:=\Hom(\Phi(\ga_1),\ga_2)$.
There is a natural action of $\Aut(\ga_1)\times \Aut(\ga_2)$ on $H$, such that the action of the subgroup $\Aut(\ga_2)$ on $H$ is simply transitive,
while the stabilizer subgroup of $\phi\in H$ under the action of the subgroup $\Aut(\ga_1)$ can be identified with $\Aut(\ga_1,\phi)$.
Hence, the size of the corresponding orbit is
$$|\Aut(\ga_1)\cdot\phi|=\frac{|\Aut(\ga_1)|}{|\Aut(\ga_1,\phi)|}.$$

Now, the coefficient of $f(\ga_1)$ in the right-hand side of \eqref{push-for-fiber-eq} is given by
\begin{align*}
&\sum_{\phi\in H/\Aut(\ga_1)} \frac{1}{|\Aut(\ga_1,\phi)|}=\sum_{\phi\in H} \frac{1}{|\Aut(\ga_1)\cdot\phi|}\cdot \frac{1}{|\Aut(\ga_1,\phi)|}=\\
&\sum_{\phi\in H} \frac{1}{|\Aut(\ga_1)|}=\frac{|H|}{|\Aut(\ga_1)|}=\frac{|\Aut(\ga_2)|}{|\Aut(\ga_1)|}.
\end{align*}
But this is equal to the coefficient of $f(\ga_1)$ in $\Phi_*f(\ga_2)$, and our assertion follows.
\end{proof}

\subsection{Double cosets groupoids}\label{double-coset-sec}

Let $G$ be a group, $H,K\sub G$. Then the set of double cosets $H\backslash G/K$ can be viewed as a set of isomorphism classes of
a groupoid. Namely, the objects of this groupoid are elements of $G$. A morphism from $g_1$ to $g_2$ is a pair of elements $(h\in H, k\in K)$ such that
$hg_1k=g_2$. The composition of $(h,k):g_1\to g_2$ with $(h',k'):g_2\to g_3$ is $(h'h,kk'):g_1\to g_3$.

\begin{lemma}\label{double-coset-group-lem} 
Let $\pi: B\to T$ be a surjection of groups with the kernel $U$. Let $B_1, B_2\sub B$ be a pair of subgroups, and let $T_1=\pi(B_1)$, $T_2=\pi(B_2)$.
Consider the induced groupoid functor
$$\Pi: B_1\backslash B/B_2\to T_1\backslash T/T_2.$$
Let us set $U_1:=U\cap B_1$. For an element $b_0$, let us set 
$$U_{2,b_0}:=U\cap b_0B_2b_0^{-1}.$$
Then the map $u\mapsto ub_0$ naturally extends to an equivalence of groupoids
$$U_1\backslash U/U_{2,b_0}\rTo{\sim} \Pi^{-1}(\pi(b_0)),$$
where on the right we take the groupoid fiber.
\end{lemma}

\begin{proof}
Set $t_0=\pi(b_0)$. By the definition, the objects of $\Pi^{-1}(t_0)$ are triples $b\in B$, $t_1\in T_1$ and $t_2\in T_2$, such that $\pi(b)=t_1t_0t_2$.
A morphism $(b,t_1,t_2)\to (b',t'_1,t'_2)$ is given by a pair of elements $b_1\in B_1$, $b_2\in B_2$, such that
$$b'=b_1bb_2, \ \ t'_1=\pi(b_1)t_1, \ \ t'_2=t_2\pi(b_2).$$

Since $T_1=\pi(B_1)$, $T_2=\pi(B_2)$, we see that every object of $\Pi^{-1}(t_0)$ is isomorphic to one with $t_1=t_2=1$. 
A morphism between such objects $(b,1,1)\to (b',1,1)$, where $\pi(b)=\pi(b')=t_0$, is given by a pair of elements $u_1\in U\cap B_1=U_1$, $u_2\in U\cap B_2$ such that 
$b'=u_1bu_2$. If we write $b=ub_0$, $b'=u'b_0$, then this equation becomes
$$u'=u_1u(b_0u_2b_0^{-1}).$$
Since $U_{2,b_0}=b_0(U\cap B_2)b_0^{-1}$, we get the claimed equivalence.
\end{proof}

%Let $\LL$ be the line bundle associated with $\a$, and let us consider the set $\EE(\LL,D)$ of isomorphism classes of pairs $L\sub V$ as before,
%together with a fixed isomorphism $L\simeq \LL\ot \OO[-D]$. Note that we then have an induced isomorphism $V/L\simeq \LL^{-1}[D]$.
%We claim that there is a natural bijection
%$$U(\A_C)/(U(F)\cdot U_t)\rTo{\sim} \EE(\LL,D).$$

\section{Local calculations for the nilpotent orbit for $\PGL_2$}\label{loc-app}

Let $G=\PGL_2$, and let $O\sub E$ be square-zero extensions of a local field and its ring of integers $\ov{O}\sub \ov{E}$. We assume that there is a fixed element 
$\eps\in O$, such that 
$\eps^2=0$, $O/\eps O=\ov{O}$, $E/\eps E=\ov{E}$, $\eps O\simeq \ov{O}$ and $\eps E\simeq \ov{E}$.
Fix a non-trivial additive character $\psi:\ov{E}\to U(1)$ with the conductor $\ov{O}$.

Let $r:G(E)\to G(\ov{E})$ denote the reduction map.
Consider the subgroup 
$$H:=r^{-1}(U(\ov{E}))\sub G(E)$$
(this is the stabilizer in $G(E)$ of the nilpotent element in $\fg(\ov{E})$).

Given $a,b\in \ov{E}$, with $a\neq 0$, there is a well defined character $\chi_{a,b}$ of $H$ given by
$$\chi_{a,b}(1+\eps X)=\psi(a X_{21}), \ \ X\in \mathfrak{sl}_2,$$
$$\chi_{a,b}(u(y))=\psi(b\ov{y}), \ \ y\in E,$$
where $u(y)=\left(\begin{matrix} 1& y \\ 0 & 1\end{matrix}\right)$.

We are interested in the representation 
$$V_{a,b}:=\ind_H^{G(E)}\chi_{a,b}$$ of $G(E)$.
Its elements are smooth functions $f$ on $G(E)$ such that
$$f(hg)=\chi_{a,b}(h)f(g)$$
and the support of $f$ in $H\backslash G(E)$ is compact.

In the split case $E=\ov{E}[\eps]/(\eps^2)$, parts (i) and (iii) of the following proposition were proved in \cite[Sec.\ 4.4]{K-YD}

\begin{prop}\label{nilp-sph-v-prop}
(i) Assume $b\neq 0$. Then 
$$\dim V_{a,b}^{G(O)}=\max(0,v(a)+v(b)+1).$$

\noindent
(ii) Still for $b\neq 0$, for any open subgroup $K\sub G(O)$, one has
$$\dim V_{a,b}^K<\infty.$$ 

\noindent
(iii) For $b=0$, the space $V_{a,0}^{G(\OO)}$ is infinite-dimensional.
\end{prop}

\begin{proof} Let us choose a uniformizer $\pi\in \ov{O}$ and its lift $\wt{\pi}\in O$. Set 
$$t_n=\left(\begin{matrix} \wt{\pi}^n & 0 \\ 0 & 1\end{matrix}\right), \ \ n\in\Z.$$

\noindent
(i) The Iwasawa decomposition for $G(\ov{E})$ gives 
$$H\backslash G(E)/G(O)\simeq U(\ov{E})\backslash G(\ov{E})/G(\ov{O})\simeq\Z$$
with representatives $(t_n)$. Hence, an element $f\in V_{a,b}^{G(O)}$ is determined by 
the values $f(t_n)$. Furthermore, the value $f(t_n)$ can be nonzero if and only
$$\chi_{a,b}|_{H\cap t_nG(O)t_n^{-1}}=1.$$
We claim that this is equivalent to 
$$-v(b)\le n\le v(a).$$

Indeed, let $N=\ker(G(E)\to G(\ov{E}))$. Then 
$$N\cap t_nG(O)t_n^{-1}=\{1+\eps X \ |\ X\in \Ad(\ov{t}_n)\mathfrak{sl}_2(\ov{O})\}.$$
Note that for $1+\eps X\in N\cap t_nG(O)t_n^{-1}$, $X_{21}$ ranges over $\pi^{-n}\ov{O}$. Hence,
$$\chi_{a,b}|_{N\cap t_nG(O)t_n^{-1}}=1$$
if and only if $a\pi^{-n}\in \ov{O}$, or equivalently $v(a)\ge n$.

On the other hand, the reduction of $U(E)\cap t_nG(O)t_n^{-1}$ is exactly $U(\pi^n\ov{O})$. Hence,
$$\chi_{a,b}|_{U(E)\cap t_nG(O)t_n^{-1}}=1$$
if and only if $b\pi^n\in\ov{O}$, or equivalently $v(b)\ge -n$.

Given any $h\in H\cap t_nG(O)t_n^{-1}$, its reduction is in $U(\pi^n\ov{O})$, which has a lift $\wt{u}\in U(\wt{\pi}^nO)$.
Hence, 
$$H\cap t_nG(O)t_n^{-1}=(N\cap t_nG(O)t_n^{-1})\cdot (U(E)\cap t_nG(O)t_n^{-1}),$$
and our condition on $n$ is also sufficient.

\noindent
(ii) For $r\ge 1$, set $K(r)=\ker(G(O)\to G(O/\wt{\pi}^rO))$. It is enough to check that $V_{a,b}^{K(r)}$ is finite-dimensional.
Choose representatives $k_1,\ldots,k_s$ of $G(O)/K(r)$. Then $(t_nk_i)$ are representatives of double cosets $H\backslash G(E)/K(r)$.
We need to check that if $f\in V_{a,b}^{K(r)}$ then $f(t_nk_i)\neq 0$ only for finite set of pairs $(n,i)$.
Since $k_i$ normalizes $K(r)$, if $f(t_nk_i)\neq 0$ then we should have
$$\chi_{a,b}|_{H\cap t_nK(r)t_n^{-1}}=1.$$

For any $z\in \ov{O}$, we have $1+\eps \pi^r z E_{21}\in K(r)$, hence we should have
$$1=\chi_{a,b}(1+\Ad(t_n)\eps \pi^r z E_{21})=(1+\eps \pi^{r-n}z E_{21})=\psi(a\pi^{r-n}z).$$
This implies that $v(a)\ge n-r$.

On the other hand, for any $y\in O$, we have $u(\wt{\pi}^r y)\in K(r)$, hence,
$$1=\chi_{a,b}(t_nu(\wt{\pi}^ry)t_n^{-1})=\chi_{a,b}(u(\wt{\pi}^{r+n}y))=\psi(b\pi^{r+n}\ov{y}),$$
which gives $v(b)\ge -n-r$. Thus, we get a finite range for $n$:
$$-v(b)-r\le n\le v(a)+r.$$ 

\noindent
(iii) Exactly as in (i), we get only one-sided constraint $n\le v(a)$, so there are infinitely many choices for $n$.
\end{proof}

Note that in global applications we use the fact that at all places where both characters (and the open subgroup in $G(\OO)$) are unramified, the
case $v(a)=v(b)=0$ of Proposition \ref{nilp-sph-v-prop}(i) gives $1$-dimensional spaces of invariants.

\section{Example of a nontrivial action of $\Aut(P,\phi)$ on $\LL_\psi$ for $\PGL_2$ in characteristic $2$}\label{Aut-ex-sec}

Let $k=\F_2$, $A=k[\eps]/(\eps^2)$.
For $g\ge 1$, consider the smooth affine curve $\ov{U}\sub\A^2$ given by the equation 
$$y^2+xy=x^{2g+1}+1$$ 
and $\ov{C}$ denote its smooth projective completion.
The map $\ov{U}\to\A^1$ given by $x$ extends to the double covering map $\ov{C}\to \P^1$ ramified at $(2g+1)$th roots of $1$ and at $\infty$, so $\ov{C}$ has genus $g$.
The curve $\ov{C}$ has a unique point at infinity $p\in \ov{C}\setminus \ov{U}$. We have an affine neighborhood $\ov{V}$ of $p$ in $\ov{C}$ with 
$$\OO(\ov{V})=k[u,v]/(v^2+vu^{g+1}+u^{2g+2}+u),$$
where $u=x^{-1}$, $v=yx^{-g-1}$, so that $p$ corresponds to $u=v=0$.
Note that
$$\eta:=\frac{dx}{x}=\frac{dy}{y+x^{2g}}=\frac{du}{u}=\frac{u^g dv}{1+gu^gv}$$
is a global regular differential on $\ov{C}$.

We define a smooth proper curve $C$ over $A$ reducing to $\ov{C}$ over $k$ as the union of two affine opens $U$ and $V$, where 
$U$ is given by 
$$y^2+xy=x^{2g+1}+1+\eps,$$
while $V$ is given by
$$v^2+u^2v=u+(1+\eps)u^{2g+2}$$
(and still $u=x^{-1}$, $v=yx^{-g-1}$). 

Note that $\div(x)=2q-2p$ on $\ov{C}$, where $q\in \ov{U}$ is the point with $x(q)=0$, $y(q)=1$. Thus, 
$$M:=\OO_{\ov{C}}(q-p)$$ 
is a nontrivial torsion line bundle on $\ov{C}$. More precisely, we have an isomorphism
$$x:M^2=\OO_{\ov{C}}(2q-2p)\rTo{\sim} \OO_{\ov{C}}.$$

Set $V=\OO_{\ov{C}}\oplus M$, and let $P$ denote the corresponding $\PGL_2$-bundle on $\ov{C}$.
Since the characteristic is $2$, the trace induces a conjugation-invariant functional
$$\tau:\fg=\mathfrak{gl}_2/k\cdot I\to k,$$
hence we can view $\tau$ as an element of $H^0(\ov{C},\fg_P^\vee)$. We define a Higgs field $\phi$ on $P$ by
$$\phi:=\tau\ot \eta\in H^0(\ov{C},\fg_P^\vee\ot\om_{\ov{C}}).$$

Let us define an automorphism $a$ of $P$ (defined over $k$) as
$$a=\left(\begin{matrix} 0 & x \\ 1 & 0\end{matrix}\right): V\ot M=M\oplus M^2\to \OO_{\ov{C}}\oplus M=V.$$
Since $\tau$ is conjugation-invariant, $a$ preserves $\phi$, so it is an automorphism of $(P,\phi)$.

Recall that the nilpotent extension $C$ defines a $k$-torsor $\TT^C(k)$ over the groupoid $\MM^{Higgs}(\ov{C})(k)$, and
hence, a homomorphism $c_{P,\phi}$ from $\Aut(P,\phi)$ to $k$ (see \eqref{c-P-phi-eq}).

\begin{prop}
One has $c_{P,\phi}(a)=1$.
\end{prop}

\begin{proof}
Let us lift the points $p\in\ov{V}$ and $q\in\ov{U}$ to the $A$-points $\wt{p}\in V(A)$ and $\wt{q}\in U(A)$
given by 
$$u(\wt{p})=v(\wt{p})=0, \ \ \ \ x(\wt{q})=\eps, \ y(\wt{q})=1.$$
Note that the ideal of $\wt{p}$ in the local ring of $C$ at $p$ is generated by $v$, while the ideal of
$\wt{q}$ in the local ring of $C$ at $q$ is generated by $y+1$.
To check the latter, it is convenient to set $z=y+1$, so that 
$$z^2+xz=(x+\eps)(1+x^{2g}+\eps x^{2g-1}),$$
which shows that $x+\eps$ belongs to the ideal generated by $z$ in a neighborhood of $q$.

We lift the line bundle $M$ on $\ov{C}$ to a line bundle
$$\wt{M}=\OO_C(\wt{q}-\wt{p})$$
on $C$. Then $\wt{M}^{\ot 2}$ has trivial reduction on $\ov{C}$, so it determines a class
$$\kappa:=[\wt{M}^{\ot 2}]\in \ker(\Pic(C)\to \Pic(\ov{C}))\simeq H^1(\ov{C},\OO).$$

Note that the function $x^{-1}=u$ is regular on $C\setminus \{q\}$, invertible away from $p$,
and near $p$, it differs from $v^2$ by an invertible function. Hence, 
$$e_{C\setminus q}:=x^{-1}$$ gives a trivialization of $\wt{M}^2|_{C\setminus\{q\}}$.
On the other hand, on a neighborhood $U_q$ of $q$ the line bundle $\wt{M}^2$ can be trivialized by any function of the form $uz^{-2}$,
with $u$ invertible near $q$. In particular,
$$e_{U_q}:=(1+z+x^{2g})z^{-2}=x^{-1}\frac{z^2+\eps}{z^2}$$
gives such a trivialization of $\wt{M}^2$ on $U_q$. Then on $U_q\setminus\{q\}$, we have
$$e_{U_q}/e_{C\setminus q}=1+\eps z^{-2},$$
which means that $\kappa$ is represented by the Cech cocycle $z^{-2}$ on the reduction $\ov{U}_q\setminus\{q\}$ (with respect to the open covering $(\ov{C}\setminus\{q\},U_q)$).

It follows that the pairing $\lan \kappa,\eta\ran\in k$ (given by the Serre duality) is
$$\lan \kappa,\eta\ran=\Res_q (z^{-2}\eta)=\Res_q(z^{-2}\cdot \frac{dz}{1+z+x^{2g}}).$$
On $\ov{C}$ near $q$ we have $x=z^2(1+z+x^{2g})$, hence, $x^{2g}\in (z^{4g})$ does not contribute to the above residue, and we get
\begin{equation}\label{phi-kappa-1-eq}
\lan \kappa,\eta\ran=1.
\end{equation}

To compute $c_{P,\phi}(a)$ we need to pick a $\PGL_2$-bundle $\wt{P}$ over $C$ together with an isomorphism (framing)
$$\b:\wt{P}|_{\ov{C}}\rTo{\sim} P.$$
Then $(P,a\circ\b)$ and $(P,\b)$ both belong to the $H^1(\ov{C},\fg_P)$-torsor of framed liftings of $P$, and we need to find a unique class
$e\in H^1(\ov{C},\fg_P)$ such that 
$$(P,a\circ\b)=e+(P,\b).$$
The character $c_{P,\phi}(a)$ is then given by $\lan e,\phi\ran$.

In our case we let $\wt{P}$ be the $\PGL_2$-bundle associated with the vector bundle
$$\wt{V}=\OO_C\oplus \wt{M},$$
with the framing induced by the obvious isomorphism 
$$\wt{V}|_{\ov{C}}\simeq \OO_{\ov{C}}\oplus M=V.$$
By the definition of the automorphism $a$, we have to compare this framed lifting with the one induced by
$\wt{V}\ot\wt{M}$, with the framing
$$\wt{V}\ot\wt{M}|_{\ov{C}}\simeq V\ot M\rTo{a} V.$$ 
The difference gives a class $e\in H^1(\ov{C},\End(V)/\OO_{\ov{C}})$. Due to our definition of $\phi$, we have
$$\lan e,\phi\ran=\lan \tr(e),\eta\ran,$$
where we use the trace morphism $\End(V)/\OO_{\ov{C}}\to \OO_{\ov{C}}$.

Computing $\tr(e)$ reduces to passing from $\wt{V}$ to the line bundle $\det(\wt{V})=\wt{M}$ (viewed modulo squares of line bundles).
This easily leads to $\tr(e)=\kappa$, and the assertion follows from \eqref{phi-kappa-1-eq}.
\end{proof}

\section{Geometric interpretation of the constant term operator}\label{geom-const-term-sec}

In this section we will provide a geometric interpretation of the constant term operator for $G=\GL_2$ and a special nilpotent extension $C$ of $\ov{C}$ of length $2$.
Recall that in the case of the reduced curve, the constant term operator is related to the moduli space of $B$-bundles, i.e., of pairs $L\sub V$,
where $L$ is a line subbundle in a rank $2$ bundle. In the case of a nilpotent extension of length $2$, we will use certain generalizations of line bundles,
which we call {\it quasi line bundles}. In Sec.\ \ref{quasi-lin-sec} we describe some basic properties of quasi line bundles. We give adelic description
of the groupoid of quasi line bundles in Sec. \ref{adelic-quasi-lin-sec}.
In Lemma \ref{Qbun-triples-lem} and Proposition \ref{QBun-mod-prop} we give geometric interpretations of the double coset groupoids and maps between them involved in the constant term
operator. Finally, in Proposition \ref{geom-const-term-prop} we provide a geometric interpretation of the constant term operator.

\subsection{Quasi line bundles}\label{quasi-lin-sec}

Let $C$ be a special nilpotent extension of $\ov{C}$ of length $2$. We denote by $\NN\sub \OO_C$ the nilradical, i.e., the ideal of $\ov{C}\sub C$. 
Note that $\NN$ is a line bundle on $\ov{C}$. When working locally we denote by $\eps\in \NN$ a generator as a module over $\OO_{\ov{C}}$.
%As before, $\eps$ denotes a generator of the maximal ideal of $A$.

\begin{definition}
A coherent sheaf $M$ on $C$ is called a {\it quasi line bundle} if locally it admits an embedding $M\hra \OO_C$ and if $\NN M\neq 0$.
\end{definition}

For a coherent sheaf $M$ on $C$ we set 
$$\ov{M}:=\ker(M\to \und{\Hom}(\NN,M):m\mapsto(x\mapsto xm)).$$ 
We view $\ov{M}$ as a coherent sheaf on $\ov{C}$. If $\eps$ is a local generator of $\NN$ then $\ov{M}=\ker(\eps:M\to M)$ and $M/\ov{M}\simeq \eps M$.

\begin{prop}\label{quasi-line-def-prop} 
The following conditions for a coherent sheaf $M$ are equivalent:
\begin{enumerate}
\item $M$ is a quasi line bundle;
\item $\NN M\neq 0$ and $\ov{M}$ is a line bundle on $\ov{C}$;
\item $\ov{M}$ is a line bundle on $\ov{C}$, and there exists a line bundle $\LL$ on $C$ and an embedding $\LL\hra M$ inducing an isomorphism $\LL/\NN\LL\rTo{\sim} M/\ov{M}$,
or equivalently, such that $\NN\LL=\NN M$;
\item $\NN M\neq 0$, and there exists a line bundle $\LL$ on $C$ and an embedding $M\hra \LL$ such that $\ov{M}=\NN\LL$;
\item $M$ is locally isomorphic to an ideal $(\eps,f)\sub \OO$, where $f\not\equiv 0 \mod (\eps)$.
%\item $M$ is locally isomorphic to an ideal $I\sub \OO$ such that $(\eps)\sub I$ and $(\eps)\neq I$.
\end{enumerate}
\end{prop}

\begin{proof}
(1)$\implies$ (2). Locally we have an embedding $M\sub \OO$. Thus, $\ov{M}$ is a subsheaf in $\ov{\OO}=\OO_{\ov{C}}$.
Since $\NN M$ is contained in $\ov{M}$, we have $\ov{M}\neq 0$, so $\ov{M}$ is a line bundle on $\ov{C}$.
%The multiplication by $t$ induces an isomorphism 
%$$M/\ov{M}\simeq \eps M.$$

\noindent
(2)$\implies$ (3). 
%We have an exact sequence
%\begin{equation}\label{M-tM-ex-seq}
%0\to \ov{M}\to M\to \eps M\to 0
%\end{equation}
%where by assumption, $\eps M\neq 0$ and $\ov{M}$ is a line bundle on $\ov{C}$. 
Since $\NN M\neq 0$ is contained in $\ov{M}$, it is also a line bundle.
Thus, the embedding $\NN M\hra \ov{M}$ identifies $\NN M$ with $\ov{M}(-D)$ for some effective divisor $D\sub \ov{C}$.

Let us consider the coherent sheaf $M/\NN M$ on $\ov{C}$. Then $\ov{M}/\NN M\sub M/\NN M$ is exactly its torsion subsheaf, and the quotient by it is isomorphic to 
the line bundle $M/\ov{M}$. We can choose a splitting $\si:M/\ov{M}\to M/\NN M$ on $\ov{C}$. Now define 
$\LL\sub M$ to be the preimage of $\im(\si)\sub M/\NN M$. We have
$$\LL+\ov{M}=M.$$
Hence, $\NN\LL=\NN M$, and 
$$\LL/\NN\LL\simeq \LL/\NN M\simeq \im(\si)\simeq M/\ov{M}$$
is a line bundle on $\ov{C}$. This implies that $\LL$ is a line bundle with the required properties.

\noindent
(3)$\implies$ (4). Consider an embedding $\LL\hra M$ with $\NN \LL=\NN M$. Then 
$$M/\LL\simeq \ov{M}/\NN\LL\simeq \ov{M}|_D,$$ 
where $D$ is the effective divisor on $\ov{C}$ corresponding to the embedding of line bundles $\NN \LL\to \ov{M}$.
Let $j:U\to C$ be the embedding of the complement to $D$. Then we have $\LL|_U=M|_U$, so we can view $M$ as a subsheaf in $j_*(\LL|_U)$.
Hence, $M/\LL$ is a subsheaf in $j_*(\LL|_U)/\LL$.
Let us choose an effective Cartier divisor $\wt{D}\sub C$ reducing to $D\sub \ov{C}$. Then the subsheaf in $j_*(\LL|_U)/\LL$ annihilated by the local
equations of $\wt{D}$ is exactly $\LL(\wt{D})/\LL$. Since $M/\LL\simeq \ov{M}|_D$, we deduce that $M/\LL$ is contained in $\LL(\wt{D})/\LL$.
In other words, we get an inclusion
$$M\sub \LL(\wt{D}).$$
Consider the induced embeddings
$$\NN \LL=\ov{\LL}\sub \ov{M}\sub \ov{\LL(\wt{D})}=\ov{L}(D).$$
%Applying (i) to $M^\vee$ we get and embedding $M^\vee\to \LL^\vee$ such that $\ov{M^{\vee}}=t\LL^\vee$. Then the dual map $\LL\to M$ is the required embedding.
We see that $\ov{M}=\ov{\LL(\wt{D})}=\NN\cdot \LL(\wt{D})$, so the inclusion $M\to \LL(\wt{D})$ satisfies the required property.

\noindent
(4)$\implies$ (5), 
Locally we get an embedding $M\hra \OO$ such that $\ov{M}=\NN=(\eps) \sub\OO$. We also know that $M\neq (\eps)$. Hence, $M/(\eps)\sub \ov{C}$ is a line bundle on $\ov{C}$.
Therefore, if $\ov{f}$ is a local generator of $M/(\eps)$, then $M=(\eps,f)$ for any lifting $f$ of $\ov{f}$.

\noindent
(5)$\implies$ (1). This is clear.
\end{proof}

As we have seen in the above proof, 
if $M$ is a quasi line bundle then we have an embedding of line bundles on $\ov{C}$,
$$\NN M\hra \ov{M}.$$
which gives an effective divisor $D$ on $\ov{C}$, such that $\NN M=\ov{M}(-D)$.

\begin{definition}\label{O[D]-def}
For an effective divisor $D$ on $\ov{C}$, we define the sheaf of $\OO_C$-algebras
$$\OO_C[D]:=\OO_C+\NN(D)$$
(the sum is taken in $\eta_*\OO_{C,\eta}$, where $\eta$ is the general point).
We also define the ideal $\OO_C[-D]\sub \OO_C$ as 
$$\OO_C[-D]:=\ker(\OO_C\to \OO_D=\OO_{\ov{C}}/\OO_{\ov{C}}(-D)).$$
\end{definition}

Note that $\OO_C[D]$ and $\OO_C[-D]$ are both quasi line bundles with the associated divisor $D$.

\begin{lemma}\label{O[D]-mod-lem}
The category of quasi line bundle with an associated divisor $D$ is equivalent to the category of locally trivial $\OO_C[D]$-modules of rank $1$.
\end{lemma}

\begin{proof}
The characterization (5) from Proposition \ref{quasi-line-def-prop} shows that the multiplication by $\OO_C[D]$ is well defined on any quasi line bundle with
the associated divisor $D$.
Furthermore, a morphism of quasi line bundles with the associated divisor $D$ is automatically a morphism of $\OO_C[D]$-modules
(since for a quasi line bundle $M$ the natural map $M\to j_*j^*M$ is injective, where $j:C\setminus D\to C$ is the open embedding).
\end{proof}

For a coherent sheaf $M$ on $C$, we set $M^\vee:=\und{\Hom}(M,\OO)$.

\begin{lemma}\label{dual-quasi-line-lem}
(i) Let $M=(\eps,f)$, where $\eps$ is a local generator of $\NN$ and $f\not\equiv 0 \mod(\eps)$. Then there is an isomorphism $M^\vee\simeq (\eps,f)$. The corresponding pairing is given by
$$(\a \eps+\b f,a \eps+bf):=(\a b+\b a)\eps+ \b bf=(\a \eps+\b f)(a\eps+bf)/f.$$

\noindent 
(ii) If $M$ is a quasi line bundle with the associated divisor $D$, then the natural map 
$$\und{\Hom}_{\OO_C[D]}(M,\OO_C[-D])\to M^\vee$$ 
is an isomorphism. Hence, $M^\vee$ is also a quasi line bundle with the associated divisor $D$, and the natural map $M\to M^{\vee\vee}$ is an isomorphism.

\noindent
(iii) For quasi line bundles $L$ and $M$ with the same associated divisor $D$, the coherent sheaf  $L\ot_{\OO_C[D]} M$
%\simeq (L\ot_{\OO_C} M)/(\tors)$$ 
is also a quasi line bundle with the associated divisor $D$. 
%(where we take the quotient by the torsion subsheaf). 
%For quasi line bundles $L$ and $M$ with the same associated divisor $D$, the coherent sheaf $\und{\Hom}(L,M)$ is also a quasi line bundle with the associated divisor $D$.
%Furthermore, the action of $\OO_C[D]$ on $M$ induces an isomorphism $\OO_C[D]\rTo{\sim}\und{\End}(M,M)$.
%$M=(\eps,f)$ as in (i), the pairing $M\times M\to M: (m_1,m_2)\mapsto m_1m_2/f$ induces an isomorphism $M\rTo{\sim}\und{\Hom}(M,M)$.
\end{lemma}

\begin{proof} (i) This can be deduced by computing $M^\vee$ using the following presentation:
$$\OO^2\rTo{A} \OO^2\rTo{(\eps,f)} M\to 0,$$
where 
$$A=\left(\begin{matrix} \eps & -f \\ 0 & \eps\end{matrix}\right).$$
Indeed, we get that elements of $M^\vee$ correspond to $(u,v)$ such that $(u,v)A=0$, i.e.,
$$u \eps=0, \ \ u f=v \eps.$$
This easily implies that $u$ is determined by $v$ and that $v$ can be any element of $(\eps,f)$. Thus, we get an isomorphism $M^\vee\simeq (\eps,f)$.

\noindent
(ii) The assertion is local, so it follows from the explicit form of the pairing in (i).

\noindent
(iii) This follows immediately from Lemma \ref{O[D]-mod-lem}.
% since tensoring $L$ and $M$ over $\OO$ is the same as tensoring them over $\OO_C[D]$.
%This can be checked by identifying locally $M$ with $(\eps,f)$ and using (i).
\end{proof}

%\begin{lemma}\label{line-vs-quasi-line-lem}
%\noindent
%\end{proof}

\begin{lemma}\label{ex-seq-lemma}
(i) Any surjective homomorphism $V\to M$, where $V$ is a rank 2 vector bundle and $M$ is a quasi line bundle,
extends to an exact sequence
$$0\to L\rTo{i} V\rTo{p} M\to 0,$$
where $L$ is a quasi line bundle. Furthermore, any such sequence is locally of the form
\begin{equation}\label{model-ex-seq}
0\to (\eps,f)\rTo{i} \OO^2\rTo{p} (\eps,f)\to 0
\end{equation}
where $f\not\equiv 0 \mod (\eps)$, $p(a,b)=a\eps+bf$, $i(a'\eps+b'f)=(a'\eps+b'f,-b'\eps)$.

\noindent
(ii)  Let $\eps$ be a generator of $\NN$ near the divisor $D$ associated with $L$.
An exact sequence as in (i) induces an isomorphism
$$\de:\ov{M}/\NN M\rTo \ov{L}/\NN L$$
such that for a local section $s$ of $V$ such that $p(s)\in \ov{M}$, one has 
$$\de(p(s) \mod \eps M)=s'\mod \eps L,$$
where $s'$ is a section of $L$ such that $i(s')=\eps s$.
\end{lemma}

\begin{proof}
(i) If $M$ is a line bundle then this is clear.
Now assume that $M=(\eps,f)$, where $f$ is in the maximal ideal $\fm$ of some point then $M=(\eps,f)$ has $2$-dimensional fiber $M/\fm M$. Hence, any surjection $\OO^2\to M$ is locally
isomorphic to $(a,b)\mapsto a\eps+bf$. 

\noindent
(ii) The long exact sequence associated with the short exact sequence of complexes
$$0\to [L\rTo{\eps} L]\to [V\rTo{\eps} V]\to [M\rTo{\eps}M]\to 0$$
gives an exact sequence
$$\eps V\to \ov{M}\to L/\eps L\to V/\eps V.$$
It remains to observe that the image of $\eps V\to \ov{M}$ is $\eps M$, while the kernel of $L/\eps L\to V/\eps V$ is $\ov{L}/\eps L$.
\end{proof}

\begin{remark}
If $L\hra V$ is an embedding of a quasi line bundle into a rank $2$ vector bundle such that $\ov{L}$ is a subbundle in $\ov{V}$ on $\ov{C}$, it is not necessarily true
that $V/L$ is a quasi line bundle. For example, take $L=(\eps,f)$, where $f\not\equiv 0 \mod (\eps)$, and consider the embedding given as the composition 
$$L\to \OO\to \OO\oplus \OO.$$
\end{remark}

\begin{lemma}\label{det-lem}
Given an exact sequence
$$0\to L\to V\to M\to 0,$$
where $V$ is a rank $2$ bundle, $L$ and $M$ are quasi line bundles,
there is a natural pairing
$$L\ot M\to {\bigwedge}^2V=\det(V): l\ot m\mapsto l\wedge \wt{m},$$
where $\wt{m}$ is any lifting of $m$. This pairing induces isomorphisms
$$L\rTo{\sim} \und{\Hom}(M,\det(V)), \ \ M\rTo{\sim} \und{\Hom}(L,\det(V)).$$
Thus, for fixed $\det(V)$, isomorphism classes of $L$ and $M$ determine each other.
\end{lemma}

\begin{proof}
It is enough to check this statement for the sequence \eqref{model-ex-seq}.
The corresponding pairing is
\begin{align*}
&(\a \eps+\b f)\ot (a\eps+bf)\mapsto i(\a \eps+\b f)\we \wt{a\eps+bf}=((\a \eps+\b f)e_1-\b \eps e_2)\we (ae_1+be_2)=\\
&[b(\a \eps+\b f)+a\b \eps]\cdot e_1\we e_2
\end{align*}
which coincides with the pairing of Lemma \ref{dual-quasi-line-lem}(i).
\end{proof}

\subsection{Adelic interpretation of quasi line bundles}\label{adelic-quasi-lin-sec}

Let us fix an effective divisor $D=\sum_p n_p[p]$ on $\ov{C}$. 
We want to describe the set of isomorphism classes of quasi-line bundles on $C$ with the associated divisor $D$ (for $D=0$ these will be usual line bundles on $C$).

For each point $p$, let us consider the $\OO_p$-subalgebra 
$$\OO_p[D]:=\OO_p+t_p^{-n_p}\ov{\OO}_p\eps_p\sub F_p,$$
where $\eps_p$ is a generator of the nilradical of $\OO_p$.
Note that these are completions of the stalks of the sheaf of $\OO_C$-algebras $\OO_C[D]$ (see Definition \ref{O[D]-def}).
We denote by $\OO[D]\sub \A$ the corresponding subalgebra in the algebra of adeles.
We also have the corresponding subgroup of invertible elements in the group of ideles,
$\OO[D]^*\sub \A^*$.

\begin{lemma}
The groupoid of quasi-line bundles on $C$ with the underlying divisor $D$ is equivalent to the double cosets groupoid
$$F^*\backslash \A^*/\OO[D]^*.$$
\end{lemma}

\begin{proof} This follows from Lemma \ref{O[D]-mod-lem} similarly to the case of line bundles.
The equivalence is obtained by choosing local isomorphisms of a quasi line bundle with $\OO[D]$. 
More precisely, the quasi line bundle $L$ associated with $a=(a_p)\in \A^*$ and the divisor $D$ is defined as follows:
$$L(U):=\{f\in F \ |\ f\in \bigcap_{p\in U} a_p\cdot \OO_p[D]\}.$$
\end{proof}

%The natural functor
%$$F^*\backslash \A_C^*/\OO^*\to F^*\backslash\A_C^*/ \OO[D]^*$$
%corresponds to sending a line bundle $\LL$ to the quasi line bundle $\LL[D]$.

\begin{lemma}\label{quasiline-Hom-adelic-lem}
Let $L$ (resp., $M$) be the quasi line bundle associated with the idele $a_1$ (resp., $a_2$) and the divisor $D$.
Let $t_p\in \OO_p$ denote some local uniformizers, and let $n_p$ be the multiplicity of $p$ in $D$.
Then 
\begin{itemize}
\item
the quasi line bundle $L\ot_{\OO_C[D]} M$ is represented by $a_1a_2$;
\item
the quasi line bundle $\und{\Hom}(M,L)$ is represented by $a_1a_2^{-1}$;
\item
the quasi line bundle $L^\vee$ is represented by the idele $a_1^{-1}(t_p^{n_p})$.
\end{itemize}
%$M^\vee\ot L$ is the quasi line bundle associated with the idele $a_1a_2^{-1}(t_p^{n_p})$,
%where $t_p\in \OO_p$ is a lift of a uniformizer in $\ov{\OO}_p$, and $D=\sum_p n_pp$.
\end{lemma}

\begin{proof}
The first two items are clear if we think of quasi line bundles with the associated divisor $D$ as locally trivial $\OO[D]$-modules (see Lemma \ref{O[D]-mod-lem}).
Since $\OO[-D]$ is represented by the idele $(t_p^{n_p})$, 
the isomorphism $L^\vee\simeq \und{\Hom}_{\OO[D]}(L,\OO[-D])$ (see Lemma \ref{dual-quasi-line-lem}(ii)) shows that $L^\vee$ is represented by $a_1^{-1}(t_p^{n_p})$.
%Hence, $M^\vee \ot L$ is represented by $a_1a_2^{-1}(t_p^{n_p})$.
\end{proof}

Now we can give a geometric interpretation of the groupoid 
$$\QBun'_T(C):=\sqcup_D \QBun_T(C,D),$$
where $\QBun_T(C,D)=T(F)\backslash T(\A)/T(\OO)[D]$ (see Sec.\ \ref{Iwasawa-sec}).
Note that the groupoids $\QBun_T(C)=(T(F)U(\A))\backslash G(\A)/G(\OO)$ and $\QBun'_T(C)$ have the same isomorphism classes of objects but different automorphisms groups.

\begin{lemma}\label{Qbun-triples-lem}
One has an identification of $\QBun'_T(C)$ with the groupoid $\QTr$ of triples $(L,M,\LL)$, where $L$ and $M$ are  quasi line bundles, $\LL$ is a line bundle on $C$, 
equipped with an isomorphism 
$$\a:\ov{L}/\NN L\rTo{\sim} \ov{M}/\NN M,$$
and a pairing
$$\b:L\ot M\to \LL$$
inducing isomorphisms
$$L\rTo{\sim} \und{\Hom}(M,\LL), \ \ M\rTo{\sim} \und{\Hom}(L,\LL).$$
Under this equivalence, the subgroupoid $\QBun_T(C,D)$ corresponds to $(L,M,\LL)$ such that the divisor in $\ov{C}$ associated with $L$ (or $M$) is $D$.
\end{lemma}

\begin{proof} 
Recall that $T(\OO)[D]$ consists of $t=\diag(a_1,a_2)$, where $a_1,a_2\in \OO[D]^*$, $a_1a_2\in \OO^*$, $a_1-a_2\in \ov{\OO}(-D)+(\eps)$.
Let $L$ (resp., $M$) be the quasi line bundle associated with the idele $a_1(t_p^{n_p})$ (resp., $a_2$), and let $\LL$ be the line bundle associated with the idele $a_1a_2$.  
In other words, $L_p=a_{1,p}\OO_p[-D]$ and $M_p=a_{2,p}\OO_p[D]$. We consider the pairing $\b:L\ot M\to \LL$ induced by the product on $F$.
Since 
$$\ov{\OO[D]}=\ker(\eps:\OO[D]\to \OO[D])=\eps \ov{\OO}(D), \ \ \ov{\OO[-D]}=\eps \ov{\OO},$$
$$\eps \OO[D]=\eps \ov{\OO}, \ \ \eps\OO[-D]=\eps \ov{\OO}(-D),$$ 
we see that
the line bundles $\ov{L}$ and $\NN L$ (resp., $\ov{M}$ and $\NN M$)
are associated with the ideles $\ov{a}_1$ and $\ov{a}_1(t_p^{n_p})$ (resp., $\ov{a}_2(t_p^{-n_p})$ and $\ov{a}_2$). Thus, we have identifications
$$\a_1:\OO_D\rTo{\sim} \ov{L}/\NN L: 1\mapsto \ov{a}_1,$$
$$\a_2:\OO_D\rTo{\sim} \ov{M}/\NN M: 1\mapsto \ov{a}_2(t_p^{-n_p}).$$ 
We then set $\a=\a_2\circ\a_1^{-1}$. This gives an object of $\QTr$.

We have to check that this defines a morphism of groupoids. First, we claim that the constructed data $(L,M,\LL,\a,\b)$ do not change 
when we multiply $t$ with an element $\diag(b_1,b_2)$ of $T(\OO)[D]$. Indeed, it is clear that $L$, $M$, $\LL$ and $\b$ do not change since $b_1,b_2\in \OO[D]^*$ and $b_1b_2\in \OO^*$.
Also, $\a_2\circ \a_1^{-1}$ depends only on the image of $\ov{a}_2/\ov{a}_1$ in $\OO_D^*$, which gets multiplied by $\ov{b}_2/\ov{b}_1\equiv 1 \mod \ov{\OO}(-D)$.

Next, if we have $t_F=\diag(c_1,c_2)\in T(F)$, then we get an isomorphism $(L,M,\LL,\a,\b)\to (L',M',\LL',\a',\b')$, where $(L',M',\LL',\a',\b')$ is the object of $\QTr$ associated with $t_Ft$.
Namely, multiplication by $c_1$ (resp., $c_2$, resp., $c_1c_2$) gives an isomorphism $L\to L'$ (resp., $M\to M'$, resp., $\LL\to \LL'$), compatible with the extra data.
To see that our functor is fully faithful, we have to show that the above map identifies automorphisms
of every object $(L,M,\LL,\a,\b)$ of $\QTr$ with $T(F)\cap T(\OO)[D]$. Such an automorphism is given by a pair of principal ideles $c_1\in F^*$, $c_2\in F^*$ which give automorphisms
of $L$ and $M$. Then by compatibility with the pairing $\b$, $c_1c_2$ should give an automorphism of $\LL$. 
This means that $c_1,c_2\in \OO[D]^*$ and $c_1c_2\in \OO^*$. The condition of compatibility with $\a$ means that $\ov{c_1}/\ov{c_2}\equiv 1 \mod \ov{\OO}(-D)$.
In other words, we get that $\diag(c_1,c_2)\in T(F)\cap T(\OO)[D]$.

For essential surjectivity, starting from the data $(L,M,\LL,\a,\b)$ we first find ideles $a_2$ and $a$ representing $M$ and $\LL$. Then we set $a_1=aa_2^{-1}$.
Then by Lemma \ref{quasiline-Hom-adelic-lem}, 
$L\simeq \und{\Hom}(M,\LL)\simeq M^\vee\ot \LL$ will be represented by $a_1(t_p^{n_p})$, so that the pairing $\b$ is induced by the product on $F$.
The isomorphism $\a$ may differ from the isomorphism $\a'$ obtained from $(a_1,a_2)$ using our functor by some element in $\ov{b}\in \ov{\OO}^*_D$. Finally, we can
find an integer idele $b\in \OO^*$ lifting $\ov{b}$. Then replacing $a_1$ by $a_1b$ (which does not change $L$) we will get an element $(a_1,a_2)$ producing the given object of $\QTr$. 
\end{proof}

\subsection{Geometric interpretation of the constant term operator}

%Let $G=\GL_2$.
%Consider the Borel subgroup $B(\A_C)\sub G(\A_C)$. We have a subgroup $B(F)\sub B(\A_C)$. On the other hand, for an effective divisor $D$ as above,

Recall that we introduced the element
$$g_D:=\left(\begin{matrix} 1 & 0 \\ (t_p^{-n_p})\eps & 1\end{matrix}\right)\in G(\A).$$
%and set 
%$$B(\OO)[D]_p:=B(F_p)\cap g_{D,p}^{-1}G(\OO)g_{D,p}.$$
%Note that if $n_p=0$ then $B(\OO)[D]_p=B(\OO_p)$. Thus, we can 
Let us consider the subgroup
$$B(\OO)[D]:=B(\A)\cap g_DG(\OO)g_D^{-1}\sub B(\A).$$
Recall that the image of this subgroup under the projection $B(\A)\to T(\A)$ is the subgroup $T(\OO)[D]\sub T(\A)$ (see Sec. \ \ref{Iwasawa-sec}). 
%\prod_p B(\OO)[D]_p\sub B(\A_C).$$

With every $b=(b_p)\in B(F)\backslash B(\A_C)/B(\OO)[D]$, we can associate canonically an exact sequence
$$0\to L\to V\to M\to 0,$$
where $V$ is vector bundle of rank $2$ and $L$ and $M$ are quasi line bundles.
Namely, we have the following morphisms of exact sequences
\begin{diagram}
0&\rTo{}& \OO_p[-D]&\rTo{e_1-t_p^{-n_p}\eps\cdot e_2}&\OO_p^2&\rTo{(t_p^{-n_p}\eps, 1)}&\OO_p[D]&\rTo{}& 0\\
&&\dTo{}&&\dTo{g_{D,p}}&&\dTo{}\\
0&\rTo{}& F_p&\rTo{e_1}& F_p^2&\rTo{e_2^*}& F_p&\rTo{}& 0\\
&&\dTo{(a_1)_p}&&\dTo{b_p}&&\dTo{(a_2)_p}\\
0&\rTo{}& F_p&\rTo{e_1}& F_p^2&\rTo{e_2^*}& F_p&\rTo{}& 0
\end{diagram}
where $\OO_p[-D]=\OO_p(-D)+\eps\OO_p$, and $\diag(a_1,a_2)\in T(\A)$ is the diagonal part of $b$.
Now we define $V=V_b$ and $L=L_b\sub V_b$ by
\begin{equation}\label{b-V-L-constr}
V(U):=\{v\in F^2 \ |\ v\in \bigcap_{p\in U} b_p\cdot g_{D,p}\OO_p^2\}, \ \  L(U)=V(U)\cap F\cdot e_1=\{\phi\cdot e_1 \ |\ \phi\in F\cap \bigcap_{p\in U}a_{1,p}\OO_p[-D]\}.
\end{equation}
If we change $(b_p)$ by $(b_p)b'$, where $b'\in B(\OO)[D]$ then the pair $(V,L)$ does not change (since $b'g_D\OO^2=g_D\OO^2$).
On the other hand, if we replace $b$ by $b_0b$, where $b_0\in B(F)$, then the multiplication by $b_0$ induces an isomorphism between the corresponding pairs $(V,L)$.
Note that $L$ is identified with the quasi line bundle associated with the idele $a_1(t_p^{n_p})$, while $M:=V/L$ is identified with the quasi line bundle associated with $a_2$.
%we have a natural isomorphism
%$$L\simeq \LL\ot \OO[-D],$$
%where $\LL$ is the line bundle corresponding to the idele $(\a_p)$, such that the diagonal part of $b_p$ is $\diag(\a_p,\a_p^{-1})$.

Let us denote by $\QBun_B(C,D)$ the groupoid of pairs $L\sub V$, where $V$ is a rank 2 vector bundle on $C$
and $L\sub V$ is a quasi line bundle with the associated divisor $D$, such that $V/L$ is also a quasi line bundle.

We have a natural projection
$$p:\QBun_B(C,D)\to \Bun_G(C)$$
sending $L\sub V$ to $V$, and the functor 
$$q:\QBun_B(C,D)\to \QTr\simeq \QBun_T(C,D)$$
associating with a pair $L\sub V$, the triple $(L,M:=V/L,\LL:=\det(V))$ together with the induced isomorphism $\ov{M}/\eps M\rTo{\sim} \ov{L}/\eps L$ 
and the induced pairing $L\ot M\to \LL$ (see Lemma \ref{ex-seq-lemma}(ii) and \ref{det-lem}). 

\begin{prop}\label{QBun-mod-prop} 
The construction \eqref{b-V-L-constr} gives an equivalence of groupoids
$$B(F)\backslash B(\A)/B(\OO)[D]\simeq \QBun_B(C,D).$$ 
Under this equivalence the functors
%Furthermore, the functors 
$p$ and $q$  
correspond to the functors of double coset groupoids
$$p:B(F)\backslash B(\A)/B(\OO)[D]\to G(F)\backslash G(\A_C)/G(\OO): b\mapsto b\cdot g_D,$$
$$q:B(F)\backslash B(\A)/B(\OO)[D]\to T(F)\backslash T(\A)/T(\OO)[D]: b\mapsto b \mod U(\A).$$
\end{prop}

\begin{proof}
{\bf Step 1}. Let us check that our functor is fully faithful.
Given elements $b_1,b_2\in B(\A)$, isomorphisms from $b_1$ to $b_2$ in the groupoid of double cosets correspond to pairs $b_0\in B(F)$, $b_D\in B(\OO)[D]$, such that
$b_0b_1b_D=b_2$. On the other hand, an isomorphism $V_{b_1}\to V_{b_2}$ compatible with quasi line subsheaves $L_{b_i}$, is determined by an element $b_0\in B(F)$ such that
$$b_0b_{1,p}\cdot g_{D,p}\OO_p^2=b_{2,p}\cdot g_{D,p}\OO_p^2$$
for all $p$. In other words, we need 
$$b_D:=b_1^{-1}b_0^{-1}b_2\in B(\OO)[D].$$
So the morphisms in both groupoids are the same.

\medskip

\noindent
{\bf Step 2}. We claim that any embedding $L\to V$, where $V$ is a rank 2 vector bundle, $L$ a quasi line bundle with the associated divisor $D$, such that $V/L$ is a quasi line bundle, locally
can be identified with 
\begin{equation}\label{model-embedding}
\OO_p[-D]\rTo{e_1-t_p^{-n_p}\eps\cdot e_2}\OO_p^2.
\end{equation}
Indeed, this follows from Lemma \ref{ex-seq-lemma}(i).

\medskip

\noindent
{\bf Step 3}. To check that our functor is essentially surjective, let us start with a pair $L\sub V$ in $\QBun_B(C,D)$. At the generic point we choose
an identification $V_\eta\simeq F^2$, so that $L_\eta$ corresponds to $F e_1\sub F^2$. Now for each point $p$, we get an embedding $V_p\ot F_p\sub F_p^2$.

By Step 2, we can choose an identification $V_p\simeq \OO_p^2$, so that $L_p$ corresponds to the image of the embedding \eqref{model-embedding}.
Consider the induced embedding $\OO_p^2\simeq V_p\to F_p^2$. It is given by some matrix $g_p\in G(F_p)$ such that
$$g_p\cdot \left(\begin{matrix} 1 \\ -t_p^{-n_p}\eps\end{matrix}\right)=g_p\cdot g_{D,p}^{-1}e_1\in F_p e_1.$$
Equivalently, $g_p\in B(F_p)g_{D,p}$, i.e., $g_p=b_pg_{D,p}$, as in the construction \eqref{b-V-L-constr}.

\noindent
{\bf Step 4}. The identification of $p$ is clear. To identify $q$, we observe that the pair $(L,M=V/L)$ coming from the object of $\QBun_B(C,D)$ associated with $b\in B(\A)$,
is exactly the pair of quasi line bundles associated with the diagonal part of $b$, $t=\diag(a_1,a_2)$. The line bundle $\det(V)$ is associated with the idele $\det(bg_D)=a_1a_2$,
and it is easy to see that the pairing $\b:L\ot M\to \det(V)$ is induced by the product on $F$. Finally, we can calculate the homomorphism $\de:\ov{M}/\eps M\to \ov{L}/\eps L$. 
The local generator $\eps t_p^{-n_p}(a_2)_p$ of $\ov{M}_p$ can be lifted to $x=b_pg_De_1\in V_p$. Then $\eps x$ comes from $\eps (a_1)_pe_1\in L_p$. Hence, the isomorphism
$\de^{-1}$ agrees with the morphism $\a$ constructed in Lemma \ref{Qbun-triples-lem}.
\end{proof}

%The projection 
%$$B(F_p)\to F_p^*: \left(\begin{matrix} a & b \\ 0 & a^{-1}\end{matrix}\right) \mapsto a$$
%induces a surjective map 
%$$B(\OO)[D]_p\to \OO[D]^*.$$ 
%Hence we have a well defined groupoid functor

%\begin{cor} For constant term operator???
%\end{cor}

Now we can give a geometric interpretation of our constant term operators $C_D:\C(\Bun_G(C))\to \C(\QBun_T(C,D))$.

\begin{prop}\label{geom-const-term-prop}
For $t\in \QBun_T(C,D)$, one has 
$$C_Df(t)=\vol(U_{t,D})\cdot q_*p^*f(t).$$
Furthermore, if $t$ corresponds to the data $(L,M,\LL)$ as in Lemma \ref{Qbun-triples-lem},
% (and an isomorphism $\ov{L}/\eps L\simeq \ov{M}/\eps M$),
then
$$\vol(U_{t,D})=\frac{|H^1(C,\OO)|}{|H^0(C,\OO)|}\cdot \frac{|H^0(C,\und{\Hom}(M,L))|}{|H^1(C,\und{\Hom}(M,L))|}.$$
%q^{2g-2} where $g$ is the genus of $\ov{C}$, $q=|k|$.
\end{prop}

\begin{proof} By Proposition \ref{QBun-mod-prop}, we can use the adelic descriptions of $p$ and $q$.
The key observation is that applying Lemma \ref{double-coset-group-lem} with $B=B(\A)$, $B_1=B(F)$, $B_2=B(\OO[D])$, we can identify the fiber groupoid $q^{-1}(t)$, for $t\in T(F)\backslash T(\A)/T(\OO)[D]$,
with the double coset groupoid $U(F)\backslash U(\A)/U_{t,D}$, where
$$U_{t,D}=U(\A_C)\cap t B(\OO[D])t^{-1}$$
is exactly the subgroup introduced before (see Sec.\ \ref{Iwasawa-sec}).
Note that the automorphism group of any object in this groupoid is the finite group $U(F)\cap U_{t,D}$.

Hence, by Lemma \ref{groupoid-push-for-lem}, we have
$$q_*p^*f(t)=\frac{1}{|U(F)\cap U_{t,D}|}\cdot \sum_{u\in U(F)\backslash U(\A)/U_{t,D}}f(utg_D).$$
Comparing this with the formula \eqref{ED-sum-def-eq} for $C_Df$, we obtain
%$$E_Df(t)=\vol(U_{t,D})\cdot \sum_{U(\A_C)/(U(F)\cdot U_{t,D})}f(utg_D).$$
$$C_Df(t)=\vol(U_{t,D})\cdot q_*p^*f(t).$$ 

It remains to compute $\vol(U_{t,D})$ in terms of the quasi line bundles $L$ and $M$ associated with $t\in \QBun_T(C,D)$.
Let $t=\diag(a_1,a_2)$. Then under the natural isomorphism $U(\A_C)\simeq \A_C$, the subgroup $U_{t,D}$
corresponds to $a_1a_2^{-1}(\OO(-D)+\eps\OO)=a_1a_2^{-1}(t_p^{n_p})\OO[D]$. 
Recall that $L$ is associated with $a_1(t_p^{n_p})$, while $M$ is associated with $a_2$.
Hence, by Lemma \ref{quasiline-Hom-adelic-lem}, the quasi line bundle $\und{\Hom}(M,L)$ is represented by
the idele $a_1a_2^{-1}(t_p^{n_p})$ (and the divisor $D$).
%where $L$ (resp., $M$) is associated with $a_1$ (resp., $a_2$) and $D$. 
Hence, we get
$$U(\A_C)/(U(F)\cdot U_{t,D})\simeq H^1(\und{\Hom}(M,L)),$$
$$U(F)\cap U_{t,D}\simeq H^0(\und{\Hom}(M,L)).$$

Now the exact sequence
$$1\to U_{t,D}/(U(F)\cap U_{t,D})\to U(\A)/U(F)\to U(\A)/(U(F)\cdot U_{t,D})\to 1$$
gives the identity
$$\vol(U(\A)/U(F))=\frac{\vol(U_{t,D})}{|U(F)\cap U_{t,D}|}\cdot |U(\A)/(U(F)\cdot U_{t,D})|.$$
Taking into account the above identifications, we get
$$\vol(U_{t,D})=\vol(U(\A)/U(F))\cdot \frac{|H^0(C,\und{\Hom}(M,L))|}{|H^1(C,\und{\Hom}(M,L))|}.$$
Finally, for $t=1$ and $D=0$ this identity shows that
$$\vol(U(\A)/U(F))=\frac{|H^1(C,\OO)|}{|H^0(C,\OO)|}.$$
%Since $|H^0(C,\OO)|=q^2$ and $|H^1(C,\OO)|=q^{2g}$, we get the claimed formula for $\vol(U_{t,D})$.
\end{proof} 
%Recall that the sets of double cosets can be naturally thought of as groupoid.

%\begin{prop} We have the equality 
%$$Ef(t)=c(t)\cdot q_*f(t)$$
%between the operators
% Next, we will prove that the value at $t$ of the groupoid push-forward with respect to the map 
%$$\SS(B(F)\backslash B(\A_C)/B(\OO)[D])\to \C(\A_C^*/(F^*\cdot \OO[D]^*)),$$
%where $c(t)=???$
%\end{prop}

%\begin{lemma} There is a natural bijection (depending on $t\in T(\A_C)$)
%$$U(\A_C)/(U(F)U_{t,D})\rTo\pi^{-1}((e_p)),$$
%where 
%$$\pi:\Ext^1(\LL^{-1}[D],\LL[-D])\to \prod_p \Ext^1_{\OO_p}(\OO_p[D_p],\OO_p[-D_p])$$
%is the natural projection and $e_p$ is the class of the extension ???
%\end{lemma}
%
%\begin{proof}
%???
%\end{proof}

\end{document}